\documentclass[a4paper, oneside]{amsbook}

\usepackage{etoolbox}

\makeatletter
\patchcmd\@part{\protect\enspace\protect\noindent}{\hspace{.5em}}{}{\ERROR}
\makeatother

\usepackage[original]{imakeidx}
\makeindex

\usepackage[pdftex, colorlinks, citecolor=black, linkcolor=black,
urlcolor=black, bookmarks=false, backref=page]{hyperref}

\usepackage{geometry}

\usepackage{graphicx} 
\usepackage{afterpage}

\usepackage{mathrsfs}
\usepackage{amssymb}
\usepackage{stmaryrd}
\usepackage{mathtools}

\usepackage{array}

\usepackage{hyphenat}
\renewcommand*{\backref}[1]{}
\renewcommand*{\backrefalt}[4]{%
  \ifcase #1 (Not cited.)%
  \or        (Cited on page~#2.)%
  \else      (Cited on pages~#2.)%
\fi}

\usepackage{bookmark}

\usepackage[table,cmyk]{xcolor}
\usepackage[mark={\rule{4em}{0.3pt}}, ]{sectionbreak}
\newcommand{\sectionbreakafterproof}{%
  \vspace{-\topsep-\partopsep}%
  \sectionbreak
}

\usepackage{multirow}
\usepackage{booktabs}

\makeatletter
\@removefromreset{figure}{chapter}
\@removefromreset{table}{chapter}
\renewcommand{\thefigure}{\@arabic\c@figure}
\renewcommand{\thetable}{\@arabic\c@table}
\makeatother

\usepackage{enumitem}

\makeatletter
\newcommand*{\centerfloat}{%
  \parindent\z@%
  \leftskip\z@\@plus 1fil\@minus\textwidth%
  \rightskip\leftskip%
\parfillskip\z@skip}
\makeatother

\newtheorem{theorem}{Theorem}[chapter]
\newtheorem{lemma}[theorem]{Lemma}

\newtheorem{corollary}[theorem]{Corollary}

\makeatletter
\renewenvironment{proof}[1][\proofname]{\par
  \pushQED{\qed}%
  \normalfont \topsep6\p@\@plus6\p@\relax
  \trivlist
  \itemindent\z@ 
\item[\hskip\labelsep
    \itshape
  #1\@addpunct{.}]\ignorespaces
}{%
  \popQED\endtrivlist\@endpefalse
}
\makeatother

\newtheoremstyle{example}%
{4pt}
{5pt}
{}
{}
{\itshape}
{.}
{.5em}
{\thmname{#1}\textup{\thmnumber{ #2}}}
\theoremstyle{example}
\newtheorem{exmp}[theorem]{Example}

\makeatletter
\let\sv@thm\@thm
\def\@thm{\let\indent\relax\sv@thm}
\makeatother

\newcommand{\mi}[1]{\mathit{#1}}

\newcommand{\defi}[1]{\textit{#1}}

\makeatletter
\newcommand*\smallbullet{\mathpalette\smallbullet@{0.7}}
\newcommand*\smallbullet@[2]{\mathbin{\vcenter{\hbox{\scalebox{#2}{$\m@th#1\bullet$}}}}}
\makeatother

\newcommand{\smallcirc}{\mathbin{\scalebox{.85}{$\circ$}}}

\newenvironment{optprob}
{
  \arraycolsep=0pt
  \begin{array}{r@{\ }l@{\quad}l}
  }%
  {
  \end{array}
}

\newcommand{\onerow}[1]{\multicolumn{2}{l}{#1}}

\DeclareMathOperator{\Aff}{Aff}
\DeclareMathOperator{\cl}{cl}

\DeclareMathOperator{\cone}{cone}
\DeclareMathOperator{\ccone}{\overline{\mathrm{cone}}}
\DeclareMathOperator{\conv}{conv}
\DeclareMathOperator{\cch}{\overline{\mathrm{conv}}}
\DeclareMathOperator{\lspan}{span}

\DeclareMathOperator{\algint}{\mathrm{algint}}

\newcommand{\Sub}[2]{\mathrm{Sub}_{#2}(#1)}
\newcommand{\ind}{\mathcal{I}}
\newcommand{\inds}[1]{\mathcal{I}_{#1}}
\newcommand{\clique}{\mathcal{K}}
\newcommand{\ud}{\overline{\delta}}

\newcommand{\Orb}{\mathrm{Orb}}
\newcommand{\Stab}{\mathrm{Stab}}
\newcommand{\ortho}{\mathrm{O}}
\newcommand{\torus}{\mathbb{T}}

\DeclareMathOperator{\dens}{dens}
\newcommand{\Id}{I}
\newcommand{\1}{\mathbf{1}}
\newcommand{\tr}{{\sf T}}
\newcommand{\Tr}{\mathrm{Tr}}
\DeclareMathOperator{\Aut}{Aut}
\DeclareMathOperator{\Diag}{Diag}

\DeclareMathOperator{\supp}{supp}

\newcommand{\B}{\mathcal{B}}
\newcommand{\K}{\mathcal{K}}
\renewcommand{\L}{\mathcal{L}}
\newcommand{\D}{\mathcal{D}}

\newcommand{\sym}{\mathrm{sym}}
\newcommand{\inv}{\mathrm{inv}}
\newcommand{\op}{\mathrm{op}}

\newcommand{\fin}{\mathrm{fin}}

\newcommand{\setof}[1]{\llbracket{#1}\rrbracket}
\newcommand{\sub}{\mathrm{sub}}
\newcommand{\lhs}{\mathrm{lhs}}
\newcommand{\smallpmatrix}[1]{\left(
    \begin{smallmatrix}#1
\end{smallmatrix}\right)}  
\DeclareMathOperator{\rank}{rank}

\newcommand{\cont}{\mathrm{c}}
\newcommand{\m}{\mathrm{m}}
\newcommand{\bt}{\mathrm{big}}
\newcommand{\st}{\mathrm{small}}
\newcommand{\SP}{\mathrm{SP}}

\newcommand{\floor}[1]{\left\lfloor #1\right\rfloor}

\newcommand{\MS}[1]{{\rm MS}(#1)}
\newcommand{\spindle}[2]{{\rm S}(#1, #2)}
\newcommand{\comp}[1]{{\rm K}_{#1}}

\newlength\claimlen%
\newcommand{\assert}[1]{%
  \begin{minipage}{\claimlen}
    #1
  \end{minipage}%
}

\newcommand{\R}{\mathbb{R}}
\newcommand{\N}{\mathbb{N}}
\newcommand{\Z}{\mathbb{Z}}

\newcommand{\cor}{\mathbin{\scalebox{.85}{$\star$}}}
\newcommand{\Av}{\mathrm{Av}}
\newcommand{\T}{\mathcal{T}}
\newcommand{\Res}{\mathrm{Res}}
\newcommand{\convol}{\mathrm{K}}
\newcommand{\symdif}{\mathbin{\scalebox{.9}{$\triangle$}}}

\newcommand{\Sym}{\mathrm{Sym}}

\newcommand{\PSD}{\mathrm{PSD}}
\newcommand{\swrtz}{\mathcal{S}}

\newcommand{\cC}{\mathcal{C}}
\newcommand{\CP}{\mathrm{CP}}
\newcommand{\COP}{\mathrm{COP}}
\newcommand{\bC}{\breve{C}}

\newcommand{\Q}{\mathcal{Q}}
\newcommand{\BQP}{\mathrm{BQP}}
\newcommand{\PT}{\mathrm{P}}

\newcommand{\lass}{\mathrm{M}}
\newcommand{\kpb}{\mathrm{blockM}}

\newcommand{\A}{\mathscr{A}}

\renewcommand{\P}{\mathcal{P}}

\makeatletter
\newcommand{\firstname}[1]{\gdef\@firstname{#1}}
\newcommand{\@firstname}{\@latex@warning@no@line{No \noexpand\firstname given}}%
\newcommand{\lastname}[1]{\gdef\@lastname{#1}}
\newcommand{\@lastname}{\@latex@warning@no@line{No \noexpand\lastname given}}%
\makeatother

\title{Optimization hierarchies for extremal geometry through complete positivity}

\firstname{Abraham Johannes Franciscus}
\lastname{Bekker}

\date{03-06-2026}

\begin{document}

\frontmatter

\begin{titlepage}

  \begin{center}

    \vspace*{2\bigskipamount}

    {\makeatletter
      \bfseries\LARGE\@title
    \makeatother}

    {\makeatletter
      \ifx\@subtitle\undefined\else
      \bigskip
      \Large\@subtitle
      \fi
    \makeatother}

  \end{center}

  \newpage
  \thispagestyle{empty}
  \mbox{}
  \newpage

  \thispagestyle{empty}

  \begin{center}


    \vspace*{2\bigskipamount}

    {\makeatletter
      \bfseries\LARGE\@title
    \makeatother}

    {\makeatletter
      \ifx\@subtitle\undefined\else
      \bigskip
      \Large\@subtitle
      \fi
    \makeatother}

    \vfill

    {\bfseries\LARGE Dissertation}

    \bigskip
    \bigskip

    for the purpose of obtaining the degree of doctor

    at Delft University of Technology

    by the authority of the Rector Magnificus Prof. dr. ir. H. Bijl;

    Chair of the Board for Doctorates

    to be defended publicly on
    
    Thursday, 22 October 2026 at 15:00

    \bigskip
    \bigskip

    by

    \bigskip
    \bigskip

    \makeatletter
    {\LARGE \@firstname\ \textsc{\@lastname}}
    \makeatother

    \vspace*{3\bigskipamount}

  \end{center}

  \clearpage
  \thispagestyle{empty}

  \noindent  This dissertation has been approved by the promotors and the external advisor.

  \bigskip
  \noindent Composition of the doctoral committee:

  \medskip\noindent
  \begin{tabular}{p{5cm}p{6.2cm}}
    Rector Magnificus & chairperson \\
    Dr.\ F.M.\ de Oliveira Filho  & Delft University of Technology, promotor \\
    Prof.\ dr.\ D.C.\ Gijswijt & Delft University of Technology, promotor \\
    Dr.\ P.\ Moustrou & University of Toulouse -- Jean Jaur\`{e}s, France, external advisor \\
    
    \noalign{\medskip}
    \emph{Independent members:} & \\

    Prof.\ dr.\ ir.\ M.C.\ Veraar & Delft University of Technology\\
    Prof.\ dr.\ F.\ Vallentin & University of Cologne, Germany\\
    Prof.\ dr.\ E.\ de Klerk & Tilburg University\\
    Prof.\ dr.\ A.\ Wiegele & University of Klagenfurt, Austria\\
    Prof.\ dr.\ J.M.A.M.\ van Neerven & Delft University of Technology, \emph{reserve member}
  \end{tabular}

  \medskip
  \noindent The research in this thesis is funded by the grant \textsc{ocenw.klein.024} of the Dutch Research Council (\textsc{nwo}).

  \vfill
  \begin{center}
      \includegraphics[height=80pt]{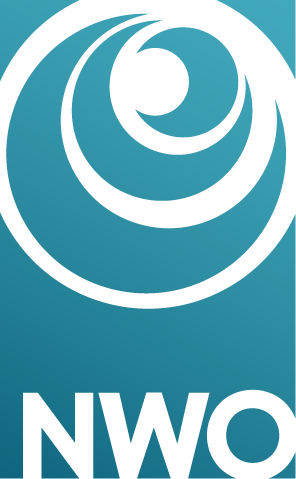}
      \hspace{2em}
      \includegraphics[height=60pt]{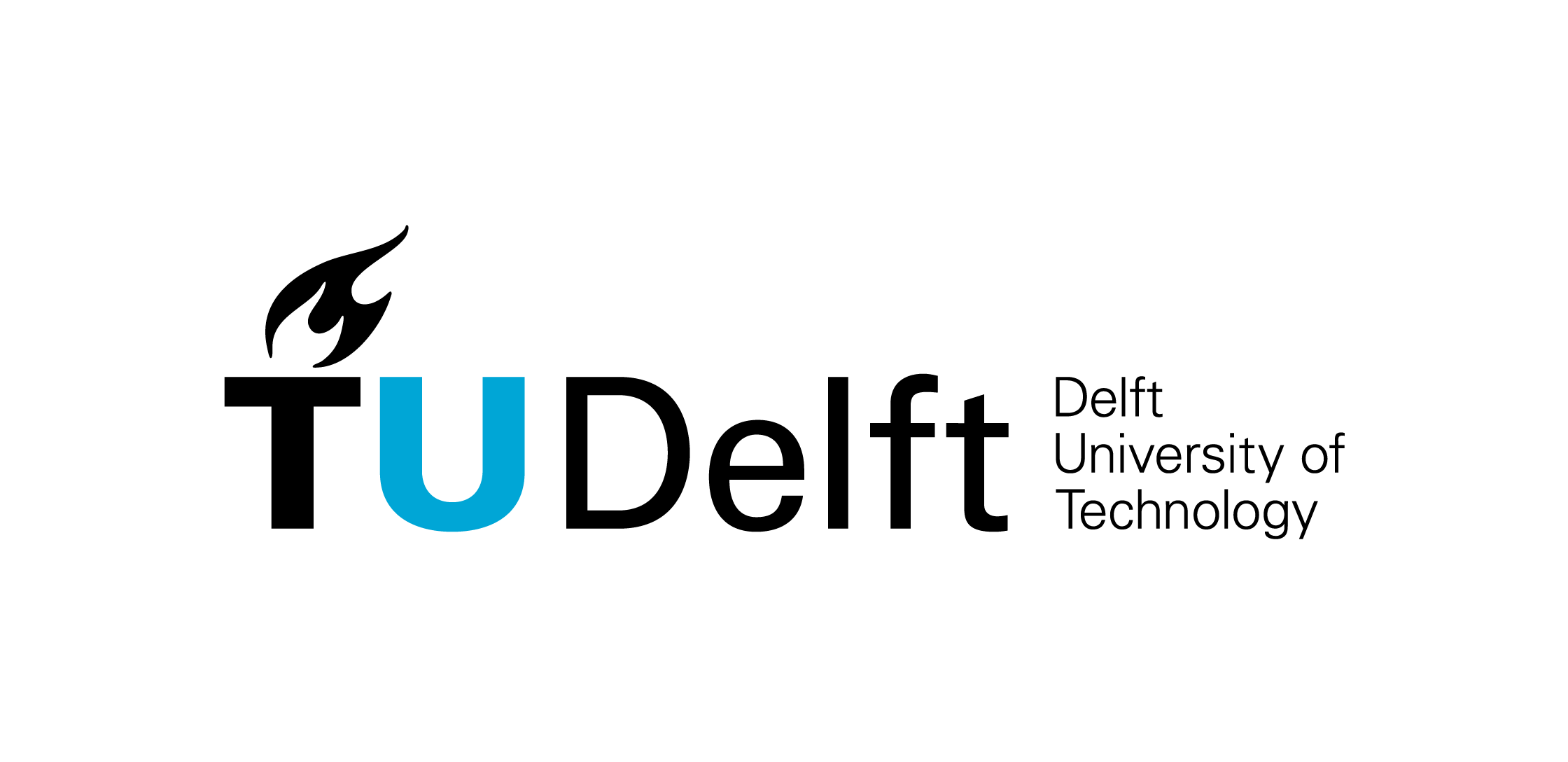}
  \end{center}
  \vfill

  \noindent Copyright \textcopyright\ 2026 by A.J.F.~Bekker

  \medskip



  \medskip

  \medskip

\end{titlepage}


\setcounter{page}{5}

\tableofcontents

\chapter*{Summary}

Completely positive functions are an extension of completely positive matrices. They are known to characterize maximal spherical codes and maximum-density distance-avoiding subsets of~$\R^n$ and certain compact metric spaces. This thesis expands this framework to related classes of problems in finite measure spaces and to the sphere-packing problem. For the latter, this is sharpened to show that the optimal sphere-packing density can be approximated using Schwartz functions.

Converging hierarchies of semidefinite programming bounds on the size of optimal spherical codes are known, based on approximations of completely positive functions and the Lovász theta number of a graph. This thesis extends these hierarchies to distance-avoiding sets and similar problems and to the sphere-packing problem, and proves their convergence to the maximum density. For distance-avoiding sets, additional hierarchies, such as the moment hierarchy, are introduced and shown to be stronger than the completely positive hierarchy, hence they also converge. These bounds are investigated for Witsenhausen's problem, which asks for the maximum fraction~$\alpha_n$ of the $n$-dimensional unit sphere that is coverable by a set avoiding orthogonal pairs, obtaining the best known upper bounds on~$\alpha_n$ in low dimensions.

The comparison of hierarchies for distance-avoiding sets moreover inspires a proof that the $k$-point bound for compact packing problems due to De Laat, Machado, Oliveira, and Vallentin is stronger than a converging completely positive hierarchy by Kuryatnikova and Vera. This proves convergence of the $k$-point bound. A related three-point bound is introduced for the $t$-almost-equiangular-set problem: finding the maximum size~$\alpha(n,t)$ of a subset of the $n$-dimensional unit sphere in which every triple contains a pair with inner product~$t \in [-1,1)$. An analytic solution to this bound yields an enumeration of optimal constructions for~$n = 2$ and~$3$ when~$t \geq 0$.

\chapter*{Samenvatting}
Volledig positieve functies zijn een uitbreiding van volledig positieve matrices. Het is bekend dat ze maximale sferische codes en afstandsvermijdende deelverzamelingen van $\R^n$ en bepaalde compacte metrische ruimten met maximale dichtheid karakteriseren. Dit proefschrift breidt dit raamwerk uit naar verwante klassen van problemen in eindige maatruimten en naar het bolstapelingsprobleem. Voor het laatste wordt dit verscherpt om aan te tonen dat de optimale bolstapelingsdichtheid benaderd kan worden met behulp van Schwartz-functies.

Convergerende hiërarchieën van bovengrenzen op de grootte van optimale sferische codes via semidefinietprogrammeren zijn bekend, gebaseerd op benaderingen van volledig positieve functies en het Lovász-thetagetal van een graaf. Dit proefschrift breidt deze hiërarchieën uit naar afstandsvermijdende verzamelingen en vergelijkbare problemen en naar het bolstapelingsprobleem, en bewijst hun convergentie naar de maximale dichtheid. Voor afstandsvermijdende verzamelingen worden aanvullende hiërarchieën, zoals de momentenhiërarchie, geïntroduceerd en aangetoond sterker te zijn dan de volledig positieve hiërarchie, en dus convergeren zij ook. Deze grenzen worden onderzocht voor het probleem van Witsenhausen, dat vraagt naar de maximale fractie $\alpha_n$ van de $n$-dimensionale eenheidssfeer die overdekt kan worden door een verzameling die orthogonale paren vermijdt. Dit leidt tot de best bekende bovengrenzen op $\alpha_n$ in lage dimensies.

De vergelijking van hiërarchieën voor afstandsvermijdende verzamelingen inspireert bovendien een bewijs dat de $k$-puntsgrens voor compacte stapelingsproblemen van De Laat, Machado, Oliveira en Vallentin sterker is dan een convergerende volledig positieve hiërarchie van Kuryatnikova en Vera. Dit bewijst de convergentie van de $k$-puntsgrens. Een verwante driepuntsgrens wordt geïntroduceerd voor het $t$-bijna-gelijkhoekige-verzamelingprobleem: het vinden van de maximale grootte $\alpha(n,t)$ van een deelverzameling van de $n$-dimensionale eenheidssfeer waarin elk drietal een paar bevat met inwendig product $t \in [-1,1)$. Een analytische oplossing van deze grens maakt een opsomming van de optimale constructies voor $n = 2$ en $3$ wanneer $t \geq 0$ mogelijk.


\mainmatter

\chapter{Introduction}%
\label{ch:introduction}
The interest in completely positive programming originates from the ubiquity of problems of the form
\begin{equation}%
  \label{eqn:SQO}
  \min_{x \in \Delta_S} x^{\tr} Q x.
\end{equation}
Here,~$S$ is a finite set,~$Q\in \R^{S \times S}$ is a real symmetric matrix, and~$\Delta_S$ is the standard simplex in~$\R^S$: the set of all nonnegative vectors whose coefficients sum to~$1$. By writing~$x^{\tr} Q x = \Tr(Q xx^{\tr})$, where~$\Tr$ denotes the matrix trace, and noting that~$x \in \Delta_S$ if and only if~$x \geq 0$ and~$\sum_{i, j} x_i x_j = 1$, it is apparent that~\eqref{eqn:SQO} is equivalent to an optimization problem over the intersection of the affine space~$\{ A \in \R^{S \times S} : \sum_{i, j} A_{i, j} = 1\, \}$ and the extreme rays of the completely positive cone, which is the convex cone generated by matrices of the form~$xx^{\tr}$, with~$x$ a nonnegative vector. The study of problems of this kind was initiated in the works by Bomze~\cite{Bomze1998OnProblems} and Bomze, Dür, De Klerk, Roos, Quist, and Terlaky~\cite{Bomze2000OnCopProblems}.

Calculating the independence number~$\alpha(G)$ of a finite graph~$G$ is an example of an NP-hard problem~\cite{Karp1972ReducibilityProblems} that can be expressed in the form~\eqref{eqn:SQO}. The independence number is the maximum cardinality a set of vertices can have without containing edges. Denoting the vertex set of~$G$ by~$V$, the adjacency matrix of~$G$ by~$A$, and the identity matrix by~$I$, Motzkin and Straus~\cite{Motzkin1965MaximaTuran} showed that
\begin{equation}%
  \label{eqn:cp-bound-motzkin-straus}
  \frac{1}{\alpha(G)} = \min_{x \in \Delta_{V}} x^{\tr}(A+I)x.
\end{equation}

This is not terribly interesting in itself: as computing the independence number is NP-hard, it just means that optimizing over the completely positive cone is difficult. However, in his thesis, Parrilo~\cite{Parrilo2000StructuredOptimization} introduced systems of linear matrix inequalities that approximate the copositive cone---the conic dual of the completely positive cone---in finitely many steps. De Klerk and Pasechnik~\cite{DeKlerk2002ApproximationProgramming,DeKlerk2007AProblem}, Bomze and De Klerk~\cite{Bomze2002SolvingProgramming}, and Peña, Vera, and Zuluaga~\cite{Pena2007ComputingProgramming} investigated this method closer, and used it to define linear and semidefinite optimization hierarchies that approximate the independence number of a finite graph, also in finitely many steps.

This thesis presents extensions of these ideas to four classes of questions from extremal geometry that are modelled as a kind of independence-number problem on an infinite hypergraph. The following are typical examples of a problem in each class. The dimension~$n$ is fixed.
\begin{enumerate}[label=\Roman*]
  \item \label{it:problem-1} What is the maximum fraction of the $(n-1)$-dimensional unit sphere that a set can cover without containing pairs of orthogonal vectors?
  \item \label{it:problem-2} How many unit balls can touch a central unit ball, if their interiors do not overlap?
  \item \label{it:problem-3} What is the maximum fraction of $n$-dimensional Euclidean space that a set can cover without containing pairs at distance~$1$?
  \item \label{it:problem-4} What is the maximum fraction of $n$-dimensional Euclidean space that a set of $(n-1)$-dimensional unit balls can cover, if their interiors do not overlap?
\end{enumerate}
Each of these has its own characteristics, but all can be formulated as an independence number of a graph~$G = (V, E)$, where the correct notion of ``size'' of an independent set is not necessarily its cardinality.

For example, for Problem~\ref{it:problem-1}, take~$V = S^{n-1}$ the unit sphere, and
\[
  E = \{\, (x,y) \in (S^{n-1})^2 : x^{\tr}y=0\, \}.
\]
In these terms, Problem~\ref{it:problem-1} asks for the largest fraction of the sphere that is covered by an independent set of~$G$. Here, ``fraction'' means the total surface measure of such a set, since such a set in general is not finite. Indeed, a spherical cap of nonzero angular radius less than~$\pi/4$ is a valid construction which contains uncountably many points. Hence, we take the uniform probability measure~$\mu$ on~$S^{n-1}$, with which the question can be formulated as: find
\[
  \sup \{\, \mu(I) : I\subseteq S^{n-1} \text{ measurable and independent}\, \}.
\]

This problem is called~\defi{Witsenhausen's problem}\index{problem!Witsenhausen's}, and was first posed in 1974 by Witsenhausen~\cite{Witsenhausen1974SphericalPairs}. It is conjectured that the optimal value is given by twice the measure of a spherical cap of angular radius~$\pi/4$~\cite[Conjecture 2.8]{Kalai2015SomeProblem}, which is~$(1/\sqrt{2} + o(1))^n$. This has only been confirmed for~$n=2$. Linear optimization bounds were introduced by Bachoc, Nebe, Oliveira, and Vallentin~\cite{Bachoc2009LowerNumbers} and Oliveira~\cite{Oliveira2009NewOptimization}. Parts~\ref{part:cones} and~\ref{part:measurable-setting} of this thesis give a theory of completely positive programming for this class of problems, which results in the best bounds known on Witsenhausen's problem in low dimensions. These results were partially published in~\cite{Bekker2026OptimizationSpaces}.

Problem~\ref{it:problem-2} is an example of a problem where the objective value is the cardinality of a set; it is called the~\defi{kissing-number problem}\index{problem!kissing number}. Here, again, we take~$V = S^{n-1}$, but now~$E = \{\, (x,y) \in (S^{n-1})^2 : x^{\tr}y \in (1/2, 1) \, \}$. The kissing number asks for the largest size~$|I|$, where~$I$ is an independent set of this graph. To see that this is the correct choice for the edge set, radially project a valid configuration of balls onto the central ball. Each ball touching the central ball corresponds to a spherical cap of angular radius~$\pi/3$ under this projection, explaining the maximally allowed inner product~$\cos(\pi/3) = 1/2$.

Delsarte, Goethals, and Seidel~\cite{Delsarte1977SphericalDesigns} introduced linear programming bounds to the kissing-number problem. They did this for a more general class of problems, where the inner product~$1/2$ is replaced by any inner product. After many improvements on this bound and adaptations to other problems, Bachoc and Vallentin~\cite{Bachoc2008NewProgramming} gave the first semidefinite programming bounds on the kissing number. More details on the history are presented in Chapter~\ref{ch:cop-prog-packing}.

A theory of copositive programming for problems of this type was initiated by Kuryatnikova and Vera~\cite{Kuryatnikova2017Approximating,Kuryatnikova2019TheProblems}. In Part~\ref{part:compact-packings} we complement their work by comparing it to other well-known optimization methods. We also explore a problem on a 3-uniform hypergraph that is similar to Problem~\ref{it:problem-2}. These results are based on the preprints~\cite{Bekker2023OnGraphs} and~\cite{Bachoc2025ObtuseSets}.

Problems~\ref{it:problem-3} and~\ref{it:problem-4} take place on a graph with vertex set~$V = \R^n$. These are examples of problems where even the Lebesgue measure of a typical set satisfying the requirements is infinite; instead, the objective value is a limit over local densities, which we make more precise in Part~\ref{part:Euclidean-space}. For now, we denote it by~$\ud(I)$ without definition.

Problem~\ref{it:problem-3} is known as the~\defi{1-avoiding-set problem}\index{problem!1-avoiding set}. The edge set associated to it is~$E = \{\, (x,y) \in (\R^n)^2 : \|x - y\| = 1\, \}$. The problem asks for the supremum~$m_1(\R^n)$ of~$\ud(I)$ ranging over all measurable independent sets of the graph~$(\R^n, E)$.  Recently, Ambrus, Csiszárik, Matolcsi, Varga, and Zsámboki~\cite{Ambrus2024TheDistances}, showed that~$m_1(\R^2) \leq 1/4$, settling a conjecture by Erd\H{o}s~\cite{ErdosProblemsGeometry}. Their method was based on convex programming bounds first introduced by Oliveira and Vallentin~\cite{Oliveira2010FourierRn}. Chapter~\ref{ch:distance-avoiding-sets-euclidean-space} contains more background on the problem.

Problem~\ref{it:problem-4} is called the~\defi{sphere packing problem}\index{problem!sphere packing}. It can be formulated as an independent set problem on the graph~$(\R^n, E)$, with~$E = \{\, (x,y) \in (\R^n)^2 : \|x - y\| \in (0, 2)\, \}$. Again, the objective is to maximize~$\ud(S)$ over some sets~$S$. However, the sets are not directly independent sets of the graph, but rather we are looking for
\[
  \sup\biggl\{\, \ud\biggl(\bigcup_{x \in I} B_1(x)\biggr) : I~\text{measurable and independent}\, \biggr\},
\]
where~$B_1(x)$ is the unit ball centered at~$x$. In other words, at every point of an independent set we attach a unit ball centered at that point; it is the density of such a set we are interested in.

The sphere packing problem is perhaps the most famous of these four problems. The most notable results in this topic are the linear programming bound by Cohn and Elkies~\cite{Cohn2003NewI} and the proof that the Cohn-Elkies bound is exact for~$n = 8$ by Viazovska~\cite{Viazovska2017The8}, and for~$n=24$ by and Cohn, Kumar, Radchenko, and Viazovska~\cite{Cohn2017The24}, including calculations of the optimal values. The introduction of Chapter~\ref{ch:sphere-packing} goes deeper into this history.

Chapter~\ref{ch:complete-positivity-under-symmetry} investigates the completely positive cone for problems like Problem~\ref{it:problem-3} and Problem~\ref{it:problem-4} on spaces similar to~$\R^n$. In Part~\ref{part:Euclidean-space} all the work in this thesis comes together to describe completely positive and copositive programming approaches to Problem~\ref{it:problem-3} and~\ref{it:problem-4}. These results are new and unpublished.

\sectionbreak

I have included extensive preliminaries. To keep the main text to the point, I have moved the most elementary of these to the appendix. Notions and theorems that cannot be found in the preliminary section of the relevant chapter can probably be found there.

\section{Some notation}
The natural numbers start at~$0$. For integer~$n \geq 1$,~$[n] = \{1, \ldots, n\}$. Given an index set~$I$, we will often abbreviate~$\{\, x_i : i \in I\, \}$ as~$\{x_i\}_{i \in I}$; when~$I$ is ordered, and we intend a collection of~$x_i$s to be ordered accordingly, we use the common notation~$(x_i)_{i \in I}$. If~$S$ is a set of real-valued vectors, and~$r \in \R$, then~$S_{\geq r}$ is the set of all vectors in~$S$ with all coordinates at least~$r$. In particular, for every~$r \in \R$,~$\N_{\geq r}$ is the set of natural numbers greater than or equal to~$r$.

Let~$k \in \N_{\geq 1}$, denote by~$\mathfrak{S}_k$\index{ Sk@$\mathfrak{S}_k$} the permutation group on~$k$ letters. Fix~$S$ a finite set. The space~$\Sym(S, k)$\index{ SymSk@$\Sym(S, k)$} is the vector space of~\defi{symmetric $k$-tensors}\index{k-tensor@$k$-tensor} on~$S$, that is, functions~$S^k \to \R$ such that for all elements~$s_1$,~$\ldots$,~$s_k \in S$ and permutations~$\pi \in \mathfrak{S}_k$:~$T(s_{\pi 1}, \ldots, s_{\pi k}) = T(s_1, \ldots s_k)$. The space~$\Sym(S, 2)$ is nothing more than the space of symmetric matrices indexed by~$S$, and we will often denote it by~$\Sym(S)$\index{ SymS@$\Sym(S)$}. If~$n \in \N$, then~$\Sym(n, k) = \Sym([n], k)$, and~$\Sym(n) = \Sym([n])$. We will use both the function notation~$T(s_1, \ldots, s_k)$ and the index notation~$T_{s_1, \ldots, s_k}$.

Denote the transpose of a vector~$x \in \R^S$ by~$x^{\tr}$. We always consider the inner product of~$x$ and~$y \in \R^S$ to be~$x^{\tr}y$, hence, unless indicated otherwise, the norm~$\|x\| = \sqrt{x^{\tr}x}$. The inner product on~$\Sym(S,k)$ is the usual Euclidean inner product, and is denoted~$\langle \cdot\, , \cdot\rangle$\index{ <>@$\langle \cdot\, , \cdot\rangle$}. Indicate the cone of positive semidefinite matrices by
\index{ SymS>0@$\Sym(S)_{\geq0}$}\[
  \Sym(S)_{\succeq 0} = \{\, M \in \Sym(S) : \langle M, xx^{\tr}\rangle \geq 0 \text{ for all } x \in \R^S\, \}.
\]
If~$M \in \Sym(S)_{\succeq 0}$, we write~$M \succeq 0$.

Let~$k \in \N_{\geq 1}$. A choice of~$k$ elements~$x_1$,~$\ldots$,~$x_k \in \R^S$ defines a $k$-tensor
\index{ 000@$\otimes$}\[
  (x_1 \otimes \cdots \otimes x_k)_{i_1, \ldots, i_k} =  (x_1)_{i_1} \cdots (x_k)_{i_k}.
\]
The vector~$\1 \in \R^S$ indicates the all-one vector---$\1_s = 1$ for all~$s \in S$---and~$J = \1\1^{\tr}$ is the all-one matrix. For a set~$I \subseteq S$, the vector~$\1_I$ is the indicator function of~$I$.

Given a set~$S$ and~$r \in \N$, the set~$\Sub{S}{r}$ is the family of all subsets of~$S$ of cardinality at most~$r$, including the empty set. Likewise,~$\Sub{S}{=r}$ is the family of subsets of cardinality exactly~$r$. Again, if~$n \in \N$, then~$\Sub{n}{r} = \Sub{[n]}{r}$ and~$\Sub{n}{=r} = \Sub{[n]}{=r}$. We will call a set of cardinality~$r$ an~\defi{$r$-set}\index{set@$r$-set}.

When we say that~$H = (V, E)$ is a hypergraph, we mean that~$V$ is its vertex set and~$E$ is its edge set, which is a collection of subsets of~$V$. For edges~$\{v_1, \ldots, v_k\} \in E$, we often omit the brackets, and write~$v_1 \cdots v_k \in E$. For an integer~$k \geq 2$, a hypergraph is~\defi{$k$-uniform}\index{hypergraph!k uniform hypergraph@$k$-uniform hypergraph} if~$E \subseteq \Sub{V}{=k}$. In this case, we will often interpret~$E$ as a symmetric subset of~$V^k$ by identifying it with the set of all $k$-tuples~$(v_1, \ldots, v_k)$ such that~$\{v_1, \ldots v_k\} \in E$. The~\defi{automorphism group}\index{automorphism group of a hypergraph} of~$H$ is the set of bijections~$\sigma: V \to V$ such that~$\sigma(v_1) \cdots \sigma(v_k) \in E$ if and only if~$v_1 \cdots v_k \in E$, and is denoted~$\Aut(H)$\index{ AutH@$\Aut(H)$}. We denote the action of an element~$\sigma \in \Aut(H)$ without brackets:~$\sigma v = \sigma(v)$ for all~$v \in V$. A hypergraph is called~\defi{vertex transitive} if the action of~$\Aut(V)$ is transitive on~$V$.

\section{The independence number of a finite graph}%
\label{sec:independence-number-finite-graph}
Let~$k \geq 2$ be an integer and~$H = (V, E)$ be a $k$-uniform hypergraph with vertex set~$V$ and edge set~$E$. An~\defi{independent set}\index{independent set} of~$H$ is a subset~$S \subseteq V$ such that no $k$-subset of~$S$ is an edge. The~\defi{independence number}\index{independence number}~$\alpha(H)$\index{ a H@$\alpha(H)$} of~$H$ is
\[
  \alpha(H) = \sup\{\, |I| : I \subseteq V \text{ independent}\, \}
\]
and is either attained or infinite. In this thesis we will discuss upper bounds on several extensions of the independence number to hypergraphs with a possibly infinite vertex set. The remainder of this chapter forms an overview of the finite-graph setting.

Fix a finite set~$V$ and a graph~$G = (V, E)$. A common starting point for optimization methods for the independence number is the~\defi{Lovász theta number}\index{Lovász theta number} of~$G$, introduced by Lovász~\cite{Lovasz1979OnGraph}. It is the semidefinite program
\index{ theta@$\vartheta(G)$}
\begin{equation}%
  \label{eqn:theta-number}
  \begin{optprob}
    \vartheta(G) = \sup &\onerow{\sum_{v, w \in V} X_{vw}}\\
    &\onerow{\sum_{v\in V} X_{vv} = 1,}\\
    &X_{vw} = 0 &\text{for all } vw \in E,\\
    &\onerow{X \in \Sym(V)_{\succeq 0}.}
  \end{optprob}
\end{equation}
We use the same name and symbol for an optimization problem and its optimal value; so, ``Lovász theta number'' refers to both the program displayed in~\eqref{eqn:theta-number} and the number it produces.

The Lovász theta number is an upper bound on the independence number of~$G$, that is~$\vartheta(G) \geq \alpha(G)$. Indeed, and take~$I \subseteq V$ an independent set. Then, the matrix~$\1_I\1_I^{\tr} / |I|$ is a feasible solution of~$\vartheta(G)$ with objective~$|I|$.

The Lovász theta number is appealing from a computational perspective, as the ellipsoid method offers a proof that it can be solved in polynomial time to any fixed precision~\cite{Groetschel1993GeometricOptimization}. In practice, interior point methods offer a polynomial-time implementation which is preferred. Moreover, the theta number is amenable to many well-known techniques, like restriction of its feasible region and sparsity arguments, to improve the bound and reduce the size of the program. Symmetries of the graph and of the constraints and objective can also often be exploited to the same end.

We will say that a bound on~$\alpha(G)$ is~\defi{sharp}\index{sharp bound} or~\defi{exact}\index{exact bound} if it is equal to~$\alpha(G)$. Rarely is~$\vartheta(G)$ a sharp bound on~$\alpha(G)$. However, there are many ways to add or modify constraints, which has produced many of the best upper bounds on~$\alpha(G)$.

One of the ways in which the theta number can be strengthened, is by replacing the positive-semidefinite cone by the~\defi{completely positive cone}\index{completely positive cone!in finite dimension}
\index{ CPV@$\CP(V)$!in finite dimension}\[
  \CP(V) = \biggl\{\, \sum_{i=1}^m x_ix_i^{\tr} : m \in \N \text{ and } x_i \in \R^V_{\geq 0}\text{ for }i \in [m]\, \biggr\}.
\]
Let for any convex cone~$\cC \subseteq \Sym(V)$,
\index{ theta@$\vartheta(G, \cC)$!finite graph}\[
  \begin{optprob}
    \vartheta(G, \cC) = \sup &\onerow{\sum_{v, w \in V} X_{vw}}\\
    &\onerow{\sum_{v \in V} X_{vv} = 1,}\\
    &X_{vw} = 0 &\text{for all } vw \in E,\\
    &\onerow{X \in \cC,}
  \end{optprob}
\]
so that~$\vartheta(G, \Sym(V)_{\succeq 0}) = \vartheta(G)$. Then,~$\vartheta(G, \CP(V))$ is a sharp bound on~$\alpha(G)$ which is closely related to the formulation~\eqref{eqn:cp-bound-motzkin-straus} of~$1/\alpha(G)$. That it is an upper bound on~$\alpha(G)$, follows from the same argument as for the theta number. It is moreover exact, since the optimal value is attained at an extreme point, and the extreme points of the feasible region are of the form~$xx^{\tr}$ with~$x \geq 0$,~$\|x\| = 1$, and the support of~$x$ is an independent set. Take an optimal solution~$xx^{\tr}$, with~$x \geq 0$ a vector with support~$I$. It has objective value~$(\1^{\tr}x)^2 \leq \|\1_I\|^2\|x\|^2 = |I|$ by the Cauchy-Schwarz inequality, hence the conclusion follows. We call~$\vartheta(G, \CP(V))$ a~\defi{completely positive formulation}\index{formulation!completely positive} of~$\alpha(G)$ for finite graphs~$G$. In Chapter~\ref{ch:completely-positive-formulations-meas-ind-num} we will see a proof of a much more general statement.

Instead of optimizing over~$\CP(V)$, upper bounds on the independence number are often obtained by optimizing over its conic dual
\index{ COPV@$\COP(V)$!in finite dimension}\[
  \COP(V) = \{\, M \in \Sym(V) : \langle M, xx^{\tr} \rangle \geq 0 \text{ for all } x \in \R^V_{\geq 0}\, \},
\]
which is called the~\defi{copositive cone}\index{copositive cone!in finite dimension}. The accompanying optimization problem is~$\vartheta^*(G, \COP(V))$, where for any convex cone~$\cC \subseteq \Sym(V)$,
\[
  \begin{optprob}
    \vartheta^*(G, \cC) = \inf &\onerow{t}\\
    &F_{vv} = t - 1 &\text{for all } v \in V,\\
    &F_{vw} = -1 &\text{for all $v \neq w$ and } vw \notin E,\\
    &F \in \cC.
\end{optprob}\]
If~$\cC$ is a closed convex cone, the program~$\vartheta^*(G, \cC)$ is the dual of~$\vartheta(G, \cC^*)$, and under mild conditions, strong duality holds; that is~$\vartheta(G, \cC) = \vartheta^*(G, \cC^*)$. In particular, we have~$\vartheta(G, \CP(V)) = \vartheta^*(G, \COP(V))$, so that the latter also is an exact bound on~$\alpha(G)$. We call~$\vartheta^*(G, \COP(V))$ the~\defi{copositive formulation}\index{formulation!copositive formulation} of~$\alpha(G)$ for finite graphs. Of course, that these programs return exactly the independence number must mean that they are hard to compute. In Section~\ref{sec:cop-cone-finite-dimension} of this introduction we go further into this.

Other well-studied ways to improve the Lovász theta number stem from the moment hierarchy, which was developed for general 0-1 programming by Lasserre~\cite{Lasserre2001AnPrograms,Lasserre2002AnPrograms} and described in more detail for the independence number by Laurent~\cite{Laurent2003AProgramming}. For~$r \in \N$, let~$M_r$ be the operator
\[
  M_r : \R^{\Sub{V}{2r}} \to \Sym(\Sub{V}{r}),\qquad (M_r \nu)_{S, T} = \nu_{S \cup T}.
\]
The~\defi{moment hierarchy}\index{moment hierarchy!finite graph}---also called the~\defi{Lasserre hierarchy}\index{Lasserre hierarchy}---is the sequence of programs
\index{ Mr@$\lass_r(H)$!finite graph}\[
  \begin{optprob}
    \lass_r(G) = \sup &\onerow{\sum_{v \in V} \nu_{\{v\}}}\\
    &\onerow{\nu_{\emptyset} = 1,}\\
    &\nu_S = 0 & \text{for all $S$ not independent,}\\
    &\onerow{\nu \in \R^{\Sub{V}{2r}}_{\geq 0},\ M_r\nu \in \Sym(\Sub{V}{r})_{\succeq 0}.}
\end{optprob}\]

For all~$r$,~$\lass_r(G)$ is an upper bound on~$\alpha(G)$, and the sequence is decreasing: for an integer~$r \geq 1$ and an independent set~$I$, define the vector~$\chi_I \in \R_{\geq 0}^{\Sub{V}{2r}}$, which is~$1$ on sets~$S \in \Sub{I}{r}$ and~$0$ otherwise. This defines a feasible solution to~$\lass_r(G)$. Moreover, extending a feasible solution~$\nu$ of~$\lass_r(G)$ by zeros gives a feasible solution of~$\lass_s(G)$ for all~$s \geq r$. In fact, Laurent~\cite{Laurent2003AProgramming} showed that
\[
  \lass_1(G) \geq \lass_2(G) \geq \cdots \geq \lass_{\alpha(G)}(G) = \alpha(G).
\]
We say that the hierarchy~\defi{converges}\index{converging hierarchy} to~$\alpha(G)$, since the sequence of numbers~$(\lass_1(G), \lass_2(G), \ldots)$ converges to~$\alpha(G)$.

The program~$\lass_1(G)$ is known as the~\defi{theta-prime number}\index{theta prime number@theta-prime number} and is equivalent to the theta number~\eqref{eqn:theta-number}, but with the additional constraint that the matrix~$X$ is nonnegative. One way to think of the higher levels~$\lass_r(G)$ is that they strengthen the theta number by including correlations between more than two points. We could thus call~$\lass_r(G)$ a~\defi{$(2r)$-point bound}\index{k point bound@$k$-point bound}.

The moment hierarchy is interesting because each level is a semidefinite program that can be solved in polynomial time, but on the other hand it converges to~$\alpha(G)$ in finitely many steps. However, the time it takes to compute~$\lass_r(G)$ goes up quickly with~$r$ for many graphs, even with efficient algorithms. Thus, already for small~$r$ the programs might not be tractable. For some graphs in this thesis, though, we can define a moment hierarchy for which low levels are tractable~\cite{deLaat2023TheAngle}.

That the moment hierarchy might be difficult to compute explains the interest in weaker versions of these programs, which might give worse bounds, but are easier to compute. Schrijver~\cite{Schrijver2005NewProgramming} described a three-point bound for a certain combinatorial problem, which Bachoc and Vallentin~\cite{Bachoc2008NewProgramming} then used as inspiration for a three-point bound for the kissing-number problem, which was a breakthrough result. Musin~\cite{Musin2014MultivariateSpheres} extended it to an $r$-point bound for the kissing-number problem for all~$r \geq 2$. Gvozdenović, Laurent, and Vallentin introduced an $r$-point bound for the independence number of finite graphs~\cite{Gvozdenovic2009BlockProgramming}, which De Laat, Machado, Oliveira, and Vallentin~\cite{deLaat2021K-PointLines} extended to an $r$-point bound for topological packing graphs, a class that includes all finite graphs.

In this thesis, we investigate the bound by De Laat, Machado, Oliveira, and Vallentin~\cite{deLaat2021K-PointLines}, defined as follows. For~$r \in \N_{\geq 2}$ and~$Q \in \Sub{V}{r-2}$, define
\[
  M_Q : \R^{\Sub{V}{r}} \to \Sym(\Sub{V}{1}),\qquad (M_Q\nu)_{S, T} = \nu_{Q \cup S \cup T}.
\]
The $r$th level of the~\defi{block moment hierarchy}\index{block moment hierarchy!finite graph} is
\index{ blockMr@$\kpb_r(H)$!finite graph}\[
  \begin{optprob}
    \kpb_r(G) = &\onerow{\sup \sum_{v \in V} \nu_{\{v\}}}\\
    &\onerow{\nu_{\emptyset} = 1,}\\
    &\nu_S = 0 &\text{for all $S$ not independent},\\
    &M_Q \nu \in \Sym(\Sub{V}{1})_{\succeq 0} &\text{for all } Q \in \Sub{V}{r-2},\\
    &\onerow{\nu \in \R^{\Sub{V}{r}}_{\geq 0}.}
\end{optprob}\]
In the literature, it is called ``the $k$-point bound'', but we will discuss many hierarchies that qualify for this name.

We may think of the matrix~$M_Q\nu$ as a principal submatrix of a matrix of the form~$M_r\mu$, where~$(M_Q\nu)_{S,T} = (M_r\mu)_{Q \cup S, Q \cup T}$. Hence, the operator~$M_Q$ selects blocks indexed by sets containing~$Q$. This shows that the block moment hierarchy is indeed weaker than the moment hierarchy, in the sense that there is a number~$n \in \N$ such that for all~$i$:~$\lass_{i+n}(G) \leq \kpb_i(G)$. We again have
\[
  \kpb_1(G) \geq \kpb_2(G) \geq \cdots \geq \alpha(G).
\]
We prove that~$\alpha(G) = \kpb_{\alpha(G)^2}(G)$ for all finite graphs~$G$ in Chapter~\ref{ch:Completely positive programming on compact spaces} and~\ref{ch:cop-prog-packing}.

\section{The copositive cone in finite dimensions}%
\label{sec:cop-cone-finite-dimension}
Let~$G = (V, E)$ be a finite graph. In the previous section, we saw the copositive formulation~$\vartheta^*(G, \COP(V))$ for the independence number of~$G$. Since computing~$\vartheta^*(G, \COP(V))$ is equivalent to computing~$\alpha(G)$, it cannot be easier. Parrilo~\cite{Parrilo2000StructuredOptimization}, and subsequently De Klerk and Pasechnik~\cite{DeKlerk2002ApproximationProgramming,DeKlerk2007AProblem}, and Peña, Vera, and Zuluaga~\cite{Pena2007ComputingProgramming}, introduced hierarchies of optimization problems based on the copositive formulation of~$\alpha(G)$. These problems are computationally much less expensive than~$\vartheta^*(G, \COP(V))$, and converge in finitely many steps to~$\alpha(G)$: the weakest of these hierarchies can be shown to converge to~$\alpha(G)$ in at most~$\alpha(G)^2$ steps, after rounding down to an integer~\cite[Theorem 4.1]{DeKlerk2002ApproximationProgramming}.

This section describes the background necessary to understand the weakest of these hierarchies better, which we will use to develop more general theory in Chapters~\ref{ch:completely-positive-cone} and~\ref{ch:complete-positivity-under-symmetry}. In particular, we define cones~$C_r(V)$ for finite~$V$ that approximate~$\COP(V)$ in the sense that~$\COP(V)$ is the closure of~$\bigcup_{r \in \N} C_r(V)$. This defines the hierarchy~$\vartheta^*(G, C_r(V))$, which converges to~$\vartheta^*(G, \COP(V))$---equal to~$\alpha(G)$---in finitely many steps. The theory will be presented more generally, and these concepts are defined for $k$-uniform hypergraphs, with~$k \in \N_{\geq 2}$. Convergence of this hierarchy is not proved here, as this thesis contains two generalizations of this statement, one in Chapter~\ref{ch:Completely positive programming on compact spaces} and another in Chapter~\ref{ch:cop-prog-packing}.

The copositive hierarchy departs from Pólya's theorem, a positivstellensatz for homogeneous polynomials. Let~$p \in \R[x]$ with~$x = (x_1, \ldots, x_n)$ be a homogeneous polynomial in~$n$ variables. If there is an~$r \in \N$ such that all coefficients of
\begin{equation}%
  \label{eqn:polya-ineq}
  (\1^{\tr}x)^rp(x) = (x_1 + \cdots + x_n)^r p(x),
\end{equation}
are nonnegative, then $p(x) \geq 0$ for all~$x \geq 0$. Pólya's theorem offers a partial converse: if~$p(x) > 0$ for all nonzero~$x \geq 0$, then there exists an~$r \in \N$ such that all coefficients of~\eqref{eqn:polya-ineq} are nonnegative~\cite{Polya1974UberPolynomen}.

Define for a finite set~$V$ and an integer~$k \geq 2$  the~\defi{completely positive cone}\index{completely positive cone!in finite dimensions}
\index{ CPVk@$\CP(V, k)$!in finite dimensions}\[
  \CP(V,k) = \biggl\{\, \sum_{i = 1}^m x_i^{\otimes k} : m \in \N \text{ and } x_i \in \R^V_{\geq 0} \text{ for all } i \in [m]\, \biggr\}
\] and the~\defi{copositive cone}\index{copositive cone!in finite dimensions}
\index{ COPVk@$\COP(V, k)$!in finite dimensions}\[
  \COP(V, k) = \{\, T \in \Sym(V,k): \langle T, x^{\otimes k} \rangle \geq 0 \text{ for all } x \in \R^V_{\geq 0}\, \},
\]
which is the conic dual of~$\CP(V,k)$.

To see the connection between Pólya's theorem and the copositive cone, let~$V$ be a finite set with~$n = |V|$, let~$k \geq 2$ be an integer, and let~$T$ be in~$\Sym(V,k)$. The function~$p: x \mapsto \langle T, x^{\otimes k}\rangle$ on~$\R^V$ defines a homogeneous polynomial of degree~$k$ in~$n$ variables over~$\R$. Conversely, for all integers~$n \geq 0$ and~$k \geq 2$, a homogeneous polynomial of degree~$k$ in~$n$ variables over~$\R$ is given in this way by a unique element of~$\Sym(n,k)$.

Let~$r \geq 0$ be an integer, and define for all~$v_1$,~$\ldots$,~$v_{k+r} \in V$,
\index{ Av S@$\Av_{\mathfrak{S}_{k+r}}$}\[
  \Av_{\mathfrak{S}_{k+r}} T (v_1, \ldots, v_{k+r}) = \frac{1}{(k + r)!} \sum_{\pi \in \mathfrak{S}_{k+r}} T(v_{\pi 1}, \ldots, v_{\pi (k+r)}).
\]
Rewrite~\eqref{eqn:polya-ineq} as
\begin{equation}%
  \label{eqn:polya-tensor-expression}
  (\1^{\tr}x)^r p(x) =  \langle T\otimes \1^{\otimes r}, x^{\otimes (k + r)} \rangle = \langle \Av_{\mathfrak{S}_{k+r}}(T \otimes \1^{\otimes r}), x^{\otimes (k+r)}\rangle,
\end{equation}
so~$\Av_{\mathfrak{S}_{k+r}}(T \otimes \1^{\otimes r})$ is the unique symmetric $(k+r)$-tensor that holds the coefficients of~$(\1^{\tr}x)^r p(x)$. With this, Pólya's theorem says that if~$T$ is in the algebraic interior of~$\COP(V, k)$, there exists an~$r \in \N$ such that
\[
  \Av_{\mathfrak{S}_{k+r}}(T \otimes \1^{\otimes r}) \geq 0.
\]

For~$r \in \N$ and~$V$ a finite set, define the~\defi{Pólya-type cone}\index{Polya type cone@Pólya-type cone!in finite dimensions} by
\[
  C_r(V, k) = \{\, T \in \Sym(V, k) : \Av_{\mathfrak{S}_{k+r}}(T \otimes \1^{\otimes r}) \geq 0\, \},
\]
which is a closed cone. This leads to the following conic formulation of Pólya's theorem, where~$\algint$\index{ algint@$\algint$} denotes the algebraic interior\index{algebraic interior}.
\index{Polyas theorem@Pólya's theorem}
\begin{theorem}[Pólya's theorem~{\cite{Polya1974UberPolynomen}}]%
  \label{thm:conic-polyas-theorem-finite-dimension}
  For a finite set~$V$ and integer and~$k \geq 2$,
  \[
    \algint{\COP(V, k)} \subseteq \bigcup_{r \in \N} C_r(V, k).
  \]
\end{theorem}

On the other hand, if~$p(x) = \langle T, x^{\otimes k}\rangle$ and there is an~$r \in \N$ such that~$(\1^{\tr}x)^rp(x) \geq 0$, then, by Equation~\eqref{eqn:polya-tensor-expression},~$\langle T, x^{\otimes k} \rangle \geq 0$ for all~$x \geq 0$ such that~$x \neq 0$. This proves~$C_r(V, k) \subseteq \COP(V, k)$ for all~$r$. If~$r' \geq r$, then~$(1^{\tr}x)^{r'} p(x) \geq 0$ for all~$x \geq 0$, so~$C_r(V, k) \subseteq C_{r'}(V, k)$. Denoting the dual cone of~$C_r(V,k)$ by~$C_r(V, k)^*$ and, by Theorem~\ref{thm:polar-identities}, taking the conic dual results in the following corollary of Pólya's theorem. We call a sequence of cones satisfying the conclusion of the Theorem an~\defi{outer approximation}\index{outer approximation} of~$\CP(V, k)$.
\begin{theorem}%
  \label{thm:completely-positive-polyas-finite-dimension}
  For a finite set~$V$ and an integer~$k \geq 2$,
  \[
    C_0(V, k)^* \supseteq C_1(V, k)^* \supseteq \cdots \supseteq \CP(V, k)\qquad \text{and}\qquad \CP(V, k) = \bigcap_r C_r(V, k)^*.
  \]
\end{theorem}

\sectionbreakafterproof

It should be said that the cones~$C_r(V, k)$ give a weak approximation of the copositive cone. For example, for a finite graph~$G = (V, E)$, the program~$\vartheta^*(G, C_r(V))$ is infeasible if~$r < \alpha(G) - 1$~\cite[Theorem 4.2]{DeKlerk2002ApproximationProgramming}. It is therefore unsurprising that since the conception of the topic, tighter hierarchies were studied. However, this thesis is mostly concerned with whether certain hierarchies of optimization problems converge, not with the rate of convergence. From this perspective it is interesting that such a weak hierarchy suffices.

\part{The completely positive cone}\label{part:cones}
\chapter[The completely positive cone of a measure space]{The completely positive cone of a measure space}%
\label{ch:completely-positive-cone}
Problems~\ref{it:problem-1},~\ref{it:problem-3}, and~\ref{it:problem-4} from the introduction have in common that they do not ask to optimize a cardinality, but a kind of density. This density is defined by a measure, which leads us to study spaces of $\mi{p}$-integrable functions.

Adams~\cite{Adams2018CopositivityDimension} and DeCorte, Oliveira, and Vallentin~\cite{DeCorte2022CompleteSets} defined a completely positive cone of square-integrable functions on a finite measure space. An outer approximation of this cone by Pólya-type cones similar to Theorem~\ref{thm:completely-positive-polyas-finite-dimension} was introduced by Bekker, Kuryatnikova, Oliveira, and Vera~\cite{Bekker2026OptimizationSpaces}.

In this chapter, we reproduce the results from~\cite{Bekker2026OptimizationSpaces}, particularly Theorem 4.1 and Theorem 5.3 therein, by explicitly showing how the finite-dimensional Pólya-type cones lift to the square-integrable setting. The approach is new and comes with tools that are more generally applicable. As opposed to the approach in~\cite{Bekker2026OptimizationSpaces}, it does not lean on the action of a locally compact group. Rather, group actions are considered later, in Chapter~\ref{ch:complete-positivity-under-symmetry}.

Especially the projective approximation lemma, Lemma~\ref{lem:projective-approximation}, is new. It is an attempt at capturing and generalizing some folklore and intuition about infinite-dimensional analogues of finite-dimensional hierarchies.

The application of the martingale convergence theorem is put to the forefront, which is another difference with the original exposition. This replaces the continuity assumption and an application of a technical theorem by Powers and Reznick in the proof of~\cite[Theorem A.1]{Bekker2026OptimizationSpaces}. Whereas~\cite{Bekker2026OptimizationSpaces} restricts itself to continuous and invariant kernels on compact spaces, the treatment here regards square-integrable kernels on finite countably generated measure spaces, but the conclusions are weaker. This approach lends itself to extension to $\sigma$-finite countably generated measure spaces, although the hierarchy there is not a direct generalization of the Pólya-type cones.

\section{Notation and preliminaries}
\label{sec:preliminiaries-operator-spaces-and-duality}
See the appendix for more background on locally convex analysis and $\mi{p}$-integrable functions. Our main reference is Simon's book~\cite{Simon2011Convexity}. All topological vector spaces are Hausdorff, and all functions and measures are real-valued.

\subsection*{Convexity and duality}
Section~\ref{ch:Appendix}.\ref{sec:appendix-topological-vector-spaces} gives more details about duality and convexity.

Let~$X$ be a vector space. The~\defi{algebraic dual}\index{dual!algebraic dual}~$X'$\index{ X '@$X'$} of~$X$ is the space of all linear functionals of~$X$. If~$X$ is a topological vector space, its~\defi{continuous dual}\index{dual!continuous dual}~$X^*$\index{ X *@$X^*$} is the subspace of~$X'$ consisting of continuous linear functionals. Let~$A : X \to Y$ is a continuous linear map of topological vector spaces. If its continuous linear adjoint exists, denote it by~$A^*$\index{ A *@$A^*$}.

For us, all cones are convex. The convex hull of a subset~$S$ of a vector space~$X$ is denoted~$\conv{S}$\index{convex hull}\index{ conv@$\conv$}. The closed convex hull of~$S$ is~$\cch{S} = \cl \conv S$\index{ clconc@$\cch$}. The conic hull is~$\cone{S}$\index{conic hull}\index{ cone@$\cone$}, and the closed conic hull of~$S$ is~$\ccone S = \cl \cone S$\index{ clcone@$\ccone$}. If~$(X,Y)$ is a dual pair of vector spaces, denote the dual cone of a set~$S$ by~$S^*$.

\subsection*{Measures and integrable functions}
See Section~\ref{ch:Appendix}.\ref{subsec:spaces-of-meas-funcs} for more background on measures and integrable functions.

Given a set~$V$ and a subset~$\P$ of the power set of~$V$, let~$\sigma(\P)$\index{ s P@$\sigma(\P)$} be the $\sigma$-algebra generated by~$\P$. Let~$(V, \A, \mu)$ be a measure space with $\sigma$-algebra~$\A$ and measure~$\mu$. One way to interpret~$\mu$ is as a function from the set of measurable functions to~$\R \cup \{\infty\}$ defined by integration over~$V$, which leads to the notation~$\mu(f) = \int_V f(x)\, d\mu(x)$. Denote the usual $\mi{p}$-norms by~$\|\cdot\|_p$ for~$1 \leq p \leq \infty$.

The measure~$\mu^k$ is the product measure on~$V^k$ with the naturally induced $\sigma$-algebra~$\A^k$; we will sometimes write it as~$\mu^{\otimes k}$. Unless specified otherwise, we always understand the product~$V^k$ to be equipped with the product measure~$\mu^k$, and we denote the naturally induced $\mi{\sigma}$-algebra on~$V^k$ by~$\A^k$. A square-integrable function~$f$ on~$V^k$ is called a~\defi{$k$-tensor}\index{k tensor@$k$-tensor}. It is called~\defi{symmetric}\index{k tensor@$k$-tensor!symmetric $k$-tensor} if it is invariant under permutation of its coordinates. We denote the space of $\mu$-equivalence classes of symmetric $k$-tensors by~$L^2_{\sym}(V, k)$\index{ L 2symVk@$L^2_{\sym}(V, k)$}, and understand~$L^2_{\sym}(V) = L^2_{\sym}(V, 2)$\index{ L 2symV@$L^2_{\sym}(V)$}. If~$f_1, \ldots, f_k \in L^2(V)$, then~$f_1 \otimes \cdots \otimes f_k$ denotes the element of~$L^2(V^k)$ given by
\[
  f_1 \otimes \cdots \otimes f_k(v_1, \ldots, v_k) = f_1(v_1) \cdots f_1(v_k)\index{ 000@$\otimes$}
\]
for all~$v_i \in V$.

The notation~$\langle f, g\rangle$\index{ <>@$\langle \cdot\, , \cdot\rangle$} refers to the integral of the pointwise product~$\mu(fg)$. If~$f \in L^p(V)$ and~$g \in L^q(V)$ with ~$p = 1$ and~$q = \infty$, or if~$1 < p < \infty$ and~$1/p + 1/q = 1$, then the product~$fg$ is in~$L^1(V)$ and~$\langle f, g\rangle \in \R$. In this case the map~$\langle \cdot\, , \cdot\rangle$ defines a duality
\index{ <>@$\langle \cdot\, , \cdot\rangle$}\[
  \langle \cdot\, , \cdot\rangle: L^p(V) \times L^q(V) \to \R
\]
under the respective norm topologies, hence the notation.

For the following, see the appendix and~\cite{Simon2005TraceApplications} for more details on Hilbert-Schmidt and trace-class operators.

We isometrically identify the subspace of Hilbert-Schmidt operators with the space of kernel~$L^2(V^2)$. Recall that a kernel is called~\defi{positive semidefinite}\index{positive semidefinite!kernel} if for all~$f \in L^2(V)$ we have~$\langle K, f \otimes f\rangle \geq 0$. We denote the cone of positive-semidefinite kernels by~$\PSD(V)$\index{ PSDV@$\PSD(V)$}.

Let~$K \in L^2_{\sym}(V)$. If~$K$ is of trace class and~$\Phi$ an orthonormal basis, the quantity
\begin{equation}%
  \label{eqn:trace}
  \sum_{\phi \in  \Phi} \langle K, \phi \otimes \phi\rangle
\end{equation}
is independent of~$\Phi$ and finite. It is called the~\defi{trace}\index{trace} of~$K$, denoted~$\Tr K$\index{ Tr@$\Tr$}. A positive-semidefinite kernel~$K$ is of trace class if and only if~$\Tr(K) < \infty$; in this case,~$\Tr(K) = \|K\|_{\B^1}$, where the latter denotes the trace norm of~$K$ as defined in~\ref{ch:Appendix}.\ref{sec:operator-spaces}. If~$V$ is a compact Hausdorff topological space and~$K$ is positive semidefinite and continuous, then~$\Tr(K) = \int_V K(x,x)\, d\mu(x)$.

\subsection*{Continuous functions and Radon measures}
Let~$V$ be a locally compact topological space. The space~$C(V)$\index{ C V@$C(V)$} is the space of continuous functions on~$V$. The space~$C_c(V)$\index{ CcV@$C_c(V)$} is the space of continuous functions with compact support; if~$V$ is compact, then~$C_c(V) = C(V)$. The space~$C_0(V)$\index{ C0V@$C_0(V)$} is the space of continuous functions on~$V$ that vanish outside compact sets: if~$f \in C_0(V)$, then for every~$\epsilon > 0$ there is a compact set~$K \subseteq V$ such that~$|f(x)| \leq \epsilon$ for all~$x \in V \setminus K$. We equip~$C_c(V)$ and~$C_0(V)$ with the supremum norm. Then,~$C_0(V)$ is the norm closure of~$C_c(V)$ in~$L^{\infty}(V)$.

Similar to spaces of square-integrable functions, we let~$C_{\sym}(V, k)$ be the subspace of~$C_0(V^k)$ of functions that are invariant under permutation of their arguments. We will also write~$C_{\sym}(V) = C_{\sym}(V,2)$.

We denote the space of Radon measures by~$M(V)$. By identifying it with the dual of~$C_0(V)$, we equip it with the operator norm. Let~$\mu \in M(V)$. The~\defi{support}\index{support of measure} of~$\mu$ is
\index{ supp mu@$\supp \mu$}\[
  \supp \mu = \bigcap_{\substack{U \subseteq V \text{ open}\\ \mu(U) = 0}} (V \setminus  U).
\]
We say that a Radon measure has~\defi{full support}\index{full support} if~$\mu(U) > 0$ for all open sets~$U$, i.e. if~$\supp \mu = V$.

\section{Projective approximation and martingales}%
\label{sec:projective-approximation}

At the core of this chapter lie two ideas. The first is that under mild conditions an outer approximation of a cone on a collection of simpler spaces lifts to a larger locally convex vector space. In this chapter, this means that we will take, for fixed~$k \geq 2$, the outer approximation of~$\CP([n], k)$ for all~$n$ given in Theorem~\ref{thm:conic-polyas-theorem-finite-dimension}, and use it to give an outer approximation of the completely positive cone in~$L^2_{\sym}(V, k)$, which we define later for suitable measure spaces~$(V, \A, \mu)$. The abstract framework for lifting these approximations is given by the projective approximation lemma, Lemma~\ref{lem:projective-approximation}. The second idea that this chapter is built on, is that the theory of martingales gives rise to such a lifting. It offers the operators that figure in the projective approximation lemma applied to~$L^2(V, k)$.

If~$X$ is a topological vector space and~$\cC \subset X$ is a cone, a continuous linear map~$\sigma: X \to X$ is called a linear automorphism of~$\cC$ if it is an automorphism of~$X$ and~$\sigma(\cC) \subseteq \cC$. If~$X$ is part of a dual pair and equipped with a dual topology, and if~$\sigma$ is an automorphism of~$\cC$, then~$\sigma^*$ is an automorphism of~$\cC^*$. For details on dual pairs and related topics, see Section~\ref{sec:preliminiaries-operator-spaces-and-duality} of this chapter and the appendix.

\begin{lemma}[Projective approximation lemma]\index{projective approximation lemma}%
  \label{lem:projective-approximation}
  Let~$(X, Y)$ be a dual pair of vector spaces,~$I$ be a directed set, and~$\{(X_i, Y_i)\}_{i \in I}$ be a collection of dual pairs of vector spaces. Let for all~$i$ the maps~$A_i : X \to X_i$ and~$B_i : Y \to Y_i$ be linear and continuous under the weak topologies,~$\cC \subseteq X$ be a closed convex cone and~$\{\cC_r\}_{r \in R}$ be a family of closed convex cones in~$X$. If~$B_i^*A_i$ converges weakly to a linear automorphism of~$\cC$, and
  \begin{enumerate}
    \item[(i)] $B_i^*A_i\cC_r \subseteq \cC_r$ for all~$i$ and $r$,
    \item[(ii)] $B_i^*A_i\cC \subseteq \cC$ for all~$i$,
    \item[(iii)] $A_i \cC = \bigcap_r A_i \cC_r$ for all~$i$,
  \end{enumerate}
  then~$\cC = \bigcap_r \cC_r$.
\end{lemma}

Before we move to the proof of this rather technical lemma, let us first look at two examples that show how it can be applied.

\begin{exmp}
  Let~$V$ be finite set and~$k \geq 2$ be an integer, let~$X = \Sym(V, k)$ and let~$\cC = \CP(V, k)$. Theorem~\ref{thm:completely-positive-polyas-finite-dimension} says that~$\CP(V, k) = \bigcap_r C_r(V, k)^*$. By applying a linear automorphism~$\sigma$ of~$\CP(V, k)$ to both sides, we find that~$\CP(V, k) = \bigcap_r \sigma(C_r(V, k)^*)$. This conclusion also follows from the projective approximation theorem with ~$I = \{0\}$ and~$A_0 = \sigma$,~$B_0 = \Id$. If~$\sigma$ does not stabilize the~$C_r(V, k)^*$, the~$\sigma(C_r(V, k))^*$ form an outer approximation that is distinct from the original.
\end{exmp}

\begin{exmp}
  Let~$X = l^2(\N)$ and~$\cC = l^2(\N)_{\geq 0}$, the closed convex cone of nonnegative sequences. Let for~$r \in \N$ the set~$\cC_r$ be the closed convex cone of sequences such that~$x_i \geq 0$ for all~$i \leq r$. It is of course more than clear that~$\cC = \bigcap_r \cC_r$. We will use the projective approximation lemma to reach the same conclusion.

  The space~$l^2(\N)$ is a Hilbert space with inner product~$\langle x, y\rangle = \sum_{i \in \N} x_i y_i$. For every~$i \in \N$, let~$p_i : l^2(\N) \to \R^i$ be the projection onto the first~$i$ coordinates. Then,~$p_i^*p_i$ is the map that sets all coordinates with index larger than~$i$ to zero. We see that~$p_i^*p_i\cC \subseteq \cC$ and~$p_i^* p_i\cC_r = \cC_r$ for all~$r$ and~$i$. Moreover, the weak limit of~$p_i^*p_i$ is the identity on~$l^2(\N)$.

  For all~$i \in \N$,~$p_i \cC = \R^i_{\geq 0}$ and if~$r \geq i$,~$p_i \cC_r = \R^i_{\geq 0}$. For all~$r < i$,~$p_i\cC_r$ is the cone of vectors in~$\R^i$ with the first~$r$ coordinates nonnegative. Thus, for all~$i$,~$\bigcap_r p_i \cC_r = p_i\cC_i = \R^i_{\geq 0} = p_i\cC$, so that all conditions of the projective approximation lemma are met, and the conclusion follows.
\end{exmp}

From these examples one might guess that it is usually not necessary---if not pedantic---to use the lemma. One would be correct. The lemma is also generic: it does not say anything what a ``best'' way of lifting an outer approximation is. Its main purpose is to offer some additional rigor and direction to this exposition.

\begin{proof}[Proof of Lemma~\ref{lem:projective-approximation}]
  For all~$i \in I$
  \begin{equation}%
    \label{eqn:incl-A_nC}
    A_i \cC = \bigcap_{r \in R} A_i\cC_r \subseteq \bigcap_{r \in R} {B_i^*}^{-1} \cC_r.
  \end{equation}
  The first equality is implied by property~(iii) and the inclusion holds by property~(i). Moreover, by Theorem~\ref{thm:polar-identities}(i) and (iii) and properties~(ii) and~(iii) above, for all~$i \in I$
  \begin{equation}%
    \label{eqn:incl-B_nC*}
    B_i\cC^* \subseteq ({B_i^*}^{-1}\cC)^*\subseteq (A_i \cC)^* = \biggl(\bigcap_{r \in R} A_i \cC_r\biggr)^*.
  \end{equation}

  Denote the duality of the pair~$(X, Y)$ by~$\langle\cdot\, , \cdot\rangle$ and that of~$(X_i, Y_i)$ by~$\langle \cdot\, , \cdot\rangle_i$. Weak convergence of~$B_i^*A_i$ to a linear automorphism~$\sigma$ of~$\cC$ means that for all~$x \in X$ and~$y \in Y$
  \begin{equation}%
    \label{eqn:duality-under-weak-limit}
    \langle x, y \rangle = \lim_{i \in I} \langle A_i \sigma^{-1} x, B_i y \rangle_i = \lim_{i \in I} \langle A_i x, B_i {\sigma^*}^{-1} y \rangle_i.
  \end{equation}

  If~$x \in \cC$ and~$y \in (\bigcap_r \cC_r)^*$, then~$A_i \sigma^{-1}x \in \bigcap_r {B_i^*}^{-1} \cC_r$ by~\eqref{eqn:incl-A_nC}. Moreover,~$B_i y \in B_i (\bigcap_r \cC_r)^* \subseteq (\bigcap_r {B_i^*}^{-1} \cC_r)^*$ according to Theorem~\ref{thm:polar-identities}(iii). It follows from~\eqref{eqn:duality-under-weak-limit} that~$\langle x, y\rangle \geq 0$, and~$\cC \subseteq \bigcap_r \cC_r$.

  Furthermore, if~$x \in \bigcap_r \cC_r$ and~$y \in \cC^*$, then~$A_i x \in A_i \bigcap_r\cC_r \subseteq \bigcap_r A_i \cC_r$ for all~$i \in I$. The adjoint~$\sigma^*$ is an automorphism of~$\cC^*$, so~$B_i {\sigma^*}^{-1} y \in ( \bigcap_r A_i \cC_r )^*$ by~\eqref{eqn:incl-B_nC*}, so that~$\langle x, y \rangle \geq 0$ by~\eqref{eqn:duality-under-weak-limit}. This establishes~$\bigcap_r \cC_r \subseteq \cC$ by Theorem~\ref{thm:polar-identities}(i) and concludes the proof.
\end{proof}

This chapter not only presents a direct $L^2$ analogue of the tried-and-true Pólya-type cones on~$C(V^2)$ for compact~$V$, as introduced in~\cite{Bekker2026OptimizationSpaces, Kuryatnikova2019TheProblems,Kuryatnikova2017Approximating}, but also a stricter outer approximation of the completely positive cone that extends to a class of $\sigma$-finite measure spaces. Both can be understood from the perspective of the projective approximation lemma.

\sectionbreak

A special case of the theory of martingales\index{martingale} offers the operators~$A_i$ and~$B_i$ of the projective approximation lemma for $L^p$ spaces. For the purposes of this chapter and the next, it simply gives a specific way of describing a function as a limit of simple functions---finite linear combinations of step functions. Chapter~5 of the book by Edwards and Gaudry~\cite{Edwards1977Littlewood-PaleyTheory} contains a more complete presentation of the topic.

For a set~$V$, we say that a partition~$\mathcal{P}'$ of~$V$ \defi{refines}\index{refinement of partition} a partition~$\mathcal{P}$, if every element of~$\mathcal{P}'$ is contained in an element of~$\mathcal{P}$; notation~$\mathcal{P}' \preceq \mathcal{P}$\index{ <@$\preceq$}. Let~$(V, \A, \mu)$ be a finite measure space. Call a sequence of partitions~$(\mathcal{P}_n)_{n \in \N}$ of~$V$ with all~$\P_n \subseteq \A$ a~\defi{finite-rank approximation}\index{finite rank approximation@finite-rank approximation} of~$(V, \A, \mu)$ if:
\begin{itemize}
  \item $\mathcal{P}_0 = \{V\}$;
  \item $\mathcal{P}_{n+1} \preceq \mathcal{P}_n$ for all~$n \in \N$;
  \item each~$\mathcal{P}_n$ is finite;
  \item $\A = \sigma(\bigcup_{n}\mathcal{P}_n)$.
\end{itemize}

Given a partition~$\P$ of~$V$, call the subset of~$\P$ consisting of those sets that have nonzero measure~$\P^+$. Define for~$1 \leq p < \infty$ the operator
\index{ EP@$E_{\P}$}\[
  E_{\P} : L^p(V) \to L^p(V),\quad E_{\P} f = \sum_{P \in \P} \langle f, \1_P\rangle / \mu(P).
\]
For a sequence of partitions~$(\P_n)_{n \in \N}$, let~$E_n = E_{\P_n}$\index{ En@$E_n$}.

\begin{theorem}[Martingale convergence theorem]\index{martingale!convergence}%
  \label{thm:martingale-convergence}
  Let~$(V, \A, \mu)$ be a finite measure space equipped with a finite-rank approximation~$(\mathcal{P}_n)_{n \in \N}$. If~$1 \leq p < \infty$, the sequence~$(E_n)_{n \in \N}$
  converges under the strong operator topology to the identity.

  If~$f \in L^p(V)$ for~$1 \leq p < \infty$, the sequence~$(E_nf(v))_{n \in \N}$ converges to~$f(v)$ for almost every~$v \in V$.
\end{theorem}
The second statement of the theorem is quite strong, so it is not surprising that its proof extends beyond what fits in this thesis. It uses a clever but technical trick, which can be found in the book by its inventor~\cite[Theorem VII.4.1]{Doob1990StochasticProcesses}. The first statement is easier; it uses that~$\bigcup_n L^p(V, \sigma(\P_n), \mu)$ is dense in~$L^p(V)$ and that the operators~$E_n$ are uniformly bounded~\cite[Theorem 5.2.6]{Edwards1977Littlewood-PaleyTheory}.

The functions~$E_nf$ are called the~\defi{conditional expectations}\index{conditional expectations} of~$f$ with respect to~$(\sigma(\P_n))_{n \in \N}$. They are usually interpreted as functions on the measure space~$(V, \sigma(\mathcal{P}_n), \mu)$, in which case the sequence~$(E_n f)_{n \in \N}$ is called the \defi{martingale}\index{martingale} associated to~$f$. For us, there is no harm in thinking of~$E_nf$ as an element of~$L^p(V, \A, \mu)$.

Let us investigate how the martingale convergence theorem interacts with the projective approximation lemma. Let~$k \geq 2$ be an integer. For given sets~$P_1, \ldots, P_k \subseteq V$ and a permutation~$\pi \in \mathfrak{S}_k$, let~$ P_{\pi 1, \ldots, \pi k}$ denote the set~$P_{\pi 1} \times \cdots \times P_{\pi k}$\index{ P1k@$P_{1, \ldots, k}$}. We will be interested in symmetric square-integrable $\mi{k}$-tensors, that is, elements of~$L^2_{\sym}(V, k)$.

Let~$(V, \A, \mu)$ be a finite measure space with a finite partition~$\P$. Note that~$\P$ induces a partition on~$V^k$, namely
\[
  \{\, P_{1,\ldots, k} : P_i \in \P\text{ for all $i$}\, \}.
\]
When~$(\P_n)_{n \in \N}$ is a finite-rank approximation of~$(V, \A, \mu)$, this construction induces a finite-rank approximation of~$(V^k, \A^k, \mu^k)$. We denote the conditional expectation operators on~$L^2_{\sym}(V, k)$ with respect to this induced partition by~$E_{\P}$ as well.

Define the bounded operators~$A_{\P}, B_{\P}: L^2_{\sym}(V, k) \to \Sym(\P^+, k)$ by
\begin{multline*}
  A_{\P} K = \bigl(\langle K, \1_{P_{1, \ldots, k}} \rangle \bigr)_{P_1,\ldots, P_k \in \P^+} \text{ and}\\
  B_{\P} K = \biggl(\frac{\langle K, \1_{P_{1, \ldots, k}} \rangle}{\mu^k(P_{1,\ldots,k})} \biggr)_{P_1,\ldots, P_k \in \P^+},
\end{multline*}
with continuous adjoints
\begin{multline*}
  A_{\P}^* T = \sum_{P_1,\ldots, P_k \in \P^+} T_{P_1,\ldots, P_k} \1_{P_{1,\ldots, k}}\text{ and}\\
  B_{\P}^* T = \sum_{P_1,\ldots, P_k \in \P^+} T_{P_1,\ldots, P_k} \frac{\1_{P_{1,\ldots, k}}}{\mu^k(P_{1,\ldots, k})}.
\end{multline*}
If~$(\P_n)_{n}$ is a finite-rank approximation of~$V$, denote~$A_n = A_{\P_n}$ and~$B_n = B_{\P_n}$. Item~(i) of the following lemma together with the martingale convergence theorem imply that for a finite-rank approximation~$(\P_n)_{n \in \N}$ the weak limit of~$B_n^*A_n$ is the identity, which is the first requirement of the projective approximation lemma.
\begin{lemma}%
  \label{lem:properties-An-Bn}
  Let~$(V, \A, \mu)$ be a finite measure space with a partition~$\P$. For all integers~$k \geq 2$:
  \begin{enumerate}
    \item[(i)] $E_{\P} = B_{\P}^* A_{\P}$;
    \item[(ii)] $A_{\P} B_{\P}^* = I$, the identity on~$\Sym(\P^+, k)$;
    \item[(iii)] $A_{\P}(f^{\otimes k}) = (A_{\P} f)^{\otimes k}$ for all~$f \in L^2(V)$ and~$B_{\P}^*(w^{\otimes k}) = (B_{\P}^* w)^{\otimes k}$ for all~$w \in \R^{\P^+}$.
  \end{enumerate}
\end{lemma}
\begin{proof}
  \textup{(i)} This follows directly from the definition of the maps.

  \textup{(ii)} This is true because~$\{\, \1_{P_{1,\ldots, k}} / \sqrt{\mu^k(P_{1,\ldots,k})} : P_1, \ldots, P_k \in \P^+\, \}$ is an orthonormal set in~$L^2(V)$.

  \textup{(iii)} Indeed,
  \[
    (A_{\P}(f^{\otimes k}))_{P_1, \ldots,P_k} = \langle f^{\otimes k}, \1_{P_{1,\ldots, k}}\rangle = \prod_{i=1}^k \langle f, \1_{P_i}\rangle = \prod_{i = 1}^k (A_{\P}f)_{P_i},
  \]
  and
  \[
    B_{\P}^*(w^{\otimes k}) = \sum_{P_1, \ldots, P_k \in \P^+} w_{P_1} \cdots w_{P_k} \frac{\1_{P_{1,\ldots, k}}}{\mu^k(P_{1,\ldots, k})} = \left(\sum_{P \in \P^+} w_P \frac{\1_P}{\mu(P)}\right)^{\otimes k}.\qedhere
  \]
\end{proof}

We end this section with a sufficient condition for a measure space to locally admit finite-rank approximations, and for~$L^p(V)$ to be separable.
\begin{lemma}%
  \label{lem:sufficient-conition-finite-rank-approximation}
  If~$(V, \A, \mu)$ is countably generated and $\sigma$-finite, then every finite measure subspace admits a finite-rank approximation and~$L^p(V)$ is separable for all~$1 \leq p < \infty$.
\end{lemma}
\begin{proof}
  Let~$A \in \A$ have finite measure, and denote the power set of~$A$ by~$\mathscr{P}(A)$. Let~$\{ U_0, U_1, \ldots \}$ be a countable generator of~$\A$. Define~$\P_0 = \{A\}$, and define~$\P_n$ recursively by
  \[
    \P_{n + 1} = \{\, P \cap U \cap A: P \in \P_n, U\in \{U_{n}, V \setminus U_{n}\}\, \}.
  \]
  Then,~$\sigma(\bigcup_{n}\P_n) = \A \cap \mathscr{P}(A)$,~$\P_{n+1}$ refines~$\P_n$ for all~$n \in \N$, and each~$\P_n$ contains only finitely many sets, so~$(\P_n)_{n \in \N}$ is a finite-rank approximation of~$(A, \A \cap \P(A), \mu)$.

  That for all~$1 \leq p < \infty$ the space~$L^p(V)$ is separable is Proposition 3.4.5. in~\cite{Cohn2013MeasureTheory}.
\end{proof}

\section{A Pólya-type approximation on finite measure spaces}%
\label{sec:Polya-for-finite-measure-spaces}
For a measure space~$(V, \A, \mu)$, the cone of \defi{completely positive}\index{completely positive k tensor@completely positive $k$-tensor} $\mi{k}$-tensors is
\index{ CPVk@$\CP(V, k)$}\index{completely positive cone!in L2symVk@in~$L^2_{\sym}(V, k)$}\[
  \CP(V, k) = \ccone\{\, f^{\otimes k} \in L^2_{\sym}(V, k) : f \in L^2(V)_{\geq 0}\, \}.
\]
A~$\mi{k}$-tensor is called \defi{copositive}\index{copositive k tensor@copositive $k$-tensor} if it is in~$\CP(V, k)^*$, that is, when it is in
\index{ COPVk@$\COP(V, k)$}\index{copositive cone!in L2VK@in~$L^2_{\sym}(V, k)$}\[
  \COP(V, k) = \{\, T \in L^2_{\sym}(V, k) : \langle T, f^{\otimes k}\rangle \geq 0 \text{ for all } f \in L^2(V)_{\geq 0}\, \}.
\]
We derive a version of Theorem~\ref{thm:completely-positive-polyas-finite-dimension} for these cones when~$(V, \A, \mu)$ is countably generated and finite.

From here on, assume~$(V, \A, \mu)$ is countably generated and finite. We follow the efforts of Kuryatnikova and Vera~\cite{Kuryatnikova2019TheProblems,Kuryatnikova2017Approximating} and Bekker, Kuryatnikova, Oliveira and Vera~\cite{Bekker2026OptimizationSpaces}. For~$r \in \N$, define~$\T_r : L^2(V^k) \to L^{2}(V^{k+r})$ as
\index{ Tr@$\T_r$}\[
  \T_rK(v_1, \ldots, v_{k + r}) = \frac{1}{(k+r)!}\sum_{\pi \in \mathfrak{S}_{k+r}} K(v_{\pi(1)},\ldots, v_{\pi(k)}).
\]
This is a bounded operator, as~$\T_rK =  \Av_{\mathfrak{S}_{k+r}}(K\otimes \1^{\otimes r})$, and the maps~$\Av_{\mathfrak{S}_{k+r}}$ and~$K \mapsto K \otimes \1^{\otimes r}$ are bounded. The~\defi{Pólya-type cones}\index{Polya type cone@Pólya-type cone!on a finite measure space} of~$L^2_{\sym}(V, k)$ are
\index{ CrVk@$C_r(V, k)$}\index{ CrVk*@$C_r(V, k)^*$}\[
  C_r(V, k) = \T_r^{-1} L^2_{\sym}(V, k + r)_{\geq 0}\qquad \text{and}\qquad C_r(V, k)^* = \T_r^*L^2_{\sym}(V, k + r)_{\geq 0}.
\]
Continuity of~$\T_r$ implies that the cones are closed. The adjoint of~$\T_r$ is given by, for almost all $k$-tuples~$(v_1, \ldots, v_k)$,
\begin{equation}%
  \label{eqn:adjoint-Polya-type-operator}
  \T_r^*(F)(v_1,\ldots, v_k) = \frac{1}{(k+r)!} \int_{V^r} \sum_{\pi \in \mathfrak{S}_{k+r}} F(\pi (v_1, \ldots, v_k, v))\, d\mu^r(v);
\end{equation}
here~$v \in V^r$ while each~$v_i$ is in~$V$, and for~$\pi \in \mathfrak{S}_{k}$ and~$v \in V^{k+r}$, we define~$\pi(v) = \bigl(v_{\pi 1},\ldots, v_{\pi (k + r)}\bigr)$. If~$F$ is a symmetric tensor, the sum over the permutation group and the factor~$1/(k+r)!$ disappear.

The claim is of course that if~$(\P_n)_{n \in }$ is a finite-rank approximation of~$V$ with associated operators~$A_n$ and~$B_n$ as defined in Section~\ref{sec:projective-approximation}, the cones~$\CP(V, k)$ and~$C_r(V, k)^*$ with these operators satisfy conditions~\textup{(i)-(iii)} in Lemma~\ref{lem:projective-approximation}, the projective approximation lemma. The following theorem is a step in this direction.
\begin{theorem}%
  \label{thm:C_r(V)^*-given-by-projections}
  Let~$(V, \A, \mu)$ be a finite measure space with a finite-rank approximation~$(\P_n)_{n \in \N}$. For all integers~$k \geq 2$ and~$r \geq 0$,
  \[
    C_r(V, k)^* = \ccone \bigcup_{n \in \N} B_n^* C_r({\P_n^+}, k)^*.
  \]
\end{theorem}

The proof of this comes down to showing a degree of compatibility between~$\T_r$ and~$B_n$. In the process of proving this, we prove  Theorem~\ref{thm:generating-set-L2-CrV*}. This theorem extends a result by Bomze and De Klerk, who described an explicit generating set for the Pólya-type cones~$(C_r(V))^*$ on~$\Sym(V)$ for finite~$V$~\cite[Theorem 2.4]{Bomze2002SolvingProgramming}.

For a finite set~$S$ and~$r \in \N$, let~$I^S(r)$\index{ Isr@$I^S(r)$} denote the set of all~$m \in \N^S$ that have~$\|m\|_1 = r$. The vectors~$m$ correspond bijectively to multisets of cardinality~$r$ of elements of~$S$, thus we will say that~$m$ represents a choice~$s_1, \ldots, s_r \in S$. Bomze and De Klerk~\cite[Theorem 2.4]{Bomze2002SolvingProgramming} showed that, for all~$r \geq 0$ and finite~$S$,
\[
  C_r(S) = \{\, M \in \Sym(S) : \langle mm^{\tr} - \Diag m, M\rangle \geq 0 \text{ for all } m \in I^S(r+2)\, \},
\]
where~$\Diag m$ is the diagonal matrix with diagonal~$m$.

If~$m$ corresponds to a choice~$s_1, \ldots, s_{r+2} \in S$, then~$r!(mm^{\tr} - \Diag{m})_{s_i s_j}$ is the multiplicity of~$e_{s_i} \otimes e_{s_j}$ in~$\T_r^* (e_{s_1} \otimes \cdots \otimes e_{s_{r+2}})$, where~$e_{s_i}$ is the standard basis vector in~$\R^S$ corresponding to~$s_i$. Investigating expression~\eqref{eqn:adjoint-Polya-type-operator} shows that~$r!(mm^{\tr} - \Diag{m})_{s_i s_j}$ is the number of permutations~$\pi \in \mathfrak{S}_{r+2}$ such that~$(s_i, s_j) = (s_{\pi 1}, s_{\pi 2})$.

We define a tensor analogue of the matrices~$mm^\tr - \Diag(m)$. Let~$S$ be a finite set,~$k \geq 2$ and~$r \geq 0$ integers, and~$m \in I^{S}(k+r)$. If~$m$ represents~$t_1 ,\ldots, t_{k+r}$, we define~$T_m$ as the $\mi{k}$-tensor such that~$r! (T_m)_{s_1,\ldots,s_k}$ is the number of permutations~$\pi \in \mathfrak{S}_{k+r}$ for which~$(s_{1}, \ldots, s_{k})=(t_{\pi 1},\ldots,t_{\pi k})$. Although we do not need the exact values of~$T_m$, for completeness's sake, if~$m' \in I^{S}(k)$ represents~$s_1, \ldots, s_k \in S$, then
\[
  (T_m)_{s_1,\ldots, s_k} =
  \begin{cases}
    \prod_{s \in \{s_1,\ldots, s_k\}}\frac{m_s!}{(m_s - m'_s)!} & \text{if } m'_{s_i} \leq m_{s_i} \text{ for } 1 \leq i \leq k, \\
    0                                                         & \text{otherwise.}
  \end{cases}
\]

The notation for~$\T_r$ and~$A_{\P}$ does not specify the domains, meaning that the operators on~$\Sym(\P^+, k)$ get the same symbol as the corresponding operators on~$L^2_{\sym}(V, k)$. In Lemma~\ref{lem:Tr-compatible} below, which ones are which should be derived from context.
\begin{lemma}%
  \label{lem:Tr-compatible}
  Let~$(V, \A, \mu)$ be a finite measure space with finite partition~$\P$, and let $k \geq 2$ and~$r \geq 0$ be integers, then
  \begin{enumerate}
    \item[(i)] $\T_r A_{\P}^* = A_{\P}^* \T_r$, and
    \item[(ii)]  if~$P_1,\ldots,P_{k+r} \in \P^+$ and~$m \in I^{\P^+}(k+r)$ represents this choice, then
      \[
        \T_r^*\frac{\1_{P_{1,\ldots, k+r}}}{\mu^{k+r}(P_{1,\ldots,k+r})} = \frac{r!}{(k+r)!} B_{\P}^* T_m.
      \]
  \end{enumerate}
\end{lemma}
\begin{proof}
  \textup{(i)} Take~$T \in (\R^{\P^+})^{k+r}$,~$P_1, \ldots, P_{k + r} \in \P^+$, and~$w \in P_{1, \ldots, k + r}$. Then
  \[
    \begin{split}
      (A_{\P}^*\Av_{\mathfrak{S}_{k+r}}T)(w) &= \sum_{Q_1,\ldots, Q_{k+r} \in \P^+}\frac{1}{(k+r)!}\sum_{\pi \in \mathfrak{S}_{k+r}}T_{Q_{\pi 1}, \ldots, Q_{\pi (k+r)}} \1_{Q_{1 ,\ldots, k+r}}(w)\\
      &= \frac{1}{(k+r)!}\sum_{\pi \in \mathfrak{S}_{k+r}}T_{P_{\pi 1}, \ldots, P_{\pi (k+r)}}\\
      &= \frac{1}{(k+r)!}\sum_{\pi \in \mathfrak{S}_{k+r}}\sum_{Q_1, \ldots, Q_{k+r} \in \P^+}T_{Q_{1}, \ldots, Q_{k+r}} \1_{Q_{1 ,\ldots, k + r}}(\pi w)\\
      &= (\Av_{\mathfrak{S}_{k+r}}A_{\P}^*T)(w).
  \end{split}\]

  Thus, for~$T \in \Sym(\P^+, k)$,~$A_{\P}^* \Av_{\mathfrak{S}_{k+r}}(T \otimes \1^{\otimes r}) = \Av_{\mathfrak{S}_{k+r}} A_{\P}^*(T \otimes \1^{\otimes r})$. Moreover,
  \[
    A_{\P}^* (T \otimes \1^{\otimes r}) = \sum_{P_1, \ldots, P_{k} \in \P^+} T_{P_1,\ldots, P_k} \1_{P_{1,\ldots,k}} \otimes \biggl(\sum_{P \in \P^+} \1_P \biggr)^{\otimes r} = (A_{\P}^* T) \otimes \1^{\otimes r},
  \]
  and the conclusion follows.

  \textup{(ii)} Let~$P_1, \ldots, P_{k+r} \in \P^+$, and let~$m \in I^{\P^+}(k+r)$ be the vector representing this choice. For all~$w \in V^k$
  \[
    \begin{split}
      \T_r^*\frac{\1_{P_{1, \ldots, k+r}}}{\mu^{k+r}(P_{1,\ldots, k+r})}(w)& = \frac{1}{(k+r)!} \int_{V^r} \sum_{\pi \in \mathfrak{S}_{k+r}} \frac{\1_{P_{\pi 1, \ldots, \pi(k+r)}}}{\mu^{k+r}(P_{1,\ldots, k+r})}(w, v)\, d\mu(v)\\
      &= \frac{1}{(k+r)!} \sum_{\pi \in \mathfrak{S}_{k+r}} \frac{\1_{P_{\pi 1},\ldots, \pi k}}{\mu^{k}(P_{\pi 1, \ldots, \pi k})}(w).
  \end{split}\]
  By definition of~$T_m$, for given~$\pi \in \mathfrak{S}_{k+r}$, the number of~$\sigma \in \mathfrak{S}_{k+r}$ such that~$(P_{\sigma 1}, \ldots, P_{\sigma k}) = (P_{\pi 1},\ldots,P_{\pi k})$ is ~$r! (T_m)_{P_{\pi 1},\ldots,P_{\pi k}}$. Group equal terms together to obtain
  \[
    \begin{split}
      \T_r^*\frac{\1_{P_{1, \ldots, k+r}}}{\mu^{k+r}(P_{1,\ldots, k+r})} &= \frac{r!}{(k+r)!}\sum_{Q_1, \ldots, Q_k \in \P^+}(T_m)_{Q_1,\ldots, Q_k}\frac{\1_{Q_{1,\ldots, k}}}{\mu^k(Q_{1,\ldots, k})}\\
      & = \frac{r!}{(k+r)!} B_{\P}^* T_m. \qedhere
  \end{split}\]
\end{proof}

The next theorem is interesting, even if~$V$ is a finite set with the counting measure. In this case, taking~$A_n = B_n = I$, it says that~$C_r(V, k)^*$ is the conic hull of the tensors~$T_m$ with~$m \in I^V(k+r)$. This is the $\mi{k}$-tensor analogue of~\cite[Theorem 2.4]{Bomze2002SolvingProgramming}.
\begin{theorem}%
  \label{thm:generating-set-L2-CrV*}
  Let~$(V, \A, \mu)$ be a finite measure space with a finite-rank approximation~$(\P_n)_{n \in \N}$. For all integers~$k \geq 2$ and~$r \geq 0$,
  \[
    C_r(V, k) = \bigcap_{n \in \N}\{\, K \in L^2_{\sym}(V, k) : \langle K, B_n^* T_m \rangle \geq 0 \text{ for all } m \in I^{\P_n^+}(k+r)\, \}
  \]
  and
  \[
    C_r(V, k)^* = \ccone \bigcup_{n \in \N} B_n^*\{\, T_m : m \in I^{\P_n^+}(k + r)\, \}.
  \]
\end{theorem}
\begin{proof}
  If~$K \in C_r(V, k)$, then for all sets~$P_1,\ldots, P_{k+r} \in \P_n^+$ the inequality~$\langle \T_r K, \1_{P_{1,\ldots k + r}}/\mu^k(P_{1,\ldots, k+r})\rangle \geq 0$ holds. So, using Lemma~\ref{lem:Tr-compatible}(ii), this implies~$\langle K, B_n^* T_m\rangle \geq 0$ for all~$m \in I^{\P_n^+}(k+r)$.

  The martingale convergence theorem, Theorem~\ref{thm:martingale-convergence}, says that if~$K$ is in~$L^2_{\sym}(V, k)$, then
  \[
    \T_rK = \lim_{n} \sum_{P_1, \ldots, P_{k+r} \in \P_n^+} \frac{\langle \T_rK, \1_{P_1,\ldots, k + r}\rangle}{\mu^k(P_{1,\ldots, k + r})}\1_{P_{1,\ldots, k + r}}
  \]
  under the norm. Lemma~\ref{lem:Tr-compatible}(ii) shows that the condition~$\langle K, B_n^* T_m\rangle \geq 0$ for all~$m \in I^{\P_n^+}(k+r)$ is also sufficient for~$K$ to be in~$C_r(V)$, proving the first statement of the theorem.

  The equality
  \[
    C_r(V, k)^* = \ccone \bigcup_{n \in \N} B_n^*\{\, T_m : m \in I^{\P_n^+}(k + r)\, \}
  \]
  follows by taking the dual on both sides, and Theorem~\ref{thm:polar-identities}(iv).
\end{proof}

\begin{proof}[Proof of Theorem~\textup{\ref{thm:C_r(V)^*-given-by-projections}}]
  This follows directly from Theorem~\ref{thm:generating-set-L2-CrV*}. Indeed,
  \[
    C_r(V, k)^* = \ccone \bigcup_{n \in \N} B_n^*\{\, T_m : m \in I^{\P_n^+}(k+r)\, \},
  \]
  and it follows from the same theorem that
  \[
    C_r(\P_n^+,k) = \ccone \{\, T_m : m \in I^{\P_n^+}(k+r)\, \}
  \]
  by taking a trivial finite-rank approximation.
\end{proof}

\begin{theorem}%
  \label{thm:Polyas-theorem-finite-measure}
  If~$(V, \A, \mu)$ is a countably generated finite measure space and~$k \geq 2$ is an integer, then
  \[
    C_0(V, k)^* \supseteq C_1(V, k)^* \supseteq \cdots \supseteq \CP(V, k)\qquad \text{and}\qquad \CP(V, k) = \bigcap_{r \in \N} C_r(V, k)^*.
  \]
\end{theorem}
\begin{proof}
  The inclusions~$C_r(V, k)^* \supseteq C_{r+1}(V, k)^*$ follow from Theorem~\ref{thm:completely-positive-polyas-finite-dimension} and Theorem~\ref{thm:C_r(V)^*-given-by-projections}.

  Martingale convergence---Theorem~\ref{thm:martingale-convergence}---together with Lemma~\ref{lem:properties-An-Bn}(i) says that~$\lim_n B_n^* A_n = \Id$ under the weak operator topology. It suffices to check conditions~(i)--(iii) of the projective approximation lemma for the cones~$C_r(V, k)^*$ and~$\CP(V, k)$ and the operators~$A_n$ and~$B_n$.

  To this end, first prove that
  \[
    A_n \CP(V, k) = \CP(\P_n^+,k)\quad \text{and}\quad A_n C_r(V, k)^* = C_r(\P_n^+,k)^*.
  \]
  By Lemma~\ref{lem:properties-An-Bn}(ii), for all~$n$,~$A_n B_n^* = I$ on~$\Sym(\P_n^+, k)$. Thus, it is enough to show that~$A_n$ maps the respective cones on~$V$ into their counterparts on~$\P_n^+$, and that~$B_n^*$ does the opposite.

  The maps~$A_n$ and~$B_n^*$ preserve pointwise nonnegativity. Hence, the equalities~$A_n \CP(V, k) = \CP(\P_n^+, k)$ for all~$n$ follow from Lemma~\ref{lem:properties-An-Bn}(iii): indeed, if~$f \geq 0$ and~$w \geq 0$, then~$A_nf \geq 0$ and~$B_n^*w \geq 0$, and~$A_n(f^{\otimes k}) = (A_n f)^{\otimes k}$ and~$B_n^*(w^{\otimes k}) = (B_n^* w)^{\otimes k}$.

  Theorem~\ref{thm:C_r(V)^*-given-by-projections} says that~$B_n^* C_r({\P_n^+},k)^* \subseteq C_r(V, k)^*$. To prove
  \[
    A_n C_r(V, k)^* \subseteq C_r({\P_n^+}, k)^*,
  \]
  by Theorem~\ref{thm:polar-identities}(i) and~(ii), it is equivalent to show that~$C_r({\P_n^+}, k)$ is contained in~${A_n^*}^{-1} C_r(V, k)$. Lemma~\ref{lem:Tr-compatible}(i) says that if~$T \in \Sym(\P_n^+, k)$, then it follows that~$\T_r A_n^* T = A_n^* \T_r T$. The latter is nonnegative if~$T \in C_r({\P_n^+}, k)$, whence it follows that~$T \in {A_n^*}^{-1} C_r(V, k)$. This proves~$A_n C_r(V, k)^* = C_r({\P_n^+}, k)^*$.

  Next, prove conditions~\textup{(i)--(iii)} of the projective approximation lemma.

  \textup{(i)} The inclusion~$B_n^*A_n C_r(V, k)^* \subseteq C_r(V, k)^*$ for all~$n, r \in \N$ follows from~$A_n C_r(V, k)^* = C_r(\P_n^+, k)^*$ and Theorem~\ref{thm:C_r(V)^*-given-by-projections}.

  \textup{(ii)} To prove~$B_n^* A_n \CP(V, k) \subseteq \CP(V, k)$, it is enough to consider the generators~$f^{\otimes k}$ with~$f \in L^2(V)_{\geq 0}$. Lemma~\ref{lem:properties-An-Bn}\textup(iii) shows that the equality~$B_n^* A_n (f^{\otimes k}) = (B_n^* A_n f)^{\otimes k}$ holds, and~$B_n^* A_n f$ is nonnegative, so property~\textup{(ii)} follows.

  \textup{(iii)} The property~$A_n \CP(V, k) = \bigcap_r A_n C_r(V, k)^*$ follows from the fact that~$A_n\CP(V, k) = \CP({\P_n^+},k)$,~$A_n C_r(V, k) = C_r({\P_n^+}, k)$, and Theorem~\ref{thm:completely-positive-polyas-finite-dimension}.  This concludes the proof.
\end{proof}

\section{\texorpdfstring{An outer approximation on $\sigma$-finite measure spaces}{An outer approximation on sigma-finite measure spaces}}
We would like to extend Theorem~\ref{thm:Polyas-theorem-finite-measure} to $\sigma$-finite measure spaces by restricting to finite-measure subspaces. Indeed, we can recognize a completely positive kernel on a $\sigma$-finite space by showing it is completely positive on every finite-measure subspace.  However, it is a well-known frustration that such a procedure is ineffective for the Pólya-type cones, even in finite dimensions; this is related to the obstruction~\cite[Theorem 3]{Laurent2023ExactnessGraph}.

\begin{exmp}%
  \label{exmp:restriction-breaks-Polya}
  Fix an integer~$n \geq 1$, and define~$\Res: \Sym(n) \to \Sym(n-1)$ as the restriction to the principal submatrix indexed by~$[n-1]$. Its adjoint lifts a matrix by appending a row and a column of zeros.

  Let~$r \in \N$ and~$A \in \Sym(n-1)$, and assume there exists~$A_{i,j} < 0$. Then,
  \[
    (\T_{r} \Res^* A)_{i, j, n, \ldots, n} = \bigl(2r!/(r+2)!\bigr)A_{i,j} < 0.
  \]
  So, the lifting~$\Res^*A$ is in~$C_r([n])$ if and only if~$A \geq 0$.

  A similar problem occurs on the dual side. Take~$A \in (C_r([n]))^*$ and~$B$ in~$\Sym(n-1)$. Then~$\langle \Res A, B\rangle = \langle A, \Res^* B\rangle$, so, by the above and Theorem~\ref{thm:polar-identities},~$\langle \Res A, B\rangle \geq 0$ if and only if~$B \geq 0$. It follows that~$\Res A \in (C_r([n-1]))^*$ if and only if~$r = 0$.
\end{exmp}

A greater plight renders the objection raised by Example~\ref{exmp:restriction-breaks-Polya} irrelevant. Applied to an infinite measure space~$(V, \A, \mu)$, except in degenerate cases, the Pólya-type cones fail to approximate anything more than the nonnegative orthant. If~$\P$ is a partition of~$V$ into sets of finite measure and~$|\P^+| = |\N|$, the codomain of the operator~$A_{\P}$ on~$L^2(V^2)$ is isomorphic to~$l^{2}(\N^2)$. Example~\ref{exmp:infinite-measure-breaks-Polya} shows we cannot simply copy-paste the definition of the Pólya-type cones to this setting.

\begin{exmp}%
  \label{exmp:infinite-measure-breaks-Polya}
  Consider~$l^2_{\sym}(\N^2)$ with the counting measure, and fix~$r \in \N$. Define the operator~$\T_r: l^2_{\sym}(\N^2) \to l^{\infty}_{\sym}(\N^{r+2})$ by
  \[
    (\T_ra)_{i_1, \ldots, i_{r+2}} = (1/(r+2)!)\sum_{\pi \in \mathfrak{S}_{r+2}}a_{i_{\pi 1}, i_{\pi 2}}.
  \]
  Take~$a$ such that~$\T_ra \geq 0$. Since~$a$ is square summable, for all~$\epsilon > 0$ and all~$j \in \N$ there is an~$n^{\epsilon}_j$ such that if~$i \geq n^{\epsilon}_j$, then~$a_{i, j} \leq \epsilon$. Take~$i, j \in \N$,~$\epsilon > 0$ and~$\epsilon' = 2\epsilon/\bigl(r(r+3)\bigr)$. Choose~$k_1$, ...,~$k_r \in \N$ such that
  \[
    \begin{split}
      k_1 &\geq \max \{n^{\epsilon'}_i, n^{\epsilon'}_j\},\\
      k_2 &\geq \max \{n^{\epsilon'}_i, n^{\epsilon'}_j, n^{\epsilon'}_{k_1}\},\\
      &\vdotswithin{\geq}\\
      k_r &\geq \max \{n^{\epsilon'}_i, n^{\epsilon'}_j, n^{\epsilon'}_{k_1}, \ldots, n^{\epsilon'}_{k_{r-1}}\}.
  \end{split}\]
  Then,
  \[
    \begin{split}
      0 &\leq \bigl((r+2)!/(2r!)\bigr)(\T_ra)_{i,j,k_1 ,\ldots, k_r}\\
      &= a_{i,j} + \sum_{l \in [r]}\biggl( a_{i,k_l} + a_{j,k_l} + \sum_{l < l' \leq r}a_{k_l,k_{l'}}\biggr)\\
      &\leq a_{i, j} + \bigl(r(r+3)/2\bigr)\epsilon' = a_{i,j} + \epsilon,
  \end{split}\]
  so~$a$ itself is nonnegative. Thus, for all~$r \geq 0$,~$\T_r^{-1}l^2_{\sym}(\N^{r+2})_{\geq 0} = l^2_{\sym}(\N^2)_{\geq 0}$.
\end{exmp}

To overcome these difficulties, we define a slightly different set of cones, which offers an outer approximation of the completely positive cone. Compared with the approximation by Pólya-type cones on a finite measure space, the approximation defined in this section is strictly stronger.

Let~$(V, \A, \mu)$ be a countably generated~$\sigma$-finite measure space, and fix an integer~$k \geq 2$. Denote the set of measurable subsets of nonzero and finite measure by~$\A_{\fin}$\index{ Afin@$\A_{\fin}$}. Write for all~$A \in \A_{\fin}$ the restriction to~$A$ as
\index{ ResA@$\Res_A$}\[
  \Res_A : L^2_{\sym}(V, k) \to L^2_{\sym}(A, k).
\]
For~$T \in L^2_{\sym}(A, k)$, the tensor~$\Res_A^*T$ agrees with~$T$ on~$A$ and is~$0$ everywhere else.

For~$r \in \N$, define the cone
\index{ CrVk@$\bC_r(V, k)$}\[
  \bC_r(V, k) = \ccone \bigcup_{A \in \A_{\fin}} \Res_A^*C_r(A, k);
\]
in words:~$\bC_r(V, k)$ is the cone generated by tensors~$T \in L^2_{\sym}(V, k)$ for which there is an~$A \in \A_{\fin}$ and a~$K \in C_r(A, k)$ such that~$T(v) = K(v)$ for all~$v \in A^k$, and~$T(v) = 0$ everywhere else. Theorem~\ref{thm:polar-identities} shows
\index{ CrVk*@$\bC_r(V, k)^*$}
\begin{equation}%
  \label{eqn:breve_C*}
  \bC_r(V, k)^* = \bigcap_{A \in \A_{\fin}} \Res_A^{-1} C_r(A, k)^*;
\end{equation}
that is,~$T \in \bC_r(V, k)^*$ if and only if~$\Res_A(T) \in C_r(A,k)^*$ for all~$A \in \A_{\fin}$. Since~$\bC_r(V, k)$ is contained in~$C_r(V, k)$ if~$\mu(V)< \infty$, if~$(V, \A, \mu)$ satisfies the conditions of Theorem~\ref{thm:Polyas-theorem-finite-measure}, then~$\bigcap_r \bC_r(V, k)^* = \CP(V, k)$. So, on countably generated finite measure spaces the~$\bC_r(V, k)$ give a tighter outer approximation of~$\CP(V, k)$ than the Pólya-type cones.

\begin{theorem}%
  \label{thm:breve-Polya-for-sigma-finite-measure-spaces}
  If~$(V, \A, \mu)$ is a countably generated $\sigma$-finite measure space, then, for every integer~$k \geq 2$,
  \[
    \bC_0(V, k)^* \supseteq \bC_1(V, k)^* \supseteq \cdots \supseteq \CP(V, k)\qquad \text{and}\qquad \CP(V, k) = \bigcap_{r \in \N} \bC_r(V, k)^*.
  \]
\end{theorem}
\begin{proof}
  The inclusions~$\bC_r(V, k)^* \supseteq \bC_{r+1}(V, k)^*$ for all~$r \in \N$ follow from the corresponding inclusions~$C_r(A,k)^* \supseteq C_{r+1}(A, k)^*$ for all~$A \in \A_{\fin}$.

  By Lemma~\ref{lem:sufficient-conition-finite-rank-approximation} and Theorem~\ref{thm:Polyas-theorem-finite-measure},
  \[
    \bigcap_{r \in \N} \bC_r(V, k)^* = \bigcap_{A \in \A_{\fin}} \Res_A^{-1} \bigcap_{r \in \N} C_r(A, k)^* = \bigcap_{A \in \A_{\fin}} \Res_A^{-1} \CP(A, k),
  \]
  so, if~$\CP(V, k) = \bigcap_{A} \Res_A^{-1} \CP(A, k)$, the theorem follows.

  If~$\phi \in L^2(V^k)_{\geq 0}$, then for all~$A \in \A_{\fin}$,~$\Res_A(\phi^{\otimes k}) = (\Res_A\phi)^{\otimes k}$. Since for all~$A$,~$\Res_A \phi \in L^2(A)_{\geq 0}$, this shows that~$\CP(V, k) \subseteq \bigcap_{A} \Res_A^{-1}\CP(A, k)$.

  On the other hand,~$\Res_A^*(\phi^{\otimes k}) = (\Res_A^*\phi)^{\otimes k}$ for all~$\phi \in L^2(A)_{\geq 0}$, so that~$\Res_A^* \CP(A, k) \subseteq \CP(V, k)$. Make~$\A_{\fin}$ into a directed set by equipping it with the opposite of the inclusion relation. Then, if~$T \in \COP(V, k)$ and if~$K \in \bigcap_{A} \Res_A^{-1} \CP(A, k)$, then,
  \[
    \langle K, T\rangle = \lim_{A \in \A_{\fin}} \langle \Res_A K, \Res_A T\rangle = \lim_{A \in \A_{\fin}} \langle \Res_A^*\Res_A K,  T\rangle \geq 0.
  \]
  Therefore,~$\bigcap_{A} \Res_A^{-1} \CP(A, k) \subseteq \CP(V, k)$, and thus these cones are equal, from which the theorem follows.
\end{proof}

\chapter[Complete positivity and symmetry]{Complete positivity and symmetry}%
\label{ch:complete-positivity-under-symmetry}
Problems~\ref{it:problem-1}--\ref{it:problem-4} in the introduction have a lot of symmetry: for example, take~$n \geq 1$ and integer and let~$S^{n-1} = \{\, x \in \R^n : \|x\| = 1\, \}$\index{ Sn1@$S^{n-1}$} be the unit sphere in~$\R^n$. If we take a set~$S \subseteq S^{n-1}$ that contains no orthogonal pairs, then neither does~$TS$ for any~$T \in \ortho(n)$, the orthogonal group on~$\R^n$. Moreover, this group action leaves the measure unchanged.

We expect therefore that the optimal value of such an optimization problem remains unchanged when we restrict the feasible region to feasible solutions that are invariant under a group action. Restricting the feasible region to invariant solutions decreases the size of such a problem significantly, and is essential for tractability. The goal of this section is to describe an outer approximation of the cone of group-invariant completely positive tensors.

This hinges on the existence of a linear averaging operator, one example of which we have already seen: the operator~$\Av_{\mathfrak{S}_{k+r}}$. Such operators deserve careful consideration. We start this chapter with some harmonic analysis and an in-depth investigation of these averaging operators.

\section{Harmonic analysis}%
\label{sec:prel-harmonic-analysis}
We largely follow Folland's book on harmonic analysis~\cite{Folland2016AAnalysis}. See Section~\ref{sec:app-invariant-measures} of the appendix for more details on invariant measures.

When we denote a group multiplicatively, we call its unit~$1$. When we denote it additively, we call its unit~$0$. All groups we study are locally compact and Hausdorff. To avoid technicalities, we also assume they are $\sigma$-compact: a \defi{$\sigma$-compact group}\index{ s compact@$\sigma$-compact} is a countable union of compact sets. Recall that any locally compact group has a left-invariant measure, the Haar measure, and is unimodular if and only if the Haar measure is also right-invariant. If~$\Gamma$ is a compact group, the Haar measure~$\mu(\Gamma)$ is finite, and we normalize~$\mu(\Gamma)=1$, unless stated otherwise.

\subsection*{Homogeneous spaces}%
\label{subsec:prel-hom-space}
Let~$\Gamma$ be a $\sigma$-compact unimodular locally compact group. An~\defi{action}\index{action} of~$\Gamma$ on a topological space~$V$ is a continuous function~$\Gamma \times V \to V$,~$(\gamma,v) \mapsto \gamma v$, such that~$v \mapsto \gamma v$ is a homeomorphism of~$V$ for all~$\gamma \in \Gamma$, and~$\gamma(\zeta v) = (\gamma \zeta)v$ for all~$\gamma, \zeta \in \Gamma$ and~$v \in V$. A space with an action of~$\Gamma$ is called a~\defi{$\Gamma$-space}\index{G space@$\Gamma$-space}. For~$V$ a $\Gamma$-space and~$v \in V$, call~$\Orb(v) = \{\, \gamma v : \gamma \in \Gamma\, \}$\index{ Orbv@$\Orb(v)$} the~\defi{orbit}\index{orbit} of~$v$ and~$\Stab(v) = \{\, \gamma \in \Gamma : \gamma v = v\, \}$\index{ Stabv@$\Stab(v)$} the~\defi{stabilizer}\index{stabilizer} of~$v$. A~$\Gamma$-space~$V$ is called~\defi{homogeneous}\index{homogeneous} if for all~$v, w \in V$ there exists~$\gamma \in \Gamma$ such that~$\gamma v = w$. In other words,~$V$ is a homogeneous space if and only if~$V$ has only one orbit.

A function~$f : V \to W$ with~$V$ a $\Gamma$-space is~\defi{$\Gamma$-invariant}\index{G invariant function@$\Gamma$-invariant function}, or just~\defi{invariant}, if for all~$\gamma \in \Gamma$,~$f(\gamma v) = f(v)$. If~$W$ is also a $\Gamma$-space, the function~$f$ is called~\defi{$\Gamma$-equivariant}\index{G equivariant function@$\Gamma$-equivariant function} if, for all~$\gamma \in \Gamma$ and~$v \in V$,~$f(\gamma v) = \gamma f(v)$.

Let~$V$ be a homogeneous $\Gamma$-space. A~$v_0 \in V$ defines a $\Gamma$-equivariant map~$\phi : \Gamma \to V$ through~$\gamma \mapsto \gamma v_0$, and a quotient map~$p: \Gamma \to \Gamma / \Stab(v_0)$ through~$\gamma \mapsto \gamma \Stab(v_0)$. The stabilizer of~$v_0$ is a closed subgroup of~$\Gamma$. Then,~$\phi$ induces a $\Gamma$-invariant homeomorphism~$\Phi : \Gamma / \Stab(v_0) \to V$ such that~$\Phi \smallcirc p = \phi$; this depends on~$\Gamma$ being $\sigma$-compact. Thus, for us, a homogeneous space is always the quotient of~$\Gamma$ by a closed subgroup~$H$.

Let~$V$ be identified with~$\Gamma / H$, with~$H$ closed and unimodular. Given a Haar measure~$\mu$ on~$\Gamma$ and a Haar measure~$\nu$ on~$H$, Section~\ref{sec:app-invariant-measures} of the appendix describes a $\Gamma$-invariant measure~$\omega$ on~$\Gamma / H$. If~$H$ is compact, we can assume it is equal to the pushforward of~$\mu$ under~$p$, and we call it the~\defi{quotient measure}\index{quotient measure} of~$\mu$ under~$p$. If~$H$ is not compact, we only have the equation
\begin{equation}%
  \label{eqn:quotient-measure-is-invariant-Radon-measure}
  \int_{\Gamma} f(\gamma)\, d\mu(\gamma) = \int_{\Gamma / H} \int_H f(v \zeta)\, d\nu(\zeta)d\omega(p(v)),
\end{equation}
for all~$f \in L^1(\Gamma)$.

If~$H$ is compact and~$\Gamma / H$ is equipped with the quotient measure, then, for all~$1 \leq p \leq \infty$ and integer~$k \geq 1$, let~$\gamma f(v_1, \ldots, v_k) = f(\gamma^{-1}v_1, \ldots, \gamma^{-1}v_k)$. This defines an action~$(\gamma, f) \mapsto \gamma f$ on~$L^p((\Gamma / H)^k)$, called the~\defi{diagonal action}\index{diagonal action} of~$\Gamma$ on~$L^2((\Gamma/H)^k)$. Indeed, for all~$1 \leq p \leq \infty$, all~$f \in L^p((\Gamma/H)^k)$, and all~$\gamma \in \Gamma$, we have~$\|f\|_p = \|\gamma f\|_p$; if~$p = \infty$ this follows from homogeneity of~$\Gamma / H$, and if~$p < \infty$ this follows from invariance of the measure. So, for all~$\gamma \in \Gamma$ the map~$f \mapsto \gamma f$ is isometric, and in particular continuous with continuous inverse~$f \mapsto \gamma^{-1}f$. Associativity,~$(\zeta \gamma) f = \zeta (\gamma f)$, follows from a direct calculation.

Finally, the quotient map~$p: \Gamma \to \Gamma / H$ is open. Indeed, let~$U \subseteq \Gamma$ be open, then~$p^{-1}(p(U)) = UH = \bigcup_{\zeta \in H} U \zeta$. Since~$U$ is open, each~$U \zeta$ is open, and their union is as well, so~$p(U)$ is open in~$\Gamma/H$. Since open and continuous surjections send compact sets to compact sets, bases to bases and local bases to local bases, this implies that if~$\Gamma$ is locally compact,~$\Gamma / H$ is also locally compact, if~$\Gamma$ is $\sigma$-compact,~$\Gamma/H$ is $\sigma$-compact, and if~$\Gamma$ is second countable,~$\Gamma / H$ is second countable.

\subsection*{Cross-correlations and functions of positive type}
Let~$\Gamma$ be a $\sigma$-compact unimodular locally compact group with Haar measure~$\mu$, and let~$V$ be a homogeneous $\Gamma$-space equipped with the quotient map~$p : \Gamma \to V$ such that~$p(1) = v_0$. Let~$f$ and~$g$ be two measurable functions on~$V$. Their~\defi{convolution}~$f*g$\index{convolution}\index{ 000@$*$} is given by
\[
  f * g (v) = \int_{\Gamma} f(\gamma v_0) g(\gamma^{-1} v)\, d\mu(\gamma).
\]
Although it is standard to work with the convolution of~$f$ and~$g$, in this exposition the~\defi{cross-correlation}\index{correlation!cross-correlation}\index{ 000@$\cor$} is more natural. It is given by
\[
  f \cor g (v) = \int_{\Gamma} f(\gamma^{-1} v_0) g(\gamma^{-1} v)\, d\mu(\gamma).
\]
By invariance of the measure---i.e. by unimodularity of the group---if clarity demands so, we may drop the~${}^{-1}$ in the formula. Call the cross-correlation~$f \cor f$ of a function~$f$ with itself the~\defi{auto-correlation}\index{correlation!auto-correlation} of~$f$.

For the convolution, we have the following inequalities.
\begin{itemize}
  \item If~$1 \leq p \leq \infty$,~$f \in L^1(V)$, and~$g \in L^p(V)$, then~$\|f * g\|_p \leq \|f\|_1 \|g\|_p$, hence~$f * g, g * f \in L^p(V)$.
  \item If~$1 \leq p, q \leq \infty$ such that~$1/p + 1/q = 1$,~$f \in L^p(V)$, and~$g \in L^q(V)$, then~$f * g \in C_0(V)$ and~$\|f * g\|_{\infty} \leq \|f\|_p \|g\|_q$. We will call this~\defi{Young's inequality}\index{inequality!Young's inequality}.
\end{itemize}
The convolution and the cross-correlation are related; let~$f$ and~$g$ be functions~$V \to \R$ such that~$f * g$ and~$f \cor g$ are well-defined. Denoting~$\overline{f}(\gamma) = f(\gamma^{-1})$, it follows immediately that~$f \cor g = \overline{f} * g$. This means that the inequalities above also hold for the correlation.

Given a $\sigma$-compact unimodular locally compact group~$\Gamma$ with Haar measure~$\mu$, an~\defi{approximate identity}\index{approximate identity} is a net~$(\psi_i)_{i \in I}$ of compactly supported and bounded functions such that
\begin{itemize}
  \item there exists a neighborhood basis~$\mathcal{U}$ of~$1$ such that for every~$U \in \mathcal{U}$ there is an~$i_0$ such that for all~$i \geq i_0$,~$\supp \psi_i \subseteq U$,
  \item $\psi_i \geq 0$ for all~$i$, and
  \item $\int_{\Gamma} \psi_i(\gamma) d\mu(\gamma) = 1$.
\end{itemize}
We moreover always assume that~$\psi_i(\gamma^{-1}) = \psi_i(\gamma)$ for all~$i\in I$ and~$\gamma\in \Gamma$. Then, for all~$f \in L^p(\Gamma)$ with~$1 \leq p < \infty$:
\[
  \lim_{i \in I} \|\psi_i * f - f\|_p = \lim_{i \in I} \|f * \psi_i - f\|_p  = 0,
\]
and likewise with the convolution replaced by the correlation. The same results hold if~$p =\infty$ and~$f$ is uniformly continuous.

A~\defi{function of positive type}\index{positive type} on~$\Gamma$ is a function~$\phi \in L^{\infty}(\Gamma)$ such that for all~$f \in L^1(\Gamma)$:
\[
  \int_{\Gamma} (f \cor f)(\gamma) \phi(\gamma)\, d\mu(\gamma) \geq 0.
\]
The set of positive-type functions is a closed convex cone in~$L^{\infty}(\R^n)$, which we denote by~$\PT(\Gamma)$. A function of positive type is continuous almost everywhere~\cite[Corollary 3.21]{Folland2016AAnalysis}, so we always assume them to be continuous. Moreover, if~$f \in L^2(\Gamma)$, then~$f \cor f \in \PT(\Gamma)$~\cite[Corollary 3.16]{Folland2016AAnalysis}.

\section{Averaging}%
\label{sec:averaging}
Fix a $\sigma$-compact unimodular locally compact group~$\Gamma$ with Haar measure~$\mu$. When~$\Gamma$ is compact, normalize~$\mu(\Gamma) = 1$. Let~$H$ be a compact subgroup of~$\Gamma$,~$p: \Gamma \to \Gamma/H$ be the quotient map,~$V = \Gamma / H$, and equip~$V$ with the quotient measure~$\nu = p_*\mu$.

We will study two different operators that send a $k$-tensor on~$V$ to an invariant counterpart. The first exists for compact~$\Gamma$ and integer~$k \geq 2$, and is an orthogonal projection of the Hilbert space~$L^2(V^k)$. The second exists only for trace-class kernel, but is defined for any $\sigma$-compact unimodular locally compact~$\Gamma$. We call these operators~\defi{averaging operators}\index{averaging operator}, and denote them by~$\Av_{\Gamma}$\index{ AvG@$\Av_{\Gamma}$}. In the literature, the term~\defi{Reynolds operator}\index{Reynolds operator} is also common.

To define such operators, for~$T \in L^2(V^k)$, the integral formula
\begin{equation}%
  \label{eqn:integral-formula-averaging}
  \Av_{\Gamma}T(v) = \int_{\Gamma} T(g^{-1} v)\, d\mu(g)
\end{equation}
is tempting, but it is a priori unclear whether this is well-defined. Even if~$\Gamma$ is compact, complications occur: the orbits of~$V^k$ under the diagonal action can have measure~$0$.

\begin{exmp}
  Let~$H$ be a compact subgroup of~$\Gamma$ and~$\mu(H)=0$. Then, the orbit~$\Orb(H, H) = \{\, (gH, gH) : g \in \Gamma\, \} \subseteq (\Gamma/H)^2$ also has measure~$0$. Indeed,
  \[
    \nu^2(\Orb(H,H)) = \mu^2\bigl(\{\, (g, g') \in \Gamma^2 : (p(g),p(g')) \in \Orb(H,H)\, \}\bigr)
  \]
  For~$g \in \Gamma$, let~$\xi^{g}(g') = 1$ if~$(p(g), p(g')) \in \Orb(H,H)$ and~$\xi^{g}(g')= 0$ otherwise. Then,~$\xi^g(g') = 1$ if and only if there are~$h_1, h_2 \in H$ such that~$g' = gh_1^{-1}h_2$, so its support is~$gH$, which has measure~$0$. By Tonelli's theorem,
  \[
    \mu^2\bigl( \{\, (g, g') \in \Gamma^2  : (p(g),p(g')) \in \Orb(H,H)\, \} \bigr) = \int_{\Gamma} \mu(\xi^{g})\, d\mu(g) = 0.
  \]
\end{exmp}

Let~$k \geq 2$ be an integer, and denote the space of $k$-tensors on~$V$ that are invariant under the diagonal action by~$L^2_{\sym}(V, k)^{\Gamma}$. The map~$T \mapsto gT$ is an isometry, so if~$T'$ is invariant and~$\|T' - T\|_2 \leq \epsilon$, then
\[
  \|T' - gT\|_2 = \|gT' - gT\|_2 = \|T' - T\|_2 \leq \epsilon.
\]
Thus, the subspace of invariant tensors is closed, so it comes with an orthogonal projection.

If~$\Gamma$ is compact, the continuous tensors lie dense in~$L^2(V^k)$, and on~$C(V^k)$ the integral formula~\eqref{eqn:integral-formula-averaging} defines a bounded linear operator. The continuous extension to~$L^2(V^k)$ of this map yields the first example of an averaging operator; it is the orthogonal projection onto~$L^2(V^k)^{\Gamma}$.

\begin{theorem}%
  \label{thm:averaging-on-compact-group}%
  \index{ AvG@$\Av_{\Gamma}$!for compact group}
  Let~$\Gamma$ be a compact group,~$H$ be a closed subgroup, and~$\Gamma/H$ be equipped with the quotient measure. The continuous extension of the operator on~$C(V^k)$ defined by~\eqref{eqn:integral-formula-averaging} to an operator on~$L^2((\Gamma/H)^k)$ is the orthogonal projection~$\Av_{\Gamma} : L^2((\Gamma/H)^k) \to L^2((\Gamma/H)^k)^{\Gamma}$.
\end{theorem}
\begin{proof}
  Let~$V = \Gamma/H$ and denote the quotient measure on~$V$ by~$\nu$. It is enough to show that the two operators coincide on~$C(V^k)$. If~$T \in C(V^k)$ and~$K \in L^2(V^k)$, then, by Fubini-Tonelli,
  \[
    \begin{split}
      \int_{V^k}\int_{\Gamma}T(g^{-1} v)\, d\mu(g) K(v)\, d\nu^k(v) &= \int_{\Gamma} \int_{V^k} T(g^{-1} v) K(v)\, d\nu^k(v) d\mu(g)\\
      & \leq \|T\|_{\infty}\|K\|_2 < \infty.
    \end{split}
  \]
  Therefore, by the Riesz representation theorem,~$v \mapsto \int_{\Gamma} T(g^{-1} v)\, d\mu(g)$ is in~$L^2(V^k)$. The map~$C(V^k) \to L^2(V^k)$ thus defined is moreover bounded, because both measures are finite.

  The image of this map lies in~$L^2(V^k)^{\Gamma}$, and when~$T$ is already invariant, the integral only adds a factor~$\mu(\Gamma)=1$. So, the extension to~$L^2(V^k)$ is indeed the identity on~$L^2(V^k)^{\Gamma}$ and equal to the orthogonal projection onto~$L^2(V^k)^{\Gamma}$.
\end{proof}

If~$\Gamma$ is not compact, the integral formula applied to a continuous function on~$\Gamma$ is not necessarily square-integrable, even when the function has compact support. So, even though the orthogonal projection~$L^2(V^k) \to L^2(V^K)^{\Gamma}$ exists, it is not defined by~\eqref{eqn:integral-formula-averaging}.

In general, we will be mainly interested in trace-class kernels. On a rank-one operator~$\phi \otimes \psi \in \B^1(L^2(\Gamma))$, Equation~\eqref{eqn:integral-formula-averaging} produces the correlation~$\phi \cor \psi$ by
\[
  \phi \cor \psi(h) = \int_{\Gamma}(\phi \otimes \psi)(g^{-1}, g^{-1}h)\, d\mu(g).
\]
Young's inequality tells us that it is bounded:~$\|\phi \cor \psi\|_{\infty} \leq \|\phi\|_2\|\psi\|_2 < \infty$. The extension of this map to~$\B^1(L^2(V))$ is the second averaging operator we define. This operator is different from the projection~$L^2(\Gamma^2) \to L^2(\Gamma^2)^{\Gamma}$: its range is contained in~$C_0(\Gamma)$, the space of continuous functions that vanish outside compact sets, and the map~$(g, h) \mapsto (\phi \cor \psi)(g^{-1} h)$ is in general not in~$L^2(\Gamma^2)$.

\begin{theorem}%
  \label{thm:avering-on-trace-class}%
  \index{ AvG@$\Av_{\Gamma}$!for trace class}
  Let~$\Gamma$ be a $\sigma$-compact unimodular locally compact group,~$H$ be a compact subgroup, and~$\Gamma / H$ be equipped with the quotient measure. The continuous linear extension of the map~$\phi \otimes \psi \mapsto \phi \cor \psi$ induces a bounded linear operator~$\Av_{\Gamma} : \B^1(L^2(\Gamma/H)) \to C_0(\Gamma/H)$.
\end{theorem}

The proof depends on the following Lemma, which will be of use to us more often.
\begin{lemma}%
  \label{lem:averaged-rank-one-is-continuous}
  Let~$\Gamma$ be a $\sigma$-compact unimodular locally compact group,~$H$ be a compact subgroup, and~$\Gamma / H$ be equipped with the quotient measure. For all~$\phi_1$,~$\ldots$,~$\phi_k \in L^k(\Gamma/H)$, the function
  \[
    (v_1, \ldots, v_k) \mapsto \int_{\Gamma} \phi_1(g^{-1} v_1) \cdots \phi_k(g^{-1} v_k)\, d\nu(g)
  \]
  is continuous and vanishes at infinity.
\end{lemma}
\begin{proof}
  From the definition of the quotient measure it follows that~$\phi \in L^k(\Gamma/H)$ if and only if~$\phi \smallcirc p \in L^k(\Gamma)$. The function
  \[
    (g_1, \ldots, g_k) \mapsto \int_{\Gamma} \phi_1(p(g^{-1} g_1)) \cdots \phi_k(p(g^{-1} g_k))\, d\mu(g)
  \]
  is constant on the cosets of~$H$, so if it is in~$C_0(\Gamma)$, it defines an element of~$C_0(\Gamma/H)$. It is therefore sufficient to take~$H = \{1\}$ by replacing~$\phi_i$ by~$\phi_i \smallcirc p$ for all~$i$. For the sake of readability, we will integrate over~$g$ instead of~$g^{-1}$, as this is equivalent.

  If~$\phi_1, \ldots, \phi_k \in C_c(\Gamma)$, each~$\phi_i$ is left uniformly continuous by~\cite[Proposition 2.6]{Folland2016AAnalysis}, therefore~$(v_1, \ldots, v_k) \mapsto \int_{\Gamma} \phi_1(g v_1) \cdots \phi_k(g v_k)\, d\mu(g)$ is continuous.

  If~$\phi_1, \ldots, \phi_k \in L^k(\Gamma)$, each~$\phi_i$ is approximated in the norm topology by elements of~$C_c(\Gamma)$. The strategy is to prove that a choice of~$\epsilon > 0$ and~$\psi_1, \ldots, \psi_k \in C_c(\Gamma)$ for which~$\|\phi_i - \psi_i\|_k \leq \epsilon$, for all~$i$, give a bound on
  \begin{equation}%
    \label{eqn:difference-of-averages}
    \biggl| \int_{\Gamma} \phi_1(gv_1) \cdots \phi_k(g v_k) - \psi_1(g v_1) \cdots \psi_k(g v_k)\, d\mu(g) \biggr|
  \end{equation}
  that goes to~$0$ as~$\epsilon$ goes to~$0$ and is uniform in~$v_1, \ldots, v_k$. The conclusion then follows.

  Apply the triangle inequality to see that~\eqref{eqn:difference-of-averages} is at most
  \[
    \begin{split}
      \int_{\Gamma} |&\phi_1(g v_1) \cdots \phi_k(g v_k) - \phi_1(gv_1) \cdots \phi_{k-1}(g v_{k-1})\psi_k(g v_k)\\
      &+ \phi_1(gv_1) \cdots \phi_{k-1}(g v_{k-1})\psi_k(g v_k) - \psi_1(g v_1) \cdots \psi_k(g v_k)|\, d\mu(g)\\
      \leq &\int_{\Gamma} |\phi_1(g v_1) \cdots \phi_{k-1}(g v_{k-1})| |\phi_{k}(g v_k) - \psi_k(g v_k)|\, d\mu(g)\\
      &+ \int_{\Gamma} | \phi_1(g v_1) \cdots \phi_{k-1}(g v_{k-1}) - \psi_1(g v_1) \cdots \psi_{k-1}(g v_{k-1})| | \psi_k(g v_k)|\, d\mu(g).
    \end{split}
  \]
  Repeatedly applying this method shows that~\eqref{eqn:difference-of-averages} is bounded from above by a sum of~$k$ terms of the form
  \begin{equation}%
    \label{eqn:telescoping-terms}
    \int_{\Gamma}|\phi_i(g v_i) - \psi_i(g v_i)| |\xi_1(g)|\cdots | \xi_{k-1}(g)|\, d\mu(g),
  \end{equation}
  where each~$\xi_j$ is one of the functions~$g \mapsto \phi_l(g v_l)$ or~$g \mapsto \psi_l(g v_l)$.

  Recursive application of Hölder's inequality reveals that each term of the form~\eqref{eqn:telescoping-terms} is at most
  \[
    \|\phi_i - \psi_i\|_k \|\xi_1\|_k \cdots \|\xi_{k-1}\|_k \leq \epsilon M^{k-1},
  \]
  where~$M = \epsilon + \max\{\|f_1\|_k, \ldots, \|f_k\|_k\}$. Since there are~$k$ such terms,~\eqref{eqn:difference-of-averages} is bounded from above by~$\epsilon k M^{k-1}$. This proves
  \[
    (v_1, \ldots, v_k) \mapsto \int_{\Gamma} \phi_1(g^{-1} v_1) \cdots \phi_k(g^{-1} v_k)\, d\mu(g)
  \]
  is a uniform limit of continuous functions with compact support and hence lies in~$C_0(\Gamma)$.
\end{proof}

\begin{proof}[Proof of Theorem~\textup{\ref{thm:avering-on-trace-class}}]
  As in the proof of Lemma~\ref{lem:averaged-rank-one-is-continuous}, it is enough to prove the theorem for~$H = \{1\}$.

  Define~$\Av_{\Gamma}$ on finite-rank operators by~$\Av_{\Gamma}(\sum_{i = 1}^n \phi_i \otimes \psi_i) = \sum_{i = 1}^n \phi_i \cor \psi_i$. This is linear, since the correlation is linear in both its arguments. Young's inequality for convolutions shows it is separately continuous in the~$\phi_i$ and~$\psi_i$ under the supremum norm on the codomain. Then,~$\Av_{\Gamma}$ is independent of an expansion of the~$\phi_i$ and~$\psi_i$ over a basis, and so well-defined, and~$\Av_{\Gamma}$ extends linearly and continuously to~$\B^1(L^2(\Gamma))$.

  If~$K \in \B^1(L^2(\Gamma))$ with singular value decomposition~$K = \sum_{i\in\N} \lambda_i \phi_i \otimes \psi_i$, then
  \[
    \|\Av_{\Gamma}K\|_{\infty} \leq \sum_{i \in \N} |\lambda_i| \|\phi_i \cor \psi_i\|_{\infty} \leq \sum_{i \in \N} |\lambda_i| \|\phi_i\|_2\|\psi_i\|_2 = \sum_{i \in \N} |\lambda_i| < \infty.
  \]
  Applying the same argument to the tails~$\sum_{i \geq n} \lambda_i \phi_i \otimes \psi_i$ shows that the series~$\sum_{i \in \N} \lambda_i \phi_i \cor \psi_i$ converges absolutely and uniformly. Lemma~\ref{lem:averaged-rank-one-is-continuous} implies that~$\Av_{\Gamma}(K)$ is continuous and vanishes at infinity, which proves the theorem.
\end{proof}

When~$\Gamma$ is compact, the two averaging operators we defined coincide on the trace class. Every continuous positive semidefinite kernel on a compact space is of trace class, which is a corollary of Mercer's theorem. So, in this case, the image of the trace class under~$\Av_{\Gamma}$ contains all invariant continuous positive semidefinite kernels.

After doing all this work, it is perfectly fine to interpret~\eqref{eqn:integral-formula-averaging} in the almost-everywhere sense. Indeed, under the conditions of Theorem~\ref{thm:averaging-on-compact-group} or Theorem~\ref{thm:avering-on-trace-class}, it is valid to use this formula in expressions of the form~$\langle \Av_{\Gamma} T, F\rangle$.

\sectionbreak

My main objection to the integral formula~\eqref{eqn:integral-formula-averaging} was that~$T \in L^p(V^k)$ might not be well-defined on orbits, as they might have measure~$0$. However, even without the above discussion, there is some sense to the integral formula when~$T$ is a trace-class kernel and~$\Gamma$ is second countable. Picking points~$v, w \in V$ defines the kernel~$K(g, h) = T(gv, hw)$ almost everywhere on~$\Gamma^2$. It is trace class, for example because
\[
  \|K\|_{\B^1} = \sup_{(\phi_n)_{n\in\N}, (\psi_n)_{n \in \N}} \|(\langle K\phi_n, \psi_n\rangle)_{n\in\N}\|_1 = \|T\|_{\B^1},
\]
with the supremum over pairs of orthonormal bases of~$L^2(V)$~\cite[Proposition 2.6]{Simon2005TraceApplications}. Brislawn showed~\cite{Brislawn1991TraceableSpaces}---using martingales---that there exists a function~$\widetilde{K}$ almost-everywhere equal to~$K$ for which~$\Tr(K) = \int_{\Gamma}  \widetilde{K}(g,g)\, d\mu(g)$, which is as close to~\eqref{eqn:integral-formula-averaging} as we can wish for.

\section{A group-invariant outer approximation}%
\label{sec:appr-under-symmetry}
We have gathered all we need to discuss group-invariant analogues of the completely positive cone and its approximations. Our approach is by reduction to Theorems~\ref{thm:Polyas-theorem-finite-measure} and~\ref{thm:breve-Polya-for-sigma-finite-measure-spaces}. The Borel algebra is countably generated if and only if the topology is second countable, and together with local compactness this implies the space is $\sigma$-compact, so second countable unimodular locally compact groups satisfy the conditions of Theorems~\ref{thm:Polyas-theorem-finite-measure} and~\ref{thm:breve-Polya-for-sigma-finite-measure-spaces} and the results in Section~\ref{sec:averaging}. Parallel to Chapter~\ref{ch:completely-positive-cone}, we first discuss compact groups, then we extend to second countable unimodular locally compact groups.

Recall that if~$\Gamma$ is second countable and locally compact, then for every subgroup~$H$,~$\Gamma/H$ is second countable and locally compact, see Section~\ref{subsec:prel-hom-space}.
\begin{theorem}%
  \label{thm:invariant-polyas-thm-compact-groups}
  If~$\Gamma$ is a second countable compact group, if~$H$ is a compact subgroup, and if~$V = \Gamma/H$ is equipped with the quotient measure, then
  \[
    \Av_{\Gamma}\CP(V, k) = \bigcap_{r \in \N} \Av_{\Gamma} C_r(V, k)^*,
  \]
  and all cones in the expression are closed.
\end{theorem}
\begin{proof}
  Recall Theorem~\ref{thm:averaging-on-compact-group}: the operator~$\Av_{\Gamma} : L^2_{\sym}(V, k) \to L^2_{\sym}(V, k)^{\Gamma}$ is an orthogonal projection. If~$\cC \subseteq L^2_{\sym}(V, k)$ is a closed convex cone, then~$\Av_{\Gamma} \cC$ is a closed convex cone in~$L^2_{\sym}(V, k)^{\Gamma}$, and its dual in~$L^2_{\sym}(V, k)^{\Gamma}$ under the inner product is~$\Av_{\Gamma} \cC^*$. Indeed, for $k$-tensors~$T, T'$,~$\langle T, T' \rangle = \langle \Av_{\Gamma} T, \Av_{\Gamma} T'\rangle$, so if~$T \in \cC$ then~$\langle \Av_{\Gamma} T, \Av_{\Gamma}T'\rangle \geq 0$ if and only if~$T' \in \cC^*$.

  This shows that the dual of~$\Av_{\Gamma} \CP(V, k)$ is~$\Av_{\Gamma}\COP(V, k)$. Moreover, by Theorem~\ref{thm:Polyas-theorem-finite-measure} and~$\Av_{\Gamma}$ being continuous, closed and linear,
  \[
    \Av_{\Gamma} \COP(V, k) = \Av_{\Gamma} \ccone \bigcup_{r \in \N} C_r(V, k) =  \ccone \bigcup_{r \in \N} \Av_{\Gamma} C_r(V, k).
  \]
  Taking the dual, by the discussion above and Theorem~\ref{thm:polar-identities},
  \[
    \Av_{\Gamma} \CP(V, k) = (\Av_{\Gamma} \COP(V, k))^* = \bigcap_{r \in \N} \Av_{\Gamma} C_r(V, k)^*.\qedhere
  \]
\end{proof}

\sectionbreakafterproof

The situation is more nuanced when~$\Gamma$ is not compact. We restrict the discussion to trace-class kernels on~$\Gamma$, with~$\Gamma$ a second countable unimodular locally compact group.

By Theorem~\ref{thm:avering-on-trace-class},~$\Av_{\Gamma}$ is a bounded operator~$\B^1(L^2(\Gamma)) \to C_0(\Gamma)$. The space~$C_0(\Gamma)$ can be difficult to work with, as it has too few compact sets; we extend the codomain of~$\Av_{\Gamma}$ to~$L^{\infty}(\Gamma)$. Since~$\Av_{\Gamma}$ is bounded, it is continuous with respect to the weak topologies on~$\B^1(L^2(\Gamma))$ and~$L^{\infty}(\Gamma)$, and so also with respect to the weak* topology on~$L^{\infty}(\Gamma)$ under the duality with~$L^1(\Gamma)$, which is weaker.

A warning: although in the compact setting the operator~$\Av_{\Gamma}$ is an orthogonal projection, in the current setting, the range~$\Av_{\Gamma}(\B^1(L^2(\Gamma)))$ is not even closed in~$L^{\infty}(V)$. Indeed, by taking an approximate identity~$(\psi_i)_{i \in I}$ and a compact neighborhood basis~$(V_j)_{j \in J}$, it can be show that~$\Av_{\Gamma}(\psi_i \otimes f\1_{V_j})$ converges to~$f$ under the weak* topology.

Let~$\Av_{\Gamma} : \B^1(L^2(\Gamma)) \to L^{\infty}(\Gamma)$, and define the completely positive cone on~$L^{\infty}(\Gamma)$ by
\index{ CPGinv@$\CP(\Gamma)_{\inv}$}\[
  \CP(\Gamma)_{\inv} = \cl \Av_{\Gamma} (\CP(\Gamma,2) \cap\B^1(L^2(\Gamma))) = \ccone \{\, \phi \cor \phi : \phi \in L^2(\Gamma)_{\geq 0}\, \},
\]
with closure under the weak* topology. The latter equality follows directly from the definition of~$\CP(V, 2)$. Define the copositive cone on~$L^1(\Gamma)$ by
\index{ COPGinv@$\COP(\Gamma)_{\inv}$}\[
  \COP(\Gamma)_{\inv} = \{\, \rho \in L^1(\Gamma) : \langle \phi * \rho, \phi\rangle \geq 0 \text{ for all } \phi \in L^2(\Gamma)_{\geq 0}\, \}.
\]
By a direct calculation and Theorem~\ref{thm:polar-identities},~$\COP(\Gamma)_{\inv} = \CP(\Gamma)_{\inv}^*$. Let
\index{ CrGinv@$\bC_r(\Gamma)_{\inv}$}\[
  \bC_r(\Gamma)_{\inv}^* = \cl \Av_{\Gamma} \biggl(\bC_r(\Gamma, 2)^* \cap \B^1(L^2(\Gamma))\biggr),
\]
with closure under the weak* topology.

Denote the Borel algebra of~$\Gamma$ by~$\B$, and the set of Borel sets with finite and nonzero measure by~$\B_{\fin}$. For any function~$f : \Gamma \to \R$ and~$A \in \B_{\fin}$, let~$\convol_Af(g,h) = f(h^{-1}g)$ for all~$g,h\in A$. If~$f \in L^{\infty}(\Gamma)$, this is the kernel of the operator~$\phi \mapsto \phi * f$ restricted to~$A^2$. Its adjoint is~$\convol_A^* = \Av_{\Gamma} \smallcirc \Res_A^*$.

We need one more property to obtain the results we want:~\defi{amenability}\index{group!amenable group}. Amenable $\sigma$-compact locally compact groups are those $\sigma$-compact locally compact groups that have an~\defi{averaging sequence}\index{averaging sequence}~\cite[Proposition 16.14 and 16.16]{Pier1984AmenableGroups}: a sequence of sets~$(A_n)_{n \in \N}$ with~$A \in \B_{\fin}$ such that for every finite subset~$S \subseteq \Gamma$
\[
  \lim_{n \to \infty} \frac{\mu(\bigcap_{s \in S} sA_n)}{\mu(A_n)} = 1.
\]
The convergence is uniform in~$S$ when the sets~$S$ considered are subsets of a fixed compact set. The second statement of the following lemma will be useful later on.
\begin{lemma}%
  \label{lem:local-definition-bC}
  If~$\Gamma$ is a second countable unimodular locally compact group and~$r \in \N$, then
  \[
    \bC_r(\Gamma)_{\inv}^* \subseteq \bigcap_{A \in \B_{\fin}} \convol_A^{-1}C_r(A, 2)^*.
  \]
  If~$\Gamma$ is moreover amenable, then
  \[
    \bC_r(\Gamma)_{\inv}^* \cap \PT(\Gamma) = \bigcap_{A \in \B_{\fin}} \convol_A^{-1}C_r(A, 2)^* \cap \PT(\Gamma),
  \]
  that is, the cones contain the same positive-type functions.
\end{lemma}
\begin{proof}
  Let~$f  = \Av_{\Gamma}K$ with~$K = \sum_{\phi \in \Phi}\lambda_{\phi}\phi \otimes \phi$ trace class, and such that for all~$A \in \B_{\fin}$,~$\Res_AK \in C_r(A, 2)^*$. If~$B \in C_r(A, 2)$ and~$\gamma \in \Gamma$, then~$\gamma B$ is in~$C_r(\gamma A, 2)$. By Fubini's theorem and invariance of the Haar measure, and a change of variables,
  \[
    \begin{split}
      \langle \convol_Af, B\rangle &= \int_{A^2} \Av_{\Gamma}K(h^{-1} g)B(g,h)\, d\mu(g)d\mu(h)\\
      &= \int_{A^2}\sum_{\phi \in \Phi}\lambda_{\phi} \int_{\Gamma}\phi(\gamma^{-1})\phi(\gamma^{-1}h^{-1}g) B(g,h)\, d\mu(\gamma)d\mu(g)d\mu(h)\\
      &= \int_{\Gamma}\int_{A^2}\sum_{\phi \in \Phi}\lambda_{\phi} \phi(\gamma^{-1})\phi(\gamma^{-1}h^{-1}g) B(g,h)\, d\mu(g)d\mu(h)d\mu(\gamma)\\
      &= \int_{\Gamma} \langle  \Res_{\gamma A} K, \gamma B\rangle\, d\mu(\gamma) \geq 0,
    \end{split}
  \]
  and~$\gamma A \in \B_{\fin}$. So, indeed,~$\convol_A f \in C_r(A, 2)^*$ for all~$A \in \B_{\fin}$.

  Since~$\Av_{\Gamma}$ is continuous, it follows that~$\bC_r(\Gamma)^*_{\inv} \subseteq \bigcap_A \convol_A^{-1}C_r(A, 2)^*$ if the latter is weak* closed. For this, it is enough that the maps~$\convol_A$ are bounded linear operators from~$L^{\infty}(\Gamma)$ to~$L^2(A^2)$, since then, for all~$A$,~$\convol_A^{-1}C_r(A, 2)^*$ is closed. For~$f \in L^{\infty}(\Gamma)$,~$\|\convol_A f\|_2 \leq \|f\|_{\infty}\mu(A)^2$, and the result follows.

  To prove the second statement of the lemma, it is enough to prove that any~$f \in \bigcap_{A} \convol_A^{-1}C_r(A,2)^*$ is the limit of a sequence of averages of trace-class kernels in~$\bC(\Gamma,2)^*$.

  Let~$f \in \bigcap_A \convol_A^{-1} C_r(V,2)^* \cap \PT(\Gamma)$; we may assume that~$f$ is continuous. Since we assume that~$\Gamma$ is amenable, take an averaging sequence~$(A_n)_n$, and define for all~$n \in \N$:~$K_n = \Res_{A_n}^*\convol_{A_n}f$, so~$K_n$ is the extension by zeros of~$\convol_{A_n}f$ to~$\Gamma$. First show that~$K_n \in C_r(\Gamma, 2)$ for all~$n$ and that it is trace class.

  To see that~$K_n \in \bC(V,2)^*$, note that for every~$A \subseteq A_n$,
  \[
    \Res_A K_n = \convol_A f \in C_r(A, 2)^*
  \]
  by the first part of this proof. Moreover, by a direct calculation, the extension by zeros of an element of~$C_r(A_n,2)^*$ to a set~$A \supseteq A_n$ is in~$C_r(A,2)^*$, so indeed~$K_n \in \bC(\Gamma,2)^*$.

  Since~$f$ is continuous,~$\convol_{A_n}f$ is so as well. Since~$f$ is positive type,~$K_n$ is positive semidefinite, see for example~\cite[Proposition 3.35]{Folland2016AAnalysis}. Thus,
  \[
    \|K_n\|_{\B^1} = \Tr K_n = \Tr \convol_{A_n}f = \mu(A_n)f(0) < \infty,
  \]
  and~$K_n$ is trace class, so~$K_n \in \bC(\Gamma,2)^* \cap \B^1(L^2(\Gamma))$.

  Define~$f_n = \Av_{\Gamma} K_n / \mu(A_n)$. An application of the Fubini-Tonelli theorem shows that if~$K$ is trace class and positive semidefinite, then~$\Av_{\Gamma} K$ is positive type, so~$f_n$ is positive type.

  It is left to show that the weak* limit~$\lim_n f_n = f$. Take~$\rho \in C_c(\Gamma)$ and let~$C_{\rho} = \supp \rho$, then
  \[
    \begin{split}
      |\langle f_n - f, \rho\rangle| &\leq \int_{\Gamma}|\Av_{\Gamma}K_n(g)/\mu(A_n) - f(g)||\rho(g)|\, d\mu(g)\\
      &\leq \int_{\Gamma} f(g)\biggl|\int_{\Gamma} \frac{\1_{A_n}(\gamma^{-1})\1_{A_n}(\gamma^{-1} g)d\mu(\gamma)}{\mu(A_n)} - 1\biggr| |\rho(g)|\, d\mu(g)\\
      &\leq \|f\|_{\infty}\|\rho\|_1 \int_{C_{\rho}} \biggl| \frac{\mu(A_n \cap g A_n)}{\mu(A_n)} - 1\biggr|\, d\mu(g).
  \end{split}\]
  Since~$\lim_n \mu(A_n \cap gA_n)/\mu(A_n)= 1$ uniformly for~$g \in C_{\rho}$, and~$|\langle f_n -f, \rho\rangle| \to 0$ follows. Since~$C_c(\Gamma)$ is dense in~$L^1(\Gamma)$, this proves the theorem.
\end{proof}

\begin{theorem}%
  \label{thm:Polyas-theorem-for-invariant-cones}
  If~$\Gamma$ is an amenable unimodular second countable locally compact group, then
  \[
    \CP(\Gamma)_{\inv} = \bigcap_{A\in \B_{\fin}} \convol_A^{-1}\CP(A, 2).
  \]
  Moreover,
  \[
    \bC_0(\Gamma)_{\inv} \supseteq \bC_1(\Gamma)_{\inv} \supseteq \cdots \supseteq \CP(\Gamma)_{\inv}\qquad \text{and}\qquad \CP(\Gamma)_{\inv} = \bigcap_r \bC_r(\Gamma)_{\inv}.
  \]
\end{theorem}
\begin{proof}
  By Lemma~\ref{lem:local-definition-bC} it is enough to prove~$\CP(\Gamma)_{\inv} = \bigcap_{r} \bigcap_{A} \convol_A^{-1}  C_r(A, 2)^*$. Since
  \[
    \bigcap_{r \in \N} \bigcap_{A \in \B_{\fin}} \convol_A^{-1}  C_r(A, 2)^* = \bigcap_{A \in \B_{\fin}} \convol_A^{-1} \bigcap_{r \in \N} C_r(A, 2)^* = \bigcap_{A\in \B_{\fin}} \convol_A^{-1} \CP(A, 2),
  \]
  the result follows from~$\CP(\Gamma)_{\inv} = \bigcap_{A} \convol_A^{-1} \CP(A, 2)$, which is the subject of the remainder of this proof. In what follows, denote~$\cC = \bigcap_{A} \convol_A^{-1} \CP(A, 2)$.

  Let~$\phi \in L^2(\Gamma)_{\geq 0}$,~$A \in \B_{\fin}$, and~$K \in \COP(A, 2)$. Then, by a change of variables and Fubini-Tonelli,
  \[
    \begin{split}
      \langle \convol_A(\phi \cor \phi), K\rangle &= \int_{A^2} \int_{\Gamma} \phi(\gamma^{-1}) \phi(\gamma^{-1} g^{-1} h)\, d\mu(\gamma) K(g, h)\, d\mu(g)d\mu(h)\\
      &= \int_{A^2} \int_{\Gamma} \phi(\gamma^{-1}g) \phi(\gamma^{-1} h)\, d\mu(\gamma) K(g, h)\, d\mu(g)d\mu(h)\\
      &= \int_{\Gamma} \int_{A^2} \phi(\gamma^{-1}g) \phi(\gamma^{-1} h) K(g, h)\, d\mu(g)d\mu(h)d\mu(\gamma)\\
      &\geq 0.
    \end{split}
  \]
  Thus,~$\phi \cor \phi$ lies in~$\cC$, which is closed, so~$\CP(\Gamma)_{\inv} \subseteq \cC$ follows.

  To prove the other inclusion, show
  \begin{equation}%
    \label{eqn:inclusion-COP-in-C*}
    \COP(\Gamma)_{\inv} \subseteq \bigcap_{A \in \B_{\fin}} \convol_A^{-1} \COP(A, 2) \subseteq \cC^*.
  \end{equation}
  Indeed, the result then follows from~$\cC^* \subseteq \CP(\Gamma)^*_{\inv}$ and Theorem~\ref{thm:polar-identities}.

  For the first inclusion of~\eqref{eqn:inclusion-COP-in-C*}, take~$\rho \in \COP(\Gamma)_{\inv}$, take~$A \in \B_{\fin}$ and take~$\phi \in L^2(A)_{\geq 0}$. Then, by a change of variables and Fubini-Tonelli,
  \[
    \begin{split}
      \langle \convol_A \rho, \phi \otimes \phi \rangle &= \int_{\Gamma^2} \rho(h^{-1} g)\Res_A^*\phi(g) \Res_A^*\phi(h)\, d\mu(g)d\mu(h)\\
      &= \int_{\Gamma} \int_{\Gamma} \rho(g) \Res_A^*\phi(hg) \Res_A^*\phi(h)\, d\mu(g)d\mu(h)\\
      &= \langle \rho, \Res_A^*\phi\cor \Res_A^*\phi \rangle \geq 0,
    \end{split}
  \]
  and~$\convol_A \rho \in \COP(A, 2)$.

  For the second inclusion in~\eqref{eqn:inclusion-COP-in-C*}, note that for~$\rho \in L^1(\Gamma)$ and~$A \in \B_{\fin}$,
  \[
    \int_{A^2} \convol_A \rho(g,h)\, d\mu(g)d\mu(h) = \int_{\Gamma} \rho(g) (\1_A \cor \1_A)(g)d\mu(g).
  \]
  Take an averaging sequence~$(A_n)_{n \in \N}$ of~$\Gamma$, and let~$\rho_n = ((\1_{A_n} \cor \1_{A_n})/\mu(A_n)) \rho $ almost everywhere. Then,~$\lim_n \rho_n(g) = \rho(g)$ almost everywhere, and for all~$n$,~$|\rho_n| \leq |\rho|$. So, by the Lebesgue dominated convergence theorem
  \[
    \lim_{n \to \infty}\frac{1}{\mu(A_n)} \int_{A_n^2} \convol_{A_n}\rho(g,h)\, d\mu(g)d\mu(h) = \int_{\Gamma} \rho(g)\, d\mu(g).
  \]

  Take~$\rho \in \bigcap_A \convol_A^{-1}\COP(A, 2)$ and~$f \in \cC$. For all~$A$,~$\convol_A (\rho f) = (\convol_A \rho) (\convol_A f)$, so apply the above to~$\rho f$ to obtain
  \[
    \langle \rho, f\rangle = \lim_{n \to \infty} \frac{1}{\mu(A_n)}\langle \convol_{A_n} \rho, \convol_{A_n} f\rangle \geq 0.
  \]

  In summary,
  \[
    \CP(\Gamma)_{\inv}^* = \COP(\Gamma)_{\inv} \subseteq \cC^* \subseteq \CP(\Gamma)_{\inv}^*,
  \]
  and since all cones are closed,~$\cC = \CP(\Gamma)_{\inv}$.
\end{proof}

\sectionbreakafterproof

The work by Pier~\cite{Pier1984AmenableGroups} contains many characterizations of amenable groups. All locally compact Abelian groups and all compact groups are amenable, and so are products of amenable groups.

The special orthogonal group~$\mathrm{SO}(2)$ is also amenable under the discrete topology, which Matolcsi, Ruzsa, Varga, and Zsámboki recently exploited to prove that the fractional chromatic number of the plane is at least~4~\cite{matolcsi2025The4}. They could not extend their method to dimensions larger than~$2$, because the group~$\mathrm{SO}(n)$ is not amenable under the discrete topology if~$n > 2$. This is closely related to the Banach-Tarski paradox, which prompted the study of amenability.

\part{Finite measure spaces}\label{part:measurable-setting}
\chapter[Completely positive formulations]{Completely positive formulations}%
\label{ch:completely-positive-formulations-meas-ind-num}
In Part~\ref{part:measurable-setting} we will study optimization problems of Class~\ref{it:problem-1} from the introduction. They all have the following form: given a finite measure space~$(V, \A, \mu)$, an integer~$k \geq 2$, and a symmetric set~$E \subseteq V^k$, find
\begin{equation}%
  \label{eqn:extremal-problem-finite-measure}
  \sup\{\, \mu(I) : I \text{ measurable and } I^k \cap E = \emptyset\, \}.
\end{equation}
For general~$V$ and~$E$, there is little we can say about such a problem. We treat two special---but still quite broad---subcategories. In this chapter, for these subcategories, we present exact completely positive formulations of~\eqref{eqn:extremal-problem-finite-measure}. These results build heavily on the arguments developed by DeCorte, Oliveira, and Vallentin~\cite{DeCorte2022CompleteSets}.

We will call the first subcategory the~\defi{thick setting}\index{thick setting}. A problem is in the thick setting if~$E$ is a thick set; as the definition is a bit technical, we postpone it until Section~\ref{sec:density-and-thick-sets}. Informally, a set in~$V^k$ is thick if it adequately described by sets with positive $\mu^k$-measure.

\begin{exmp}%
  \label{exmp:thick-sets-1}
  For finite~$V$ equipped with the counting measure, any~$E \subseteq V^k$ is thick, and we find ourselves practicing combinatorics. Thus, the thick setting extends the finite-graph setting.
\end{exmp}

\begin{exmp}%
  \label{exmp:thick-sets-2}
  An example of a graph in the thick setting that has an infinite vertex set is as follows. If~$V = [0,1]$ with the Lebesgue measure~$\lambda$, then the set
  \[
    E_{\geq} = \{\, (x, y) \in [0,1]^2 : |x - y| \geq 1/2\, \}
  \]
  is a thick set. It is thick because to prove that a point is not in~$E_{\geq}$, it is enough to find an open neighborhood whose intersection with~$E_{\geq}$ has $\lambda^2$-measure~$0$. Figure~\ref{fig:thick-R2} demonstrates that this is in contrast with a set like
  \[
    E_{=} = \{\, (x, y) \in [0,1]^2 : |x-y| = 1/2\, \}.
  \]
  Indeed,~$\mu^2(E_{=}) = 0$, so there is no hope of a similar argument.
  \begin{figure}[tb]
    \centering
    \includegraphics[width=0.95\textwidth]{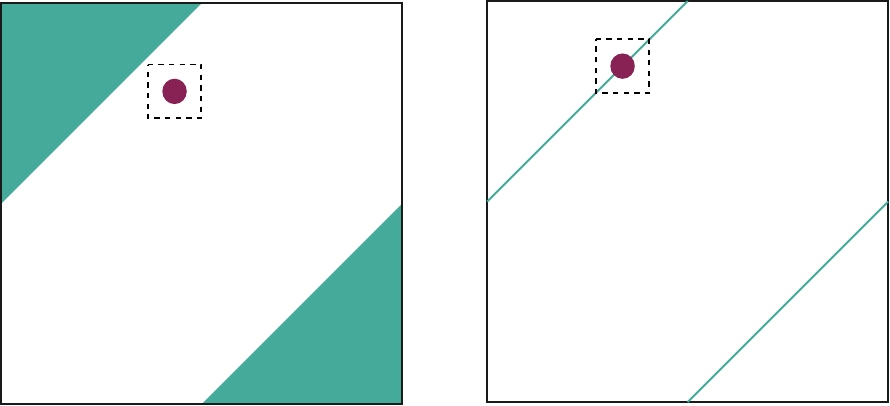}
    \caption{On the left, the square~$[0,1]^2$ with~$E_{\geq}$ shaded. To see that the point is not in~$E$, it suffices to find an open neighborhood such that its intersection with~$E$ has measure~$0$. On the right,~$E_{=}$ is shown, and we see that although the point is in~$E$, the intersection of any open neighborhood with~$E$ has measure~$0$.}
    \label{fig:thick-R2}
  \end{figure}
\end{exmp}

A problem is in the second subcategory if~$(V, E)$ displays sufficient symmetry; we will call this the~\defi{homogeneous setting}\index{homogeneous setting}. More precisely,~$V$ is a compact space that is homogeneous under a compact metrizable group~$\Gamma$,~$\Gamma E \subseteq E$ under the naturally induced action on~$V^k$, and~$V$ is equipped with the quotient of the Haar measure. See Section~\ref{ch:complete-positivity-under-symmetry}.\ref{sec:prel-harmonic-analysis} and the appendix for more background. The metric on~$\Gamma$ should furthermore meet technical assumptions which will be explained as we encounter them.

Our main example of an optimization problem in the homogeneous setting is~\defi{Witsenhausen's problem}: what is the largest surface measure of a set on the unit sphere~$S^{n-1} \subseteq \R^n$ that does not contain a pair of orthogonal points? This is exactly~\eqref{eqn:extremal-problem-finite-measure} with~$V$ the unit sphere~$S^{n-1} \subseteq \R^n$,~$\mu$ the standard uniform surface measure,~$E = \{\, (x, y) \in S^{n-1}: x^{\tr}y = 0\, \}$, and~$\Gamma = \ortho(n)$---the group of orthogonal $(n \times n)$-matrices.

Since the orthogonal group preserves inner products, indeed~$\ortho(n) E \subseteq E$. Moreover,~$S^{n-1}$ is homeomorphic to~$\ortho(n) / \Stab(e)$ with~$\Stab(e)$ the stabilizer subgroup of an arbitrary point~$e \in S^{n-1}$, thus the $(n-1)$-dimensional unit sphere is~$\ortho(n)$-homogeneous. Witsenhausen's problem is the subject of the paper~\cite{Bekker2026OptimizationSpaces}, and will be treated in depth in Chapter~\ref{ch:witsenhausen}.

We can modify Witsenhausen's problem to produce homogeneous examples with~$k > 2$. Take~$V$ and~$\mu$ the same, but consider~$E$ to be the set of all $k$-tuples~$\{v_1 ,\ldots, v_k\}$ such that~$v_i^{\tr} v_j = 0$ for all~$i \neq j$. Expression~\eqref{eqn:extremal-problem-finite-measure} then gives the maximal measure of a set of unit vectors not containing an orthonormal set of size~$k$. Castro-Silva, Oliveira, Slot, and Vallentin~\cite{Castro-Silva2022ASets} present an upper bound to~\eqref{eqn:extremal-problem-finite-measure} for such problems, which is a recursive version of the Lovász $\vartheta$-number. These questions are relevant as they are compact versions of questions in Euclidean Ramsey theory; see~\cite{Castro-Silva2022ASets} and references therein.

The example with the thick edge set~$\{\, (x,y) \in [0,1]^2 : |x-y| \geq 1/2\, \}$ is not homogeneous, and the examples in the homogeneous setting given above are not thick. However, the two settings have much overlap. For instance, the latter example with the condition~$v_i^{\tr} v_j = 0$ replaced by~$v_i^{\tr} v_j \in [-1/2, 1/2]$ displays the required symmetry, but also has a thick edge set.

We will use the language of hypergraphs to present completely positive formulations of~\eqref{eqn:extremal-problem-finite-measure}. The proofs of exactness stick closely to the proof for~$k=2$ by DeCorte, Oliveira, and Vallentin, but hold in greater generality: Theorem 5.1 of~\cite{DeCorte2022CompleteSets} is the special case of Theorem~\ref{thm:completely-positive-is-exact-homogeneous-graph} in Section~\ref{sec:exact-cp-bound-homogeneous-hypergraph} where~$k = 2$ and~$E$ is a closed set. We reproduce their proof for~$k \geq 2$ and without the requirement that~$E$ is closed.

\section{Upper bounds on the measurable independence number}%
\label{sec:upper-bounds-measurable-independence-number}
Let~$k \geq 2$ an integer. We say~$H = ((V, \A, \mu), E)$ is a \defi{$k$-uniform measurable hypergraph}\index{hypergraph!k uniform measurable@$k$-uniform measurable hypergraph} if~$(V, E)$ is a $k$-uniform hypergraph,~$(V, \A, \mu)$ is a measure space, and~$E \subseteq V^k$ is measurable. A~\defi{uniform measurable hypergraph}\index{hypergraph!uniform measurable hypergraph} is a $k$-uniform hypergraph for some integer~$k \geq 2$. If~$H$ is a uniform measurable hypergraph with measure~$\mu$, the~\defi{measurable independence number}\index{independence number!measurable independence number} of~$H$ is
\index{ aH@$\alpha(H)$}\[
  \alpha(H) = \sup \{\, \mu(I) : I \subseteq V\text{ measurable and independent}\, \}.
\]
A set~$I \subseteq V$ is independent if and only if~$I^k \cap E = \emptyset$, thus~$\alpha(H)$ is exactly~\eqref{eqn:extremal-problem-finite-measure}.

The measure is part of the data of~$H$, so there can be no confusion with the traditional independence number, and we will often just call~$\alpha(H)$ the independence number of~$H$. Having said that, we usually abuse notation, and say~$H = (V, E)$ is a uniform measurable hypergraph if the measure on~$V$ is clear.

Similarly, we say~$H = (V, E, \Gamma)$ is a~\defi{$k$-uniform homogeneous hypergraph under~$\Gamma$}\index{hypergraph!k uniform homogenenous hypergraph@$k$-uniform homogeneous hypergraph} if~$(V, E)$ is a $k$-uniform hypergraph,~$\Gamma \subseteq \Aut((V, E))$ is a $\sigma$-compact unimodular locally compact group,~$V$ is a homogeneous~$\Gamma$-space that admits a $\Gamma$-invariant Radon measure, and~$E \subseteq V^k$ is Borel. A~\defi{uniform homogeneous hypergraph under~$\Gamma$}\index{hypergraph!uniform homogeneous hypergraph} is a $k$-uniform homogeneous hypergraph under~$\Gamma$ for some~$k$. We again abuse notation and just say that~$(V, E)$ is a uniform homogeneous hypergraph under a group~$\Gamma$.

In particular, we consider a uniform homogeneous hypergraph to be a vertex-transitive hypergraph and a uniform measurable hypergraph with a~$\Gamma$-invariant measure. When~$\Gamma$ is in addition compact, the $\Gamma$-invariant Radon measure on~$V$ is always the pushforward of a Haar measure of~$\Gamma$ under the quotient map; recall that we named it the quotient measure. See Sections~\ref{ch:complete-positivity-under-symmetry}.\ref{sec:prel-harmonic-analysis} and~\ref{ch:Appendix}.\ref{sec:app-invariant-measures} for more details on harmonic analysis and invariant measures.

Finally, if there is no group action, we need another condition that ensures compatibility between a topology, a measure, and an edge set. We say~$H = (V, E)$ is a~\defi{$k$-uniform locally independent hypergraph}\index{hypergraph!k uniform locally independent hypergraph@$k$-uniform locally independent hypergraph} if~$V$ is a Hausdorff space and under the induced topology on~$V^k$,~$E$ is Borel and every compact independent set is contained in an open independent set. We say~$H = ((V, \B, \mu), E)$ is a~\defi{$k$-uniform measurable locally independent hypergraph}\index{hypergraph!k uniform measurable locally independent hypergraph@$k$-uniform measurable locally independent hypergraph} if~$((V, \B, \mu), E)$ is a $k$-uniform measurable hypergraph,~$(V, E)$ is locally independent,~$\B$ is the Borel algebra of~$V$, and~$\mu$ is a Borel measure. Again, denote such hypergraphs just by~$(V, E)$ if all other data is clear.

In all the definitions above, we replace ``hypergraph'' by ``graph'' and drop the adjective~$k$-uniform when~$k=2$.

\sectionbreak

We first describe a family of programs that give upper bounds on the independence number of a uniform measurable hypergraph. Let~$k \geq 2$ be an integer and~$H = (V, E)$ be a $k$-uniform measurable hypergraph with measure~$\mu$. Recall the operators
\[
  \T_{k-2}^*: L^2_{\sym}(V, k) \to L^2_{\sym}(V, 2),\qquad \T_{k-2}^*T(x,y) = \int_{V^{k-2}} T(x, y, v)\, d\mu^{k-2}(v).
\]
Say that a symmetric $k$-tensor~$T \in L^2_{\sym}(V, k)$ is \defi{slice positive}\index{slice positive tensor} if and only if for all~$G \in L^2(V^{k-2})_{\geq 0}$ and~$f \in L^2(V)$
\[
  \int_{V^{k-2}} \int_{V} \int_V T(x, y, v) G(v) f(x) f(y)\, d\mu^{k-2}(v)d\mu(x)d\mu(y) \geq 0.
\]
When~$V$ is additionally a topological space,~$T$ is continuous, and~$\mu$ is regular and has full support, slice positivity is equivalent to~$T(\, \cdot\, ,\, \cdot\, , v)$ being a positive-semidefinite kernel for all~$v \in V^{k-2}$. In particular, completely positive tensors are slice positive.

For a convex cone~$\cC \subseteq L^2_{\sym}(V, k)$, let
\index{ thetabig@$\vartheta_{\bt}(H, \cC)$}
\begin{equation}%
  \label{eqn:big-theta-number}
  \begin{optprob}
    \vartheta_{\bt}(H, \cC) = \sup &\onerow{\langle A, \1^{\otimes k} \rangle}\\
    &\onerow{\Tr(\T_{k-2}^*A)= 1,}\\
    &A(v_1, \ldots, v_k) = 0 & \text{for all } v_1\cdots v_k \in E,\\
    &\onerow{A \in \cC \text{ is slice positive}.}
  \end{optprob}
\end{equation}
The normalization constraint should be interpreted as: ``the trace of~$\T_{k-2}^*A$ exists and is equal to~$1$''. The edge constraints on~$A$ are pointwise; this means that only certain combinations of cones and edge sets produce useful optimization problems, as the support of~$A$ as an $L^2$ class is only defined up to a set of measure~$0$.

On the other hand, the requirements for~$\vartheta_{\bt}(H, \cC) \geq \alpha(H)$ to hold are mild. For an integer~$k \geq 2$ and~$H = (V, E)$ a $k$-uniform measurable hypergraph with~$0 < \alpha(H) < \infty$, it is not so hard to prove that, if~$I$ is a measurable independent set with~$\mu(I) > 0$, then~$A = \1_I^{\otimes k} / \mu(I)^{k-1}$ has~$\Tr(\T_{k-2}^*A) = 1$ and~$\langle A, \1^{\otimes k}\rangle = \mu(I)$. Thus, any~$\cC$ containing these tensors results in an upper bound.

However, in applications, we usually ask tensors to be continuous. One of the reasons is that the most practical way to obtain feasible solutions is by use of polynomials. Another reason is that the edge set might have measure~$0$, in which case the constraint on edges does not affect the optimal value, and the upper bound given by~$\vartheta_{\bt}$ is trivial.

DeCorte, Oliveira and Vallentin~\cite{DeCorte2022CompleteSets} showed that when~$G = ((V, \B, \mu), E)$ is a measurable locally independent graph equipped with~$\mu$ an inner regular Borel measure and~$V$ compact such that~$0 < \alpha(G) < \infty$, then Urysohn's Lemma gives continuous kernels~$f^{\otimes k}$ with~$0 \leq f \leq 1$  and objective value arbitrarily close to~$\alpha(G)$. Thus, in this case, the independence number is indeed bounded from above by continuous kernels satisfying the constraints of~$\vartheta_{\bt}(H, \cC)$ for a suitable cone~$\cC$. Their proof goes through without issue in the $k$-uniform setting.
\begin{lemma}%
  \label{lem:upper-bound-thick-edges}
  If~$k \geq 2$ is an integer,~$H = ((V, \B, \mu), E)$ is a $k$-uniform measurable locally independent hypergraph with~$\mu$ a finite inner regular Borel measure such that~$\alpha(H) > 0$,~$V$ is compact, and~$\cC \subseteq C_{\sym}(V,k)$ is a cone containing all continuous tensors of the form~$f^{\otimes k}$ with~$f \in L^2(V)_{\geq 0}$, then
  \[
    \vartheta_{\bt}(H, \cC) \geq \alpha(H).
  \]
\end{lemma}

Lemma~\ref{lem:averaged-rank-one-is-continuous} offers a second way of obtaining continuous tensors. It says that when~$H = (V, E)$ is a $k$-uniform homogeneous hypergraph under a compact group~$\Gamma$, then for every measurable independent set~$I$ with~$\mu(I) > 0$, the map~$A = \Av_{\Gamma} (\1_I^{\otimes k})/\mu(I)^{k-1}$ is continuous, and we can repeat the arguments above for the tensor~$A$, obtaining an upper bound on~$\alpha(H)$ if~$\cC$ contains all tensors of this form. However, we will see that, in the homogeneous setting, for a proof of exactness of the completely positive formulation we introduce later, it is necessary to replace the condition~$\Tr(\T_{k-2}^*A) = 1$ in~\eqref{eqn:big-theta-number} by~$\int_V A(v, \cdots, v)\, d\mu(v) = 1$, obtaining an upper bound on~$\alpha(H)^{k-1}$. For a convex cone~$\cC \subseteq C_{\sym}(V, k)$, let
\index{ thetasmall@$\vartheta_{\st}(H, \cC)$}
\begin{equation}%
  \label{eqn:small-theta-number}
  \begin{optprob}
    \vartheta_{\st}(H, \cC) = \sup &\onerow{\langle A, \1^{\otimes k} \rangle}\\
    &\onerow{\int_V A(v, \ldots, v)\, d\mu(v)= 1,}\\
    &A(v_1, \ldots, v_k) = 0 & \text{for all } v_1\cdots v_k \in E,\\
    &\onerow{A \in \cC \text{ is slice positive}.}
  \end{optprob}
\end{equation}
For~$k = 2$ and~$\cC \subseteq C_{\sym}(V, k)$, we have~$
\vartheta_{\bt}(H, \cC) = \vartheta_{\st}(H, \cC)$.
\begin{lemma}%
  \label{lem:upper-bound-homogeneous}
  If~$k \geq 2$ is an integer,~$H = (V, E, \Gamma)$ is a $k$-uniform homogeneous hypergraph with~$0 < \alpha(H) < \infty$, and~$\cC \subseteq C_{\rm sym}(V, k)$ is a cone containing all tensors of the form~$\Av_{\Gamma} (f^{\otimes k})$ with~$f\in L^k(V)_{\geq 0}$, then
  \[
    \vartheta_{\st}(H, \cC) \geq \alpha(H)^{k-1}.
  \]
\end{lemma}
\begin{proof}
  See the discussion above.
\end{proof}

\section{Density and thick sets}%
\label{sec:density-and-thick-sets}
The strategies for proving sharpness of a completely positive bound on the independence number for the thick and the homogeneous setting are similar: given a feasible solution~$A$, find a function~$f \geq 0$ such that its support is independent and whose measure gives an upper bound on the integral of~$A$ over the space. Since the suitable space for the optimization variables is a space of square-integrable functions, by functional-analytic arguments we can only describe values of integrals; we have little control over pointwise properties. We introduce the concept of density to deal with this.

In the thick setting, although the definitions are technical, this approach is rather straightforward. We obtain a zero-measure-removal lemma, Lemma~\ref{lem:removal-lemma-thick-edge-set}. It implies that if we find a function~$f \geq 0$ such that~$\mu^k((\supp{f})^k \cap E) = 0$, we can remove a subset of~$\supp{f}$ of measure~$0$ to obtain an independent set. In the homogeneous setting, things are a little more involved. We also obtain a removal lemma there, though we do not explicitly state it.

Let~$(V, \A, \mu)$ be a measure space, and denote the family of all measurable sets with nonzero measure by~$\A^+$\index{ A+@$A^+$}. Given a subfamily~$\D \subseteq \A^+$ that is a directed set under the opposite of the inclusion relation, define for measurable~$A \in \A^+$ and~$v \in V$,
\index{ densvA@$\dens_v(A)$}\[
  \dens_v(A) = \lim_{v \in D \in \D} \frac{\mu(D \cap A)}{\mu(D)},
\]
when the limit exists. We call~$\dens_v(A)$ the~\defi{density}\index{density} of~$v$ in~$A$ with respect to~$\D$. Denote the set of points~$v \in V$ such that~$\dens_v(A) = 1$ by~$D(A)$\index{ DA@$D(A)$}. Note that~$D(A)$ is not necessarily contained in~$A$.

Such a directed set~$\D \subseteq \A^+$ is called a~\defi{density system}\index{density system} if~$\mu(A \symdif D(A)) = 0$ for all~$A \in \A^+$, where~$\symdif$ is the symmetric difference of sets\index{ 000@$\symdif$}. If~$\D$ is a density system, we call a point~$v \in D(A)$ a~\defi{density point}\index{density point} of~$A$. In particular, for every~$A \in \A^+$, almost every~$v \in A$ is a density point of~$A$; on the other hand, almost every density point of~$A$ is in~$A$, from which we conclude~$\mu(A) = \mu(D(A)) = \mu(A \cap D(A))$.

For us, there are three main examples of a measure space with a  density system. The first is that of a set with a discrete measure, that is, combinations of point measures. The set of singletons forms a density system for such a measure.

The second example is~$\R^n$ with the Lebesgue measure. The Lebesgue density theorem states that the set of open balls forms a density system of~$\R^n$.

The third example is when the measure space has a finite-rank approximation. Let~$(V, \A, \mu)$ be a measure space with a finite-rank approximation~$(\P_n)_n$ with associated conditional expectation operators~$E_n$. Let~$\D = \bigcup_n \P_n^+$, the set of the parts with positive measure of all partitions in~$\P_n$. For all~$v \in V$, denote the set in~$P_n$ containing~$v$ by~$P_n(v)$. Let~$A \in \A^+$. We have
\[
  \lim_n E_n \1_A(v) = \lim_n \sum_{P \in \P_n} \frac{\langle \1_P, \1_A\rangle }{\mu(P)} \1_P(v) = \lim_n \frac{\mu(P_n(v) \cap A)}{\mu(P_n(v))} = \dens_v(A).
\]
The martingale convergence theorem, Theorem~\ref{thm:martingale-convergence}, then implies that, indeed,~$\mu(D(A) \symdif A) = 0$.

Let~$k \geq 1$ be an integer. A set~$E \subseteq V^k$ is called~\defi{$\D$-thick}\index{D-thick set@$D$-thick set} if and only if for all~$(v_1, \ldots, v_k) \in E$ we have
\begin{equation}%
  \label{eqn:density-in-E}
  \lim_{\substack{D_1, \ldots, D_k \in \D \\ (v_1, \ldots, v_k) \in D_1 \times \cdots \times D_k}} \frac{\mu^k((D_1 \times \cdots \times D_k) \cap E)}{\mu(D_1)\cdots \mu(D_k)} > 0.
\end{equation}
We simply call~$E$ a~\defi{thick set}\index{thick set} if there exists a density system~$\D$ for which~$E$ is $D$-thick. We also call the quantity~\eqref{eqn:density-in-E} the~\defi{density}\index{density} of~$(v_1, \ldots, v_k)$ in~$E$ with respect to~$\D$.

A measure space may have many density systems~$\D$; which sets are $\D$-thick depends on the choice of~$\D$. This means that, although a finite-rank approximation defines a density system, the density system thus obtained is not necessarily the one we want to use, as the next example shows.
\begin{exmp}%
  \label{exmp:thick-notthick}
  Take the set~$E_{\geq} = \{\, (x,y) : |x-y| \geq 1/2\, \}$ in the square~$[0,1]^2$ from the introduction. By the Lebesgue density theorem, the set of open intervals~$\D_I$ forms a density system of~$[0,1]$. Under this system,~$E_{\geq}$ is~$\D_I$-thick. On the other hand, the dyadic decomposition of~$[0,1]$ provides a finite-rank approximation of~$[0,1]$, that is,~$(\P_n)_n$ with
  \[
    P_n = \{\, [2^{-n} i, 2^{-n}(i+1)) : 0 \leq i \leq 2^n-1\, \} \cup \{\{1\}\}.
  \]
  The set~$\D_D = \bigcup_n \P_n \setminus \{\{1\}\}$ is a density system, but~$E$ is not~$\D_D$-thick. For example, the point in Figure~\ref{fig:thick-notthick} has density~$0$ under~$\D_D$, because it is unlucky enough to fall on the corner of a dyadic square that contains no other points of~$E_{\geq}$.
  \begin{figure}[tb]
    \centering
    \includegraphics[width=0.95\textwidth]{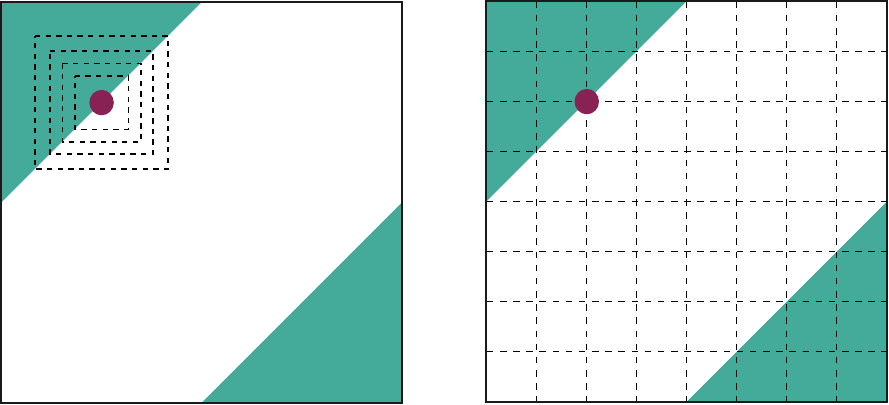}
    \caption{The square~$[0,1]^2$ with the set~$E_{\geq}$ shaded. The figure on the left shows that if we take~$\D_I$ for the density system, the density in~$E_{\geq}$ of the point is~$1/2$. On the right, the point coincides with the corner of a dyadic square that contains no other points of~$E_{\geq}$. This remains true for any subdivision of the grid, so the point has density~$0$ in~$E_{\geq}$ under~$\D_D$.}
    \label{fig:thick-notthick}
  \end{figure}
\end{exmp}

The usefulness of density systems and thick sets stems from the following removal lemma, which is reminiscent of the graph removal lemma from combinatorics~\cite{Erdos1986TheExponent}. It implies that if~$S$ is a set of vertices such that the set of edges in~$S^k$ has measure~$0$, then~$D(S)$ is independent.
\begin{lemma}[Removal lemma]%
  \label{lem:removal-lemma-thick-edge-set}
  Let~$k \geq 1$ be an integer,~$(V, \A, \mu)$ be a measure space with a density system~$\D$, and~$E \subseteq V^k$ be a $\D$-thick set. If~$S \in \A^+$, then~$D(S)^k \cap E \neq \emptyset$ if and only if~$\mu^k(S^k \cap E) > 0$.
\end{lemma}
\begin{proof}
  Let~$S \in \A^+$ and~$E \subseteq V^k$ be $\D$-thick.  Note first that if~$S_1, S_2 \subseteq V$ such that~$\mu(S_1 \symdif S_2) = 0$, then~$\mu^k(S_1^k \symdif S_2^k) = 0$. Indeed, to see this, note that~$S_1^k \symdif S_2^k \subseteq (S_1 \cup S_2)^k \setminus (S_1 \cap S_2)^k$, so
  \begin{multline*}
    \mu^k(S_1^k \symdif S_2^k) \leq \mu^k((S_1 \cup S_2)^k \setminus (S_1 \cap S_2)^k)\\
    = \mu^k((S_1 \cup S_2)^k) - \mu^k((S_1 \cap S_2)^k) = \mu(S_1)^k - \mu(S_1)^k = 0.
  \end{multline*}
  Thus, if~$D(S)^k \cap E = \emptyset$, then~$\mu^k(D(S)^k \cap E) = 0$, and since~$\mu^k(S^k \symdif D(S)^k) = 0$, we have~$\mu^k(S^k \cap E) = 0$.

  On the other hand, suppose there is a~$(v_1, \ldots, v_k) \in D(S)^k \cap E$. Since~$E$ is $\D$-thick, there exists a constant~$1 \geq c > 0$ and for each~$i$ a set~$D_{i,1} \in \D$ such that for all choices~$D_1, \ldots, D_k \in \D$ with~$v_i \in D_i \subseteq D_{i,1}$,
  \[
    \frac{\mu^k((D_1 \times \ldots \times D_k) \cap E)}{\mu(D_1) \cdots \mu(D_k)} \geq c.
  \]
  Moreover, each~$v_i$ is a density point of~$S$, so for each~$i$ there exists a set~$D_{i,2} \in \D$ such that for all~$D$ with~$v_i \in D \subseteq D_{i,2}$,
  \[
    \frac{\mu(D \cap S)}{\mu(D)} > (1-c/2)^{1/k}.
  \]
  The set~$\D$ is directed, so take for each~$i$ a set~$D_{i, 0} \in \D$ contained in~$D_{i,1} \cap D_{i,2}$.

  For each~$i$, choose~$D_i \in \D$ such that~$D_i \subseteq D_{i,0}$. Use the following version of the inclusion-exclusion principle. For a measure~$\nu$ and $\nu$-measurable sets~$S_1$,~$S_2$, and~$S_3$,
  \begin{multline*}
    \nu(S_1 \cap S_2 \cap S_3) = \nu\bigl((S_1 \cap S_2) \cap (S_1 \cap S_3)\bigr) = \nu(S_1 \cap S_2) + \nu(S_1 \cap S_3) -\\
    \nu\bigl((S_1 \cap S_2) \cup (S_1 \cap S_3)\bigr) \geq \nu(S_1 \cap S_2) + \nu(S_1 \cap S_3) - \nu(S_1).
  \end{multline*}
  Apply this with~$S_1 = D_1 \times \cdots \times D_k$,~$S_2 = S^k$ and~$S_3 = E$, to obtain
  \[
    \begin{split}
      \frac{\mu^k(S^k \cap E)}{\mu(D_1) \cdots \mu(D_k)}  &\geq \frac{\mu^k\bigl((D_1 \times \cdots \times D_k) \cap S^k \cap E\bigr)}{\mu(D_1) \cdots \mu(D_k)}\\
      &\geq \frac{\mu(D_1 \cap S) \cdots \mu(D_k \cap S)}{\mu(D_1) \cdots \mu(D_k)} +\\
      &\mathrel{\phantom{\geq}} \frac{\mu^k((D_1 \times \cdots \times D_k) \cap E)}{\mu(D_1) \cdots \mu(D_k)} - \frac{\mu(D_1) \cdots \mu(D_k)}{\mu(D_1) \cdots \mu(D_k)}\\
      &>(1-c/2) + c - 1\\
      &= c/2.
  \end{split}\]
  It follows that~$\mu^k(S^k \cap E) > 0$.
\end{proof}

\section{The thick setting}%
\label{sec:extactness-completely-positive-bound}
We first prove~$\vartheta_{\bt}(H, \CP(V, k)) = \alpha(H)$ for uniform measurable hypergraphs~$H$ with a thick edge set. In computational implementations of upper bounds,~$V$ is usually equipped with a topology and a Borel measure, and we only consider the continuous tensors in each cone. However, in our treatment of the thick setting, leaving out the continuity condition simplifies matters a little. This suffices, because the programs~$\vartheta_{\bt}(H, \cC)$ defined in~\eqref{eqn:big-theta-number} are maximization problems, so restricting the feasible region decreases the optimal value, and Lemma~\ref{lem:upper-bound-thick-edges} ensures that restricting to continuous tensors still produces a valid bound, which is also sharp.

A key tool in the proof of Theorem 5.1 in~\cite{DeCorte2022CompleteSets} is the identification of a compact region of the completely positive cone, its tip, whose extreme points have a simple description. We extend this definition here to $k$-tensors in a weak sense. An oddity of this extension is that the tip is not necessarily a part of the cone, but once we accept this, everything works as it should.

Let
\index{ Ebig@$\mathcal{E}_{\bt}$}\[
  \mathcal{E}_{\bt} = \{0\} \cup \{\, f^{\otimes k} : f \in L^{k/(k-1)}(V)_{\geq 0} \text{ and } \|f\|_{k/(k-1)} = 1\, \}
\]
and define the~\defi{big tip}\index{big tip@big tip of~$\CP(V,k)$} of~$\CP(V,k)$ by~$ \T(\CP(V, k)) = \cch \mathcal{E}_{\bt}$, with closure in the weak topology on~$L^{k/(k-1)}$\index{ TCPVk@$\T(\CP(V,k))$}. The appearance of the exponent~$k / (k-1)$---the $L^p$ exponent conjugate to~$k$---is a bit mysterious. As we will see in the proof of the lemma below, it comes from the following interpolation trick. Let~$m$ be a natural number, let~$\{k_i\}_{i \in [m]}$ be positive real numbers such that~$\sum_i k_i = 1$, let~$1 \leq p_1,\ldots, p_m \leq \infty$ be real numbers, and define~$p$ by~$1/p = \sum_i (k_i / p_i)$. For all measurable~$f$,
\begin{equation}%
  \label{eqn:interpolation-Holders-ineq}
  \|f\|_p \leq \prod_{i = 1}^m \|f\|_{p_i}^{k_i},
\end{equation}
which follows from repeated application of Hölder's inequality.

The proof of the lemma makes use of the conditional expectation operators associated to a finite-rank approximation. Recall Lemma~\ref{lem:sufficient-conition-finite-rank-approximation}, which says in particular that a finite countably generated measure space has a finite-rank approximation.
\begin{lemma}%
  \label{lem:trace-1-is-in-big-tip}
  If~$(V, \A, \mu)$ is a finite countably generated measure space, ~$k\geq 2$ is an integer, and~$A \in \CP(V, k)$ with~$\Tr(\T_{k-2}^*A) \leq 1$, then~$A \in \T(\CP(V, k))$.
\end{lemma}
\begin{proof}
  Let~$(\P_n)_{n \in \N}$ be a finite-rank approximation of~$V$ with the conditional expectation operators~$E_n$ and the associated operators~$A_n$ as in Section~\ref{ch:completely-positive-cone}.\ref{sec:projective-approximation}. Since~$\mu$ is finite, using Hölder's inequality,~$\|A\|_{k/(k-1)}$ is finite. Indeed, if~$k \geq 2$ then~$k/(k-1) \leq 2$, and for any~$1 \leq p \leq q \leq \infty$,
  \[
    \|A\|_p = \langle |A|^p, \1\rangle^{1/p} \leq c\||A|^p\|_{q/p}^{1/p} = c\|A\|_q < \infty,
  \]
  where~$c = \mu^k(V^k)^{(q/p)/((q/p)-1)}$. By taking~$p = k/(k-1)$ and~$q = 2$, it follows that~$A \in L^{k/(k-1)}(V^k)$, and~$A = \lim_n E_n A$ in the~$L^{k/(k-1)}$ norm. Hence, it is enough to prove that each~$E_nA \in \T(\CP(V, k))$.

  Recall that~$A_n A \in \CP(\P_n^+, k)$. It has finite expansion~$A_n A = \sum_m \lambda_{m} x_{m}^{\otimes k}$ with~$\lambda_{m} \in \R_{\geq 0}$ and~$x_{m} \in \R^{\P_n^+}_{\geq 0}$. Denote~$f_m = \sum_{P} ({x_m})_{P} \1_{P} / \mu(P)$ so that the equality~$E_n A = \sum_{m} \lambda_{m}f_m^{\otimes k}$ holds, and rescale each~$\lambda_m$ and~$f_m$ such that~$\|f_m\|_{k/(k-1)}  = 1$. Because~$0 \in \T(\CP(V, k))$, if~$\sum_m \lambda_m \leq 1$ the lemma follows.

  By Parseval's identity, for any orthonormal basis~$\Phi$,
  \[
    \Tr E_n K = \sum_{P_1, P_2 \in \P_n^+} \frac{\langle K\1_{P_1}, \1_{P_{2}}\rangle}{\sqrt{\mu^2(P_{1,2})}} \sum_{\phi \in \Phi} \frac{\langle \1_{P_1}, \phi\rangle \langle \phi, \1_{P_2} \rangle}{\sqrt{\mu(P_1)\mu(P_2)}} \leq \Tr(K).
  \]
  The inequality uses that the set of functions~$\1_{P}/\sqrt{\mu(P)}$ with~$P \in \P_n^+$ is also orthonormal.

  Since the sum over~$m$ is finite and again using Parseval's identity,
  \begin{multline*}
    1 \geq \Tr(E_n \T_{k-2}^* A) = \Tr(\T_{k-2}^*E_n A) \\
    = \sum_{m} \lambda_{m} \|f_m\|_1^{k-2} \sum_{\phi \in \Phi} \langle f_m, \phi\rangle^2 =  \sum_{m} \lambda_{m} \|f_m\|_1^{k-2} \|f_m\|_2^2.
  \end{multline*}
  Now use the interpolation inequality~\eqref{eqn:interpolation-Holders-ineq} with~$p_1 = 1$,~$p_2 = 2$,~$k_1 = (k-2)/k$ and~$k_2 = 2/k$. Then~$1/p = (k-1)/k$, so that
  \[
    \sum_{m} \lambda_{m} \bigl(\|f_m\|_1^{(k-2)/k} \|f_m\|_2^{2/k}\bigr)^k \geq \sum_m \lambda_m \|f_m\|_{k/(k-1)}^k = \sum_m \lambda_m.
  \]
  The  conclusion~$A = \lim_n E_n A \in \T(\CP(V, k))$ follows.
\end{proof}

\begin{lemma}%
  \label{lem:extreme-points-big-tip}
  If~$(V, \A, \mu)$ is a finite countably generated measure space and~$k \geq 2$ is an integer, the extreme points of~$\T(\CP(V, k))$ lie in~$\mathcal{E}_{\bt}$.
\end{lemma}
\begin{proof}
  Since every element of~$\T(\CP(V, k))$ is nonnegative,~$0$ certainly is an extreme point.

  Milman's theorem~\cite[Theorem 9.4]{Simon2011Convexity} says that the extreme points of the big tip~$\T(\CP(V, k))$ lie in the weak closure~$\cl\mathcal{E}_{\bt}$. Suppose that~$(f_i^{\otimes k})_{i \in I}$ is a convergent net in~$\mathcal{E}_{\bt}$. The space~$L^{k/(k-1)}(V)$ is reflexive, so by the Banach-Alaoglu theorem its unit ball is compact~\cite[Theorem V.4.2]{Conway2010AAnalysis}. So, since~$\|f_i\|_{k/(k-1)} = 1$ for all~$i$, the sequence~$(f_i)_i$ has a weakly converging subnet; without loss of generality assume it converges with limit~$f$. It follows that~$\|f\|_{k/(k-1)} \leq 1$.

  The weak limit of~$(f_i^{\otimes k})_i$ is~$f^{\otimes k}$. Indeed, the linear span of the set of functions~$h_1 \otimes \cdots \otimes h_k \in L^k(V^k)$ is dense in~$L^k(V^k)$---for example, use the martingale convergence theorem. So it is enough to prove that
  \[
    |\langle f_i^{\otimes k} - f^{\otimes k}, h_1 \otimes \cdots \otimes h_k\rangle| \to 0
  \]
  with~$h_j \in L^k(V^k)$ arbitrary. For~$i \in I$ and~$h_1,  \ldots, h_k \in L^k(V)$, use the triangle inequality to obtain
  \begin{multline*}
    |\langle f_i^{\otimes k} - f^{\otimes k}, h_1 \otimes \cdots \otimes h_k\rangle|
    \leq |\langle f_i^{\otimes (k-1)} \otimes (f_i - f), h_1 \otimes \cdots \otimes h_k\rangle| +\\
    |\langle \bigl(f_i^{\otimes (k-1)} - f^{\otimes (k-1)}\bigr)\otimes f, h_1 \otimes \cdots \otimes h_k\rangle|.
  \end{multline*}
  By repeated application of this inequality,~$|\langle f_i^{\otimes k} - f^{\otimes k}, h_1 \otimes \cdots \otimes h_k\rangle|$ is bounded from above by~$k$ terms of the form
  \[
    |\langle f_i - f, h_j\rangle| \prod_{l \neq j}|\langle \xi_l, h_l \rangle| \leq |\langle f_i - f, h_j\rangle| M^{k-1},
  \]
  where each~$\xi_l$ is either~$f_i$ or~$f$, and the final inequality is Hölder's inequality with~$\|\xi_l\|_{k/(k-1)} \leq 1$ and~$M = \max_{1 \leq i \leq k}\|h_i\|_k$. Then,~$\lim_i f_i^{\otimes k} = f^{\otimes k}$ with~$\|f\|_{k/(k-1)} \leq 1$. Because~$0 \in \T(\CP(V, k))$, it follows that~$f^{\otimes k}$ is an extreme point if and only if~$\|f\|_{k/(k-1)} = 1$ or~$f = 0$; conclude that~$\mathcal{E}_{\bt}$ contains all extreme points of~$\T(\CP(V, k))$.
\end{proof}

We have now gathered everything we need to state and prove the main theorem of this section. The following result is independent on the choice of the density system~$\D$.
\begin{theorem}%
  \label{thm:completely-positive-theta-is-exact-thick-edges}
  Let~$k \geq 2$ an integer,~$H = ((V, \A, \mu), E)$ be a $k$-uniform measurable hypergraph with~$(V,\A,\mu)$ finite and countably generated measure space. If~$\alpha(H) > 0$ and if~$E$ is thick, then
  \[
    \alpha(H) = \vartheta_{\bt}(H, \CP(V,k)).
  \]
\end{theorem}
\begin{proof}
  The inequality~$\vartheta_{\bt}(H, \CP(V, k)) \geq \alpha(H)$ is established in Section~\ref{sec:upper-bounds-measurable-independence-number}. The remaining objective of this proof is to show that given a feasible solution~$A$ for~$\vartheta_{\bt}(H, \CP(V, k))$, there exists a measurable independent set~$I$ such that~$\langle A, \1^{\otimes k}\rangle \leq \mu(I)$.

  By Lemma~\ref{lem:trace-1-is-in-big-tip}, a feasible solution~$A$ of~$\vartheta_{\bt}(H, \CP(V, k))$ is in the big tip of~$\CP(V, k)$. As~$L^{k/(k-1)}(V)$ is reflexive,~$L^k(V^k)$ is separable, and~$\T(\CP(V, k))$ is closed and bounded, the big tip is weakly metrizable and compact~\cite[Theorem V.5.1 and Theorem V.4.2]{Conway2010AAnalysis}, thus Choquet's theorem~\cite[Theorem 10.7]{Simon2011Convexity} says that there exists a probability measure~$\rho$ on~$\mathcal{E}_{\bt}$ so that for all~$T$ in~$L^k(V^k)$ the equality~$\langle A, T\rangle = \int_{\mathcal{E}_{\bt}} \langle f^{\otimes k}, T\rangle\, d\rho(f^{\otimes k})$ holds.

  Since~$A(v) = 0$ for all~$v \in E$ and~$E$ is measurable,
  \[
    0 = \langle A, \1_E\rangle = \int_{\mathcal{E}_{\bt}} \langle f^{\otimes k}, \1_E\rangle\, d\rho(f^{\otimes k}).
  \]
  Hence,
  \[
    \rho(\{\, f^{\otimes k} \in \mathcal{E}_{\bt} : \langle f^{\otimes k},\1_E \rangle \neq 0\, \}) = 0.
  \]
  Furthermore,~$\langle A, \1^{\otimes k}\rangle = \int_{\mathcal{E}_{\bt}} \langle f^{\otimes k}, \1^{\otimes k}\rangle \, d\rho(f^{\otimes k})$ and~$\rho$ is a probability measure, so there exists a function~$f \in L^{k/(k-1)}(V)_{\geq 0}$ with~$\|f\|_{k/(k-1)} = 1$ such that~$\langle f^{\otimes k}, \1_E\rangle = 0$ and
  \[
    \langle f,  \1\rangle^k = \langle f^{\otimes k}, \1^{\otimes k}\rangle \geq \langle A, \1^{\otimes k}\rangle.
  \]

  Take a density system~$\D$ on~$(V, \A, \mu)$ with respect to which~$E$ is $\D$-thick. Denote the set of density points of~$\supp{f}$ by~$D(\supp{f})$. Then
  \[
    \langle A, \1^{\otimes k} \rangle \leq \langle f, \1\rangle^k = \langle f, \1_{\supp f}\rangle^k \leq \|f\|_{k/(k-1)}^k \|\1_{\supp f}\|_k^k = \mu(D(\supp{f})).
  \]
  It remains to prove that~$D(\supp{f})$ is independent. Since~$\int_E f^{\otimes k}(v)\, d\mu^k(v) = 0$ and~$f \geq 0$,~$\mu^k((\supp{f})^k \cap E)=0$. The removal lemma, Lemma~\ref{lem:removal-lemma-thick-edge-set}, says that~$D(\supp{f})^k \cap E = \emptyset$. Conclude that~$D(\supp{f})$ is an independent set with~$\mu(D(\supp{f})) \geq \langle A, \1^{\otimes k}\rangle$.
\end{proof}

\section{The homogeneous setting}%
\label{sec:exact-cp-bound-homogeneous-hypergraph}
We move on to the homogeneous setting. These results rely on a compact Hausdorff space being metrizable if and only if it is second countable.

We need a final result on martingales. For a finite measure space~$(V, \A, \mu)$ with finite-rank approximation~$(\P_n)_n$ and conditional expectations~$E_n$, define the~\defi{martingale maximal function}\index{martingale maximal!function} of~$f \in L^p(V)$ with~$1 \leq p < \infty$ by
\index{ M@$M$}\[
  Mf(v) = \sup_n E_nf(v),
\]
and call~$M$ the~\defi{martingale maximal operator}\index{martingale maximal!operator}. A proof of the following can be found in~\cite[Theorem 5.2.7]{Edwards1977Littlewood-PaleyTheory}, the important implication being that the maximal function of a $\mi{p}$-integrable function is again $\mi{p}$-integrable.
\begin{theorem}[Martingale maximal theorem]%
  \label{thm:martingale-maximal-theorem}%
  \index{martingale maximal!theorem}
  Let~$(V, \A, \mu)$ be a finite measure space with finite-rank approximation~$(\P_n)_{n \in \N}$. The martingale maximal
  function~$M$ defines a bounded operator~$L^p(V) \to L^p(V)$ for~$1 < p < \infty$, and a continuous operator~$L^1(V) \to L^1(V)$ under the weak topology.
\end{theorem}

For~$V$ a compact metric space with finite Borel measure~$\mu$ and~$k \geq 2$ an integer, let
\index{ E small@$\mathcal{E}_{\st}$}\[
  \mathcal{E}_{\st} = \{0\} \cup \{\, f^{\otimes k} : f \in L^k(V)_{\geq 0} \text{ and } \|f\|_k = 1\, \},
\]
and define the~\defi{small tip} of~$\CP(V, k)$ by~$\tau(\CP(V, k)) = \cch \mathcal{E}_{\st}$\index{small tip@small tip of~$\CP(V, k)$}\index{ tCPVk@$\tau(\CP(V, k))$} with closure under the weak topology of $L^k$. With a proof similar to that for the big tip, the extreme points of the small tip are contained in~$\mathcal{E}_{\st}$. The proof of an analogue of Lemma~\ref{lem:trace-1-is-in-big-tip} is slightly different.

For a set~$S \subseteq L^2(V^k)$, let~$S_{\cont} = S \cap C(V^k)$\index{ Sc@$S_{\cont}$}.
\begin{lemma}%
  \label{lem:trace-1-is-in-small-tip}
  If~$V$ is a compact metrizable space with a finite Borel measure~$\mu$,~$k \geq 2$ is an integer,~$T \in \CP(V, k)_{\cont}$, and~$\int_V T(x, \ldots, x)\, d\mu(x) \leq 1$, then~$T \in \tau(\CP(V,k))$.
\end{lemma}
\begin{proof}
  The space~$V$ is compact, second countable, and Hausdorff, so take  a finite-rank approximation of~$V$ with conditional expectation operators~$E_n$. Let~$T \in \CP(V, k)_{\cont}$. Since~$T$ is continuous and~$V$ is compact,~$T$ is a bounded function, thus~$\lim_n E_n T = T$ under the $L^k$ norm. Similar to the proof of Lemma~\ref{lem:trace-1-is-in-big-tip}, write~$E_n T = \sum_m \lambda_{m, n} f_{m, n}^{\otimes k}$, which is a finite sum with every~$\lambda_{m, n} \in \R_{\geq 0}$ and~$f_{m, n} \in L^k(V)_{\geq 0}$ such that~$\|f_{m, n}\|_k = 1$. As opposed to the proof of Lemma~\ref{lem:trace-1-is-in-big-tip}, it is not necessarily true that~$\sum_m \lambda_m \leq 1$, but it can be renormalized such that this follows.

  Indeed, by the martingale maximal theorem---Theorem~\ref{thm:martingale-maximal-theorem}---combined with the dominated convergence theorem,
  \[
    1 \geq \int_V T(x, \ldots, x)\, d\mu(x) = \int_V \lim_n (E_nT(x, \ldots, x))\, d\mu(x) = \lim_n \sum_m \lambda_{m, n}.
  \]
  Hence, for all~$\epsilon > 0$ there is an~$n_{\epsilon} \in \N$ such that~$\sum_m \lambda_{m, n} \leq 1 + \epsilon$ for all~$n \geq n_{\epsilon}$. Vice versa, there is a sequence~$(\epsilon_n)_n$ with limit~$0$ such that each~$\epsilon_n \geq 0$ and~$(1/(1+\epsilon_n))\sum_m \lambda_{m,n} \leq 1$. Thus,~$(1 / (1 + \epsilon_n)) E_n T \in \tau(\CP(V, k))$ for all~$n \in \N$ large enough, and since~$\lim_n 1/(1+\epsilon_n) = 1$,~$\lim_n (1/(1+\epsilon_n))E_nT = T$ under the $L^k$ norm, which concludes the proof.
\end{proof}

As in the thick setting, we make use of a density system; see Section~\ref{sec:density-and-thick-sets} for more details. The difference is that the edge set is not necessarily thick with respect to this density system. This is replaced by the metric being compatible with the group action.

Let~$V$ be a metric space with a metric~$d$,~$\B$ be its family of Borel sets, and~$\mu$ be a Borel measure with full support. Then~$d$ is called a~\defi{density metric}\index{metric!density metric} if the set of open balls forms a density system of~$(V, \B, \mu)$. In particular, for all~$A \in \B^+$ and almost all~$v \in A$ we have
\[
  \dens_v(A) = \lim_{\delta \to 0} \frac{\mu(A \cap B_{\delta}(v))}{\mu(B_{\delta}(v))} = 1,
\]
were~$B_{\delta}(v) = \{\, u \in V: d(u,v) \leq \delta\, \}$. The most important example of this is the Euclidean metric on~$\R^n$, which is a density metric by the Lebesgue density theorem.

Call a metric on a locally compact group~$\Gamma$~\defi{right-invariant}\index{metric!right-invariant metric} if for all~$g \in \Gamma$ and~$v_1$,~$v_2 \in V$ we have~$d(v_1, v_2) = d(v_1 g, v_2 g)$. This implies that for all~$v \in V$ and~$\delta \geq 0$,~$B_{\delta}(v g) = B_{\delta}(v)$. The metric on the orthogonal group~$\ortho(n) \subseteq \R^{n \times n}$ inherited from the Euclidean metric on~$\R^{n \times n}$ is an example of a right-invariant density metric on a compact group.

Recall that if a group~$\Gamma$ is compact and second countable, then every Hausdorff $\Gamma$-homogeneous space is compact and second countable as well.
\begin{theorem}%
  \label{thm:completely-positive-is-exact-homogeneous-graph}
  If~$k \geq 2$ is an integer, if~$H = (V, E, \Gamma)$ is a $k$-uniform homogeneous hypergraph with~$\Gamma$ a compact group that is metrizable by a right-invariant density metric, and if~$\alpha(H) > 0$, then
  \[
    \alpha(H)^{k-1} = \vartheta_{\st}(H, \CP(V, k)_{\cont}).
  \]
\end{theorem}
\begin{proof}
  The inequality~$\alpha(H)^{k-1} \leq\vartheta_{\st}(H, \CP(V, k)_{\cont})$ holds by Lemma~\ref{lem:upper-bound-homogeneous}.

  For the other inequality, the strategy is the same as that of the proof of Theorem~\ref{thm:completely-positive-theta-is-exact-thick-edges}, closely following the steps in~\cite{DeCorte2022CompleteSets}: denoting the quotient measure on~$V$ by~$\mu$, and given a feasible solution~$A$ of~$\vartheta_{\st}(H, \CP(V, k)_{\cont})$, prove that there is a measurable independent set~$I$ such that~$\mu(I)^{k-1} \geq \langle A, \1^{\otimes k}\rangle$.

  As opposed to when~$E$ is thick, the linear functional~$\langle \cdot, \1_E\rangle$ is identically 0. The first objective is to find a workaround for this problem, which goes as follows. The set~$E$ has a countable dense subset under the relative topology: for a countable basis~$\mathcal{U}$ of the topology of~$V^k$, let~$S$ be a set that contains for every~$U \in \mathcal{U}$ such that~$U \cap E \neq \emptyset$ exactly one point~$v \in U \cap E$. Since~$\mathcal{U}$ is a countable basis for the topology,~$S$ is a countable dense subset of~$E$. Order it by~$S = \{(s_{1,1}, \ldots, s_{1,k}), (s_{2,1}, \ldots, s_{2,k}), \ldots\}$ with each~$s_{i,j} \in V$.

  Since~$V$ is metrizable, fix a metric on~$V$ producing the topology, and, for fixed~$n$, using compactness, let~$\mathcal{U}_n$ be a finite cover of~$V$ by open balls of radius~$1/n$ with respect to this metric. Choose for all~$n \in \N$ and~$v \in V$ a set in~$\mathcal{U}_n$ that contains~$v$, and denote it~$U_n(v)$.

  Define for~$n \in \N$ the tensor
  \[
    T_n = \sum_{i \in \N} 2^{-i} \frac{\1_{U_n(s_{i,1})\times \cdots \times U_n(s_{i,k})}}{\mu(U_n(s_{i,1}))\cdots  \mu(U_n(s_{i,k}))}.
  \]
  It is $k/(k-1)$-integrable, for if~$M_n = \max\{\, \mu(U) : U \in \mathcal{U}_n\, \}$, then by Minkowski's inequality
  \[
    \begin{split}
      \|T_n\|_{k/(k-1)} &\leq \sum_{i \in \N} \biggl(\int_{V^k} \biggl( 2^{-i} \frac{\1_{U_n(s_{i,1})\times \cdots \times U_n(s_{i,k})}(v)}{\mu(U_n(s_{i,1}))\cdots \mu(U_n(s_{i,k}))}\biggr)^{\frac{k}{k-1}}\, d\mu^k(v)\biggr)^{\frac{k-1}{k}}\\
      &= \sum_{i \in \N} 2^{-i} \frac{\bigl(\mu(U_n(s_{i,1}))\cdots \mu(U_n(s_{i,k}))\bigr)^{(k-1)/k}}{\mu(U_n(s_{i,1}))\cdots \mu(U_n(s_{i,k}))}\\
      &= \sum_{i \in \N} 2^{-i} \bigl(\mu(U_n(s_{i,1}))\cdots\mu(U_n(s_{i,k}))\bigr)^{-1/k}\\
      &\leq 2/M_n.
  \end{split}\]

  Since any continuous~$A: V^k \to \R$ is uniformly continuous, for every~$\epsilon > 0$ there is an~$n_{\epsilon} \in \N$ so that for all~$n \geq n_{\epsilon}$
  \[
    |A(v_1, \ldots, v_k) - A(s_{i,1}, \ldots, s_{i,k})| \leq \epsilon\qquad \text{for all } i \in \N \text{ and }v_j \in U_n(s_{i,j}).
  \]
  Whence, if~$A : V^k \to \R$ is continuous, then
  \begin{equation}%
    \label{eqn:limit-T_n-gives-pointwise-values}
    \lim_n \langle T_n, A\rangle = \sum_{i \in \N} 2^{-i} A(s_{i,1}, \ldots, s_{i, k}).
  \end{equation}

  Let~$A$ be a feasible solution of~$\vartheta_{\st}(H, \CP(V,k)_{\cont})$. Since~$H$ is homogeneous and all constraints and the objective are invariant under the action of~$\Gamma$ on~$A$,~$A$ can be assumed invariant under the diagonal action of~$\Gamma$, that is,~$\Av_{\Gamma} A = A$.

  Since~$\int_V A(x, \ldots, x)\, d\mu(x) = 1$, by Lemma~\ref{lem:trace-1-is-in-small-tip},~$A \in \tau(\CP(V, k))$. Because the Banach space~$L^k(V^k)$ is reflexive,~$L^{k/(k-1)}(V^k)$ is separable, and the small tip is closed and bounded in~$L^k(V^k)$, it is compact, and the weak topology on the small tip is metrizable~\cite[Theorem V.4.2 and Theorem V.5.1]{Conway2010AAnalysis}. Then, Choquet's theorem~\cite[Theorem 10.7]{Simon2011Convexity} says that there exists a probability measure~$\rho$ on~$\mathcal{E}_{\st}$ such that for all~$F \in L^{k/(k-1)}(V^k)$,
  \begin{equation}%
    \label{eqn:extreme-point-decomposition}
    \langle A, F\rangle = \int_{\mathcal{E}_{\st}} \langle f^{\otimes k},  F\rangle\, d\rho(f^{\otimes k}).
  \end{equation}

  Since~$A$ is~$0$ on~$E$, by $\Gamma$-invariance of~$A$ and Fatou's lemma,
  \[
    \begin{split}
      0 &= \lim_{n \in \N} \langle \Av_{\Gamma} A, T_n \rangle\\
      &= \lim_{n \in \N} \langle A, \Av_{\Gamma} T_n \rangle\\
      &\geq \int_{\mathcal{E}_{\st}} \liminf_{n \in \N} \langle f^{\otimes k}, \Av_{\Gamma} T_n\rangle\, d\rho(f^{\otimes k})\\
      &= \int_{\mathcal{E}_{\st}} \liminf_{n\in \N} \langle  \Av_{\Gamma} f^{\otimes k}, T_n\rangle\, d\rho(f^{\otimes k}).
    \end{split}
  \]
  Thus, since in the above~$T_n$ and all~$f$s are nonnegative, the set
  \[
    \{\, f^{\otimes k} : \liminf_{n \in \N} \langle \Av_{\Gamma} f^{\otimes k}, T_n\rangle \neq 0\, \}
  \]
  has measure~$0$ under~$\rho$.

  This, together with Equation~\eqref{eqn:extreme-point-decomposition} applied to~$\langle A, \1^{\otimes k}\rangle$, shows that there exists an~$f \in L^k(V)_{\geq 0}$ with~$\|f\|_k = 1$ such that~$\langle f^{\otimes k}, \1^{\otimes k}\rangle \geq \langle A, \1^{\otimes k}\rangle$ and
  \[
    \liminf_{n \in \N} \langle\Av_{\Gamma} f^{\otimes k}, T_n \rangle = 0.
  \]
  Since~$\Av_{\Gamma}f^{\otimes k}$ is continuous by Lemma~\ref{lem:averaged-rank-one-is-continuous}, Equation~\eqref{eqn:limit-T_n-gives-pointwise-values}, and density of~$S$ in~$E$ imply that~$\Av_{\Gamma}f^{\otimes k}(v_1, \ldots, v_k) = 0$ if~$v_1 \cdots v_k \in E$.

  Identify~$V$ with a quotient of~$\Gamma$ and let~$p : \Gamma \to V$ be the quotient map. Let~$I = D(p^{-1}(\supp{f})) \cap p^{-1}(\supp{f})$, where~$D(p^{-1}(\supp{f}))$ is the set of density points of~$p^{-1}(\supp f)$ with respect to the right-invariant density metric on~$\Gamma$. Denote the Haar measure on~$\Gamma$ by~$\nu$. Then, since~$p(I) \subseteq \supp{f}$,
  \[
    \mu(\supp{f}) = \nu(p^{-1}(\supp{f})) = \nu(I) \leq \nu(p^{-1}(p(I))) = \mu(p(I)) \leq \mu(\supp{f}),
  \]
  so that the inequalities are equalities, and
  \[
    \langle A, \1^{\otimes k}\rangle \leq \langle f, \1\rangle^k \leq \|f\|_k^k \|\1_{\supp{f}}\|_{k/(k-1)}^k = \mu(p(I))^{k-1}.
  \]
  If~$p(I)$ is independent, the theorem follows.

  Let~$h_1$,~$\ldots$,~$h_k \in I$ be distinct points. Because~$\Av_{\Gamma} f^{\otimes k}$ is~$0$ on edges, it suffices to show~$\Av_{\Gamma} (f \smallcirc p)^{\otimes k}(h_1, \ldots, h_k) > 0$. The function~$f \smallcirc p$ is strictly positive on~$I$, and
  \[
    \Av_{\Gamma} (f \smallcirc p)^{\otimes k}(h_1, \ldots, h_k) = \int_{\bigcap_{i=1}^k I h_i^{-1}} (f \smallcirc p)^{\otimes k}(g h_1, \ldots, g h_k)\, d\nu(g),
  \]
  so that it suffices to show that~$\nu(\bigcap_{i=1}^k I h_i^{-1}) > 0$.

  Since all~$h_i$ are density points of~$p^{-1}(\supp{f})$, and so also of~$I$, there exists a~$\delta > 0$ such that for all~$i$
  \[
    \frac{\nu(I \cap B_{\delta}(h_i))}{\nu(B_{\delta}(h_i))} \geq \frac{k}{k+1}.
  \]
  Let~$N_i = I h_i^{-1}$, then
  \[
    \nu\biggl( \bigcap_{i=1}^k N_i \cap B_{\delta}(1)\biggr) \geq \sum_{i=1}^k \nu(N_i \cap B_{\delta}(1)) - (k-1)\nu(B_{\delta}(1)).
  \]
  Indeed, for two subsets~$S_1$,~$S_2 \subseteq V$,~$\nu(S_1 \cap S_2) = \nu(S_1) + \nu(S_2) - \nu(S_1 \cup S_2)$ by the inclusion-exclusion principle. For~$\{S_i\}_{i \in [l]}$ with~$l > 2$, apply induction by
  \[
    \nu\biggl(S_l \cap \bigcap_{i < l} S_i \biggr) = \nu(S_l) + \nu\biggl(\bigcap_{i < l} S_i\biggr) - \nu\biggl(S_l \cup \bigcap_{i < l} S_i\biggr),
  \]
  and the claim follows by taking~$S_i = N_i \cap B_{\delta}(1)$ and~$N_i \cap B_{\delta}(1) \subseteq B_{\delta}(1)$ for all~$i$. Since the metric and the Haar measure are right-invariant, it follows that
  \begin{multline*}
    \nu\biggl(\bigcap_{i=1}^k N_i \cap B_{\delta}(h_i)\biggr) \geq \sum_{i=1}^k \nu( I \cap B_{\delta}(h_i)) - (k-1)\nu(B_{\delta}(1))\\
    \geq \nu(B_{\delta}(1))\biggl(\frac{k^2}{k+1} - (k-1)\biggr) = \frac{\nu(B_{\delta}(1))}{k + 1} > 0,
  \end{multline*}
  which completes the proof.
\end{proof}

\chapter[Convergence through completely positive programming]{Convergence through completely positive programming}%
\label{ch:Completely positive programming on compact spaces}
We work from the general to the specific in this chapter. First, we define a completely positive hierarchy and versions of the moment hierarchy and the block moment hierarchy for the measurable independence number of certain uniform measurable hypergraphs, which we then compare  to each other. The proofs are algebraic in nature, and we can state them in great generality.

We also prove that the moment and block moment hierarchies give an upper bound on the measurable independence number, but only for uniform homogeneous hypergraphs. It is unclear whether a similar result holds without the presence of a group.

After that, we investigate convergence of these hierarchies to the measurable independence number, based on the exactness results of the previous chapter. These proofs depend on functional-analytic specifics, and in the homogeneous setting even on details of the representation theory of the group, limiting our results to only some measurable graphs. Most results in this chapter are based on~\cite{Bekker2026OptimizationSpaces}, but the extension to hypergraphs is new.

\section{Spaces of subsets}%
\label{sec:spaces-of-subsets}
We use the convention~$V^0 = \{\emptyset\}$. If~$v \in V^r$, denote by~$\setof{v}$ the set of its coefficients; that is,~$\setof{\emptyset} = \emptyset$, and~$\setof{(v_1, \ldots, v_r)} = \{v_1, \ldots, v_r\}$. This defines a map~$\setof{\, \cdot\, } : V^r \to \Sub{V}{r}$. We omit the index~$r$ from the notation, but to avoid ambiguity, we do include it in the notation for the inverse image:
\[
  \setof{S}^{-1}_r = \{\, v \in V^r : \setof{v} = S\, \},
\]
for all~$S \in \Sub{V}{r}$.

If~$(V, \A, \mu)$ is a measure space, we can turn~$\Sub{V}{r}\setminus\{\emptyset\}$ into a measure space with measure~$\mu^r_{\sub}$ by taking the pushforward of~$\mu^r$ under~$\setof{\, \cdot\, }$. That is,~$S \subseteq \Sub{V}{r} \setminus \{\emptyset\}$ is measurable if and only if~$\setof{S}^{-1}_r \in \A^r$, and its measure is~$\mu^r_{\sub}(S) = \mu(\setof{S}^{-1}_r)$. Define furthermore~$\mu_{\sub}^r(\{\emptyset\})=1$. In other words, if~$f: \Sub{V}{r} \to \R$ is a measurable function, then
\[
  \int_{\Sub{V}{r}} f(S)\, d\mu^r_{\sub}(S) = f(\emptyset) + \int_{V^r} f(\setof{v})\, d\mu(v).
\]
Therefore,~$f$ is integrable if and only if~$f \smallcirc \setof{\, \cdot\, }$ is.

Let~$V$ be a Hausdorff space. For a topology on~$V^0 = \Sub{V}{0} = \{\emptyset\}$ there is no choice. Using the map~$\setof{\, \cdot\, }$, a topology on~$V$ induces a topology on~$\Sub{V}{r}$, called the~\defi{standard topology}\index{topology!standard topology} of~$\Sub{V}{r}$; indeed, for~$r \geq 1$,~$\setof{\, \cdot\, }$ is surjective onto~$\Sub{V}{r}\setminus\{\emptyset\}$, so we equip it with the quotient topology: a collection~$S \subseteq \Sub{V}{r}$ is open if and only if~$\setof{S}^{-1}_r$ is open in~$V^r$. We then take the disjoint union with~$\{\emptyset\}$ to obtain a topology on all of~$\Sub{V}{r}$. The spaces~$\Sub{V}{r}$ are Hausdorff, and if~$V$ is compact, so are all~$\Sub{V}{r}$~\cite{Handel2000SomeSpaces}.

If~$1 \leq r \leq k$ and~$U_1, \ldots, U_r \subseteq V$ are open sets, then the set
\[
  (U_1, \ldots, U_r)_k = \{\, A \in \Sub{V}{k} : A \cap U_i \neq \emptyset \text{ for all } i \text{ and } A \subseteq U_1 \cup \cdots \cup U_r\, \}
\]
is open in~$\Sub{V}{k}$. If the~$U_i$ are pairwise disjoint, call such a set~\defi{basic open}\index{basic open set of SubVk@basic open set of~$\Sub{V}{k}$}. Handel showed \cite[Proposition 2.11]{Handel2000SomeSpaces} that the collection of all basic open sets forms a basis for the topology on $\Sub{V}{k}$. For more background on this topology, see Handel~\cite{Handel2000SomeSpaces}. The spaces~$\Sub{V}{=r}$ inherit measures and topologies from~$\Sub{V}{r}$ by restriction.

If~$V$ is equipped with a topology, the union map
\[
  \cup : \Sub{V}{r} \times \Sub{V}{s} \to \Sub{V}{r + s}, \qquad (S, T) \mapsto S \cup T
\]
is continuous. So, it induces a bounded linear operator
\[
  M: C(\Sub{V}{r + s}) \to C(\Sub{V}{r} \times \Sub{V}{s}), \qquad Mf(S, T) = f(S \cup T)
\]
for all~$r$,~$s \geq 0$~\cite[Proposition 2.14]{Handel2000SomeSpaces}.

Let~$k \geq 2$ be an integer and~$H = (V, E)$ be a $k$-uniform measurable, or locally independent, or measurable locally independent hypergraph. Denote the set of all independent subsets of~$H$ with cardinality at most~$r$ by~$\ind_r$\index{ I r@$\ind_r$}. A subset~$K$ of~$V$ is called a~\defi{clique}\index{clique} of~$H$ if every $k$-subset of~$K$ is an edge. Denote the set of all cliques of cardinality at most~$r$ by~$\clique_r$\index{ K r@$\clique_r$}. If~$|K| < k$, then~$K$ vacuously satisfies this definition, so if~$r < k$,~$\clique_r = \Sub{V}{r}$. Define the spaces~$\ind_{=r} = \ind_r \cap \Sub{V}{=r}$\index{ I =r@$\ind_{=r}$} and~$\clique_{=r} = \clique_r \cap \Sub{V}{=r}$\index{ K =r@$\clique_{=r}$}.

We will often switch between the identification~$E = \clique_{=k}$ and the identification of~$E$ with a subset of~$V^k$. Call~$\overline{E} = \Sub{V}{=k} \setminus E = \ind_{=k}$ the set of~\defi{nonedges} of~$H$, and define the~\defi{complement hypergraph}\index{complement hypergraph}~$\overline{H}$\index{ Hbar@$\overline{H}$} of~$H$ by~$\overline{H} = (V, \overline{E})$.

\section{Three hierarchies and how they compare}
We now formulate analogues of the moment hierarchy and the block moment hierarchy from Section~\ref{sec:independence-number-finite-graph} of the introduction for the independence number of a uniform measurable hypergraph, and show that for uniform homogeneous hypergraphs, these hierarchies are stronger than the completely positive hierarchy.

Let~$k \geq 2$ be an integer,~$V$ be a compact Hausdorff space, and~$H = (V, E)$ a $k$-uniform measurable hypergraph with Borel measure~$\mu$. For an integer~$r \geq 1$, let~$M : C(\Sub{V}{2r}) \to C_{\sym}(\Sub{V}{r})$ be the operator induced by the union map:
\[
  (M\phi)(S, T) = \phi(S \cup T).
\]
The~\defi{$r$th level} of the~\defi{moment hierarchy}\index{moment hierarchy!locally independent hypergraph} for~$H$ is
\index{ Mr@$\lass_r(H)$!measurable hypergraph}\[
  \begin{optprob}
    \lass_r(H) = \sup &\onerow{\int_{\Sub{V}{=1}} \phi(S)\, d\mu^1_{\sub}(S)}\\
    &\phi(\emptyset) = 1,\\
    &\phi(S) = 0  &\text{ if }S \in \Sub{V}{2r} \text{ not independent},\\
    &\onerow{\phi \in C(\Sub{V}{2r}),\quad M\phi \in C_{\sym}(\Sub{V}{r})_{\succeq 0}.}
  \end{optprob}
\]

For~$r \geq 2$ and~$Q \in \Sub{V}{r-2}$, let~$M_{Q} : C(\Sub{V}{r}) \to C_{\sym}(\Sub{V}{1})$ be the operator
\[
  M_{Q}K(S, T) = K(Q \cup S \cup T).
\]
The~\defi{$r$th level} of the~\defi{block moment hierarchy}\index{block moment hierarchy!measurable hypergraph} is the program
\index{ blockMr@$\kpb_r(H)$!measurable hypergraph}\[
  \begin{optprob}
    \kpb_r(H) = \sup &\onerow{\int_{\Sub{V}{=1}} \phi(S)\, d\mu^1_{\sub}(S)}\\
    &\phi(\emptyset) = 1,\\
    &\phi(S) = 0  &\text{if }S \in \Sub{V}{r} \\
    &&\text{not independent},\\
    &\onerow{\phi \in C(\Sub{V}{r}),}\\
    &M_Q\phi \in C_{\sym}(\Sub{V}{1})_{\succeq 0} & \text{for all } Q \in \ind_{r-2}.
  \end{optprob}
\]

The restriction of a feasible solution of~$\lass_{r+1}(H)$ to~$\Sub{H}{2r}$ is continuous~\cite[Proposition 2.4]{Handel2000SomeSpaces}, hence is a feasible solution of~$\lass_r(H)$. A similar statement holds for~$\kpb_{r+1}(H)$ and~$\kpb_r(H)$. Thus,
\[
  \lass_1(H) \geq \lass_2(H) \geq \cdots \qquad \text{ and }\qquad \kpb_1(H) \geq \kpb_2(H) \geq \cdots.
\]
Moreover, if we have~$M\phi \in C_{\sym}(\Sub{V}{r})_{\succeq 0}$, then for all sets~$Q \in \ind_{r-1}$ and~$S$,~$T \in \Sub{V}{1}$,
\[
  M_Q\phi(S, T) = \phi(Q \cup S \cup T) = M\phi(Q \cup S, Q \cup T),
\]
so~$M_Q \phi \succeq 0$ for all~$Q \in \ind_{r-1}$, and~$\lass_{r}(H) \leq \kpb_{r+1}(H)$.
\begin{theorem}%
  \label{thm:lasser-and-kpn-upper-bound}
  Let~$k \geq 2$ be an integer, and~$H = (V, E, \Gamma)$ be a $k$-uniform homogeneous hypergraph with~$\Gamma$ a compact group. If~$\alpha(H) > 0$ and~$r \geq k - 1$, then~$\lass_r(H) \geq \alpha(H)$ and~$\kpb_{r+1}(H) \geq \alpha(H)$.
\end{theorem}
\begin{proof}
  In light of the inequality~$\lass_r(H) \leq \kpb_{r+1}(H)$, it suffices to prove that~$\lass_r(H) \geq \alpha(H)$ for all~$r \geq k-1$.

  Suppose~$H$ is homogeneous under a compact group~$\Gamma$ and~$\mu$ is the quotient of the Haar measure of~$\Gamma$. Fix~$r \geq k - 1$ an integer and~$I \subseteq V$ an independent set with nonzero measure. Define~$F(v) = \Av_{\Gamma}\1_I^{\otimes 2r}$, which by Lemma~\ref{lem:averaged-rank-one-is-continuous} is continuous. Since~$\1_I$ only takes values in~$\{0, 1\}$, the function~$F$ only depends on~$\setof{v}$, that is, if~$v$,~$w \in V^{2r}$ such that~$\setof{v} = \setof{w}$, then~$F(v) = F(w)$.

  Define~$\phi \in C(\Sub{V}{2r})$ by~$\phi(\emptyset) = 1$ and~$\phi(\setof{v}) = F(v)$ for all~$v \in V^r$. To show that~$\phi$ is indeed continuous, it is enough to show it is continuous on~$\Sub{V}{2r}\setminus \{\emptyset\}$. This is true, because~$F$ is continuous and~$\Sub{V}{2r}$ has the quotient topology under~$\setof{\, \cdot\, }$. It then follows that~$\phi \in M^{-1} C_{\sym}(\Sub{V}{r})_{\succeq 0}$, for example by looking at the restrictions of~$M\phi$ to finite principle submatrices.

  Now,~$\int_{\Sub{V}{=1}} \phi(S)\, d\mu^1_{\sub}(S) = \int_V F(\setof{v})\, d\mu(v) = \mu(I)$. Moreover,~$\Gamma$ preserves edges, so~$\phi(S) = 0$ if~$S$ is not independent, which concludes the proof.
\end{proof}

Recall definitions~\eqref{eqn:big-theta-number} of~$\vartheta_{\bt}(H, \cC)$ and~\eqref{eqn:small-theta-number} of~$\vartheta_{\st}(H, \cC)$, and recall that~$S_{\cont} = S \cap C(V, k)$ for any~$S \subseteq L^2(V, k)$.
\begin{theorem}%
  \label{thm:comparison-lass-kpb-crv-locally-independent-graphs}
  Suppose~$k \geq 2$ is an integer and~$H = ((V, \A, \mu), E)$ is a $k$-uniform measurable hypergraph. If~$V$ is a compact Hausdorff space and~$\mu$ is a finite Borel measure on~$V$, then, for every~$r \geq 1$,
  \[
    \lass_{k+r}(H) \leq \kpb_{k+r+1}(H) \leq \vartheta_{\st}(H, C_r(V, k)^*_{\cont})^{1/(k-1)}
  \]
  and
  \[
    \lass_{k+r}(H) \leq \kpb_{k+r+1}(H) \leq \vartheta_{\bt}(H, C_r(V, k)^*_{\cont}).
  \]
\end{theorem}
\begin{proof}
  Fix an integer~$r \geq 1$. First show the inequality
  \[
    \kpb_{k+r + 1}(H) \leq \vartheta_{\st}(H, C_r(V)^*_c)^{1/(k-1)}.
  \]

  Let~$\phi \in C(\Sub{V}{k+r+1})$ be a feasible solution of~$\kpb_{k+r+1}(H)$ with nonzero objective value. Define the function~$F : V^{k+r} \to \R$ by
  \[
    F(v, w) = \phi(\setof{v} \cup \setof{w}) = M_{\{v_3, \ldots, v_k\} \cup \setof{w}}(v_1, v_2),
  \]
  where~$v = (v_1, \ldots, v_k) \in V^k$ and~$w \in V^r$. Then,~$F$ is a continuous and symmetric~$(k + r)$-tensor, so~$\T_r^*F$ is continuous and for~$v = (v_1, \ldots, v_k)\in V^k$,
  \[
    \begin{split}
      (\T_r^*F)(v) &= \int_{V^r} F(v, w)\, d\mu^r(w)\\
      &=\int_{V^r} \phi(\setof{v} \cup \setof{w})\, d\mu^r(w)\\
      &=\int_{V^r} M_{\{v_3, \ldots, v_k\} \cup \setof{w}}\phi(v_1, v_2)\, d\mu(w).
    \end{split}
  \]

  Since~$M_{\{v_3, \ldots, v_k\} \cup \setof{w}}\phi$ is positive semidefinite for every~$w \in V^r$ and~$v_i \in V$ for~$3 \leq i \leq k$,~$\T_r^*F$ is slice positive. To see that~$\T_r^*F \in C_r(V,k)^*_c$, it is enough to see that~$F \geq 0$. This is true, because if~$\setof{v} \cup \setof{w}$ is not independent,~$F(v, w) = \phi(\setof{v} \cup \setof{w}) = 0$; otherwise,~$F(v, w) = M_{\{v_2, \ldots, v_k\} \cup \setof{w}}(v_1, v_1)$ is a diagonal entry of a continuous positive-semidefinite kernel, thus nonnegative.

  Let~$A = \T_r^*F$. The above also shows that~$A(v) = 0$ for all~$v \in E$. Let~$\tau = \int_{V}A(v, \cdots, v)\,d\mu(v)$. The objective of the remainder of the proof is to show that~$\tau > 0$, and that~$\tau^{-1}A$ is a feasible solution of~$\vartheta(H, C_r(V,k)^*)$ with objective value at least~$\int_{\Sub{v}{=1}} \phi(S)\, d\mu^1_{\sub}(S)$.

  For an integer~$0 \leq t \leq k + r + 1$, write
  \[
    \Phi_t = \int_{V^t} \phi(\setof{v})\, d\mu^t(v).
  \]
  First, show that the matrix
  \begin{equation}%
    \label{eqn:normalization-matrix}
    \begin{pmatrix}
      \Phi_t & \Phi_{t+1}\\
      \Phi_{t+1} & \Phi_{t+2}
    \end{pmatrix}
  \end{equation}
  is positive semidefinite for all integers~$0 \leq t \leq k+r-1$.

  Indeed, fix an integer~$t$ so that~$0 \leq t \leq k + r - 1$ and let~$B: \Sub{V}{1}^2 \to \R$ be such that
  \[
    B(S, T) = \int_{V^t} \phi(\setof{v} \cup S \cup T)\, d\mu^t(v) = \int_{V^t} (M_{\setof{v}})(S, T)\, d\mu^t(v).
  \]
  Then~$B$ is a positive-semidefinite kernel, for example by Fubini's theorem, and~$B(S, T)$ only depends on~$S \cup T$. Moreover,
  \[
    \begin{split}
      B(\emptyset, \emptyset) &= \int_{V^t} \phi(\setof{v})\, d\mu(v) = \Phi_t,\\
      \int_{V} B(\emptyset, \{x\})\, d\mu(x) &= \int_V\int_{V^t} \phi(\setof{v} \cup \{x\})\, d\mu^t(v)d\mu(x) = \Phi_{t+1},\text{ and}\\
      \int_{V^2} B(\{x\}, \{y\})\, d\mu^2(x,y) &= \int_{V^2} \int_{V^t} \phi(\setof{v} \cup \{x, y\})\, d\mu^t(v) d\mu^2(x, y)= \Phi_{t+2},
    \end{split}
  \]
  so the matrix in~\eqref{eqn:normalization-matrix} is positive semidefinite. Indeed, for all~$x_1$ and~$x_2 \in \R$,
  \[
    (x_1 , x_2)
    \begin{pmatrix}
      \Phi_t & \Phi_{t+1}\\
      \Phi_{t+1} & \Phi_{t+2}
    \end{pmatrix}
    \begin{pmatrix}
      x_1\\
      x_2
    \end{pmatrix} = \langle B, (x_1 \1_{\Sub{V}{=0}} + x_2 \1_{\Sub{V}{=1}})^{\otimes 2}\rangle \geq 0.
  \]

  Since~$\Phi_0 = 1$ and since~$\phi$ has objective value~$\Phi_1 > 0$, it follows that~$\Phi_2 > 0$ as well. Repeating the argument, it follows that~$\Phi_t > 0$ for all~$t$. Hence, for fixed~$t$,~$\Phi_t \Phi_{t+2} - \Phi_{t+1}^2 \geq 0$, whence~$\Phi_{t+2}\Phi_{t+1}^{-1} \geq \Phi_{t+1}\Phi_t^{-1}$. Apply the inequality repeatedly to get
  \[
    \Phi_{k + r + 1} \Phi_{r+1}^{-1} = \Phi_{k + r + 1} \Phi_{k+ r}^{-1} \Phi_{k+r}\Phi_{k+r-1}^{-1} \Phi_{k+r-1}\cdots \Phi_{r+2} \Phi_{r+1}^{-1} \geq (\Phi_1 \Phi_0^{-1})^{k - 1}.
  \]
  Moreover,~$\Phi_{r+1} = \tau$ and~$\Phi_{k+r+1} = \langle A, \1^{\otimes k}\rangle$, hence~$\tau^{-1}A$ is a feasible solution of~$\vartheta_{\st}(H, C_r(V, k)^*_{\cont})$ with objective value at least~$\Phi_1^{k-1}$, proving the inequality~$\kpb_{k+r+1}(H) \leq \vartheta_{\st}(H, C_r(V, k)^*_{\cont})^{1/(k-1)}$.

  The proof of the inequality~$\kpb_{k+r+1}(H) \leq \vartheta_{\bt}(H, C_r(V,k)_{\cont})^*$ is entirely the same, but with the final step suitably adjusted. This completes the proof.
\end{proof}

\section{Convergence of the completely positive hierarchy}%
\label{sec:conv-cp-hierarchy}
The main objective of this section is to show that for certain measurable graphs~$G$ the completely positive hierarchy converges to~$\alpha(G)$. These results are based on Bekker, Kuryatnikova, Oliveira, and Vera~\cite[Theorems 5.5 and 5.6]{Bekker2026OptimizationSpaces}, with the difference that the result for thick edge sets no longer depends on a group action, and is complemented by Theorem~\ref{thm:completely-positive-theta-is-exact-thick-edges} to show the hierarchy indeed converges to the independence number. Another small difference with~\cite{Bekker2026OptimizationSpaces} is that convergence in the $\ortho(n)$-homogeneous setting on~$S^{n-1}$ no longer requires the edge set to be closed.

The ideal situation for convergence of the hierarchy is when the feasible regions of each level lie in a mutual compact set and the constraints and objective are continuous. These kinds of problems are rare, as compactness of the feasible region and continuity of the constraints are competing properties: the more continuous linear functionals a topological vector space has, the fewer compact sets there are. It turns out, however, that we get a little lucky in this regard. The situation is particularly favorable when the edge set is thick.

As everything in this section happens on a graph, write~$\CP(V) = \CP(V, 2)$\index{ CPV@$\CP(V)$} and~$C_r(V) = C_r(V,2)$\index{ CrV@$C_r(V)$}. A slice-positive symmetric kernel is just a positive-semidefinite kernel. Moreover,~$\vartheta_{\st}(G, \cC)$ and~$\vartheta_{\bt}(G, \cC)$ coincide for~$k=2$, so we denote both by~$\vartheta(G, \cC)$.

\begin{theorem}%
  \label{thm:convergence-cp-hierarchy-thick-edges}
  Let~$G = ((V, \A, \mu), E)$ be a measurable graph with~$(V, \A, \mu)$ a finite countably generated measure space. If~$\alpha(G) > 0$ and~$E$ is thick, then~$\alpha(G) = \lim_r \vartheta(G, C_r(V)^*)$.
\end{theorem}
\begin{proof}
  That~$\lim_r \vartheta(G, C_r(V)^*) \geq \alpha(G)$ follows from the discussion in \S~\ref{ch:completely-positive-formulations-meas-ind-num}.\ref{sec:upper-bounds-measurable-independence-number}.

  For the other inequality, take for all~$r$ a feasible solution~$A_r \in L^2_{\sym}(V)$ of~$\vartheta(G, C_r(V)^*)$ with objective value at least~$\alpha(G)/2$. Each~$A_r$ is positive semidefinite, so that~$\|A_r\|_2 \leq \|A_r\|_{\B^1} = \Tr(A_r) = 1$ for all~$r$, hence the sequence~$(A_r)_r$ lies in the unit ball of~$L^2_{\sym}(V)$. The unit ball in~$L^2_{\sym}(V)$ is compact since~$L^2_{\sym}(V)$ is reflexive~\cite[Theorem V.4.2]{Conway2010AAnalysis}. By Lemma~\ref{lem:sufficient-conition-finite-rank-approximation},~$L^2_{\sym}(V)$ is separable, and its unit ball is closed and bounded, so it is metrizable~\cite[Theorem V.5.1]{Conway2010AAnalysis}. Thus,~$(A_r)_r$ has a weakly converging subsequence; assume the sequence itself has weak limit~$A$.

  Each~$C_r(V)^*$ is weakly closed, thus Theorem~\ref{thm:Polyas-theorem-finite-measure} says that~$A \in  \bigcap_r C_r(V)^*$ and~$\bigcap_r C_r(V)^* = \CP(V)$. In particular, since~$\langle A, \1_E\rangle = \lim_r\langle A_r, \1_E\rangle = 0$ and~$A \geq 0$, we can consider~$A$ to be the representant of its $L^2$-equivalence class such that~$A(x,y) = 0$ for all~$xy \in E$.

  Finally,~$A$ is nonzero:~$\langle A, J\rangle = \lim_r \langle A_r, J\rangle = \alpha(G)/2 > 0$. Take the spectral decomposition~$A = \sum_n \lambda_n f_n \otimes f_n$, then for all~$N \in \N$
  \[
    \sum_{n=1}^N \lambda_n = \langle A, \sum_{n=1}^N f_n\otimes f_n\rangle = \lim_{r \to \infty} \langle A_r, \sum_{n=1}^{N} f_n \otimes f_n\rangle \leq \lim_r \Tr(A_r) = 1,
  \]
  so~$0 \neq \Tr(A) \leq 1$. Therefore,~$\Tr(A)^{-1}A$ is a feasible solution of~$\vartheta(G, \CP(V))$, with~$\langle \Tr(A)^{-1}A, J\rangle \geq \lim_r \langle A_r, J\rangle$, as required.  Conclude by Theorem~\ref{thm:completely-positive-theta-is-exact-thick-edges}.
\end{proof}

Combining Theorem~\ref{thm:convergence-cp-hierarchy-thick-edges} above with Lemma~\ref{lem:upper-bound-thick-edges}, we obtain the following theorem which restricts the feasible region to only continuous kernels.
\begin{theorem}%
  \label{thm:convergence-cp-hierarchy-locally-independent-graphs}
  Let~$G = ((V, \B, \mu), E)$ be a measurable locally independent graph with~$V$ compact and second countable with~$\mu$ a finite inner regular Borel measure. If~$\alpha(G) > 0$ and~$E$ is thick, then
  \[
    \alpha(G) = \lim_{r \to \infty} \vartheta(G, C_r(V)^*_{\cont}).
  \]
\end{theorem}
\begin{proof}
  This follows immediately from Theorem~\ref{thm:convergence-cp-hierarchy-thick-edges} and Lemma~\ref{lem:upper-bound-thick-edges}.
\end{proof}

\sectionbreakafterproof

For nonthick edge sets it is difficult to make a general statement, even under the assumption that~$V$ is homogeneous under a compact subgroup of~$\Aut(G)$. Indeed, it seems that all we have at our disposal is weak compactness of the unit ball in~$L^2_{\sym}(V)$, but weak limits do not necessarily preserve zeros on sets that are not thick.

In the following case this can be salvaged through Schoenberg's theorem. For~$D \subseteq [-1,1)$, define~$E(D) \subseteq \Sub{S^{n-1}}{=2}$ by~$xy \in E(D)$  if and only if~$x^{\tr}y \in D$. We always understand~$S^{n-1}$ to be equipped with the uniform surface measure~$\omega$, which is the quotient of a Haar measure of the orthogonal group~$\ortho(n)$ on~$\R^n$. That this graph satisfies all conditions of Theorem~\ref{thm:completely-positive-is-exact-homogeneous-graph} is explained in~\cite[\S 5]{DeCorte2022CompleteSets} and references therein.

\begin{theorem}%
  \label{thm:convergence-cp-hierarchy-sphere}
  If~$n \geq 3$ is integer,~$D \subseteq (-1, 1)$, and~$G = (S^{n-1}, E(D), \ortho(n))$ is a measurable homogeneous graph with the uniform surface measure~$\omega$ such that~$\alpha(G) > 0$, then
  \[
    \alpha(G) = \lim_r\vartheta(G, C_r(V)_c).
  \]
\end{theorem}

The following details are found in Andrews, Askey, and Roy~\cite{Andrews1999SpecialFunctions}. Denote by~$\omega_n$ the total measure~$\omega(S^{n-1})$.

For~$k \geq 0$, let~$H^n_k \subseteq C(S^{n-1})$ be the space of $n$-variable spherical harmonics of degree~$k$; denote the dimension of~$H^n_k$ by~$h^n_k$, and let~$S^{n}_{k,1}, \ldots S^n_{k,h^n_k}$ be a complete orthogonal system of~$H^n_k$. The set~$\{S^n_{k,i}\}_{k, i}$ forms an orthonormal basis of~$L^2(S^{n-1})$.

Let~$P^n_k$ be the Jacobi polynomial of degree~$k$ with the parameters~$\alpha = \beta = (n-3)/2$ normalized by~$P^n_k(1) = 1$. The addition formula~\cite[Theorem 9.6.3]{Andrews1999SpecialFunctions} states that
\[
  P^n_k(x^{\tr} y) = \frac{\omega_n}{h^n_k}\sum_{i=1}^{h^n_k} S^n_{k,i}(x) S^n_{k,i}(y)
\]
for all~$x, y \in S^{n-1}$.

Write~$E^n_k(x, y) = P^n_k(x^{\tr}y)$. The kernels~$E^n_k$ are $\ortho(n)$-invariant and form a complete orthogonal system of the space~$L^2_{\sym}(S^{n-1})^{\ortho(n)}$ of $\ortho(n)$-invariant kernels. We have~$\langle E^n_k, E^n_l\rangle = 0$ if~$k \neq l$ and~$\langle E^n_k, E^n_k \rangle = \omega_n^2 / h^n_k$. So if we write~$A = \sum_k f(k) E^n_k$ and~$B = \sum_k g(k) E^n_k$, then
\begin{equation}%
  \label{eqn:inner-product-formula-unit-sphere}
  \langle A, B\rangle = \sum_{k=0}^{\infty}f(k)g(k)\omega^2_n/h^n_k.
\end{equation}

By the addition formula, each of the kernels~$E^n_k$ is positive semidefinite, hence~$\sum_kf(k)E_k^n$ is positive semidefinite if and only if~$f(k) \geq 0$ for all~$k$. Schoenberg's theorem~\cite{Schoenberg1942PositiveSpheres} states that if a kernel~$K \in L^2_{\sym}(S^{n-1})$ is continuous, $\ortho(n)$-invariant, and positive semidefinite, then there is a nonnegative sequence~$f \in \ell^1$ such that
\[
  K(x, y) = \sum_{k \in \N} f(k) E^n_k(x, y) = \sum_{k \in \N} f(k)P^n_k(x^{\tr}y)
\]
with absolute and uniform convergence on~$S^{n-1}\times S^{n-1}$.

\begin{proof}[Proof of Theorem~\textup{\ref{thm:convergence-cp-hierarchy-sphere}}]
  By Lemma~\ref{lem:upper-bound-homogeneous}, for each~$r$,~$\vartheta(G, C_r(S^{n-1})^*_c) \geq \alpha(G)$, hence~$\lim_r \vartheta(G, C_r(S^{n-1})^*_c) \geq \alpha(G)$.

  To prove the other inequality, see that~$\Av_{\ortho(n)} C_r(S^{n-1})^* \subseteq C_r(S^{n-1})^*$ and that each feasible solution~$A$ of~$\vartheta(G, C_r(S^{n-1})^*_c)$ is trace class, so by Theorem~\ref{thm:avering-on-trace-class},~$\Av_{\ortho(n)}A \in  C_r(S^{n-1})_c^*$. Hence, we may restrict ourselves to $\ortho(n)$-invariant feasible solutions, which are automatically continuous. We show that~$\lim_r\vartheta(G, C_r(S^{n-1})^*_{\cont}) \leq \vartheta(G, \CP(S^{n-1})_{\cont})$.

  For every integer~$r \geq 1$, let~$A_r$ be an $\ortho(n)$-invariant feasible solution of~$\vartheta(G, C_r(S^{n-1})^*_c)$ with objective value at least~$\alpha(G) / 2 > 0$. Use Schoenberg's theorem to obtain a nonnegative sequence~$f_r \in \ell^1$ such that~$A_r = \sum_k f_r(k) E^n_k$. Then,
  \begin{equation}%
    \label{eqn:trace-invariant-polynomial}
    \omega_n \sum_{k \in \N} f_r(k) = \int_{S^{n-1}} A_r(x,x)\, d\omega(x) = 1.
  \end{equation}

  Let~$c_0$ be the space of all real-valued sequences that vanish at infinity. Under the supremum norm,~$c_0$ is a Banach space and its dual is~$\ell^1$ by the duality~$(f, g) = \sum_k f(k)g(k)$.

  By~\eqref{eqn:trace-invariant-polynomial} it follows that~$\omega_nf_r$ is in the unit ball of~$\ell^1$ for all~$r$. By Banach-Alaoglu~\cite[Theorem V.4.2]{Conway2010AAnalysis}, the unit ball of~$\ell^1$ is weak* compact, and the weak* topology on~$\ell^1$ is separable, thus the unit ball is weak* metrizable~\cite[Theorem V.5.1]{Conway2010AAnalysis}. Therefore, the sequence~$(f_r)_r$ has a weak*-converging subsequence; without loss of generality, assume the sequence itself converges with limit~$f$.

  It follows that~$f$ is nonnegative and~$\omega_n\|f\|_1 \leq 1$. By Schoenberg's theorem, the kernel~$A = \sum_k f(k) E^n_k$ is continuous, $\ortho(n)$-invariant and positive semidefinite. Moreover, if~$B \in L^2_{\sym}(S^{n-1})^{\ortho(n)}$, then~$B = \sum_k g(k)E^n_k$ for some real-valued sequence~$g$, and since the sequence~$(g(k)\omega_n^2/h^n_k)_k$ vanishes at infinity, from~\eqref{eqn:inner-product-formula-unit-sphere}~$\lim_r \langle A_r, B\rangle = \langle A, B\rangle$ follows.

  Thus,~$\langle A, J\rangle = \lim_r\langle A_r, J\rangle > 0$, so that~$A$ is nonzero and
  \[
    0 < \tau = \int_V A(x,x)\, d\omega(x) = \omega_n\|f\|_1 \leq 1.
  \]
  It is left to show that~$A$ is completely positive and vanishes on the edges, thus showing that~$\tau^{-1}A$ is a feasible solution of~$\vartheta(G, \CP(V)_c)$ and that the inequality~$\vartheta(G, \CP(V)) \geq
  \lim_r \vartheta(G, C_r(V)^*_c)$ holds, finishing the proof by~\cite[Theorem 5.1]{DeCorte2022CompleteSets}.

  Since the sequence~$A_r$ converges under the weak-$L^2$ topology to~$A$, and each~$C_r(V)^*$ is weakly closed, Theorem~\ref{thm:Polyas-theorem-finite-measure} states that~$A \in \CP(V)_c$. Finally, for~$n \geq 3$ the asymptotic formula for the Jacobi polynomials~\cite[Theorem 8.21.8]{Szego1975OrthogonalPolynomials} implies that~$(P_k^n(t))_k$ vanishes at infinity for all~$t \in (-1, 1)$. So, for all~$xy \in E(D)$ with~$D \subseteq (-1, 1)$,
  \[
    A(x,y) = \sum_{k \in \N} f(k) P^n_k(x^{\tr}y) = \lim_{r \to \infty} \sum_{k \in \N} f_r(k)P^n_k(x^{\tr}y) = 0,
  \]
  which concludes the proof by Theorem~\ref{thm:completely-positive-is-exact-homogeneous-graph}.
\end{proof}

\section{Discussion and future directions}%
\label{sec:discussion-finite-measure-spaces}
The results on convergence of the completely positive hierarchy are very incomplete. In the thick setting, it is clear that the square-integrable setting is insufficient for $k$-uniform hypergraphs with~$k > 2$, as the normalization~$\Tr\T_{k-2}^*A = 1$ only bounds the $\mi{L^{k/(k-1)}}$-norm, and not the $L^2$-norm. It is also clear that the space~$L^{k/(k-1)}_{\sym}(V, k)$ is too large; even for~$k = 2$ we really only work in the trace class. Thus, a well-structured study of~$\vartheta_{\bt}$ comprises a study of $k$-tensor analogues of the trace norm. To extend the convergence of the completely positive hierarchy to $k$-uniform homogeneous hypergraphs on the unit sphere, we have to look for a $k$-tensor version of Schoenberg's theorem. Tensor analogues of the necessary arguments are known; in particular, Castro-Silva~\cite{Castro-Silva2023GeometricalConfigurations} seems to have produced the important details.  This is discussed in more detail in the concluding chapter of this thesis.

The most challenging problem in the homogeneous setting is extending the convergence result to other groups. There are several compact groups whose representation theory is similar to that of the sphere, so that the proof of Theorem~\ref{thm:convergence-cp-hierarchy-sphere} goes through with minimal changes. These are the continuous, compact, two-point homogeneous spaces with real dimension at least~$2$: the sphere, the real, complex and quaternionic projective spaces, and the octonionic projective plane. For these spaces, a theorem like Schoenberg's theorem is available~\cite[Theorem 3.1]{Oliveira2013ASpaces}.

Extension beyond these spaces requires a different argument. For example, on the real circle~$S^1$ the sequence of continuous, $O(2)$-invariant, and positive-semidefinite kernels
\[
  \bigl((x, y) \mapsto e^{inx^{\tr}y}\bigr)_{n \in \N}
\] has weak* limit~$0$, but for orthogonal~$x$ and~$y$ the sequence~$(e^{in x^{\tr}y})_{n \in \N}$ is the constant-1 sequence. Since the proof of Theorem~\ref{thm:convergence-cp-hierarchy-sphere} hinges on weak* convergence, even though it is known that~$\vartheta((S^1, E(0)), C_1(V)_c) = \alpha((S^1, E(0)))$, the method fails!

It is therefore necessary to study positive type functions of compact groups more systematically. This requires moving away from the generic setting of $L^k$ functions. It seems that the Fourier algebra offers a suitable setting, as it is spanned by the positive type functions. Again, see the concluding chapter of this thesis for more details.

\chapter[Witsenhausen's problem]{Application: Witsenhausen's problem}%
\label{ch:witsenhausen}
In this chapter, we apply completely positive-programming methods to find upper bounds on the maximum fraction of a unit sphere that can be covered by a set containing no orthogonal pairs---Problem~\ref{it:problem-1} from the introduction. It is an example of an independent-set problem on a homogeneous graph under a compact but infinite group. We will see how the completely positive-programming bounds can be implemented using semidefinite programming. This offers the first bounds for this class of problems that improve on the linear programming bounds introduced by Bachoc, Nebe, Oliveira, and Vallentin~\cite{Bachoc2009LowerNumbers}, Oliveira~\cite{Oliveira2009NewOptimization}, DeCorte, Oliveira, and Vallentin~\cite{DeCorte2022CompleteSets}, and DeCorte~\cite{DeCorte2015TheGraphs}, and results in the best bounds known. The results and exposition are taken from Bekker, Kuryatnikova, Oliveira, and Vera~\cite{Bekker2026OptimizationSpaces}.

\sectionbreak

A subset of the unit sphere~$S^{n-1} = \{\, x\in \R^n : \|x\|=1\, \}$~\defi{avoids orthogonal pairs} if it does not contain pairs of orthogonal vectors. \defi{Witsenhausen's problem}\index{problem!Witsenhausen's}~\cite{Witsenhausen1974SphericalPairs} asks for the maximum density that a measurable subset of~$S^{n-1}$ can have if it avoids orthogonal pairs. That is, we want to know the value of
\index{ a n@$\alpha_n$}\[
  \alpha_n = \{\, \omega(I)/\omega_n : I \subseteq S^{n-1} \text{ is measurable and avoids orthogonal pairs}\, \},
\]
where~$\omega$ denotes the standard surface measure of~$S^{n-1}$, and~$\omega_n = \omega(S^{n-1})$. We refer to the quantity~$\omega(I) / \omega_n$ for a measurable~$I \subseteq S^{n-1}$ as the~\defi{density} of~$I$.

Fix~$d \in S^{n-1}$. Witsenhausen~\cite{Witsenhausen1974SphericalPairs} observed that the union of two open antipodal spherical caps of spherical radius~$\pi/4$, i.e. the set
\[
  \{\, x \in S^{n-1} : |e^{\tr} x| > \cos{\pi/4}\, \},
\]
avoids orthogonal pairs, hence~$\alpha_n$ is at least the density of this set, which is~$O(n^{-1/2}2^{-n/2})$. Kalai~\cite[Conjecture 2.8]{Kalai2015SomeProblem} conjectured that this construction is optimal, that is, that~$\alpha_n$ is exactly the density of these two caps; this is known as the~\defi{double-cap conjecture}\index{double cap conjecture@double-cap conjecture}. A version of the double-cap conjecture for the complex unit sphere has an interpretation in quantum information theory~\cite{Montina2011CommunicationMeasurement}.

The canonical basis vectors of~$\R^n$ are~$n$ pairwise-orthogonal unit vectors. Any set that avoids orthogonal pairs can contain at most one of them. It then follows from a simple averaging argument that~$\alpha_n \leq 1/n$. This upper bound was also given by Witsenhausen~\cite{Witsenhausen1974SphericalPairs}; it is quite far from the lower bound of the double-cap conjecture for all~$n \geq 3$. For~$n = 2$, the lower and upper bounds coincide. Frankl and Wilson~\cite{Frankl1981IntersectionConsequences} were the first to give an asymptotic upper bound for~$\alpha_n$ that decreases exponentially with the dimension~$n$.

On the unit sphere, distance and inner product are related; a set of points on the sphere avoids orthogonal pairs if it avoids pairs of points at distance~$\sqrt{2}$. More generally, let~$V$ be metric space with metric~$d$ and let~$D \subseteq (0, \infty)$ be a set of~\defi{forbidden} distances. We say that a set~$I \subseteq  V$~\defi{avoids} the distances in~$D$ or that it is a~\defi{$D$-avoiding set}\index{distance avoiding set@distance-avoiding set} if~$d(x,y) \notin D$ for all~$x, y \in I$. In these terms, Witsenhausen's problem asks for the maximum density of a~$\sqrt{2}$-avoiding set on the sphere equipped with the Euclidean distance.

Distance-avoiding sets can be modeled as independent sets of graphs. Given a metric space~$V$ with metric~$d$ and a set~$D$ of forbidden distances, let~$G$ be the graph with vertex set~$V$ in which~$x$,~$y \in V$ are adjacent if and only if~$d(x, y) \in D$; call such a graph~$G$ a~\defi{distance graph}\index{distance graph}. The independent sets of~$G$ are exactly the~$D$-avoiding sets.

Denote the distance graph with vertex set~$S^{n-1}$ and forbidden distance~$\sqrt{2}$ by~$G_n$. Witsenhausen's problem can be seen as an independent-set problem on~$G_n$ with the measure~$\omega / \omega_n$, i.e.~$\alpha_n = \alpha(G_n)$. We will use the hierarchies derived in Chapter~\ref{ch:Completely positive programming on compact spaces} to bound~$\alpha_n$ from above.

There is a special role for the completely positive hierarchy: although Theorem~\ref{thm:comparison-lass-kpb-crv-locally-independent-graphs} states that~$\kpb_3(G_n)$ should give a bound at least as good as~$\vartheta(G_n, C_1(V,2)^*_c)$, the usual way of implementing this three-point bound as a polynomial optimization problem does not result in a rigorous upper bound. However, a particular restriction of the completely positive hierarchies does, and this leads to the first use of optimization hierarchies for Witsenhausen's problem, giving the best upper bounds known.

\section{Invariant positive-semidefinite kernels}%
\label{sec:inv-kernels-and-tensor}
The orthogonal group on~$\R^n$ is denoted~$\ortho(n)$\index{ O n@$\ortho(n)$}. Witsenhausen's problem is invariant under the action of~$\ortho(n)$, in the sense that it is defined on a graph that is vertex transitive under the action of~$\ortho(n)$. We exploit this to significantly reduce the size of the optimization problems. The thesis of De Muinck Keizer~\cite{DeMuinckKeizer2025OnGeometry} gives a more structured and detailed account of the contents of this section. For us, though, a straight-forward ad-hoc approach suffices.

We have already seen that averaging a feasible solution of a program in the block moment hierarchy and the completely positive hierarchy preserves feasibility and the objective value. We may thus assume all tensors to be invariant under~$\ortho(n)$. As we will see, we can describe $\ortho(n)$-invariant 3-tensors by $\ortho(n)$- and $\Stab(e)$-invariant kernels, where~$e \in S^{n-1}$ is an arbitrary point.

Positive-semidefinite kernels on~$S^{n-1}$ invariant under~$\ortho(n)$ are easily described by Schoenberg's theorem~\cite{Schoenberg1942PositiveSpheres}; also see the discussion in~\S\ref{ch:Completely positive programming on compact spaces}.\ref{sec:conv-cp-hierarchy}. Let~$n \geq 1$ and~$k \geq 0$. If~$P^n_k$ is the Jacobi polynomial of degree~$k$ with parameters~$\alpha = \beta = (n-3)/2$ normalized by~$P^n_k(1) = 1$, write~$E^n_k(x, y) = P^n_k(x^{\tr}y)$. Schoenberg's theorem~\cite{Schoenberg1942PositiveSpheres} states that a kernel~$K \in L^2_{\sym}(V, 2)$ is continuous, $\ortho(n)$-invariant, and positive semidefinite, if and only if there is a nonnegative sequence~$f \in \ell^1$ such that
\begin{equation}%
  \label{eqn:O(n)-invariant-kernels}
  K(x, y) = \sum_{k \in \N} f(k) E^n_k(x, y) = \sum_{k \in \N} f(k)P^n_k(x^{\tr}y)
\end{equation}
with absolute and uniform convergence on~$S^{n-1}\times S^{n-1}$.

Fix~$e \in S^{n-1}$ and let~$\Stab(e)$ be its stabilizer under the action of~$\ortho(n)$; see~\S\ref{ch:complete-positivity-under-symmetry}.\ref{sec:prel-harmonic-analysis} for more. The kernels invariant under this action were described by Musin~\cite{Musin2014MultivariateSpheres} and Bachoc and Vallentin~\cite{Bachoc2008NewProgramming} as follows.

With~$P^n_k$ as in~\eqref{eqn:O(n)-invariant-kernels}, for integers~$n \geq 2$ and~$k \geq 0$ consider the polynomial
\[
  Q_k^n(u, v, t) = (1-u^2)^{k/2}(1-v^2)^{k/2}P^n_k\biggl(\frac{t - uv}{(1-u^2)^{1/2}(1-v^2)^{1/2}}\biggr)
\]
and let~$Y^n_{k,d}$ be the $(d-k+1)\times(d-k+1)$ matrix given by
\begin{equation}%
  \label{eqn:bachoc-vallentin-kernel}
  (Y^n_{k,d})_{i,j}(u,v,t) = u^iv^jQ^{n-1}_k(u,v,t)
\end{equation}
for~$0 \leq i,j \leq d-k$. If~$K \in C((S^{n-1})^2)$ is~$\Stab(e)$-invariant, then~$K(x, y)$ depends only on the inner products~$e^{\tr}x$,~$e^{\tr}y$, and~$x^{\tr}y$. Bachoc and Vallentin showed that, for any~$d$ and any choice of positive semidefinite matrices~$F_k \in \R^{(d-k+1) \times (d-k+1)}$, the kernel
\begin{equation}%
  \label{eqn:Stab(e)-inv-kernels}
  K(x, y) = \sum_{k=0}^d \langle F_k, Y^n_{k,d}(e^{\tr}x, e^{\tr}y, x^{\tr}y)\rangle
\end{equation}
is~$\Stab(e)$-invariant and positive semidefinite. It is continuous by construction, since it is a polynomial on the three inner products.

Although the above decomposition looks similar to that in Schoenberg's theorem, there is no guarantee that every kernel~$K$ is a pointwise converging sum of the form~\eqref{eqn:Stab(e)-inv-kernels} with~$d = \infty$ and the~$F_k$ infinite matrices. However, we may uniformly approximate a continuous, positive-semidefinite, and~$\Stab(e)$-invariant kernel by kernels of the form~\eqref{eqn:Stab(e)-inv-kernels}.

\section{The failure of the block moment hierarchy}%
\label{sec:failure-3-point-bound}
The three-point bound~$\kpb_3(G_n)$ fails to give an implementable bound. For simplicity, we use an alternative normalization, i.e.
\[
  \begin{optprob}
    \sup & \onerow{\int_{\Sub{S^{n-1}}{=2}}\phi(S)\, d\omega^2_{\sub}(S)}\\
    &\onerow{\int_{\Sub{S^{n-1}}{=1}}\phi(S)\, d\omega^1_{\sub}(S) = 1,}\\
    &\phi(S) = 0 & \text{if $S \in \Sub{S^{n-1}}{3}$ is not independent},\\
    &M_{\{u\}}\phi \succeq 0 & \text{ for all~$u \in S^{n-1}$},\\
    &\onerow{M_{\emptyset}\phi \succeq 0,}\\
    &\onerow{\phi \in C(\Sub{S^{n-1}}{3}).}
  \end{optprob}
\]
In the proof of Theorem~\ref{thm:comparison-lass-kpb-crv-locally-independent-graphs} we showed that the matrix
\[
  \begin{pmatrix}
    \phi(\emptyset) & \int_{\Sub{S^{n-1}}{=1}} \phi(S)\, d\omega^1_{\sub}(S)\\
    \int_{\Sub{S^{n-1}}{=1}} \phi(S)\, d\omega^1_{\sub}(S) & \int_{\Sub{S^{n-1}}{=2}} \phi(S)\, d\omega^2_{\sub}(S)
  \end{pmatrix}
\]
is positive semidefinite, which implies that~$\kpb_3(G_n)$ is bounded from above by this program. We further relax the bound by changing the operators~$M_Q$ for~$|Q| \leq 1$ so that it sends~$\phi$ to the kernel~$K \in C((S^{n-1})^2)$ such that~$K(x,y) = \phi(Q \cup \{x, y\})$, i.e., we disregard the empty set.

We may restrict to~$\ortho(n)$-invariant functions~$\phi$ by applying an averaging operator, because $\Sub{V}{3}$ comes with a natural group action and an invariant measure.

Let~$e \in S^{n-1}$. Given any~$u \in S^{n-1}$, there is an orthogonal transformation~$T$ such that~$Tu = e$, hence if~$\phi$ is an~$\ortho(n)$-invariant feasible solution, then
\[
  (M_{\{u\}}\phi)(x,y) = \phi(\{u, x, y\}) = \phi(\{e, Tx, Ty\}) = (M_{\{e\}} \phi)(Tx, Ty).
\]
It follows that if~$M_{\{e\}}\phi$ is positive semidefinite, so is~$M_{\{u\}}\phi$ for every~$u \in S^{n-1}$. Since~$\phi$ is~$\ortho(n)$-invariant,~$M_{\{e\}}\phi$ is~$\Stab(e)$-invariant. This allows us to rewrite the problem by considering two kernels~$A = M_{\emptyset}\phi$ and~$K = M_{\{e\}}\phi$:
\begin{equation}%
  \label{eqn:invariant-3pb}
  \begin{optprob}
    \sup & \onerow{\int_{\Sub{S^{n-1}}{=2}}\phi(S)\, d\omega^2_{\sub}(S)}\\
    &\onerow{\int_{\Sub{S^{n-1}}{=1}}\phi(S)\, d\omega^1_{\sub}(S) = 1,}\\
    &A(x, y) = K(e, Ty) &\text{for all $x$, $y \in S^{n-1}$}\\
    &&\text{and $T \in \ortho(n)$ with $Tx = e$},\\
    &K(e, x) = K(x, x) &\text{for all $x \in S^{n-1}$},\\
    &K(x, y) = 0 &\text{if $\{e, x, y\}$ is not independent},\\
    &\onerow{A \in C_{\sym}((S^{n-1}))_{\succeq 0}\text{ is $\ortho(n)$-invariant},}\\
    &\onerow{K \in C_{\sym}((S^{n-1}))_{\succeq 0}\text{ is $\Stab(e)$-invariant}.}
  \end{optprob}
\end{equation}

Since~$A$ is $\ortho(n)$-invariant, Schoenberg's theorem can be used to express~$A$ in terms of Jacobi polynomials as in~\eqref{eqn:O(n)-invariant-kernels}. The kernel~$K$ is invariant under~$\Stab(e)$, so the expansion~\eqref{eqn:Stab(e)-inv-kernels} parametrizes a large class of the required kernels.

However, as pointed out, even though the kernels of the form~\eqref{eqn:Stab(e)-inv-kernels} approximate the $\Stab(e)$-invariant kernels uniformly, a pointwise converging sum of the same form with~$d = \infty$ is in general not guaranteed. The constraint ``$K(x, y) = 0$ if~$\{e, x, y\}$ is not independent'' can therefore not be written equivalently in terms of such an expansion, hence it is unclear that the resulting problem would give an upper bound to~$\alpha_n$.

Even when we replace this constraint by a relaxation, for example requiring that~$K(x, y) \in [-\epsilon, \epsilon]$ for some fixed~$\epsilon > 0$, it remains difficult to get a rigorous upper bound on~$\alpha_n$.

Indeed, to solve the modified problem~\eqref{eqn:invariant-3pb} we have to fix the degrees of the polynomials at some point. To get an upper bound, we have to solve a problem of this form to optimality. To do so rigorously we have to use polynomials of high degree, and since~$K$ is parametrized by $3$-variable polynomials, the variable matrices become prohibitively large.

\section{The fix: another hierarchy}%
\label{sec:the-fix}
Recall the definition of slice positive tensors from Section~\ref{ch:completely-positive-formulations-meas-ind-num}.\ref{sec:upper-bounds-measurable-independence-number}. If~$V$ is a compact Hausdorff space equipped with a Radon measure with full support, a function~$F \in C(V^{r})$ with~$r \geq 2$ integer is slice positive if and only if for all~$v \in (S^{n-1})^{r-2}$ the kernel~$(x, y) \mapsto F(x, y, v)$ is a positive-semidefinite kernel.

For an integer~$r \geq 1$, let
\index{ QrVk@$\Q_r(V, k)$}
\begin{multline*}
  \Q_r(V, k) = \{\, A \in L^2_{\sym}(V, k) : \Av_{\mathfrak{S}_{k + r}}(A \otimes \1^{\otimes r} - F) \geq 0\\
  \text{for some slice-positive } F \in L^2(V^{k+r})\, \}.
\end{multline*}

Immediately we see that~$\Q_r(V, k) \subseteq \cC_r(V, k)$ for all~$k$ and~$r$. Moreover,~$\Q_r(V, k) \subseteq \COP(V, k)$ for all~$r$ and~$k$. Indeed, take~$A \in \Q_r(V, k)$ and~$F \in L^2(V^{k+r})$ with~$\Av_{\mathfrak{S}_{k+r}}(A \otimes \1^{\otimes r} - F) \geq 0$. Given a nonnegative~$f \in L^2(V)$ with~$\langle \1, f\rangle \geq 0$ we have
\[
  \begin{split}
    0 &\leq \langle \Av_{\mathfrak{S}_{k+r}}(A \otimes \1^{\otimes r} - F), f^{\otimes (k+r)}\rangle\\
    &= \langle A \otimes \1^{\otimes r}, \Av_{\mathfrak{S}_{k+r}} f^{\otimes (k+r)}\rangle - \langle F, \Av_{\mathfrak{S}_{k+r}}f^{\otimes (k+r)}\rangle\\
    &= \langle A, f^{\otimes k}\rangle\langle\1, f\rangle^r - \langle F, f^{\otimes (k+r)}\rangle.
  \end{split}
\]
Since~$F$ is slice positive,
\[
  \begin{split}
    \langle F, f^{\otimes (k+r)}\rangle &= \int_{V^{k+r-2}} \int_{V^2} F(v, w) f^{\otimes 2}(v)\, d\omega^2(v) f^{\otimes k + r - 2}(w)\, d\omega^{k+r-2}(w)\\
    &\geq 0,
  \end{split}
\]
and we see that~$\langle A, f^{\otimes k}\rangle \geq 0$, so~$A$ is copositive.

One shows, as for example in the proof of~\cite[Theorem 4.1]{Bekker2026OptimizationSpaces} that
\[
  \Q_1(V, k) \subseteq \Q_2(V, k) \subseteq \cdots \subseteq \COP(V, k),
\]
is a hierarchy of inner approximation of~$\COP(V, k)$ stronger than the~$\cC_r(V, k)$ hierarchy; it was proposed by Peña, Vera, and Zuluaga~\cite{Pena2007ComputingProgramming} and extended to the infinite-dimensional setting by Kuryatnikova and Vera~\cite{Kuryatnikova2019TheProblems}.

Given a graph~$G = (V, E)$ such that~$V$ is compact and Hausdorff and equipped with a Radon measure, write~$\Q_r(V) = \Q_r(V,2)$\index{ QrV@$\Q_r(V)$}, and consider the programs~$\vartheta(G, \Q_r(V)^*_c)$. Under the conditions of Chapters~\ref{ch:completely-positive-formulations-meas-ind-num} and~\ref{ch:Completely positive programming on compact spaces}, this gives a hierarchy of bounds for the measurable independence number, namely
\[
  \vartheta(G, \Q_1(V)^*_c) \geq \vartheta(G, \Q_2(V)^*_c)\geq \cdots \geq \alpha(G),
\]
that is at least as strong as the hierarchy~$\vartheta(G, C_r(V,2)^*_c)$. In particular, the convergence results from Chapter~\ref{ch:Completely positive programming on compact spaces} hold.

\sectionbreak

We will implement a version of the bound~$\vartheta(G_n, \Q_1(V)^*)$. For this, we first have to figure out how to describe slice-positive 3-tensors effectively.

Let~$V$ be a~$\Gamma$-space with~$\Gamma$ a compact group. Let~$p$ be the quotient map onto the set of orbits~$p: V^k \to V^k/\Gamma$ and~$R : V^k / \Gamma \to V^k$ be a section of~$p$.

Suppose~$F \in C(V^{k+2})$ is slice-positive and~$\Gamma$-invariant. Consider the function~$K : (V^k/\Gamma)\times V^2 \to \R$ such that
\[
  K(\xi, x, y) = F(x, y, R(\xi))
\]
and for every orbit~$\xi$ let~$K_{\xi}(x, y) = K(\xi, x, y)$; note that~$K$ depends on the choice of~$R$. The kernel~$K_{\xi}$ is continuous and positive semidefinite for every~$\xi$. Moreover, since~$F$ is~$\Gamma$-invariant,~$K_{\xi}$ is~$\Stab({R(\xi)})$-invariant. If we equip~$V^k / \Gamma$ with the quotient topology,~$K$ is continuous if~$R$ is.

Conversely, say~$K : (V^k / \Gamma) \times V^2 \to \R$ is a continuous function such that~$K_{\xi}$ is a positive-semidefinite $\Stab(R(\xi))$-invariant kernel for all~$\xi$. Then we may define a function~$F : V^{k+2} \to \R$ by
\[
  F(x, y, v) = K(p(v), \sigma x, \sigma y)
\]
for all~$\sigma \in \Gamma$ such that~$\sigma v = R(v)$. Such~$F$ is well-defined: if~$\tau v = R(\xi)$, then~$\sigma\tau^{-1}R(\xi) = R(\xi)$, so from the $\Stab(R(\xi))$-invariance of~$K_{\xi}$ we obtain~$K(\xi, \tau x, \tau y) = K(\xi, \sigma x, \sigma y)$.

By construction,~$F$ is slice-positive and~$\Gamma$-invariant. To see the latter, given~$\sigma \in \Gamma$, let~$\tau \in \Gamma$ be such that~$\tau\sigma v = R(p(v))$. Then,
\[
  F(\sigma x, \sigma y, \sigma v) = K(p(v), \tau\sigma x, \tau \sigma y) = F(x, y, v).
\]
If there exists a continuous function~$s: V^k \to \Gamma$ such that~$s(v)v = R(p(v))$ for all~$v \in V^k$, then~$F$ is continuous.

We will now see how~$\vartheta(G_n, \Q_1(S^{n-1})_c^*)$ behaves better than~$\kpb_3(G_n)$ with respect to approximation by the kernels from~\eqref{eqn:Stab(e)-inv-kernels}; indeed, we will see that restricting to kernels of this form gives a subset of~$\Q_1(S^{n-1})$ and therefore a superset of~$\Q_1(S^{n-1})^*$, thus relaxing the program.

If~$Z \in \Q_1(S^{n-1})$, then there is a continuous slice positive~$F : (S^{n-1})^3 \to \R$ such that~$\Av_{\mathfrak{S}_3}(Z \otimes \1 - F) \geq 0$. If~$Z$ is~$\ortho(n)$-invariant, we can assume that~$F$ is $\ortho(n)$-invariant as well, otherwise we simply take~$\Av_{\ortho(n)}F$, which is continuous and slice positive.

There is only one orbit for the action of~$\ortho(n)$ on~$S^{n-1}$; pick~$e \in S^{n-1}$ as its representative. The invariant function~$F$ is continuous and slice positive if and only if there is a continuous, positive-semidefinite, and~$\Stab(e)$-invariant kernel~$K : (S^{n-1})^2 \to \R$ such that~$F(x, y, z) = K(Tx, Ty)$, where~$T$ is any orthogonal matrix such that~$Tz = e$. So the value of~$F(x, y, z)$ depends only on~$e^{\tr}Tx = x^{\tr}z$,~$e^{\tr}Ty = y^{\tr}z$, and~$(Tx)^{\tr}Ty = x^{\tr}y$.

The kernels~\eqref{eqn:Stab(e)-inv-kernels} are positive semidefinite and~$\Stab(e)$-invariant. Fix an integer~$d \geq 1$. Say~$Z$ is the $\ortho(n)$-invariant kernel given by
\begin{equation}%
  \label{eqn:expansion-Z}
  Z(x,y) = \sum_{k=0}^{2d} f(k) P^n_k(x^{\tr}y).
\end{equation}
Let~$\overline{Y}^n_{k,d} = \Av_{\mathfrak{S}_3}Y^n_{k,d}$ be the matrix obtained from~$Y^n_{k,d}$ of~\eqref{eqn:bachoc-vallentin-kernel} by averaging over all permutations of~$(u, v, t)$.

If there are positive-semidefinite matrices~$F \in \R^{(d-k+1) \times (d-k+1)}$ for~$k = 0$,~$\ldots$,~$d$, such that
\begin{equation}%
  \label{eqn:symmetric-condition-Q1}
  \sum_{k=0}^{2d}f(k)(1/3)(P^n_k(u) + P^n(k)(v) + P^n_k(t)) - \sum_{k=0}^d\langle F_k, \overline{Y}^n_{k,d}(u,v,t)\rangle \geq 0
\end{equation}
for all~$(u, v, t) \in \Delta = \{\, (x^{\tr}z, y^{\tr}z, x^{\tr} y) : x,\ y,\ z \in S^{n-1}\, \}$, then~$Z \in \Q_1(S^{n-1})$.

Indeed, for~$x,\ y,\ z \in S^{n-1}$ with~$u = x^{\tr}z$,~$v = y^{\tr}z$, and~$t = x^{\tr}y$ we have
\[
  \Av_{\mathfrak{S}_r}(Z \otimes 1)(x, y, z) = \sum_{k=0}^{2d} f(k)(1/3)(P^n_k(u) + P^n_k(v) + P^n_k(t)).
\]
The function~$F$ given by
\[
  F(x, y, z) = \sum_{k=0}^d\langle F_k, Y^n_{k,d}(u,v,t)\rangle
\]
is slice positive and continuous  and
\[
  (\Av_{\mathfrak{S}_3}F)(x,y,z) = \sum_{k=0}^d \langle F_k, \overline{Y}^n_{k,d}(u,v,t)\rangle.
\]
Putting it all together,~$Z \in \Q_1(S^{n-1})$.

The left-hand side of~\eqref{eqn:symmetric-condition-Q1} is a polynomial~$p \in \R[u,v,t]$ of degree at most~$2d$ that should be nonnegative on~$\Delta$. The polynomial~$p$ is invariant under the permutation action of~$\mathfrak{S}_3$ on the variables. The domain~$\Delta$ is also invariant under~$\mathfrak{S}_3$; it is a semi-algebraic set:
\[
  \Delta = \{\, (u,v,t) : g_i(u,v,t) \geq 0\text{ for } i=1,\ \ldots,\ 4\, \},
\]
where
\begin{align*}
  &g_1 = g(u) + g(v) + g(t), & &g_2 = g(u)g(v) + g(u)g(t) + g(v) g(t),\\
  &g_3 = g(u)g(v)g(t),       & &g_4 = 1 + 2uvt - u^2 - v^2 - t^2,
\end{align*}
where~$g(w) = 1-w^2$. So, if there are sums-of-squares polynomials~$q_0,\ \ldots,\ q_4$ in~$\R[u,v,t]$ such that
\begin{equation}%
  \label{eqn:sos-expression-p}
  p = q_0 + g_1q_1 + g_2q_2 + g_3q_3 + g_4q_4,
\end{equation}
then~$p$ is nonnegative on~$\Delta$. Moreover, since~$p$ and the~$g_i$ are all invariant under~$\mathfrak{S}_3$, we may assume without loss of generality that the~$q_i$ are also invariant.

Let~$V_r$ be the matrix indexed by the monomials on~$u$,~$v$, and~$t$ of degree at most~$\lfloor r/2\rfloor$ such that~$V_r(m_1,m_2) = m_1m_2$ for any two such monomials. Every entry of~$V_r$ is a polynomial of degree at most~$r$. A polynomial~$q$ of degree~$2k$ is a sum of squares if and only if there is a positive-semidefinite matrix~$Q$ such that~$q = \langle Q, V_{2k}\rangle$.

Using this equivalence and restricting the degrees of the polynomials~$q_i$ appearing in~\eqref{eqn:sos-expression-p}, we can write a sufficient condition for~$p$ to be nonnegative on~$\Delta$ in terms of positive-semidefinite matrices. Namely, if there are positive-semidefinite matrices~$F_k$ and ~$Q_i$ such that
\begin{multline}%
  \label{eqn:explicit-sos-condition-p}
  \sum_{k=0}^{2d} f(k)(1/3)(P^n_k(u) + P^n_k(v) + P^n_k(t)) - \sum_{k=0}^d \langle F_k, \overline{Y}^n_{k,d}(u,v,t)\rangle\\
  = \langle Q_0, V_{2d}\rangle + \langle Q_1, g_1V_{2d-2}\rangle + \langle Q_2, g_2V_{2d-4}\rangle\\
  + \langle Q_3, g_3V_{2d-6}\rangle + \langle Q_4, g_4 V_{2d-3}\rangle,
\end{multline}
then~$Z$ given in~\eqref{eqn:expansion-Z} is in~$\Q_1(S^{n-1})$. This leads us to the definition of the following cone for every fixed~$d$:
\begin{multline*}
  \Q^d_1 = \{\, (f(0), \ldots, f(2d),0 , \ldots) \in \R^{\N} : \text{there are positive-semidefinite}\\
  \text{matrices $F_k$ and $Q_i$ such that~\eqref{eqn:explicit-sos-condition-p} holds}\, \}
\end{multline*}

We were careful to describe the domain~$\Delta$ with invariant polynomials so that we could assume that all polynomials~$q_i$ are likewise invariant. This can be used to simplify~\eqref{eqn:explicit-sos-condition-p}, so that we can work with block-diagonal positive-semidefinite matrices~$Q_i$. The original idea was presented by Gatermann and Parrilo~\cite{Gatermann2004SymmetrySquares}; see also Machado and Oliveira~\cite{Machado2018ImprovingSymmetry} and Leijenhorst and De Laat~\cite[\S4]{Leijenhorst2024SolvingOptimization} for more recent descriptions of the method and an application to this exact situation. This use of symmetry to reduce the problem's size is essential to reach high degrees.

\section{Implementation and verification of the bound}%
\label{sec:implementation-of-the-bound}
To make our bound on~$\alpha_n$ as good as possible, we combine the cone~$\Q^d_1$ with constraints from the Boolean quadratic polytope, which for a finite set~$V$ is defined as
\[
  \BQP(V) = \conv\{\, xx^{\tr} : x \in \{0, 1\}^V\, \}.
\]
Such constraints were used before by DeCorte, Oliveira, and Vallentin~\cite{DeCorte2022CompleteSets}.

Given a measurable independent set~$I \subseteq S^{n-1}$ of~$G_n$, define the kernel~$A = \Av_{\ortho(n)}(\1_I \otimes \1_I)$. Then,
\begin{enumerate}
  \item[(i)] $A$ is an~$\ortho(n)$-invariant continuous kernel (by Lemma~\ref{lem:averaged-rank-one-is-continuous}),
  \item[(ii)] $A(x, y) = 0$ for all orthogonal~$x$,~$y \in S^{n-1}$,
  \item[(iii)] $A$ is positive semidefinite and~$A \in \Q_r(S^{n-1})^*$ for all~$r \geq 1$,
  \item[(iv)] $\bigl(A(x,y)\bigr)_{x,y \in U} \in \BQP(U)$ for every finite~$U \subseteq S^{n-1}$, and
  \item[(v)] $\int_{S^{n-1}} A(x,x)\, d\omega(x) = \omega(I)$ and~$\langle A, J\rangle = \omega(I)^2$.
\end{enumerate}

We use Schoenberg's theorem to express~$A$ in terms of Jacobi polynomials as in~\eqref{eqn:O(n)-invariant-kernels}, so
\[
  A(x,y) = \sum_{k=0}^{\infty} a(k) P^n_k(x^{\tr}y)
\]
for some sequence~$a \geq 0$. Recall that we normalize the polynomials so~$P^n_k(1)=1$; together with the addition formula, see~\ref{sec:conv-cp-hierarchy}, this gives
\begin{equation}%
  \label{eqn:normalization-objective-by-schoenberg}
  \int_{S^{n-1}} A(x,x)\, d\omega(x) = \omega_n \sum_{k=0}^{\infty} a(k)\qquad \text{and}\qquad \langle A, J\rangle = \omega_n^2 a(0),
\end{equation}
where~$\omega_n = \omega(S^{n-1})$.

Let~$U \subseteq S^{n-1}$ be a finite set and let~$L \in \R^{U \times U}$ and~$\beta \in \R$ such that~$\langle L, X\rangle \leq \beta$ for all~$X \in \BQP(U)$. Then, defining~$r : \N \to \R$ by
\begin{equation}%
  \label{eqn:BQP-kernel}
  r(k) = \sum_{x, y \in U} L(x, y) P^n_k(x^{\tr}y)
\end{equation}
we have
\begin{equation}%
  \label{eqn:BQP-inequality}
  \sum_{k=0}^{\infty} a(k)r(k) \leq \beta.
\end{equation}
We call~$(r, \beta)$ a \defi{$\BQP(S^{n-1})$-inequality}\index{BQPSn1 inequality@$\BQP(S^{n-1})$-inequality}, and we call the points in~$U$ the~\defi{support points}\index{support points of BQPSn1 inequality@support points of $\BQP(S^{n-1})$-inequality} of the inequality.

Finally, if~$(f(0), \ldots, f(2d), 0, \ldots) \in \Q^d_1$ and~$Z(x, y) = \sum_{k=0}^{2d} f(k) P^n_k(x^{\tr}y)$, then~$Z \in \Q_1(S^{n-1})$, and from~\eqref{eqn:inner-product-formula-unit-sphere} we get
\[
  \sum_{k=0}^{2d}(a(k) / h^n_k)f(k) = \omega_n^{-2}\langle A, Z\rangle \geq 0,
\]
that is,~$k \mapsto a(k) / h^n_k$ belongs to~$(\Q^d_1)^*$.

Let~$(r_1, \beta_1),\ \ldots,\ (r_N, \beta_N)$ be any~$\BQP(S^{n-1})$-inequalities and fix some integer~$d \geq 1$. Put together, our developments lead us to the following optimization problem, whose optimal value gives an upper bound on~$\alpha_n$:
\begin{equation}%
  \label{eqn:explicit-problem-Q1-plus-BQP}
  \begin{optprob}
    \sup &\onerow{\sum_{k=0}^{\infty} a(k)}\\
    &\onerow{\sum_{k=0}^{\infty}a(k)P^n_k(0) = 0,}\\
    &\sum_{k=0}^{\infty} a(k)r_i(k) \leq \beta_i &\text{for } 1 \leq i \leq N,\\[3pt]
    &\onerow{
      \begin{pmatrix} 1 & \omega_n\sum_{k=0}^{\infty}a(k)\\ \omega_n\sum_{k=0}^{\infty}a(k) & \omega_n^2a(0)
    \end{pmatrix} \text{ is positive semidefinite},}\\[7pt]
    &\onerow{a \geq 0 \text{ and } k \mapsto a(k) / h^n_k \in (\Q^d_1)^*.}
  \end{optprob}
\end{equation}

The $2 \times 2$ matrix comes from~\eqref{eqn:normalization-objective-by-schoenberg} and~(v) and is used to normalize the problem. The objective function is divided by~$\omega_n$, ensuring that we get a bound for~$\alpha_n = \alpha(G_n)$ under the measure~$\omega/\omega_n$. Finally, our problem has infinitely many variables~$a$, but only the first~$2d+1$ of them appear in the cone constraint with~$(\Q^d_1)^*$. Contrast this with the situation of the block moment hierarchy from~\S\ref{sec:failure-3-point-bound}.

The dual of this problem is
\begin{equation}%
  \label{eqn:dual-Q1-plus-BQP}
  \begin{optprob}
    \inf &\onerow{z_{11} + \sum_{i=1}^N y_i \beta_i}\\
    &\onerow{\lambda + \sum_{i=1}^N y_i r_i(0) - \omega_nz_{12} - \omega_n^2 z_{22} - f(0) \geq 1,}\\
    &\lambda P^n_k(0) + \sum_{i=1}^N y_i r_i(k) - \omega_n z_{12} - f(k) \geq 1 & \text{for all } 1 \leq k \leq 2d,\\
    &\lambda P^n_k(0) + \sum_{i=1}^N y_i r_i(k) - \omega_nz_{12} \geq 1 & \text{for all } k \geq 2d+1,\\[3pt]
    &\onerow{
      \begin{pmatrix} z_{11} & z_{12}/2 \\ z_{12}/2 & z_{22}
    \end{pmatrix} \text{ is positive semidefinite},}\\[7pt]
    &y \geq 0 \text{ and } k \mapsto h^n_kf(k) \in \Q^d_1.
  \end{optprob}
\end{equation}
The objective value of any feasible solution of~\eqref{eqn:dual-Q1-plus-BQP} is greater than or equal to the objective value of any feasible solution of~\eqref{eqn:explicit-problem-Q1-plus-BQP}. So, any feasible solution of~\eqref{eqn:dual-Q1-plus-BQP} gives an upper bound on~$\alpha_n$. Furthermore, the constraint given by~``$k \mapsto h^n_kf(k) \in \Q^d_1$'' is expressed in terms of~\eqref{eqn:explicit-sos-condition-p}, namely we require there to be positive-semidefinite matrices~$F_k$ and~$Q_i$ such that
\begin{multline}%
  \label{eqn:explicit-sos-for-dual-problem}
  \sum_{k=0}^{2d} f(k)(h^n_k/3)(P^n_k(u) + P^n_k(v) + P^n_k(t)) - \sum_{k=0}^d \langle F_k, \overline{Y}^n_{k,d}(u, v, t)\rangle\\
  -\langle Q_0, V_{2d}\rangle - \langle Q_1, g_1 V_{2d-2}\rangle - \langle Q_2, g_2V_{2d-4} \rangle\\
  -\langle Q_3, g_3 V_{2d-6}\rangle - \langle Q_4, g_4 V_{2d-3}\rangle = 0.
\end{multline}
So~\eqref{eqn:dual-Q1-plus-BQP} is a semidefinite programming problem with finitely many variables but infinitely many constraints.

\sectionbreak

To solve~\eqref{eqn:dual-Q1-plus-BQP} we use the package \texttt{ClusteredLowRankSolver.jl} of Leijenhorst and De Laat~\cite{Leijenhorst2024SolvingOptimization}; the input for the solver is generated by a Julia program. The program and all data files used are available in the Harvard Dataverse repository~\cite{Oliveira2023DataParameter}.

To find good~$\BQP(S^{n-1})$-inequalities, we use a separation heuristic described by DeCorte, Oliveira, and Vallentin~\cite{DeCorte2022CompleteSets}. The inequalities used are also included in the repository and need not be recomputed.

Table~\ref{tab:detailed} contains a detailed account of all the bounds computed from~\eqref{eqn:dual-Q1-plus-BQP}.  Solving the problem for~$d = 14$ and~$18$ takes time and memory, so files with the corresponding solutions are also available in the repository.

\begin{table}[p]
  \centering
  \begin{tabular}{@{}lllllll@{}} \toprule
    \multicolumn{2}{c}{} & \multicolumn{2}{c}{\textsl{Old upper bound}} & & \multicolumn{2}{c}{\textsl{New upper bound}}\\
    $n$ & \textsl{Lower bound} & \textsl{Simple} & \textsl{Best} & $d$ & \textsl{No BQP} & \textsl{With BQP}\\ \midrule
    3 & 0.2928\ldots & 0.3333\ldots & 0.30153 & 6 & 0.316925 & 0.300708\\
    &&&& 10 & 0.309298 & 0.298998\\
    &&&& 14 & 0.305627 & 0.298341\\
    &&&& 18 & 0.303294 & 0.297742\\ \midrule
    4 & 0.1816\ldots & 0.25 & 0.21676 & 6 & 0.223633 & 0.207617\\
    &&&& 10 & 0.211825 & 0.199402\\
    &&&& 14 & 0.205479 & 0.196162\\
    &&&& 18 & 0.201445 & 0.194297\\ \midrule
    5 & 0.1161\ldots & 0.2 & 0.16765 & 6 & 0.167357 & 0.151541\\
    &&&& 10 & 0.153819 & 0.141539\\
    &&&& 14 & 0.146612 & 0.137142\\
    &&&& 18 & 0.142349 & 0.134588\\ \midrule
    6 & 0.0755\ldots & 0.1666\ldots & 0.13382 & 6 & 0.130829 & 0.116599\\
    &&&& 10 & 0.116509 & 0.105200\\
    &&&& 14 & 0.109989 & 0.100374\\
    &&&& 18 & 0.106727 & 0.098095\\ \midrule
    7 & 0.0498\ldots & 0.1428\ldots & 0.11739 & 6 & 0.106059 & 0.093031\\
    &&&& 10 & 0.091477 & 0.081221\\
    &&&& 14 & 0.086656 & 0.077278\\
    &&&& 18 & 0.084787 & 0.075751\\ \midrule
    8 & 0.0331\ldots & 0.125 & 0.09981 & 6 & 0.088750 & 0.076801\\
    &&&& 10 & 0.074309 & 0.064919\\
    &&&& 14 & 0.071676 & 0.063287\\
    &&&& 18 & 0.070607 & 0.061178\\ \bottomrule
  \end{tabular}
  \bigskip

  \caption{Low and upper bounds for Witsenhausen's parameter~$\alpha_n$. The lower bound is given by the double-cap conjecture. The simple upper bound was given by Witsenhausen~\cite{Witsenhausen1974SphericalPairs} and is just~$1/n$. The best previous upper bounds are by DeCorte, Oliveira, and Vallentin~\cite{DeCorte2022CompleteSets}. The table gives upper bounds obtained from solving~\eqref{eqn:dual-Q1-plus-BQP} with and without~$\BQP(S^{n-1})$-inequalities and for several values of the maximum degree~$d$.}%
  \label{tab:detailed}
\end{table}

Since~\eqref{eqn:dual-Q1-plus-BQP} has infinitely many linear constraints, to solve it, we select some finite set~$S \subseteq \{0, 1, \ldots\}$ and consider only the constraints for~$k \in S$.  After a solution is found it has to be verified, that is, we need to check that all constraints are indeed satisfied.

Let~$(\lambda, y, z, f, F, Q)$ be a candidate solution to~\eqref{eqn:dual-Q1-plus-BQP}, where~$F$ and~$Q$ are as in~\eqref{eqn:explicit-sos-for-dual-problem}, returned by the solver.  The
first step is to certify ourselves that~$f$, $F$, and~$Q$ indeed
satisfy~\eqref{eqn:explicit-sos-for-dual-problem}.

This is certainly not true: the solver uses floating-point arithmetic,
so~\eqref{eqn:explicit-sos-for-dual-problem} will not hold.  Rather, the left-hand side of~\eqref{eqn:explicit-sos-for-dual-problem} will be a polynomial with coefficients close to~0. Since the \texttt{ClusteredLowRankSolver.jl} uses high-precision floating-point
arithmetic, the coefficients will be quite small; let~$\eta$ be the largest
absolute value of any such coefficient.

It is always possible to perturb the matrices~$Q_i$ in order to satisfy the
constraint; the order of the perturbation depends on~$\eta$.  We want to do so
and keep the~$Q_i$ positive semidefinite; as long as the minimum eigenvalues
of the matrices~$Q_i$ are large enough compared to~$\eta$, this is always
possible.  The solutions stored in the repository have large minimum
eigenvalues, several orders of magnitude larger than~$\eta$, so this
perturbation of the~$Q_i$ can always be carried out.  We do not have
to actually change the~$Q_i$; it suffices to know that such a perturbation is
possible, since then we know we can get a feasible solution if we want to. This
procedure was used before by De Laat, Oliveira, and Vallentin~\cite{DeLaat2014UpperRadii}.

Checking that the linear constraints for all~$k \geq 0$ are satisfied is more
difficult; we use the approach outlined in DeCorte, Oliveira, and
Vallentin~\cite{DeCorte2022CompleteSets}.

The idea is as follows.  Let~$\lhs(k)$ be the left side of the~$k$th linear
constraint in~\eqref{eqn:dual-Q1-plus-BQP} and
write~$\lhs(\infty) = \lim_{k \to \infty} \lhs(k)$; we will see that this limit
exists.  We then take the following steps.
\begin{enumerate}
  \item[(i)] As long as~$z_{22} > 0$, we can change~$z_{11}$ and~$z_{12}$ to
    get~$\lhs(\infty) \geq 1 + \eta$ for some~$\eta > 0$.  The more we
    change~$z_{12}$, the more we have to change~$z_{11}$, and the worse the bound
    gets.

  \item[(ii)] Next, for some~$\epsilon < \eta$ we find a~$k_0$ such
    that~$|\lhs(k) - \lhs(\infty)| \leq \epsilon$ for all~$k \geq k_0$.
    Then~$\lhs(k) \geq \lhs(\infty) - \epsilon \geq 1 + \eta - \epsilon > 1$, and
    so all constraints are satisfied for~$k \geq k_0$.

  \item[(iii)] Finally, we check the constraints for~$k = 0$, \dots,~$k_0$, and by
    changing~$z$ again we can make all these constraints satisfied.
\end{enumerate}

If we choose our initial sample~$S$ well, then all constraints will be almost
satisfied, and we will not have to change~$z$ too much in order to get a
feasible solution.  This is the procedure implemented by the
\texttt{fix\_linear\_constraints} function in the Julia program in the
repository~\cite{Oliveira2023DataParameter}.

Let us see the details of the procedure.  The asymptotic formula for the Jacobi
polynomials~\cite[Theorem~8.21.8]{Szego1975OrthogonalPolynomials} implies that~$P_k^n(t) \to 0$
as~$k \to \infty$ for all~$t \in (-1, 1)$.  We make sure that all the
$\BQP(S^{n-1})$-inequalities~\eqref{eqn:BQP-inequality} we use have support points~$U$
such that distinct~$x$, $y \in U$ have inner product~$x^{\tr} y$ bounded away
from~$\pm 1$.  So if~$r$ is given as in~\eqref{eqn:BQP-kernel}, then
\[
  r(\infty) = \lim_{k\to\infty} r(k) = \Tr L
\]
and
\begin{equation}%
  \label{eq:rk-infty-diff}
  |r(k) - r(\infty)| \leq \sum_{\substack{x,y \in U\\x \neq y}} |L(x, y)|
  |P_k^n(x^{\tr} y)|.
\end{equation}

We also have
\[
  \lhs(\infty) = \sum_{i=1}^N y_i r_i(\infty) - \omega_n z_{12}.
\]
Given~$\epsilon > 0$ we want to get~$k_0$ as in~(ii).  Note that
\[
  |\lhs(k) - \lhs(\infty)| \leq |\lambda| |P_k^n(0)| + \sum_{i=1}^N y_i |r_i(k)
  - r_i(\infty)|.
\]
Fix~$k_0$.  Using~\eqref{eq:rk-infty-diff}, we see that to find an upper bound
for the left side above for all~$k \geq k_0$, it suffices to find for
all~$k \geq k_0$ an upper bound on~$|P_k^n(t)|$ for~$t = 0$ and all
other~$t \in (-1, 1)$ that occur as inner products between distinct support
points of the $\BQP(S^{n-1})$-inequalities we use.

To do so rigorously, we use an integral representation for the ultraspherical polynomials due to Gegenbauer (take~$\lambda = (n - 2)/2$ in Theorem~6.7.4 from
Andrews, Askey, and Roy~\cite{Andrews1999SpecialFunctions}):
\[
  P_k^n(\cos\theta) = R(n)^{-1} \int_0^\pi F(\phi)^k
  \sin^{n-3}\phi\, d\phi,
\]
where
\[
  F(\phi) = \cos\theta + i\sin\theta\cos\phi\qquad\text{and}\qquad
  R(n) = \int_0^\pi \sin^{n-3} \phi\, d\phi.
\]

Then,~$|F(\phi)|^2 = \cos^2\theta + \sin^2\theta \cos^2\phi$, so
\[
  |P_k^n(\cos\theta)| \leq R(n)^{-1} \int_0^\pi (\cos^2\theta + \sin^2\theta
  \cos^2\phi)^{k/2} \sin^{n-3}\phi\, d\phi.
\]
The right side is decreasing in~$k$, and we can estimate the integrals rigorously
using interval arithmetic.

For~(iii) we need to compute~$\lhs(k)$ for all~$k \leq k_0$.  We would like to
do this rigorously, using for instance interval arithmetic.  The most
time-consuming step here is to compute the~$r_i$ functions.  In practice, this
step involves evaluating the polynomials~$P_k^n$ for values of~$k$ that can
exceed $100{,}000$.

The Jacobi polynomials~$P_k^n$ are given by a simple recurrence, namely
\[
  P_k^n(u) = a_k^n(u) P_{k-1}^n(u) + b_k^n P_{k-2}^n(u)
\]
for~$k \geq 2$ with~$P_1^n(u) = u$ and~$P_0^n(u) = 1$, where
\[
  a_k^n(u) = \frac{2k + 2\alpha - 1}{k + 2\alpha} u\qquad\text{and}\qquad b_k =
  -\frac{k-1}{k + 2\alpha}
\]
with~$\alpha = (n-3) / 2$.  This recurrence comes from formula~(4.5.1) in
Szeg\H{o}~\cite{Szego1975OrthogonalPolynomials}, adapted to our normalization of~$P_k^n(1) = 1$.

The recurrence is very stable: even using double-precision floating-point
arithmetic it is possible to accurately evaluate the polynomial for very high
degrees for points in~$[-1, 1]$.  If we use this recurrence with interval
arithmetic though, the error estimation quickly gets out of hand: if
$\lim_{k\to\infty} a_k(u) > 1$, the error bound grows exponentially.

Using interval arithmetic then requires very high precision and is very slow,
though not prohibitively so.  In any case, we can trust floating-point
computations.  Using the recurrence amounts to solving
a linear system with a triangular matrix whose entries are the numbers
$a_k(u)$,~$b_k$, and~$1$ by backward substitution, and this matrix is well conditioned, so the error we
make in solving the system is very small.  The error was analyzed for instance
by Barrio~\cite{Barrio2002RoundingSeries}.  The Julia program that performs the verification
uses high-precision floating-point arithmetic.

\section*{Acknowledgments}
We would like to thank David de Laat, Nando Leijenhorst, Fabrício Caluza
Machado, and Willem de Muinck Keizer for fruitful discussions.  David de Laat
and Nando Leijenhorst also gave some much-needed technical support regarding the
\texttt{ClusteredLowRankSolver}.  The optimization problems were solved in a
computational cluster at TU Delft maintained by Joffrey Wallaart.

\part{Packings in compact spaces}\label{part:compact-packings}
\chapter[Copositive programming for compact packing graphs]{Copositive programming for compact packing graphs}%
\label{ch:cop-prog-packing}
Part~\ref{part:compact-packings} concerns problems in Class~\ref{it:problem-2} from the introduction. These are counting problems on compact, but infinite, spaces. Prime examples of these are the spherical-codes problems, of which Problem~\ref{it:problem-2}---the kissing number problem---is a special case.

A spherical-codes problem is as follows. For~$\theta \in (0, \pi)$, a~\defi{spherical code}\index{spherical code} with angle~$\theta$ is a subset~$S \subseteq S^{n-1}$ such that the angle between two distinct points~$x$,~$y \in S$ is at least~$\theta$. The~\defi{spherical-codes problem}\index{problem!spherical codes} with angle~$\theta$ asks what is the largest cardinality~$A(n, \theta)$ that a spherical code with angle~$\theta$ can have. This is the same as asking how many spherical caps of angular radius~$\theta/2$ fit on~$S^{n-1}$ without overlapping, from which it follows that this maximum cardinality is indeed finite and attained.

We model such problems as independent-set problems on graphs with the following properties. A~\defi{packing graph}\index{packing graph} is a graph~$G = (V, E)$ with~$V$ a topological space such that every finite clique of~$G$ is contained in an open clique. It is called a~\defi{compact packing graph}\index{packing graph!compact packing graph} if~$V$ is compact. A compact packing graph has a finite independence number, i.e.
\[
  \alpha(G) = \sup\{\, |I| : I \subseteq V \text{ is independent}\, \} < \infty.
\]

For a spherical-codes problem, the corresponding packing graph is the graph~$G_{\theta} = (S^{n-1}, E(\theta))$, where~$xy \in E(\theta)$ if and only if the angle between~$x$ and~$y$ is in the interval~$(0,\theta)$. A set~$S \subseteq S^{n-1}$ is an independent set of~$G_{\theta}$ if and only if it is a spherical code with angle~$\theta$. Hence,~$\alpha(G_{\theta})$ is indeed~$A(n, \theta)$.

Linear programming methods for spherical codes and other compact packing problems were introduced to the topic early on. Delsarte~\cite{Delsarte1972BoundsProgramming} introduced linear programming methods to the \defi{binary-codes problem}\index{binary-codes problem} in~1972; this is the spherical-codes problem, but with the sphere replaced by the Hamming cube~$(\Z/2\Z)^n$, and the angle between two points replaced by the~\defi{Hamming distance}\index{Hamming distance}: the number of distinct entries between vectors. His method was adapted to the spherical-codes problem in~1977 by Delsarte, Goethals, and Seidel~\cite{Delsarte1977SphericalDesigns}. McEliece, Rodemich, and Rumsey~\cite{McEliece1978TheGeneralizations}, and independently Schrijver~\cite{Schrijver1979ABounds},  observed that these types of linear programming bounds are symmetry-reduced versions of the Lovász $\vartheta$-number, respectively in 1978 and 1979. By today's measures, it is a fairly simple method, but already gave rise to sharp bounds in many cases~\cite{Levenshtein1979BoundariesSpace,Odlyzko1979NewDimenions}.

Schrijver~\cite{Schrijver2005NewProgramming} introduced semidefinite programming to the topic in~2005 by describing a three-point bound that gives upper bounds on the binary codes-problem. In their landmark paper, Bachoc and Vallentin~\cite{Bachoc2008NewProgramming} extended the three-point bound to spherical codes in~2008.

The moment hierarchy was studied in depth for the independence number of finite graphs in~2003 by Laurent~\cite{Laurent2003AProgramming}, and based on this, Gvozdeninovi\'c, Laurent, and Vallentin~\cite{Gvozdenovic2009BlockProgramming} introduced a $k$-point bound for~$k \geq 2$. This enabled the extension of the three-point bound to general compact packing graphs, and for those graphs, a moment hierarchy was introduced by De Laat and Vallentin~\cite{deLaat2015AGeometry} in~2015 and a block moment hierarchy by De Laat, Machado, Oliveira, and Vallentin~\cite{deLaat2021K-PointLines} in~2022. Low levels of these hierarchies were implemented successfully for several problems on the unit sphere by De Laat, Machado, and De Muinck Keizer~\cite{deLaat2023TheAngle} in 2023 and Cohn, De Laat, and Leijenhorst~\cite{Cohn2024OptimalityBounds} in~2024.

Dobre, Dür, Frerick, and Vallentin introduced copositive optimization to the topic in~2016, by describing an exact copositive formulation of the independence number of a compact metrizable packing graph. Based on this result, Kuryatnikova and Vera~\cite{Kuryatnikova2017Approximating} defined a copositive hierarchy for compact packing graphs with metrizable vertex set in~2017, and they showed it converges to the independence number. These results are analogous to the approach for finite graphs Section~\ref{ch:introduction}.\ref{sec:cop-cone-finite-dimension} of this thesis; the hierarchy is based on an extension of Pólya's theorem to continuous kernels on compact spaces. See also the PhD thesis of Kuryatnikova~\cite{Kuryatnikova2019TheProblems}.

In this chapter we study the relation between the moment hierarchy, the block moment hierarchy, and the copositive programming hierarchy for compact packing graphs, and show that they converge to the independence number if the vertex set is in addition metrizable. To be precise, the novel contribution in this chapter is a proof that the copositive hierarchy introduced by Kuryatnikova and Vera is weaker than the block moment hierarchy by De Laat, Machado, Oliveira and Vallentin. The convergence of the copositive hierarchy thus implies convergence of the block moment hierarchy. The exposition is based on the preprint~\cite{Bekker2023OnGraphs}.

\section{Spaces of measures}
Let~$V$ be a compact Hausdorff space. We say that~$\mu \in M(V^k)$ is~\defi{symmetric}\index{measure!symmetric measure} if for all~$E \subseteq V^k$ and all~$\pi \in \mathfrak{S}_k$ we have~$\mu(\pi E) = \mu(E)$, where
\[
  \pi E = \{\, (v_{\pi 1}, \ldots, v_{\pi k}) : (v_1, \ldots, v_k) \in E\, \}.
\]
Denote the space of symmetric signed Radon measures on~$V^k$ by~$M_{\sym}(V,k)$\index{ MsymVk@$M_{\sym}(V,k)$}.

The pair~$(C(V), M(V))$ is a dual pair with the duality
\[
  \langle f, \mu\rangle = \mu(f) = \int_V f(v)\, d\mu(v).
\]
This defines a duality between~$C_{\sym}(V,k)$ and~$M_{\sym}(V,k)$. All asterisks in this chapter refer to the dual with respect to this duality, i.e. if~$\cC \subset C_{\sym}(V,k)$ a cone, then~$\cC^* \subset M_{\sym}(V,k)$. The notation~$C_{\sym}(V)$ and~$M_{\sym}(V)$ denotes~$C_{\sym}(V,2)$ and~$M_{\sym}(V,2)$ respectively.

We define cones of Radon measures on a compact Hausdorff space~$V$ through this duality. Thus, the cone of nonnegative Radon measures~$M(V)_{\geq 0}$ is the cone of signed Radon measures for which~$\langle f, \mu\rangle \geq 0$ for all~$f \in C(V)_{\geq 0}$. Likewise, the cone~$M_{\sym}(V)_{\succeq 0}$ of~\defi{positive-semidefinite}\index{positive semidefinite!measure} measures on~$V^2$ consists of signed Radon measures~$\mu$ such that~$\langle F, \mu\rangle \geq 0$ for each positive-semidefinite kernel~$F \in C_{\sym}(V)$.

Let~$k \geq 0$ be an integer and~$G = (V, E)$ be a graph with~$V$ a compact Hausdorff space. Recall from Section~\ref{ch:Completely positive programming on compact spaces}.\ref{sec:spaces-of-subsets} that~$\ind_k$ denotes the subset of~$\Sub{V}{k}$ of independent sets. Let~$C_{\sym}(V^2 \times \ind_k)$ be the space of functions in~$C(V^2\times \ind_k)$ that are invariant under permutation of the first two arguments. Define the space~$M_{\sym}(V^2\times \ind_k)$ as the space of all signed Radon measures on~$V^2 \times \ind_k$ such that~$\mu(F^{\tr}) = \mu(F)$ for all measurable~$F \subseteq V^2 \times \ind_k$, where
\[
  F^{\tr} = \{\, (u, v, w) : (v, u, w) \in F\, \}.
\]
Let~$C_{\sym}(V^2\times \ind_k)_{\succeq 0}$ be the cone of continuous functions~$F$ such that the function~$(u, v) \mapsto F(u, v, w)$ is a positive semidefinite kernel for all~$w \in \ind_k$, and denote its dual by~$M_{\sym}(V^2\times \ind_k)_{\succeq 0}$.

\section{Compact packing graphs}
Recall the preliminaries on the standard topology on~$\Sub{V}{r}$ for a topological space~$V$ from Section~\ref{sec:spaces-of-subsets} of Chapter~\ref{ch:Completely positive programming on compact spaces}, in particular that a basis of the standard topology is given by sets of the form
\[
  (U_1, \ldots, U_t)_r = \{\, A \in \Sub{V}{r} : A \cap U_i \neq \emptyset \text{ for all } i \text{ and } A \subseteq U_1 \cup \cdots \cup U_t\, \},
\]
with~$t$ an integer such that~$0 \leq t \leq r$ and~$U_1$,~$\ldots$,~$U_r \subseteq V$ disjoint open sets. Moreover, recall that, for a given graph~$G$, the notation~$\clique_r$ denotes the set of cliques of size at most~$r$.

The following characterization of packing graphs was not found in the literature yet. It will be of use later.
\begin{theorem}%
  \label{thm:packing-graph-if-and-only-if-Kk-open}
  If~$H = (V, E)$ is a graph and~$V$ is a Hausdorff space, the following are equivalent:
  \begin{enumerate}
    \item[(i)] $G$ is a packing graph;
    \item[(ii)] $\clique_r$ is open in~$\Sub{V}{r}$ for every~$r \geq 0$;
    \item[(iii)] $\clique_2$ is open in~$\Sub{V}{2}$.
  \end{enumerate}
\end{theorem}
\begin{proof}
  To see that~(i) implies~(ii), let~$C = \{x_1, \ldots, x_s\}$ be a clique of cardinality~$s \leq r$. It is contained in an open clique~$K$. Since~$V$ is a Hausdorff space, there are disjoint open sets~$U_1, \ldots, U_s$ such that~$x_i \in U_i$ for all~$i$. By taking the intersection of each~$U_i$ with~$K$, assume~$\bigcup_i U_i$ is a clique. The set~$(U_1, \ldots, U_s)_r$  is an open set of~$\Sub{V}{r}$ containing~$C$ and consisting only of cliques. This proves that~$\clique_r$ is open in~$\Sub{V}{r}$ for all~$r \geq 0$.

  That~(ii) implies~(iii) is immediate.

  Assume~(iii) holds. Let~$C$ be a finite clique and~$\mathbf{x} = \{x_1, x_2\}$ be a subset of size~$\leq 2$. Then there exists a basic open set~$(U_{\bf{x}}^{x_1}, U_{\bf{x}}^{x_2})_2 \subseteq \clique_2$, where~$U_{\mathbf{x}}^{x_i}$ is an open neighborhood of~$x_i$, because sets of this form produce a basis of the topology. Choose such a basic open set for each~$\mathbf{x} \subseteq C$ of size at most~$2$, and define
  \[
    K = \bigcup_{x \in C} \bigcap_{\substack{\mathbf{x} \in \Sub{C}{2} \\ \text{with } x \in \mathbf{x}}} U_{\mathbf{x}}^x.
  \]

  The set~$K$ is a union of finite intersections of open sets that contains~$C$, thus it is an open neighborhood of~$C$. Moreover,~$K$ is a clique: let~$\mathbf{y} \subseteq K$ be a set of cardinality~$2$, say~$\mathbf{y} = \{y_1, y_2\}$. For~$i \in \{1,2\}$ there exists an~$x_i \in C$ such that~$y_i \in U_{\mathbf{x}}^{x_i}$ for every~$\mathbf{x} \subseteq C$ containing~$x_i$. Choose such~$x_1$ and~$x_2$, and take~$\mathbf{x} = \{ x_1, x_2 \}$. Then, per definition of the~$U_{\bf{x}}^{x_i}$,~$\mathbf{y} \in (U_{\bf{x}}^{x_1},U_{\bf{x}}^{x_2})_2 \subseteq \clique_2$, so~$\mathbf{y}$ is an edge. Thus, we found an open clique containing~$C$, and~$G$ is a packing graph.
\end{proof}

The following lemma is essential for much of the analysis in this chapter.
\begin{lemma}[{\cite[Lemma 2]{deLaat2015AGeometry}}]%
  \label{lem:space-of-ind-subset-disjoint-union}
  If~$G = (V, E)$ is a compact packing graph, then~$\ind_{=r}$ is both open and closed in~$\ind_k$ for all~$r \leq k$. Hence, if~$Y$ is a topological space, then~$f: \ind_k \to Y$ is continuous if and only if the restriction of~$f$ to~$\ind_{=r}$ is continuous for all~$r \leq k$.

  In particular, if~$G = (V, E)$ is a compact packing graph, then
  \[
    C(\ind_k) \cong \bigoplus_{r = 0}^kC(\ind_{=r}).
  \]
\end{lemma}

\section{The moment and block moment hierarchies}
The definitions of the moment hierarchy and the block moment hierarchy are similar to those of Chapter~\ref{ch:Completely positive programming on compact spaces}, but there are notable differences. For compact packing graphs, we may formulate the problem over the space of Radon measures. If we do this, it turns out to be favorable to realize the edge constraint by only considering elements of~$M(\ind_r)$ rather than~$M(\Sub{V}{r})$. We use the same notation~$\lass_r$ and~$\kpb_r$ as in Chapter~\ref{ch:Completely positive programming on compact spaces}; because the programs there were defined for measurable graphs, this should not lead to confusion.

Recall the operator on~$\R^V$ for finite~$V$
\[
  M_r : \R^{\Sub{V}{2r}} \to \Sym(\Sub{V}{r}),\qquad (M_r\nu)_{S, T} = \nu_{S \cup T}.
\]
To define the analogue of this operator on a space of measures, it is easier to first extend its adjoint to function spaces. A direct calculation shows that for all~$A \in \Sym(\Sub{V}{r})$,
\[
  (M_r^* A)_S = \sum_{\substack{J, J'\in \Sub{V}{r} \\ J \cup J' = S}} A_{J, J'}.
\]
To understand this operator, consider that for generic~$V$, the only additional structure given is the union map~$\Sub{V}{r} \times \Sub{V}{r} \to \Sub{V}{2r}$. We want a map~$\Sym(\Sub{V}{r}) \to \R^{\Sub{V}{2r}}$ that treats all elements in a fixed fiber of the union map the same. This is achieved by giving all elements the same weight, weight~1, and summing the values together; i.e., we average the function values in a fiber of the union map under a uniform distribution. A similar interpretation holds for the block moment hierarchy, but with the union maps~$\Sub{V}{1}^2 \times \Sub{V}{r-2} \to \Sub{V}{r}$, which results in the operators described later.

Let us proceed with defining the hierarchies. Let~$G = (V, E)$ be a compact packing graph and let~$r \geq 1$ be an integer. Equip the spaces~$\ind_r$ with the topology induced by the standard topology of~$\Sub{V}{r}$. Define the operator
\[
  A_r : C_{\sym}(\ind_r) \to C(\ind_{2r}),\qquad A_rF(I) = \sum_{\substack{J, J' \in \ind_r \\ J \cup J' = I}} F(J, J').
\]
It is a priori not clear that the codomain of this operator is correct:~$A_rF$ might not be a continuous function. We postpone the proof that the functions~$A_rF$ are indeed continuous for now. Theorem~\ref{thm:well-definedness-hierarchy-operators} gives a related result. Assuming this for now,~$A_r$ is a bounded linear operator under the supremum norm, thus its continuous adjoint~$A_r^* : M(\ind_{2r}) \to M_{\sym}(\ind_r)$ exists. Define the $r$th level of the~\defi{moment hierarchy}\index{moment hierarchy!packing graph} as the optimization problem
\[
  \begin{optprob}
    \lass_r(G) = \sup &\onerow{\nu(\ind_{=1})}\\
    &\onerow{\nu(\{\emptyset\}) = 1,}\\
    &\onerow{A_r^*\nu \in M_{\sym}(\ind_r)_{\succeq 0}},\\
    &\nu \in M(\ind_{2r})_{\geq 0}.
  \end{optprob}
\]

Let~$r \geq 2$ be an integer, and define the operator
\index{ Mr@$\lass_r(H)$!packing graph}
\begin{multline*}
  B_r : C_{\sym}(\ind_1^2 \times \ind_{r-2}) \to C_{\sym}(\ind_r),\\
  B_rF(I) = \sum_{Q \in \Sub{V}{r-2}} \sum_{\substack{J, J' \in \Sub{I}{1} \\ Q \cup J \cup J' = I}} F(Q, J, J').
\end{multline*}
By Theorem~\ref{thm:well-definedness-hierarchy-operators},~$B_r$ is a bounded linear operator which has a continuous adjoint~$B_r^* : M_{\sym}(\ind_r) \to M_{\sym}(\ind_1^2 \times \ind_{r-1})$. Define the~\defi{block moment hierarchy}\index{block moment hierarchy!packing graph} as the optimization problem
\index{ blockMr@$\kpb_r(H)$!packing graph}\[
  \begin{optprob}
    \kpb_r(G) = \sup &\onerow{\nu(\ind_{=1})}\\
    &\onerow{\nu(\{\emptyset\}) = 1,}\\
    &\onerow{B_r^*\nu \in M_{\sym}(\ind_1^2\times \ind_{r-2})_{\succeq 0},}\\
    &\onerow{\nu \in M(\ind_r)_{\geq 0}.}
  \end{optprob}
\]

The definition of the block moment hierarchy deviates in two ways from the one by De Laat, Machado, Oliveira, and Vallentin~\cite{deLaat2021K-PointLines}, resulting in a bound that is at least as strong as theirs. First, the original formulation excludes the empty set from~$\ind_r$, and second, a different normalization is used. Including the empty set is necessary for our proof of convergence; it seems to give stronger, nonequivalent problems. Changing the normalization does not affect convergence, and in fact, the proof of convergence hinges on this fact.

Like for the measurable independence number, we can restrict a feasible solution of~$\lass_{r+1}(G)$ to one of~$\lass_r(G)$, and the same is true for~$\kpb_r(G)$, which shows that
\[
  \lass_1(G) \geq \lass_2(G) \geq  \cdots\qquad \text{ and }\qquad \kpb_1(G) \geq \kpb_2(G) \geq \cdots.
\]
As was the case in Chapter~\ref{ch:Completely positive programming on compact spaces},~$\lass_r(G) \leq \kpb_{r+1}(G)$.

Moreover, if~$I \subseteq V$ is an independent set of~$G$, then
\[
  \nu = \sum_{R \in \ind_r,\, R \subseteq I} \delta_R,
\]
where~$\delta_R$ is the Dirac measure at~$R$, is a feasible solution to~$\kpb_r(G)$ with objective value~$|I|$, so~$\kpb_r(G) \geq \alpha(G)$ for all~$r$. Similarly,~$\lass_r(G) \geq \alpha(G)$ for all~$r$.

\sectionbreak

We now show that the operator~$B_r : C_{\sym}(\ind_1^2 \times \ind_{r-2}) \to C(\ind_r)$ is well-defined, that is, if~$F \in C_{\sym}(\ind_1^2 \times \ind_{r-2})$, then~$B_rF$ is continuous. The proof for~$A_r$ is similar. Lemma~\ref{lem:space-of-ind-subset-disjoint-union} is essential here: it implies we only have to prove continuity on~$\ind_{=r}$ for all~$r$.
\begin{theorem}%
  \label{thm:well-definedness-hierarchy-operators}
  If~$G = (V, E)$ is a graph with~$V$ a compact Hausdorff space, if~$r \geq 2$ is an integer, and if ~$F \in C(\ind_1^2 \times \ind_{r-2})$, then~$B_rF$ is continuous on~$\ind_{=s}$ for all~$s \leq r$.
\end{theorem}
\begin{proof}
  Fix an integer~$0 \leq s \leq r$, and let~$(I_{\alpha})_{\alpha}$ be a net in~$\ind_{=s}$ that converges to~$I \in \ind_{=s}$; the task is to show that~$B_rF(I_{\alpha})$ converges to~$B_rF(I)$.

  Say~$I = \{x_1, \ldots, x_s\}$ and take disjoint open neighborhoods~$U_1,  \ldots, U_s$ of~$x_1, \ldots, x_s$ respectively. The set~$W = (U_1,\ldots, U_s)_r$ is an open neighborhood of~$I$, so there exists an~$\alpha_0$ such that~$S_{\alpha} \in W$ for all~$\alpha \geq \alpha_0$. This implies that for each~$i$,~$S_{\alpha} \cap U_i \neq \emptyset$. Hence, since the~$U_i$ are disjoint, write the set~$I_{\alpha} = \{x_{\alpha,1}, \ldots, x_{\alpha,s}\}$, where~$x_{\alpha, i} \in U_i$.

  This shows that for every~$\alpha \geq \alpha_0$, the double sum appearing in~$B_rF(I_{\alpha})$ can be ordered such that the elements~$x_{\alpha,i}$ in a fixed place. The net~$(x_{\alpha, i})_{\alpha}$ converges in~$V$ for all~$i$, and so each of the sets being summed over converges in its respective space. Since~$F$ is continuous, the theorem follows.
\end{proof}

\section{The copositive hierarchy}
In the remainder of the chapter we assume that the vertex sets are compact and metrizable. The reason that we assume that they are metrizable is that exactness of the copositive formulation for compact packing graphs was only proved for this case, see~\cite{Dobre2016AGraphs}.

Let~$V$ be a compact metrizable space and~$k \geq 2$ an integer. Define the~\defi{completely positive cone}\index{completely positive cone!in MsymVk@in~$M_{\sym}(V,k)$} on~$M_{\sym}(V, k)$ by
\index{ CPVkm@$\CP(V, k)_{\m}$}\[
  \CP(V, k)_{\m} = \ccone\{\, \mu^{\otimes k} : \mu \in M(V)_{\geq 0}\, \}
\]
with closure under the weak* topology on~$M(V^k)$. Define the~\defi{copositive cone} on~$C_{\sym}(V, k)$\index{copositive cone!in CsymVk@in~$C_{\sym}(V,k)$} as~$\COP(V, k)_{\cont} = \CP(V, k)_{\m}^*$\index{ COPVkc@$\COP(V, k)_{\cont}$}. So, a continuous $k$-tensor~$T$ is copositive if and only if for every nonnegative measure~$\mu \in M(V)$ we have~$\int_{V^k} T(v)\, d\mu^k(v) \geq 0$.

When~$V$ is compact and metrizable, there exists a Radon measure~$\omega$ with full support on~$V$. With~$\COP(V, k) \subseteq L^2(V^k, \omega)$ as in Chapter~\ref{ch:completely-positive-cone} we have~$\COP(V, k)_{\cont} = \COP(V, k) \cap C_{\sym}(V, k)$. This is independent of the choice of such~$\omega$. This was argued by Dobre, Dür, Frerick, and Vallentin~\cite{Dobre2016AGraphs} for~$k = 2$, but also holds for~$k \geq 2$.

It is well known that a continuous tensor~$T$ is in~$\COP(V, k)_{\cont}$ if and only if~$T[U] \in \COP(U, k)$ for all finite~$U \subseteq V$, where
\[
  T[U] = (T(v_1, \ldots, v_k))_{v_1, \ldots, v_k \in U};
\] see~\cite[Lemma 2.1]{Dobre2016AGraphs} for a proof for~$k=2$, which holds for general~$k$. This is equivalent to the identity
\[
  \CP(V, k) = \ccone\{\, \chi^{\otimes k} : \chi \in M(V)_{\geq 0} \text{ discrete and has finite support}\, \}.
\]

An important distinction with the $L^2$ setting is that~$\COP(V, k)_{\cont}$ has an algebraic interior.
\begin{lemma}%
  \label{lem:COPc-has-algebraic-interior}
  If~$k \geq 2$ is an integer and~$V$ is a compact metrizable space, then~$\1^{\otimes k} \in \algint \COP(V, k)_{\cont}$.
\end{lemma}
\begin{proof}
  It suffices to check that for all~$F \in C_{\sym}(V,k)$ there exists~$\lambda > 0$ such that for every nonzero~$\chi \in M(V)_{\geq 0}$ that is discrete and has finite support,~$\langle \1^{\otimes k} - \lambda F, \chi^{\otimes k}\rangle \geq 0$. Take~$F \in C_{\sym}(V, k)$ nonzero and~$\lambda = \|F\|_{\infty}^{-1}$, then, for all~$\chi \in M(V)_{\geq 0}$ discrete and with finite support,
  \[
    \langle \1^{\otimes k} - \lambda F, \chi^{\otimes k}\rangle = |\supp \chi|^k - \lambda \langle F, \chi^{\otimes k}\rangle \geq |\supp \chi|^k(1 - \lambda \|F\|_{\infty}) = 0,
  \]
  and the conclusion follows.
\end{proof}

For~$r \in \N$, define the operators~$\T_r: C(V^k) \to C_{\sym}(V, k+r)$ similar to the~$\T_r$ operators of Chapter~\ref{ch:completely-positive-cone}:
\[
  \T_rF(v_1, \ldots, v_{k+r}) = \frac{1}{(k+r)!} \sum_{\pi \in \mathfrak{S}_{k+r}} (F \otimes \1^{\otimes r})(v_{\pi 1}, \ldots, v_{\pi (k+r)}).
\]
They are bounded and linear. Define the cones
\[
  C_r(V, k)_{\cont} = \T_r^{-1} C_{\sym}(V, k + r)_{\geq 0},
\]
which are weakly closed and convex. Again, if~$V$ is metrizable, then these cones are the intersections~$C_r(V, k) \cap C_{\sym}(V, k)$ with~$C_r(V,k) \subseteq L^2(V^k, \omega)$ with~$\omega$ a measure with full support.

Kuryatnikova and Vera~\cite[Theorem 2.9]{Kuryatnikova2019TheProblems} showed that for~$k=2$
\[
  C_1(V,k)_{\cont} \subseteq C_2(V, k)_{\cont} \subseteq \cdots \subseteq \COP(V, k)_{\cont},
\]
which we confirmed in Chapter~\ref{ch:completely-positive-cone} for all~$k \geq 2$. They also proved, again for~$k=2$, that~$\algint \COP(V, k)_{\cont} \subseteq \bigcup_r C_r(V, k)_{\cont}$; their proof goes through for all~$k \geq 2$, which was shown in the appendix of~\cite{Bekker2026OptimizationSpaces}. This results in a version of Pólya's theorem for continuous tensors, but under the presence of a Radon measure with full support. Since our definition of~$\COP(V, k)_{\cont}$ does not depend on this measure, we also do not need it for the conclusion to hold; in fact, the proof goes through with little change. It is really just the second half of that of~\cite[Theorem A.1]{Bekker2026OptimizationSpaces}; see also~\cite[Theorem 2.9]{Kuryatnikova2017Approximating}.

\begin{theorem}[Pólya's theorem for continuous tensors]%
  \label{thm:Polyas-theorem-for-continuous-tensors}
  Let~$V$ be a compact metrizable space. Given~$T \in C_{\sym}(V, k)$, let~$M = \|T\|_{\infty}$ and
  \[
    \lambda = \inf \{\, \langle T[U], x^{\otimes k} \rangle : U \subseteq V \text{ finite },\ x \in \R^U_{\geq 0},\text{ and } \1^{\tr}x = 1\, \}.
  \]
  If~$\lambda > 0$ then, for every~$r > k(k-1)M/(2\lambda) - k$,~$\Av_{\mathfrak{S}_{k+r}}(T \otimes \1^{\otimes r}) \geq 0$.
\end{theorem}
\begin{proof}
  See the proof of Theorem A.1 of~\cite{Bekker2026OptimizationSpaces}.
\end{proof}

This leads to the following approximation result for the copositive cone.
\begin{theorem}%
  \label{thm:conic-polyas-theorem-for-continuous-tensors}
  If~$V$ is a compact metrizable space, then
  \begin{multline*}
    C_0(V, k)_{\cont} \subseteq C_1(V, k)_{\cont} \subseteq \cdots \subseteq \COP(V, k)_{\cont} \text{ and}\\
    \algint \COP(V, k)_{\cont} \subseteq \bigcup_{r \in \N } C_r(V, k)_{\cont}.
  \end{multline*}
\end{theorem}
\begin{proof}
  Again, see the appendix of~\cite{Bekker2026OptimizationSpaces}.
\end{proof}

\sectionbreakafterproof

Let us now discuss the completely positive and copositive hierarchies for packing graphs. From now on~$k = 2$ everywhere, and~$G = (V, E)$ is a packing graph with~$V$ a compact and metrizable space. Denote the diagonal of~$V^2$ by~$\Delta$. Let~$C_r(V)_{\cont} = C_r(V, 2)_{\cont}$, and similarly for~$\COP(V,2)_{\cont}$ and~$\CP(V,2)_{\m}$.

For a convex cone~$\cC \subseteq M_{\sym}(V)$, consider optimization problems of the form
\index{ theta@$\vartheta(G, \cC)$!packing graph}
\begin{equation}%
  \label{eqn:theta-number-compact-packing-graphs}
  \begin{optprob}
    \vartheta(G, \cC) = \sup & \onerow{\alpha(V^2)}\\
    &\onerow{\alpha_{E} = 0,}\\
    &\onerow{\alpha(\Delta) = 1,}\\
    &\onerow{\alpha \in \cC,}
  \end{optprob}
\end{equation}
where~$\alpha_E$ is the restriction of~$\alpha$ to the edge set. For~$\cC$ a convex cone in~$C_{\sym}(V)$,
\index{ theta*@$\vartheta^*(G, \cC$)}\[
  \begin{optprob}
    \vartheta^*(G, \cC) = \inf & \onerow{t}\\
    &T(v,v) \leq t - 1 &\text{for all } v \in V\\
    &T(v_1,v_2) \leq -1 &\text{for all }v_1v_2 \in E\\
    &\onerow{T \in \cC.}
  \end{optprob}
\]
As before, there should be no confusion with~$\vartheta$ from Part~\ref{part:measurable-setting}, since there we only considered measurable graphs.

Given a closed convex cone~$\cC \subseteq M_{\sym}(V)_{\geq 0}$,~$\vartheta(G, \cC)$ and~$\vartheta^*(G, \cC^*)$ are dual in the sense of~\cite[Ch. IV]{Barvinok2002AConvexity}. Weak duality~$\vartheta(H, \cC) \leq \vartheta^*(H, \cC^*)$ follows in the usual way, namely by taking feasible solutions~$\alpha$ and~$(t, T)$ of the respective programs, and evaluating~$\langle T, \alpha\rangle$ to obtain~$\alpha(V^2) \leq t$. Dobre, Dür, Frerick and Vallentin~\cite{Dobre2016AGraphs} showed that if~$G = (V, E)$ is a packing graph with~$V$ compact and metrizable, then we also have strong duality and exactness.

\begin{theorem}[\cite{Dobre2016AGraphs}]%
  \label{thm:cop-theta-exact-packing-graphs}
  If~$G = (V, E)$ is a packing graph with~$V$ compact and metrizable, then
  \[
    \alpha(G) = \vartheta(G, \CP(V)_{\m}) = \vartheta^*(G, \COP(V)_{\cont}).
  \]
\end{theorem}

We shall refer to the sequence~$(\vartheta(G, C_r(V)_{\cont}^*))_r$ as the completely positive hierarchy of~$G$, and to~$(\vartheta(G, C_r(V)_{\cont}))_r$ as its copositive hierarchy. Kuryatnikova and Vera proved~\cite[Theorem 2.17]{Kuryatnikova2019TheProblems} that for certain compact packing graphs, the copositive hierarchy converges to the independence number. Their proof depends on the existence of a point~$T_0 \in \algint \COP(V)_{\cont}$ such that for all~$xy \in \overline{E}$,~$Z(x, y) \leq -1$. The following theorem extends this result to all packing graphs with a compact and metrizable vertex set, by the fact that such a point exists for every compact Hausdorff space~$V$.
\begin{theorem}%
  \label{thm:convergence-copositive-hierarchy-packing-graphs}
  If~$G = (V, E)$ is a packing graph with~$V$ compact and metrizable, then~$\alpha(G) = \lim_r \vartheta(G, C_r(V)_{\cont}^*) = \lim_r \vartheta^*(G, C_r(V)_{\cont})$.
\end{theorem}
\begin{proof}
  It is enough to prove the inequality~$\leq$ for each~$\vartheta(G, C_r(V)^*)$, since weak duality ensures that then~$\vartheta^*(G, C_r(V))$ for all~$r$ follows as well. Since for all~$r$,~$\vartheta(G, C_r(V)^*)$ is a relaxation of~$\vartheta(G, \CP(V)_{\m})$, the inequality indeed holds.

  For the other inequality, first prove that there exists~$T_0 \in \algint\COP(V)_{\cont}$ such that~$T_0(x,y) \leq -1$ for all~$ xy \in \overline{E}$. De Laat and Vallentin~\cite[Lemma 7]{deLaat2015AGeometry} showed that there is a positive-semidefinite kernel~$F$ such that~$F(x, y) \leq -1$ for all~$xy \in \overline{E}$. Then~$2F$ is positive semidefinite, and~$2F(x, y) \leq -2$ for~$xy \in \overline{E}$. Now take a suitable convex combination of~$2F$ and~$\1^{\otimes 2}$ to obtain the required kernel~$T_0$.

  Dobre, Dür, Frerick, and Vallentin~\cite[Theorem 1.2]{Dobre2016AGraphs} showed that for~$V$ compact and metrizable,~$\vartheta^*(G, \COP(V)_{\cont}) = \alpha(G)$. So, take a feasible solution~$(t, T)$ of~$\vartheta^*(G, \COP(V)_{\cont})$. For every~$0 < \epsilon \leq 1$ the kernel~$\epsilon T_0 + (1-\epsilon)T$ is in the algebraic interior of~$\COP(V)_{\cont}$. So by Theorem~\ref{thm:conic-polyas-theorem-for-continuous-tensors} it belongs to~$C_r(V)_{\cont}$ for some~$r \geq 0$. Thus, by taking~$\epsilon \to 0$, it follows that~$\lim_r \vartheta^*(G, C_r(V)_{\cont}) \leq t$, and the conclusion follows.
\end{proof}

\section{Comparing the hierarchies}
Throughout this section, fix~$G = (V, E)$ a packing graph with~$V$ compact and metrizable. This section contains the main subject of this chapter, namely the proof of the following theorem.
\begin{theorem}%
  \label{thm:comparison-cp-hierarchy-and-kpb}
  If~$G = (V, E)$ is a packing graph with~$V$ compact and metrizable, then for every~$r \geq 1$,~$\lass_{r + 1}(G) \leq \kpb_{r+2}(G) \leq \vartheta(G, C_r(V)^*_{\cont})$. In particular~$\alpha(G) = \kpb_r(G) = \lass_{r-1}(G)$ for all~$r \geq \alpha(G) + 2$.
\end{theorem}

That~$\kpb_r(G) = \alpha(G)$ for all~$r \geq \alpha(G) + 2$ follows from the first part of the theorem and that if~$r \geq \alpha(G) + 2$, then~$\ind_{r - 2}$ is the space of all independent sets, thus the sequence~$(\kpb_r(G))_r$ stabilizes. We will prove Theorem~\ref{thm:comparison-cp-hierarchy-and-kpb} by taking a feasible solution of~$\kpb_{r+2}(H)$ and making a feasible solution of~$\vartheta(H, C_r(V)_{\cont}^*)$ with at least the same objective value.

Let~$r \geq  0$ be an integer and let~$S \subseteq V$. Denote by~$N_r(S)$ the number of tuples~$v \in V^r$ such that~$\setof{v} = S$. Recall the convention~$V^0 = \{\emptyset\}$, so the definition includes~$r = 0$. We do not specify a domain for~$N_r$; we will use it as a function on~$\ind_t$ for any~$t$ we need. Note that~$N_r(S)$ depends only on the cardinality of~$S$, so the function~$N_r : \ind_t \to \R$ is continuous.

Fix integers~$r \geq 0$ and~$0 \leq t \leq s$. Let~$Q_{s, t} : C(V^t) \to C(\ind_{r})$ be the map such that
\[
  Q_{s,t}F(I) = \sum_{\substack{v \in V^s \\ \setof{v} = I}} F(v_1, \ldots, v_t).
\]
A proof similar to that of Theorem~\ref{thm:well-definedness-hierarchy-operators} shows that~$Q_{s,t}F$ is indeed continuous for every continuous~$F$, and each~$Q_{s,t}$ is a bounded linear operator. Of course, these maps also depend on~$r$, but we omit it from notation.
\begin{lemma}%
  \label{lem:properties-Qst}
  Let~$r$,~$s$, and~$t$ be integers such that~$r \geq 0$,~$s \geq 1$ and~$0 \leq t \leq s$, and let~$G = (V, E)$ be a compact packing graph. Then,
  \begin{enumerate}
    \item[(i)]  for all~$\nu \in M(\ind_r)$,~$Q_{s,t}^*\nu(V^t) = \langle N_s, \nu\rangle$, and
    \item[(ii)] if~$t+2 \leq s$, then~$Q_{s, t+2} \mathcal{T}_t  =Q_{s, 2}$, hence~$\T_t^* Q_{s, t+2}^* = Q_{s, 2}^*$.
  \end{enumerate}
\end{lemma}

\begin{proof}
  That~$Q_{s,t}$ is linear and bounded is clear, so its continuous adjoint is well-defined. If~$\nu \in M(\ind_r)$, then
  \[
    Q_{s,t}^*\nu(V^t) = \langle Q_{s,t}\1_{V^t}, \nu\rangle =\int_{\ind_k} \sum_{\substack{v \in V^s \\ \setof{v} = I}} 1\, d\nu(I) = \langle N_s, \nu\rangle,
  \]
  proving~\textup{(i)}.

  To see~(ii), let~$K \in C(V^2)$. For every~$I \in \ind_r$,
  \[
    \begin{split}
      (Q_{s, t+2}\T_t F)(I) &= \sum_{\substack{v \in V^s \\ \setof{v} = I}} \frac{1}{(t+2)!} \sum_{\pi \in \mathfrak{S}_{t+2}} F(v_{\pi 1}, v_{\pi 2})\\
      &=\frac{1}{(t+2)!}\sum_{\pi \in \mathfrak{S}_{t+2}} \sum_{\substack{v \in V^s \\ \setof{v} = I}} F(v_{\pi 1}, v_{\pi 2})\\
      &= \sum_{\substack{v \in V^s \\ \setof{v} = I}} F(v_1, v_2)\\
      &= Q_{s, 2} F (I),
  \end{split}\]
  and~(ii) follows.
\end{proof}

\begin{proof}[Proof of Theorem~\textup{\ref{thm:comparison-cp-hierarchy-and-kpb}}]
  Fix an integer~$r \geq 1$ and a feasible solution~$\nu$ of the problem~$\kpb_{r+2}(H)$ with positive objective value. Write~$Q_{t} = Q_{r+2, t}$ for short, and take~$C(\ind_{r+2})$ for its codomain.

  Write~$\Phi_t = \langle N_t, \nu\rangle$ for~$t \in \N$. Assume for now that~$\Phi_t > 0$ for all~$t$; the proof of this will follow later. Write
  \[
    \beta =  Q^*_{r+2} \nu \in M(V^{r+2})\qquad \text{and}\qquad \alpha = \T_r^*\beta \in M_{\sym}(V).
  \]
  The first objective is to show that~$\Phi_{r+1}^{-1} \alpha$ is feasible for~$\vartheta(H, C_r(V)_{\cont}^*)$.

  To begin, if~$F \in C(V^{r+2})$ is nonnegative, then so is~$Q_{r+2}F$, and it follows that~$\langle F, Q_{r+2}^*\nu\rangle = \langle Q_{r+2}F, \nu\rangle \geq 0$, hence~$\beta$ is nonnegative. If~$F \in C_{\sym}(V)$ is nonnegative, then so is~$\T_{r}F$, whence~$\langle F, \alpha\rangle = \langle F, \T_r^*\beta\rangle = \langle \T_rF, \beta\rangle \geq 0$, and~$\alpha \in C_r(V)_{\cont}^*$.

  Next, if~$F \in C_{\sym}(V)$ is a function with support in~$E$, using Lemma~\ref{lem:properties-Qst}
  \[
    \langle F, \alpha\rangle = \langle F, \T_r^*Q_{r+2}^* \nu\rangle\ = \langle F, Q_2^*\nu\rangle = \int_{\ind_{r+2}} \sum_{\substack{v \in V^{r+2} \\ \setof{v} = I}} F(v_1, v_2)\, d\nu(I) = 0.
  \]
  Since~$E$ is open in~$V^2$, it is itself a locally compact Hausdorff space, and~$\alpha_E$ is a signed Radon measure on~$E$. It then follows from the Riesz representation theorem that~$\alpha_E = 0$.

  Next, calculate~$\alpha(\Delta)$. Every vertex is contained in an open clique, so by compactness there are open cliques~$C_1$, $\ldots$, $C_m$ such that their union is~$V$. The set~$U = \bigcup_i C_i^2$ is an open set in~$V^2$ whose union contains~$\Delta$. Moreover, if~$(v_1 , v_2) \in U$ such that~$v_1 \neq v_2$, then~$v_1v_2 \in E$; thus,~$U \setminus \Delta \subseteq E$.

  Since~$\Delta$ is closed, Urysohn's lemma gives a continuous~$F : V^2 \to [0,1]$ such that~$F(v) = 1$ for all~$v \in \Delta$ and~$F(v) = 0$ for all~$v \notin U$. Since~$\alpha_E = 0$ and using Lemma~\ref{lem:properties-Qst}
  \begin{multline*}
    \alpha(\Delta) = \langle F, \alpha\rangle = \langle F, \T_r^* Q_{r+2}^* \nu\rangle = \int_{\ind_{r+2}}\sum_{\substack{v \in V^{r+2} \\ \setof{v} = I}} F(v_1, v_2)\, d\nu(I)\\
    =
    \int_{\ind_{r+2}} \sum_{\substack{v \in V^{r+1} \\ \setof{v} = I}} 1\, d\nu(I) = Q_{r+1, r+2}^*\nu(V^{r+2}) = \Phi_{r+1}
  \end{multline*}

  To finish, it suffices to show that~$\Phi_t > 0$ for all~$t$ and~$\Phi_{r+2}\Phi_{r+1}^{-1} \geq \nu(\ind_{=1})$, as the latter shows that~$\alpha(\Delta)^{-1} \alpha$ is a feasible solution with objective at least~$\nu(\ind_{=1})$. This will follow from the following claim: if~$t \leq r$, the matrix
  \begin{equation}%
    \label{eqn:Phi-matrix}
    \begin{pmatrix}
      \Phi_t & \Phi_{t+1}\\
      \Phi_{t+1} & \Phi_{t+2}
    \end{pmatrix}
  \end{equation}
  is positive semidefinite.

  Indeed, assume the claim. Since~$\Phi_0 = 1$ and~$\Phi_1 = \nu(\ind_{=1}) > 0$,~$\Phi_2 > 0$ immediately follows. Repeating this argument results in~$\Phi_t > 0$ for all~$t$.

  For the objective value, if the matrix~\eqref{eqn:Phi-matrix} is positive semidefinite, the inequalities~$\Phi_{t+2}\Phi_{t+1}^{-1} \geq \Phi_{t+1}\Phi_t^{-1}$ follow for all~$t$. Repeated application of these inequalities yields
  \[
    \alpha(V^2)\alpha(\Delta)^{-1} = \Phi_{r+2}\Phi_{r+1}^{-1} \geq \Phi_1\Phi_0^{-1} = \nu(\ind_{=1}),
  \] as required.

  Recall that a measure~$\mu \in M_{\sym}(V)$ is positive semidefinite if~$\langle F, \mu\rangle \geq 0$ for every positive semidefinite kernel~$F\in C_{\sym}(V)$. To prove that~\eqref{eqn:Phi-matrix} is positive semidefinite, we will prove that for every~$t$ there  is a positive semidefinite measure~$\mu$ such that~$\mu(\ind_{=i} \times \ind_{=j}) = \Phi_{t + i + j}$ for~$i$,~$j \in \{0,1\}$.

  Fix an integer~$t$ such that~$0 \leq t \leq r$. Employ the Riesz representation theorem to define~$\mu$ as the measure such that~$\langle F, \mu\rangle = \langle F \otimes N_t, B_{r+2}^*\nu\rangle$ for every~$F \in C(\ind_1^2)$, where~$\ind_{r}$ is the domain of~$N_t$, and~$(F \otimes N_t)(S, T) = F(S)N_t(T)$ for all~$S \in \ind_1^2$ and~$T \in \ind_r$.

  To see that~$\mu$ is positive semidefinite, let~$F \in C_{\sym}(\ind_1)$ be a positive semidefinite kernel. Since~$N_t \geq 0$, the kernel~$(S, T) \mapsto F(S, T)N_t(Q)$ is positive semidefinite for every~$Q \in \ind_{r}$. Therefore,~$F \otimes N_t \in C(\ind_1^2\times \ind_{r})_{\succeq 0}$, and~$\langle F, \mu\rangle \geq 0$ since~$B_{r+2}^*\nu \in C(\ind_1^2 \times \ind_{r})_{\succeq 0}$.

  Next, calculate~$\mu(\ind_{=i} \times \ind_{=j})$ for~$i$,~$j \in \{0,1\}$. For every~$I \in \ind_{r+2}$ and for~$i = j = 0$,
  \[
    \bigl(B_{r+2}(\1_{\{\emptyset\}}^{\otimes 2} \otimes N_t)\bigr)(I) = \sum_{Q \in \Sub{I}{r}} \sum_{\substack{ J, J'\in \Sub{I}{1} \\  Q \cup J \cup J' = I}} \1_{\{\emptyset\}}(J) \1_{\emptyset}(J') N_t(Q) = N_t(I).
  \]

  For~$i = 0$ and~$j = 1$,
  \[
    \begin{split}
      \bigl(B_{k+r}(\1_{\{\emptyset\}}\otimes \1_{\ind_{=1}} \otimes N_t)\bigr)(I) &= \sum_{Q \in \Sub{I}{r}} \sum_{\substack{x \in I \\ Q \cup \{x\} = I}} N_t(Q)\\
      &= \sum_{Q \in \Sub{I}{t}} \sum_{\substack{ x \in I \\ Q \cup \{x\} = I}} \sum_{\substack{v \in V^t \\ \setof{v} = Q}} 1.
    \end{split}
  \]
  The map~$(Q, x, v) \leftrightarrow (x, v)$ is a bijection between the set of triples~$(Q, x, v)$ in~$\Sub{I}{t} \times I \times V^t$ such that~$Q \cup \{x\} = I$ and~$\setof{v} = Q$ and the set of tuples~$v \in V^{t+1}$ such that~$\setof{v} = I$. Hence,
  \[
    \bigl(B_{k+r}(\1_{\{\emptyset\}} \otimes \1_{\ind_{=1}} \otimes N_t)\bigr)(I) = \sum_{\substack{v \in V^{t+1} \\ \setof{v} = I}} 1 = N_{t+1}(I).
  \]

  Similarly, for~$i = j = 1$,
  \[
    \begin{split}
      \bigl(B_{k+r}(\1_{\ind_{=1}}^{\otimes 2} \otimes N_t)\bigr)(I) &= \sum_{Q \in \Sub{I}{r}} \sum_{\substack{ x, y \in I \\ Q \cup \{x, y\}  = I}} N_t(Q)\\
      &= \sum_{Q \in \Sub{I}{t}} \sum_{\substack{x, y \in I \\ Q \cup \{x, y\} = I}} \sum_{\substack{ v \in V^t \\ \setof{v} = Q}} 1\\
      &= \sum_{\substack{v \in V^{t+2} \\ \setof{v} = I}} 1\\
      &= N_{t+2}(I).
    \end{split}
  \]

  Putting it all together,~\eqref{eqn:Phi-matrix} is equal to
  \[
    A =
    \begin{pmatrix}
      \mu(\ind_{=0}^2) & \mu(\ind_{=0} \times \ind_{=1})\\
      \mu(\ind_{=0} \times \ind_{=1}) & \mu(\ind_{=1}^2)
    \end{pmatrix}.
  \]
  This matrix is positive semidefinite, since for~$x_0$,~$x_1 \in \R$
  \[
    (x_0, x_1)^{\tr}A(x_0, x_1) = \int_{\ind_1^2}(x_0 \1_{\ind_{=0}} + x_1 \1_{\ind_{=1}})^{\otimes 2}(S, T)\, d\mu(S, T) \geq 0,
  \]
  from which the theorem follows.
\end{proof}

\section{What about packing hypergraphs?}%
\label{sec:what-about-hypergraphs}
Much of the analysis in this chapter works because of Lemma~\ref{lem:space-of-ind-subset-disjoint-union}. On the side of functions, it implies that continuity on~$\ind_k$ is ensured by separate continuity on each~$\ind_{=r}$ for all~$r \leq k$. On the side of the measures, together with Urysohn's lemma, it gives a degree of control over the measure of lower dimensional sets---like the diagonal in~$V^2$---that is otherwise out of reach. Not only are these properties important for Theorem~\ref{thm:comparison-cp-hierarchy-and-kpb}, they are also necessary for the definition of the copositive formulation, the moment hierarchy, and the block moment hierarchy.

For the conclusion of Lemma~\ref{lem:space-of-ind-subset-disjoint-union} to hold for $k$-uniform packing hypergraphs, it is sufficient to ask that again every finite clique is contained in an open clique, which is equivalent to the properties of Theorem~\ref{thm:packing-graph-if-and-only-if-Kk-open} with~$\clique_2$ replaced by~$\clique_k$. In this setting, the copositive formulation and all other definitions and results seem to go through unchanged.

However, in the next chapter we will see a 3-uniform hypergraph that reasonably can be called a packing hypergraph, but for which there are pairs---which are vacuously cliques---that are not contained in an open clique; they are not even contained in any clique. Thus, there are independent triples that converge to a pair without every passing through~$\clique_{=3}$. As a result, it can be shown that there are~$k$ for which~$\ind_k$ is not open in~$\ind_{k+1}$.

A simple solution to this is to replace the topology of~$\ind_k$ by the disjoint-union topology~$\bigsqcup_{r \leq k} \ind_{=r}$. This seems correct for the formulation of a moment and a block moment hierarchy, but it comes with other complications. Most importantly, the space~$\bigsqcup_{r \leq k} \ind_{=r}$ is not compact, which means that many steps we took might not make sense anymore.

\chapter[Application: obtuse almost-equiangular sets]{Application: obtuse almost-equiangular sets}%
\label{ch:almost-equiangular}

This chapter presents a geometric question that can be formulated as the independence number of a 3-uniform hypergraph. This hypergraph behaves like a compact packing graph and can therefore be thought of as an example of a ``packing hypergraph''. What follows is a lightly edited version of the preprint by Bachoc, Bekker, Moustrou, and Oliveira~\cite{Bachoc2025ObtuseSets}.

\sectionbreak

Given~$t \in [-1, 1)$, a set~$S \subseteq S^{n-1}$ is \defi{$t$-almost-equiangular}\index{almost equiangular set@almost-equiangular set} if every $3$-subset~$\{x, y, z\}$ of~$S$ is such that~$t \in \{x^{\tr}y, x^{\tr}z, y^{\tr}z\}$.  In the literature, the word ``almost'' is often replaced by ``nearly''.  An \defi{obtuse almost-equiangular set}\index{almost equiangular set@almost-equiangular set!obtuse} is a $t$-almost equiangular set with~$t \leq 0$.  A $0$-almost-equiangular set is also often called \defi{almost-orthogonal}\index{almost equiangular set@almost-equiangular set!almost-orthogonal set}.  Similarly, one may define \defi{almost-equidistant} subsets of a metric space, of which almost-equiangular sets are a special case.

Denote the maximum cardinality of a $t$-almost-equiangular set in~$S^{n-1}$ by~$\alpha(n, t)$\index{ a n t@$\alpha(n, t)$}.  The problem of finding~$\alpha(n, t)$ is called the \defi{$t$-almost-equiangular-set problem}\index{problem!almost equiangular set@almost-equiangular set}.  For~$t = 0$, this problem  first appears in a paper by Rosenfeld~\cite{Rosenfeld1991AlmostEd}, who attributes the question to Erdős. Rosenfeld showed that~$\alpha(n, 0) = 2n$; a lower bound is given by the union of two disjoint orthogonal bases and an upper bound is given through an interesting argument involving the spectrum of a matrix associated to an almost-equiangular set. Pudlák~\cite{Pudlak2002CyclesMatrices} and Deaett~\cite{Deaett2011TheGraph} reproved this result by simpler methods.

Later, Bezdek and Lángi~\cite{Bezdek1999AlmostSd-1} extended Rosenfeld's spectral bound to all~$t \in [-1, \varepsilon]$, where~$\varepsilon > 0$ is a number close to~$0$ that depends on the dimension. In particular, they proved that~$\alpha(n, t) \leq 2(n+1)$ on this interval with equality at~$t = -1/n$. An example of an optimal construction at this inner product is the union of two disjoint regular $n$-simplices. Polyanskii~\cite{Polyanskii2017OnII} mentioned a simple lifting argument to obtain~$\alpha(n,t) \leq 2(n+1)$ for~$t \leq 0$ directly from Rosenfeld's original result.

The goal of the current work is two-fold.  First, to obtain better upper bounds on the number~$\alpha(n, t)$, which is done through semidefinite programming and a closer investigation of the spectral bound of Bezdek and Lángi. Both methods reproduce known bounds, and improve many others.  Second, to list all $t$-almost-equiangular subsets of~$S^{n-1}$ of size~$\alpha(n, t)$ for small~$n$. The spectral bound of Bezdek and Lángi again plays an important role; it is used to derive characterizing properties of those $t$-almost-equidistant sets in~$S^{n-1}$ that are maximum for all~$t \in [-1,0]$.

\subsection*{Upper bounds through semidefinite programming} For~$t \in [-1, 1)$, the \defi{equiangular-lines problem}\index{problem!equiangular lines@equiangular-lines} asks for the maximum number of vectors in~$S^{n-1}$ such that any two distinct vectors have inner product~$\pm t$. This problem can be rephrased as an independence-number problem on a compact packing graph, similar to the spherical-codes problem. The methods discussed in Chapter~\ref{ch:cop-prog-packing} are thus applicable, and in fact result in the best known bounds on the equiangular-lines problem, see also De Laat, Machado, and De Muinck Keizer~\cite{deLaat2023TheAngle}.

The $t$-almost-equiangular-set problem can be rephrased as a question on independent sets of a hypergraph. Given~$n \geq 2$ and~$t \in [-1, 1)$, let~$H(n, t)$ be the $3$-uniform hypergraph whose vertex set is~$S^{n-1}$ and in which a 3-set~$\{x, y, z\}$ of points is an edge if~$t \notin \{x^{\tr} y, x^{\tr} z, y^{\tr} z\}$.  Then independent sets of~$H$ correspond to $t$-almost-equiangular sets and vice versa.  It follows that~$\alpha(n, t) = \alpha(H(n, t))$.

This connection  again opens the door to the development of optimization upper bounds for~$\alpha(n, t)$.  Castro-Silva, Oliveira, Slot and Vallentin~\cite{Castro-Silva2023AHypergraphs} proposed an extension of the Lovász theta number to finite hypergraphs.  A further extension to infinite hypergraphs by the same authors~\cite{Castro-Silva2022ASets} has applications in Euclidean Ramsey theory. The underlying hypergraphs are unlike packing graphs, and the setting is more in the spirit of what we studied in Part~\ref{part:measurable-setting} of this thesis.  The current chapter proposes an alternative extension of the theta number to infinite hypergraphs like~$H(n, t)$ based on the moment hierarchy and the block moment hierarchy~\cite{deLaat2015AGeometry, deLaat2021K-PointLines}. This bound is strongly related to the semidefinite programming methods developed in~\cite{Bilyk2023OptimalPotentials, Bilyk2024OptimizersSets}, where similar techniques were used to reprove Rosenfeld's original bound, and further apply them to energy minimization questions on hypergraphs.

This allows for the computation of upper bounds for~$\alpha(n, t)$ through the use of sums of squares and semidefinite programming.  Analytic bounds can be obtained by interpolating solutions of the resulting semidefinite programming problems, leading to the following theorem proved in Section~\ref{sec:Bounds}.
\begin{theorem}%
  \label{thm:interpolation}
  If~$t \in [-1, 0]$ and~$n \geq 3$, and if
  \[
    f(n,t) = p^2n(1-t)^2/(2(nt^2+1)),
  \]
  where
  \[
    p = \frac{8n^2t^4(2n-1) - 9n^2t^3(n-1) + (2nt^2 -3t + 4)(7n+1)}{2(1-t)(1+7n - 2n^2t^3(2n-1))},
  \]
  then $\alpha(n,t) \leq \floor{f(n,t)} \leq \floor{(16t-9)^2/(128t^2)}$.
\end{theorem}

The bound~$(16t-9)^2/(128t^2)$ in this theorem is an asymptotic bound; it is the limit of~$f(n,t)$ as~$n$ goes to infinity. That there exists an upper bound that does not depend on the dimension~$n$ is consistent with the existence of the constructions considered in this chapter not explicitly depending on the embedding dimension~$n$ if~$t$ is far enough removed from~$-1/n$.


\subsection*{Lower bounds through constructions}

If~$S$ is a $t$-almost-equiangular set, then its \defi{distance-$t$ graph}\index{distance graph!distance t graph@distance-$t$ graph}, namely the graph with vertex set~$S$ in which~$x$ and~$y$ are adjacent if~$x^{\tr}y = t$, is \defi{anti-triangle free}\index{anti triangle free graph@anti-triangle-free graph}, that is, its complement does not contain triangles.

Necessary and sufficient conditions for some anti-triangle-free graphs to be the distance graph of an almost-equiangular set are given in Section~\ref{sec:realizability}.  Together with the optimization bound of Theorem~\ref{thm:interpolation} and the results of Section~\ref{sec:properties-and-uniqueness}, this leads to constraints for the existence of $t$-almost-equiangular sets of certain sizes, making it possible to list all optimal such sets for dimensions~$n = 2$ and~$3$.  This search leads to the optimal constructions listed in Section~\ref{sec:low-dim} and summarized in Figure~\ref{fig:dim2and3}.

\begin{figure}[t]%
  \centerfloat
  \includegraphics[width=1.1\textwidth]{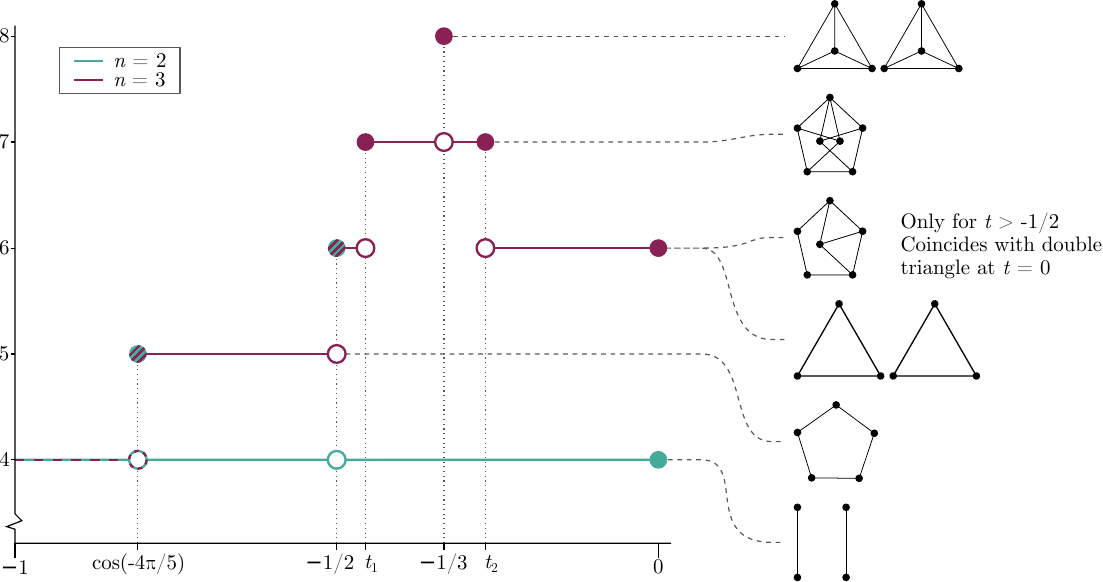}
  \caption{The classification of maximum-cardinality $t$-almost-equiangular sets in~$S^1$ and~$S^2$ for~$t \in [-1, 0]$. The vertical axis is the cardinality of the set, the horizontal axis is the inner product~$t$. Open bullets indicate that a point is excluded from the interval while closed bullets indicate that the point is included. Stripes indicate that in both dimensions the same maximum cardinality is attained. The numbers~$t_1$ and~$t_2$ are the first two roots of~\eqref{eq:MoserSpindlePolynomial} with~$k=2$.}
  \label{fig:dim2and3}
\end{figure}


\subsection*{Maximum obtuse almost-equiangular sets}
Both the semidefinite programming bound of Theorem~\ref{thm:interpolation} and the spectral bound of Rosenfeld~\cite{Rosenfeld1991AlmostEd} and Bezdek and Lángi~\cite{Bezdek1999AlmostSd-1} show that~$\alpha(n,t) \leq 2(n+1)$ for all~$n$ and~$t \leq 0$. In Section~\ref{sec:properties-and-uniqueness} the spectral bound is investigated further to show that equality for nonpositive~$t$ is only attained at~$t = -1/n$. In light of this, call the maximum $(-1/n)$-almost-equidistant sets on~$S^{n-1}$ \defi{maximum obtuse almost-equiangular sets}\index{almost equiangular set@almost-equiangular set!maximum obtuse}. Inspection of the matrices that are associated to the maximum obtuse almost-equiangular sets in the proof of the spectral bound reveals several interesting properties of these sets, like the following result.

\begin{theorem}%
  \label{cor:optimal-construction-are-2-designs}
  A maximum obtuse almost-equiangular set on $S^{n-1}$ is a spherical $2$-design.
\end{theorem}

Deaett proved~\cite{Deaett2011TheGraph} that there is a bijection between the maximum almost-orthogonal sets in~$S^{n-1}$ and certain~$2n \times 2n$ symmetric orthogonal matrices. Any $t$-almost-equidistant set in~$S^{n-1}$ with~$t \leq 0$ can be lifted to an almost-orthogonal set on~$S^n$~\cite{Polyanskii2017OnII}, and so it is expected that there is a version of this bijection for maximum nonpositive almost-equidistant sets as well. The bijection is made precise in the following theorem, which  is Deaett's correspondence with the addition of the eigenvector condition~(i).  Here,~$e$ is the all-ones vector.

\begin{theorem}%
  \label{thm:O-matrix}
  There exists a bijection between the maximum obtuse almost-equiangular subsets of $S^{n-1}$ up to orthogonal transformations and symmetric, orthogonal matrices~$O \in \Sym(2(n+1))$ such that
  \begin{enumerate}
    \item[(i)] $Oe=e$;

    \item[(ii)] $O_{ii}=0$ for all~$i$;

    \item[(iii)] $O_{ij}O_{jk}O_{ki}=0$ for all $i$, $j$, and~$k$.
  \end{enumerate}
\end{theorem}

The union of two disjoint regular $n$-simplices is called a~\defi{double-regular $n$-simplex}\index{double regular simplex@double-regular simplex}. It remains an open question whether a maximum obtuse almost-equiangular set is always a double-regular $n$-simplex. However, with the help of Theorem~\ref{thm:O-matrix} the question is settled for $2 \leq n \leq  5$.

\begin{theorem}%
  \label{thm:double-simplex-uniqueness}
  If~$2 \leq n \leq 5$, then any maximum obtuse almost-equiangular set in~$S^{n-1}$ is a double-regular $n$-simplex.
\end{theorem}

\section{Preliminaries}%
\label{sec:almost-equiangular-sets-prelims}

\subsection*{Hypergraphs}
Given a hypergraph~$H$, we will denote its vertex set by~$V(H)$\index{ V H@$V(H)$} and its edge set by~$E(H)$\index{ E H@$E(H)$}. Given~$S \subseteq V(H)$, the subgraph of~$H$~\defi{induced}\index{induced subgraph} by~$S$, denoted by~$H[S]$, is the hypergraph with vertex set~$S$ whose edges are all edges of~$H$ contained in~$S$. For all~$S \subseteq V(H)$, write~$H - S = H[V \setminus S]$.

\subsection*{Geometry}
Given~$V \subseteq S^{n-1}$ and~$t \in [-1, 1)$, the \defi{distance-$t$ graph}\index{distance graph!distance t graph@distance-$t$ graph} of~$V$ is the graph whose vertex set is~$V$ and in which~$x$ and~$y$ are adjacent if~$x^{\tr}y = t$.  A graph~$G = (V, E)$ is \defi{$(n, t)$-realizable}\index{n t realizable@$(n, t)$-realizable}\index{realizable} if there is an injection~$f\colon V \to S^{n-1}$ such that~$f(x)^{\tr} f(y) = t$ for every~$xy \in E$.  If~$xy \notin E$, then there is no constraint on~$f(x)^{\tr} f(y)$.

An~\defi{$(n - 1)$-sphere}\index{n 1 sphere@$(n-1)$-sphere} is a translated and scaled copy of~$S^{n-1}$. Let~$S$ be an $(n-1)$-sphere~$S$ with radius~$r$, and let~$k \leq n - 1$. A \defi{great $k$-sphere of~$S$}\index{great sphere} is a $k$-sphere with radius~$r$ contained in~$S$. A great $k$-sphere of~$S^{n-1}$ is then the intersection of~$S^{n-1}$ with a~$(k + 1)$-dimensional linear subspace of~$\R^n$. A great $1$-sphere is called a \defi{great circle}\index{great circle}.

An~\defi{$n$-simplex}\index{simplex!n simplex@$n$-simplex} is the convex hull of~$n + 1$ affinely independent points in Euclidean space.  An $n$-simplex is often identified with its set of~$n+1$ vertices.
A~\defi{regular $n$-simplex}\index{simplex!regular n simplex@regular $n$-simplex} with~\defi{inner product~$t$} is a regular simplex whose vertices all lie on a unit sphere and have pairwise inner product~$t$. The $t$-distance graph of a regular~$n$-simplex with inner product~$t$ is isomorphic to~$\comp{n+1}$, the complete graph on~$n+1$ vertices.  Conversely, for all~$n \geq k$ and~$d > 0$, the graph~$\comp{k+1}$ is $(n,t)$-realizable, and its realization is a regular $k$-simplex with inner product~$t$.

The~\defi{circumsphere}\index{circumsphere} of an~$n$-simplex in~$\R^n$ is the unique sphere that goes through all the vertices of the simplex~\cite[Section 1.4]{Fiedler2011MatricesGeometry}. For~$S \subseteq \R^n$, let~$\Aff{S}$ denote  the affine span of~$S$. In general, if~$S$ is a~$k$-simplex contained in~$\R^{n}$, define its circumsphere as the circumsphere of~$S$ in~$\Aff{S}$.  With this definition, the circumsphere of a $k$-simplex in~$\R^n$ is unique.


\section{The block moment hierarchy for almost-equiangular sets}
\label{sec:optimization-bounds}


For integer~$n \geq 2$ and~$t \in [-1, 1)$, let~$H = H(n, t)$ be the $3$-uniform hypergraph whose vertex set is~$S^{n-1}$
and in which three distinct points~$x$, $y$, and~$z$ form an edge if $t
\notin \{x^{\tr} y, x^{\tr} z, y^{\tr} z\}$.  Then the $t$-almost-equiangular sets are exactly the independent sets of~$H$, and so the goal is to compute
the independence number~$\alpha(H) = \alpha(n, t)$ of~$H$.

Recall the three-point bound from the block moment hierarchy for graphs from Chapter~\ref{ch:cop-prog-packing}. In Section~\ref{ch:cop-prog-packing}.\ref{sec:what-about-hypergraphs} we saw that the standard topology puts unnecessary restrictions on continuous functions on~$\ind_3$. The solution suggested there is to replace the standard topology by the disjoint union topology~$\bigsqcup_{i=0}^3 \ind_{=i}$ where each~$\ind_{=i}$ has the standard topology. This suffices for the definition of a block moment hierarchy. Thus, instead of optimizing over~$C(\ind_3)$, we optimize over~$\bigoplus_{i=0}^3C(\ind_{=3})$.

Let
\[
  B_3 : C(\ind_1^3) \to \bigoplus_{k=0}^3 C(\ind_{=3}),
\]
be defined by
\[
  (B_3A)(I) = \sum_{\substack{Q,J, J' \in \ind_1 \\ J \cup J' \cup Q = I}} A(J, J', Q)
\]
for all~$A \in C(\ind_3)$ and~$I \in \ind_3$. Then,~$B_rA$ is continuous on each~$\ind_{=i}$ by exactly the same argument with which we proved Theorem~\ref{thm:well-definedness-hierarchy-operators}. Define the optimization problem~$\kpb_3(H(n,t))$ as
\[
  \begin{optprob}
    \kpb_3(H(n, t)) = \sup &\onerow{\nu(\ind_{=1})}\\
    &\onerow{\nu(\{\emptyset\}) = 1,}\\
    &\onerow{B_3^*\nu \in M_{\sym}(\ind_1^3)_{\succeq 0},}\\
    &\onerow{\nu \in M(\ind_3)_{\geq 0}.}
  \end{optprob}
\]
we will in fact implement its dual, which is
\begin{equation}%
  \label{opt:sphere-kpb-dual}
  \begin{optprob}
    \min&(B_3 A)(\emptyset)\\
    &(B_3 A)(\{x\}) \leq -1&\text{for all~$x \in S^{n-1}$,}\\
    &(B_3 A)(S) \leq 0&\text{for all~$S \in \inds{3}$ with~$|S| \geq 2$,}\\
    &\onerow{A \in C(\inds{1}^3)\text{ is slice positive.}}
  \end{optprob}
\end{equation}
Recall that~$A$ being slice positive means that for every~$Q \in \ind_1$ the kernel~$(S, T) \mapsto A(S, T, Q)$ is positive semidefinite.

One benefit of using the minimization formulation~\eqref{opt:sphere-kpb-dual} is that every feasible solution gives an upper bound on~$\alpha(n, t)$. Contrast this with the situation for Witsenhausen's problem in Chapter~\ref{ch:witsenhausen}, where we had to go to great lengths to obtain and verify an upper bound.
\begin{theorem}
  If~$A$ is a feasible solution of~\eqref{opt:sphere-kpb-dual},
  then~$\alpha(H(n, t)) \leq (B_3 A)(\emptyset)$.
\end{theorem}
\begin{proof}
  Let~$I \subseteq S^{n-1}$ be an independent set of~$H(n, t)$.  On the one
  hand,
  \[
    \begin{split}
      \sum_{\substack{J \subseteq I\\|J| \leq 3}} (B_3 A)(J) &=
      \sum_{\substack{J \subseteq I\\|J| \leq 3}} \sum_{\substack{S, T, Q \in
      \inds{1}\\S \cup T \cup Q = J}} A(S, T, Q)\\
      &=\sum_{\substack{S, T, Q \subseteq I\\|S|, |T|, |Q| \leq 1}} A(S, T, Q)\\
      &\geq 0,
    \end{split}
  \]
  where the last inequality follows from~$A$ being positive semidefinite.

  On the other hand,
  \[
    \sum_{\substack{J \subseteq I\\|J| \leq 3}} (B_3 A)(J) =
    (B_3 A)(\emptyset) + \sum_{x \in I} (B_3 A)(\{x\}) + \sum_{\substack{J
    \subseteq I\\|J| \geq 2}} (B_3)(J)
    \leq (B_3 A)(\emptyset) - |I|,
  \]
  whence~$|I| \leq (B_3 A)(\emptyset)$.
\end{proof}

We will now see how to use a semidefinite programming solver to obtain feasible solutions to~\eqref{opt:sphere-kpb-dual}. As in Chapter~\ref{ch:witsenhausen}, we exploit the action of~$\ortho(n)$ to greatly reduce the size of the program. Denote the Haar measure on~$\ortho(n)$ by~$\mu$, and assume~$\mu(\ortho(n)) = 1$. The orthogonal group~$\ortho(n)$ acts on~$S^{n-1}$ by rotation.  Extend this
action to~$\inds{1}$ by acting trivially on~$\emptyset$. The induced action on tensors is the diagonal action. See Section~\ref{ch:complete-positivity-under-symmetry}.\ref{sec:prel-harmonic-analysis} and the appendix for more details on group actions. Simplify notation by identifying~$\inds{1}$ with~$\{\emptyset\} \cup S^{n-1}$, so below~$x \in \inds{1}$ is either~$\emptyset$ or an element
of~$S^{n-1}$.

Any feasible solution of~\eqref{opt:sphere-kpb-dual}, and in particular any $\ortho(n)$-invariant feasible solution, gives an upper bound
for~$\alpha(H)$, where~$H = H(n, t)$.  Moreover, nothing is lost by
restricting~\eqref{opt:sphere-kpb-dual} to invariant solutions.  Indeed, every
rotation in~$\ortho(n)$ is an automorphism of~$H$, and the objective of~\eqref{opt:sphere-kpb-dual} is preserved under this action.  It follows that, if~$A$ is a
feasible solution of~\eqref{opt:sphere-kpb-dual}, then
\[
  \Av_{\ortho(n)}(x, y, z) = \int_{\ortho(n)} A(Tx, Ty, Tz)\, d\mu(T),
\]
is an $\ortho(n)$-invariant feasible solution providing the same bound as~$A$.

In Chapter~\ref{ch:witsenhausen} we saw how to parametrize continuous slice-positive 3-tensors on~$S^{n-1}$ by spherical harmonics. Since we are working on~$\ind_1$, we need to slightly extend this parametrization to account for the empty set. For a full discussion of the parametrization of such tensors on~$\ind_k$, see the thesis~\cite{DeMuinckKeizer2025OnGeometry}.

Consider an $\ortho(n)$-invariant slice-positive
function~$A\colon \inds{1}^3 \to \R$. The kernel
$K_\emptyset\colon \inds{1}^2 \to \R$ defined by
\[
  K_\emptyset(x, y) = A(x, y, \emptyset)
\]
for~$x$,~$y \in \inds{1}$ is positive semidefinite and $\ortho(n)$-invariant.

Fix~$e \in S^{n-1}$, then for all~$z \in S^{n-1}$ there is a~$T \in \ortho(n)$ such
that~$Tz = e$, so~$A(x, y, z) = A(Tx, Ty, e)$.  Let~$K_e\colon \inds{1}^2 \to
\R$ be the kernel such that
\[
  K_e(x, y) = A(x, y, e).
\]
This kernel is positive semidefinite and invariant under the
\defi{stabilizer subgroup of~$e$}, namely the subgroup~$\Stab(e)$ of~$\ortho(n)$
that fixes~$e$.

It follows that an $\ortho(n)$-invariant slice-positive function~$A \in
C(\inds{1}^3)$ can be represented by two positive-semidefinite kernels
in~$C(\inds{1}^2)$, namely~$K_\emptyset$ and~$K_e$, the kernel~$K_\emptyset$
being $\ortho(n)$-invariant and the kernel~$K_e$ being $\Stab(e)$-invariant.  The
correspondence is simply
\begin{align*}
  A(x, y, \emptyset) &= K_\emptyset(x, y)\quad\text{and}\\
  A(x, y, z) &= K_e(Tx, Ty),
\end{align*}
where~$T$ is any element of~$\ortho(n)$ such that~$Tz = e$.  It follows from the
invariance of~$K_e$ that~$A$ is well-defined, since if~$T_1 z = T_2 z = e$,
then~$T_2 T_1^{-1} \in \Stab(e)$ and $K_e(T_1 x, T_1 y) = K_e(T_2 x, T_2 y)$.

In Section~\ref{sec:inv-kernels-and-tensor} of Chapter~\ref{ch:witsenhausen}, we saw that Schoenberg's theorem~\cite{Schoenberg1942PositiveSpheres} characterizes $\ortho(n)$-invariant positive-semidefinite kernels on~$S^{n-1}$ in terms of Gegenbauer polynomials, and that a theorem of Bachoc and Vallentin~\cite{Bachoc2008NewProgramming} characterizes $\Stab(e)$-invariant positive-semidefinite kernels on~$S^{n-1}$ using multivariate Gegenbauer polynomials.  Both characterizations can be easily adapted to kernels
on~$\inds{1}$.  For this the following lemma is useful.

\begin{lemma}%
  \label{lem:psd-exp}
  Let~$V$ be a topological space,~$f_1$, \dots,~$f_N\colon V \to \R$ be
  continuous functions, and for~$x$, $y \in V$ consider the~$N \times N$ matrix
  such that
  \[
    Z(x, y)_{ij} = f_i(x) f_j(y).
  \]
  If~$A \in \R^{N \times N}$ is positive semidefinite, then the kernel~$K\colon
  V^2 \to \R$ such that
  \[
    K(x, y) = \langle A, Z(x, y)\rangle
  \]
  is positive semidefinite.
\end{lemma}

\begin{proof}
  Let~$x_1$, \dots,~$x_k \in V$ and take~$u \in \R^k$.  Since~$A$ is positive semidefinite, the matrix~$A \otimes u u^{\tr}$ is also positive semidefinite; its rows and columns are indexed by~$I = \{1, \ldots, N\} \times \{1, \ldots, k\}$.  Setting~$g(i, k) = f_i(x_k)$, it follows that
  \[
    \begin{split}
      \sum_{k,l=1}^k K(x_k, x_l) u_k u_l &= \sum_{k, l=1}^k u_k u_l \sum_{i,j=1}^N A_{ij} f_i(x_k) f_j(x_l)\\
      &\sum_{(i,k), (j, l) \in I} (A \otimes uu^{\tr})_{(i,k), (j,l)} g(i, k) g(j, l)\\
      &\geq 0,
    \end{split}
  \]
  as wanted.
\end{proof}

Start with~$K_\emptyset$.  Let~$P_k^n$ denote the Jacobi polynomial of
degree~$k$ with parameters~$\alpha = \beta = (n-3)/2$ normalized so~$P_k^n(1) =
1$.  For~$k \geq 1$, let~$Z^\emptyset_k\colon \inds{1}^2 \to \R$ be such that
\[
  Z^\emptyset_k(x, y) =
  \begin{cases}
    P_k^n(x^{\tr} y)&\text{if~$x$, $y \in S^{n-1}$;}\\
    0&\text{otherwise.}
  \end{cases}
\]
Let~$Z^\emptyset_0\colon \inds{1}^2 \to \R^{2\times 2}$ be such that, for~$x$,
$y \in S^{n-1}$,
\begin{align*}
  Z^\emptyset_0(\emptyset, \emptyset)& = \smallpmatrix{1&0\\0&0},&
  Z^\emptyset_0(x, \emptyset)& = \smallpmatrix{0&0\\1&0},\\
  Z^\emptyset_0(\emptyset, x)& = \smallpmatrix{0&1\\0&0},&
  Z^\emptyset_0(x, y)& = \smallpmatrix{0&0\\0&1}.
\end{align*}

It follows from the addition formula for Gegenbauer polynomials~\cite[\S9.6]{Andrews1999SpecialFunctions} that
for every~$k > 0$ the kernel~$(x, y) \mapsto Z^\emptyset_k(x, y)$ is positive
semidefinite.  From Lemma~\ref{lem:psd-exp} it follows that if~$A \in \R^{2\times 2}$ is positive semidefinite, then the kernel $(x, y) \mapsto \langle
A, Z^\emptyset_0(x, y)\rangle$ is positive semidefinite.  So, for every~$d \geq
0$, any kernel of the form
\begin{equation}%
  \label{eq:kempty-kernel}
  (x, y) \mapsto \langle A_0^\emptyset, Z^\emptyset_0(x, y)\rangle + \sum_{k=1}^d a_k
  Z^\emptyset_k(x, y)
\end{equation}
for positive-semidefinite~$A_0^\emptyset \in \R^{2 \times 2}$ and nonnegative numbers~$a_k$
is $\ortho(n)$-invariant and positive semidefinite.
The only difference with Schoenberg's theorem is that~$Z^\emptyset_0$ is then a single number.

Next consider~$K_e$.  Recall, for~$k \geq 0$, the polynomials
\[
  Q_k^n(u, v, t) = (1 - u^2)^{k/2} (1 - v^2)^{k/2} P_k^{n-1}\biggl(\frac{t - uv}{(1
  - u^2)^{1/2} (1 - v^2)^{1/2}}\biggr);
\]
this is a polynomial on~$u$, $v$, and~$t$ of degree~$2k$, as defined in \S\ref{ch:witsenhausen}.\ref{sec:inv-kernels-and-tensor}. Also recall the matrices
\[
  (Y^n_{k, d})_{i, j}(u,v,t) = u^i v^j Q^{n-1}_k(u, v, t);
\]
here, we allow~$d = \infty$, so that the above is defined for any integers~$i$ and~$j$, and we denote the corresponding infinite matrix by~$Y^n_k$.

For~$k > 0$ and~$x$, $y \in S^{n-1}$, let~$Z^e_k(x, y)$ be the infinite matrix
indexed by integers~$i$, $j \geq 0$ such that
\[
  Z^e_k(x, y)_{i, j} = (Y^n_{k})_{i,j}(e^{\tr} x, e^{\tr} y, x^{\tr} y).
\]
Note that this is a polynomial on~$e^{\tr} x$, $e^{\tr} y$, and~$x^{\tr} y$ of degree~$i + j + 2k$.  If~$x = \emptyset$ or~$y = \emptyset$, set~$Z^e_k(x, y)_{ij} = 0$.

For integer~$i \geq 0$, let~$f_i\colon \inds{1} \to \R$ be such that
\[
  f_i(x) =
  \begin{cases}
    0&\text{if~$x = \emptyset$;}\\
    (e^{\tr} x)^i&\text{otherwise.}
  \end{cases}
\]
Let~$f_\emptyset\colon \inds{1} \to \R$ be such that~$f_\emptyset(\emptyset) = 1$
and~$f_\emptyset(x) = 0$ if~$x \in S^{n-1}$. Define the infinite
matrix~$Z^e_0(x, y)$, indexed by~$U = \{\emptyset\} \cup \{\, i \in \Z : i \geq
0\,\}$, by setting
\[
  Z^e_0(x, y)_{\alpha\beta} = f_\alpha(x) f_\beta(y)
\]
for~$\alpha$, $\beta \in U$.

Let~$A$ be a positive-semidefinite matrix indexed by a finite set of nonnegative
integers.  For~$k > 0$, Bachoc and Vallentin~\cite{Bachoc2008NewProgramming} showed that the
kernel
\begin{equation}%
  \label{eq:bachoc-v-kernel}
  (x, y) \mapsto \langle A, Z^e_k(x, y)\rangle
\end{equation}
on~$S^{n-1}$ is positive semidefinite.  In the trace inner product in~\eqref{eq:bachoc-v-kernel}, the
matrix~$Z_k^e$ is truncated, that is, only the finite submatrix
corresponding to the rows and columns of~$A$ is considered.  From this it
immediately follows that the kernel~\eqref{eq:bachoc-v-kernel} is positive
semidefinite as a kernel over~$\inds{1}$ as well.

As for~$k = 0$, if~$A$ is a positive-semidefinite matrix indexed by a finite
subset of the index set~$U$, then the kernel $(x, y) \mapsto \langle A, Z^e_k(x,
y)\rangle$ is positive semidefinite, as follows directly from
Lemma~\ref{lem:psd-exp}.  So, if~$A^e_0$, \dots,~$A^e_d$ are positive-semidefinite
matrices, with~$A^e_0$ indexed by a subset of~$U$ and~$A^e_k$ indexed by a subset of
the nonnegative integers for~$k > 0$, then
\begin{equation}%
  \label{eq:ke-kernel}
  K_e(x, y) = \sum_{k=0}^d \langle A^e_k, Z^e_k(x, y)\rangle
\end{equation}
is positive semidefinite and, by construction, $\Stab(e)$-invariant.  Every
$\Stab(e)$-invariant positive-semidefinite continuous kernel~$K_e$ can be
uniformly approximated by kernels with the above expression, see for example the appendix of~\cite{deLaat2019MomentProblems}.

With this, it is possible to express the function~$A \in C(\inds{1}^3)$
of~\eqref{opt:sphere-kpb-dual} in terms of polynomials.  Here,~$d$
in~\eqref{eq:kempty-kernel} and~\eqref{eq:ke-kernel} is fixed and the
matrices~$A_k$ in~\eqref{eq:ke-kernel} are truncated appropriately to bound the
total degree of the polynomials used.  The constraints
of~\eqref{opt:sphere-kpb-dual} are modeled as polynomial constraints using sums
of squares.  In this way,~\eqref{opt:sphere-kpb-dual} can be solved numerically
with the computer, and solutions can even be found analytically.  Both approaches are discussed in Section~\ref{sec:Bounds}.


\section{Upper bounds from the three-point bound}%
\label{sec:Bounds}

As shown in Section~\ref{sec:optimization-bounds}, the bound~\eqref{opt:sphere-kpb-dual} can be expressed in terms of a polynomial optimization problem once~$d$ is fixed in~\eqref{eq:kempty-kernel} and~\eqref{eq:ke-kernel} and the~$Z_k^\emptyset$ and~$Z_k^e$ matrices are truncated to finite matrices.
\afterpage{
  \begin{figure}[b]%
    \centerfloat
    \vbox to\textheight{
      \vfill
      \includegraphics[scale = 0.85]{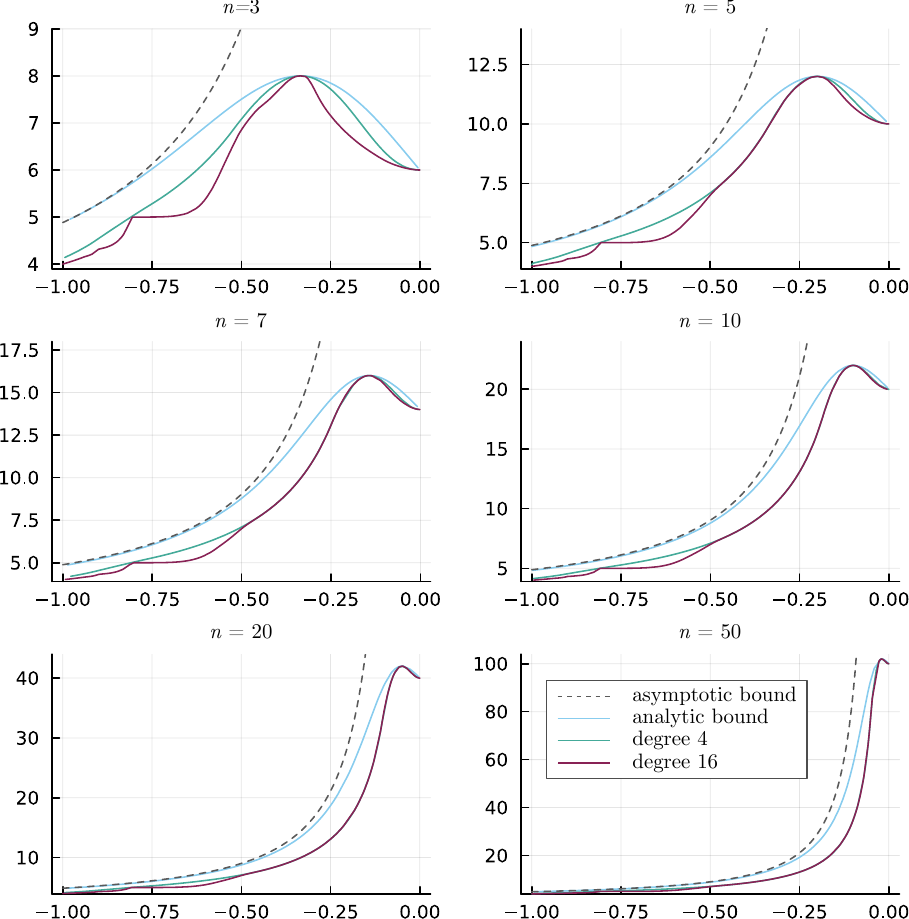}
      \caption{The numeric solutions of degree~$4$ and~$16$, and the analytic (Theorem~\ref{thm:interpolation}) and asymptotic solutions to the bound~\eqref{opt:sphere-kpb-dual}. The asymptotic bound is the limit as $n \to \infty$ of the analytic bound.}
      \label{fig:SolutionToBound}
    }
  \end{figure}
  \clearpage
}

So implemented, the three-point bound~\eqref{opt:sphere-kpb-dual} gives particularly good results for~$t \leq 0$.  Figure~\ref{fig:SolutionToBound} shows a plot of this bound as a function of~$t \in [-1, 0]$; it was computed by a Julia program using the package \texttt{ClusteredLowRankSolver.jl} \cite{Leijenhorst2024SolvingOptimization}.  These are numerical results of very high precision that can be turned into rigorous results with some effort.  The Julia package \texttt{AlmostEquiangular.jl}, contained in the arXiv supplement to~\cite{Bachoc2025ObtuseSets}, includes a function to compute the three-point bound.

Using \texttt{ClusteredLowRankSolver.jl}~\cite{Leijenhorst2024SolvingOptimization} and its rounding routine~\cite{Cohn2024OptimalityBounds}, it is possible to obtain a rational analytic solution for fixed dimension~$n \geq 3$ and for inner products~$0$ and~$-1 / n$. At these points the bound is exactly equal to the maximum size of an almost-equiangular set. These solutions can then be interpolated to obtain a rational function in~$n$ and~$t$ that gives an upper bound for the size of a $t$-almost-equiangular set in~$S^{n-1}$ for~$t \in [-1, 0]$, leading to Theorem~\ref{thm:interpolation}.

A union of two disjoint regular $(n-1)$-simplices in~$S^{n-1}$ gives a $0$-almost-equiangular set with~$2n$ points; Rosenfeld~\cite{Rosenfeld1991AlmostEd} showed that this construction is optimal.  A union of two disjoint regular $n$-simplices in~$S^{n-1}$ gives a $(-1/n)$-almost-equiangular set with~$2(n+1)$ points; Bezdek and Lángi~\cite{Bezdek1999AlmostSd-1} showed that this construction is optimal.  The bound of Theorem~\ref{thm:interpolation} is sharp in both cases, providing a new proof of the optimality of these constructions.

\begin{proof}[Proof of Theorem~\textup{\ref{thm:interpolation}}]
  The proof of the theorem is by exhibiting a solution to the three-point bound that has the objective value in the statement. To keep the solution as simple as possible, use a degree-$0$ kernel~$K_{\emptyset}$ and a degree-$4$ kernel~$K_e$. Thus, the set of positive-semidefinite variables is~$A_0^{\emptyset}$ and~$A_k^e$ with~$0 \leq k \leq 2$.

  Let
  \[
    p = \frac{8n^2t^4(2n-1) - 9n^2t^3(n-1) + (2nt^2 -3t + 4)(7n+1)}{2(1-t)(1+7n - 2n^2t^3(2n-1))}
  \]
  and
  \begin{align*}
    &A_0^{\emptyset} =
    \begin{pmatrix}
      \frac{n(1 - t)^2}{2(nt^2 + 1)} p^2  & -\frac{1}{2}p \\ \mathbf{*} & \frac{nt^2 + 1}{2n(1 - t)^2}
    \end{pmatrix},\\
    &A_0^e =
    \begin{pmatrix}
      p & -\frac{1}{4 n (1 - t)^2} -\frac{t^2}{(1 - t)^2} & \frac{3 t}{2 (1 - t)^2} & -\frac{3}{4 (1 - t)^2}\\
      \mathbf{*} & \frac{1}{4 n (n - 1) (1 - t)^3} - \frac{t^3}{2 (1 - t)^3} & \frac{3 t^2}{4 (1 - t)^3} & -\frac{n + 1}{8 n (n - 1) (1 - t)^3}\\
      \mathbf{*} & \mathbf{*} & -\frac{3 t}{2 (1 - t)^3} & 0\\
      \mathbf{*} & \mathbf{*} & \mathbf{*} & \frac{2 n - 1}{4 (n - 1) (1 - t)^3}
    \end{pmatrix},\\
    &A_1^e =
    \begin{pmatrix}
      0 & 0\\ \mathbf{*} & \frac{n + 1}{2 n (1 - t)^3}
    \end{pmatrix},\\
    &A_2^e =
    \begin{pmatrix}\frac{n-2}{4 n (n - 1) (1 - t)^3}
    \end{pmatrix}.
  \end{align*}
  The~$\mathbf{*}$s indicate that the entries are determined by the symmetry of the matrices.

  All matrices above, except for~$A_0^e$, can be checked by hand to be positive semidefinite in the domain given by~$n \geq 3$ and~$t \in [-1, 0]$.  To check that~$A_0^e$ is positive semidefinite in the required domain, first decompose it as~$A_0^e = L D L^{\tr}$, where~$L$ and~$D$ are matrices of rational functions on~$n$ and~$t$ and~$D$ is diagonal, and then check that the diagonal entries of~$D$ are nonnegative in the domain.

  These diagonal entries are rational functions, which can be rigorously checked to be nonnegative by a sum-of-squares approach.  The arXiv supplement to~\cite{Bachoc2025ObtuseSets} contains the Julia package \texttt{AlmostEquiangular.jl}, which provides sum-of-squares certificates for the nonnegativity of the diagonal entries of~$D$.  The same package also provides a sum-of-squares certificate for the inequality~$f(n, t) \leq (16t-9)^2 / (128t^2)$.

  The Julia package also checks if, for the corresponding function~$A \in C(\inds{1}^3)$,
  \begin{align*}
    &B_3A(\{x\}) = -1,\\
    &B_3A(\{x, y\}) = 0 &&\text{ for all~$x \not= y$, and}\\
    &B_3A(\{x, y, z\}) = \frac{3(t - x^{\tr} z)(t - y^{\tr} z)(t - x^{\tr} y)}{(t - 1)^3} &&\text{ for all~$x$,~$y$, and~$z$ distinct.}
  \end{align*}
  In particular, if~$t \in \{x^{\tr} z, y^{\tr} z, x^{\tr} y\}$, then~$B_3A(\{x, y, z\}) = 0$.
\end{proof}

The solution constructed in the proof above can in principle be improved; the issue is to get a good compromise between simplicity and quality.  For instance, by forcing some matrix entries to be zero as done above, it becomes possible to find a simple rational expression as given in the theorem.

\section{Realizability of anti-triangle-free graphs}
\label{sec:realizability}

A graph is \defi{anti-triangle free}\index{anti triangle free graph@anti-triangle-free graph} if its complement is triangle free. This is equivalent to saying that every triple of vertices contains an edge.  The distance graphs of almost-equiangular sets are anti-triangle free and, conversely, realizable anti-triangle-free graphs give almost-equiangular sets.  Hence, to construct good almost-equiangular sets, one has to show that given anti-triangle-free graphs are realizable.

Recall the definition of realizability from Section~\ref{sec:almost-equiangular-sets-prelims}.  The goal of this section is to determine whether certain anti-triangle-free graphs are~$(n, t)$-realizable. A construction of interest is the $(k,l)$-spindle, denoted by~$\spindle{k}{l}$ with~$k,l \geq 1$, defined later in this section, of which the Moser spindle is a special case.  In order to bound the inner products at which~$\spindle{k}{l}$ is realizable, and to offer some tools for other calculations, it is useful to derive realizability of some commonly appearing subgraphs of the spindle, namely the simplex and the rhombus.

\subsection*{The simplex}
A nice reference for simplex geometry is Fiedler~\cite{Fiedler2011MatricesGeometry}; see in particular Theorem~4.5.1 of this book for the following facts. The inner products of distinct vertices of a regular $n$-simplex inscribed in~$S^{n - 1}$ is~$-1 / n$. So~$\comp{n+1}$ is~$(n,t)$-realizable if and only if~$t = - 1/ n$.

If~$k < n$, then~$\comp{k + 1}$ is $(n, t)$-realizable if and only if~$t \geq - 1 / k$. Indeed, the circumradius of a regular $k$-simplex with inner product~$t$ is
\[
  r_k(t) = \sqrt{\frac{(1 - t) k}{k + 1}}
\]
and the circumsphere of a $k$-simplex is a $(k - 1)$-sphere. For~$k < n$, the sphere~$S^{n-1}$ contains a $(k - 1)$-sphere of every radius less than or equal to~1, so~$K_{k+1}$ is $(n, t)$-realizable if and only if~$r_k(t) \leq 1$. This happens if and only if~$t \geq -1 / k$.

The~$k+1$ vertices of a regular $k$-simplex on~$S^{n-1}$ are by definition affinely independent, and so a regular~$k$-simplex contains at least~$k$ linearly independent points. If~$t = -1/k$, then~$r_k(t) = 1$, and the circumsphere is a great sphere, which lies on a linear subspace of dimension~$k$. However, if~$k < n$ and~$t > -1/k$, then~$r_k(t) < 1$, and so the linear span of the~$k$-simplex has dimension~$k+1$. In this case, the vertices of the $k$-simplex are linearly independent.

\subsection*{The rhombus}

A useful subgraph of a spindle is the union of two complete graphs on~$k+1$ vertices that have exactly~$k$ vertices in common.  This is the distance graph of a pair of regular $k$-simplices that share exactly one facet.  Alternatively, it is the complete graph~$\comp{k+2}$ with one edge removed.  Call this graph a \defi{$k$-rhombus}\index{rhombus}\index{rhombus!k rhombus@$k$-rhombus}. By the previous paragraph, necessary conditions for realizability are~$k \leq n$,~$t = -1 / k$ if~$k = n$, and~$t \geq - 1 / k$ otherwise.

In what follows, let~$R$ be a~$k$-rhombus that is the union of two instances of~$\comp{k+1}$, denoted by~$\Sigma_1$ and~$\Sigma_2$, let~$e$ be the unique vertex of~$\Sigma_1 - V(\Sigma_2)$, let~$p$ be the unique vertex in~$\Sigma_2 - V(\Sigma_1)$, and let~$B = V(\Sigma_1) \cap V(\Sigma_2)$. Refer to~$R[B]$ as the \defi{base}\index{rhombus!base of rhombus} of the rhombus. It is an instance of~$\comp{k}$. Up to orthogonal transformations, an $(n, t)$-realization of~$R[B]$ is uniquely determined, so assume its vectors are known and denote the realization by~$B$ as well. The following lemma is comparable to~\cite[Lemma 7]{Balko2020Almost-EquidistantSets}.

\begin{lemma}%
  \label{lem:kRhombusRealizable}
  With~$e$, $p$ as above, the~$k$-rhombus is~$(n,t)$-realizable if and only if~$k \leq n-1$ and~$t > -1 / k$. If these conditions hold, then~$e$ and~$p$ lie on an~$(n - k - 1)$-sphere of radius~$\sqrt{1 - 2kt^2/((k - 1)t + 1)}$. In particular, let~$k$ be an integer such that~$k \leq n-1$ and~$t > -1 / k$, and
  \[
    \tau = \frac{2 k t^2}{(k - 1) t + 1} - 1.
  \]
  If~$k < n - 1$, then~$e^{\tr} p \geq \tau$, and if~$k = n - 1$, then~$e^{\tr} p = \tau$.

  Conversely, if~$k < n-1$ and~$t' \in [\tau, 1)$, then there exists an $(n, t)$-realization of the~$k$-rhombus in which~$e^{\tr} p = t'$.  If~$k = n - 1$, then the points~$e$ and~$p$ are uniquely determined.
\end{lemma}

\begin{proof}
  The $k$-rhombus with base~$B$ has a subgraph isomorphic to~$\comp{k+1}$, and so necessary conditions for realizability are~$t \geq - 1/k$ and~$k \leq n$.  Assume that these hold.  If~$k = n$, the vectors in~$B$ already determine a full rank system, so then~$p$ will coincide with~$e$. Consequently, another necessary condition is~$k \leq n-1$.

  Since~$t \geq -1 / k$, the~$k$-rhombus is realizable if and only if the affine space
  \[
    A = \{\, x \in \R^n : x^{\tr} b = t \text{ for all } b \in B\, \}
  \]
  intersects~$S^{n-1}$ in more than one point, that is, if and only if
  \[
    \inf{\{\, \|a\| : a \in A\, \}} < 1;
  \]
  this infimum is attained in~$A$.

  Let~$U$ be the linear span of~$B$ and let~$W$ be its orthogonal complement. Then the shortest vector~$a_0$ in~$A$ is in~$U$. Indeed, if~$a_0 = \sum_{b \in B} \lambda_b b + w \in A$ with~$w \in W$, then by orthogonality~$\|a_0\|^2 = \|\sum_{b \in B} \lambda_b b\|^2 + \|w\|^2$. Translating by a vector orthogonal to~$U$ does not change the inner product with any of the elements in~$B$. So, if~$\|a_0\|$ is minimal, then~$w = 0$.

  All that is left is to calculate the coefficients~$\lambda_b$. Since~$A$ is convex,~$a_0$ is the unique shortest vector. Because~$b^{\tr} b' = t$ for all distinct~$b, b' \in B$, applying a permutation to the coefficients gives another vector in~$A$ with the same norm. By uniqueness, this forces all~$\lambda_b$ to have the same value~$\lambda$. For every~$b \in B$,
  \[
    t = a_0^{\tr} b = \lambda \sum_{b' \in B} b'^{\tr} b = \lambda((k - 1) t + 1),
  \]
  so~$\lambda = t / ((k - 1)t + 1)$ and
  \[
    \|a_0\|^2 = \frac{kt^2}{(k - 1)t + 1}.
  \]
  Since~$t \geq -1 / k$, it follows that~$\|a_0\| < 1$ if and only if~$kt^2 - (k - 1)t - 1 < 0$. As a polynomial in~$t$ it has roots~$1$ and~$-1 / k$, so the~$k$-rhombus is realizable if and only if~$-1 / k < t < 1$.

  The intersection of~$A$ with~$S^{n-1}$ gives an $(n - k - 1)$-sphere~$S$ whose radius~$r$ is~$\sqrt{1 - \|a_0\|^2} = \sqrt{1 - kt^2 / ((k - 1)t + 1)}$.  Any two distinct points on~$S$ are valid realizations of~$p$ and~$e$.  If~$\tau$ is the minimum possible inner product between points on~$S$, then~$2r = \sqrt{2(1 - \tau)}$, so
  \[
    \tau = 1 - 2r^2 = \frac{2kt^2}{(k - 1)t + 1} - 1.\qedhere
  \]
\end{proof}

In Lemma~\ref{lem:kRhombusRealizable}, the inner product~$t$ does not depend on the embedding dimension, something that often happens for these types of constructions.

\subsection*{The spindle}

The \defi{$(k,l)$-spindle}\index{spindle}\index{spindle!k l spindle@$(k,l)$-spindle}, notation~$\spindle{k}{l}$\index{ S k l@$\spindle{k}{l}$}, is described as follows: let~$R_1$ be a $k$-rhombus; say~$e$ and~$p_1$ are the vertices of its unique nonedge. Attach at~$e$ an $l$-rhombus~$R_2$ with nonedge~$e p_2$, so~$V(R_1) \cap V(R_2) = \{e\}$. Finally, add the edge~$p_1p_2$. Figure~\ref{fig:spindles} shows several spindles. For~$k \geq 1$, the spindle~$\spindle{k}{k}$ is called the \defi{$k$-Moser spindle}\index{spindle!Moser spindle}\index{spindle!k Moser spindle@$k$-Moser spindle}, denoted by~$\MS{k}$\index{ MSk@$\MS{k}$}. If~$k \leq l$, then~$\spindle{k}{l} \subseteq \MS{l}$.

\begin{figure}[tb]
  \centerfloat
  \includegraphics{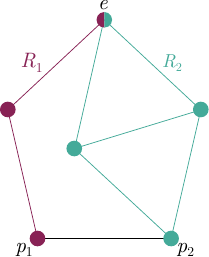}\qquad
  \includegraphics{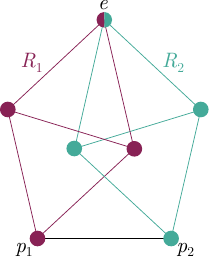}\qquad
  \includegraphics{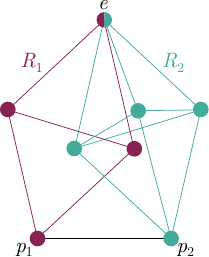}
  \caption{From left to right:~$\spindle{1}{2}$,~$\MS{2}$, and~$\spindle{2}{3}$. The~$R_i$ are indicated with their respective colors.}
  \label{fig:spindles}
\end{figure}

The~$(k,l)$-spindle is an anti-triangle-free graph of order~$k + l + 3$. The Moser spindle in particular is well studied. For example, the spindle~$\MS{n-1}$ was already pointed out by Bezdek and Lángi~\cite{Bezdek1999AlmostSd-1} as a $t$-almost-equiangular set for~$t$ close to~1. However, they did not attempt to calculate for which~$t$ the graph is realizable, and did not consider the case of negative~$t$ or~$k \neq l$.  This is done in the following theorem.

\begin{theorem}%
  \label{thm:RealizabilityOfSpindles}
  If~$k \geq 1$ and~$i \in \{1,2,3\}$, then all roots~$t_{k,i}$ of the polynomial
  \[
    8k^2t^3 - (k^2 -10k + 1)t^2 - 2(k - 1)t - 1
  \]
  with respect to~$t$ are real and can be ordered such that~$t_{k, 1} \leq t_{k, 2} < 0 < t_{k, 3}$.  The~$(k,l)$-spindle is~$(n,t)$-realizable if and only if~$k, l \leq n - 1$ and~$t \in [-1, 1)$ satisfy
  \begin{align}
    \label{eq:spindle-a} &t = (-1/4)(1 \pm \sqrt{5}) &&\text{ if } n = 2\text{ and }k = l = 1,\\
    \label{eq:spindle-b} &t_{k, 1} \leq t \leq t_{k, 2}\text{\quad or\quad} t_{k , 3} \leq t &&\text{ if } n > 2 \text{ and }k = l = n - 1,\\
    \label{eq:spindle-c} &t_{k,1} \leq t&&\text{ if }n > 2\text{ and } k = l < n - 1,\text{ or}\\
    \label{eq:spindle-d} &-1/l < t &&\text{ if } n > 2\text{ and }k < l \leq n - 1.
  \end{align}
\end{theorem}

The following simple lemma does a lot of the work in the proof of Theorem~\ref{thm:RealizabilityOfSpindles}.
\begin{lemma}
  If~$S_1$ and~$S_2$ are subsets of~$S^{n-1}$ that are invariant under the subgroup of~$\ortho(n)$ that stabilizes a point~$e$ and if $\inf{\{\, x^{\tr} y : x \in S_1, y \in S_2\, \}}$ is attained by points~$p_1 \in S_1$ and~$p_2 \in S_2$, then~$e$, $p_1$, and~$p_2$ lie on a great circle~$C$. Moreover, if~$f \in C$ is orthogonal to~$e$ and if~$p_1$, $p_2 \neq \pm e$, then~$f^{\tr} p_1$ and~$f^{\tr} p_2$ have opposite signs.
  \label{lem:MaxDistanceWhenOpposite}
\end{lemma}
\begin{proof}
  If~$p_1$ or~$p_2$ is~$\pm e$, then the result is clear.  So assume~$p_1$, $p_2 \neq \pm e$.

  Let~$U = \lspan\{ e, p_1 \}$ and let~$f$ be a unit vector in~$U$ orthogonal to~$e$ such that~$f^{\tr} p_1 > 0$. Write~$p_1 = \alpha e + \beta f$ and~$p_2 = \lambda e  + \kappa f + w$ with~$w \in U^{\perp}$, so~$|\kappa| \leq \sqrt{1 - \lambda^2}$.  By invariance under the stabilizer of~$e$, any point~$p_2'$ on the sphere with~$p_2^{\tr} e = p_2'^{\tr} e$ is also in~$S_2$. Let~$p_2' = \lambda e - \sqrt{1-\lambda^2}f \in S_2$. Then
  \[
    p_1^{\tr} p_2' = \alpha \lambda - \beta \sqrt{1 - \lambda^2} \leq \alpha \lambda + \beta \kappa = p_1^{\tr} p_2.
  \]
  It follows that~$w = 0$ and that~$f^{\tr} p_2 = \kappa < 0$, as wanted.
\end{proof}

\begin{proof}[Proof of Theorem~\textup{\ref{thm:RealizabilityOfSpindles}}]
  Let~$n \geq 2$ and~$1 \leq k \leq l \leq n - 1$ be integers and let~$t$ be in~$(- 1 / l, 1)$. A~$(k, l)$-spindle contains the union of a~$k$- and an~$l$-rhombus that intersect in a single point. Let~$R_1$ be the~$k$-rhombus with unique nonedge~$e p_1$ and~$R_2$ the~$l$-rhombus with unique nonedge~$e p_2$, so~$V(R_1) \cap V(R_2) = \{ e \}$. A necessary and sufficient condition for realizability is that there are realizations of~$R_1$ and~$R_2$ such that~$p_1^{\tr} p_2 = t$.

  Let~$S_i$ be the set of all possible images of~$p_i$ under $(n, t)$-realizations of~$R_i$ that map~$e$ to the north pole~$(1, 0, \ldots, 0)$, that is,
  \[
    S_i = \{\, f(p_i) : f \text{ is an $(n, t)$-realization of } R_i\text{ such that }f(e) = (1, 0, \ldots, 0)\, \}.
  \]
  Let
  \begin{equation}%
    \label{eq:tau-formula}
    \tau_1 = \frac{2 k t^2}{(k - 1) t + 1} - 1\quad\text{and}\quad \tau_2 = \frac{2 l t^2}{(l - 1) t + 1} - 1.
  \end{equation}
  If~$n > 2$, then Lemma~\ref{lem:kRhombusRealizable} guarantees the existence of an $(n, t)$-realization of~$R_1$ with~$e^{\tr} p_1 = \tau_1$. By rotating the realization,~$e$ can be placed at the north pole. If~$k < n - 1$, the lemma similarly guarantees the existence of an $(n,t)$-realization of~$R_1$ with~$e^{\tr} p_1 = t'$ for all~$t'\in [\tau_1, 1)$ with~$e$ at the north pole. This goes through analogously for~$R_2$. Since the action of the stabilizer of~$e$ in~$\ortho(n)$ is transitive on the set of points~$p$ that have inner product~$t'$ with~$e$ for all~$t' \in [-1, 1]$, this shows that if~$k = n - 1$ or~$l = n - 1$, the corresponding~$S_i$ is
  \[
    S_i = \{\, p \in S^{n-1} : e^{\tr} p = \tau_i\}
  \]
  and if~$k < n - 1$ or~$l < n - 1$, the corresponding~$S_i$ is
  \[
    S_i = \{\, p \in S^{n-1} : e^{\tr} p \in [\tau_i, 1)\, \}.
  \]
  In particular, they are invariant under the stabilizer of~$e$ in~$\ortho(n)$.

  Furthermore, if~$k \leq l \leq n-1$ and~$t > -1/l$, then~$\tau_1 \leq \tau_2$ for fixed~$t$, so that~$S_2 \subseteq S_1$.  It follows that there is~$\xi$ such that
  \[
    \{\, p^{\tr} q : p \in S_1,\ q \in S_2\,\} = [\xi, 1].
  \]

  Note that~$\xi$ is a function of~$k$, $l$, and~$t$.  Given~$n > 2$ and~$k$ and~$l$, it is then enough to find the values of~$t$ for which~$\xi \leq t$.  Let~$q_1 \in S_1$ and~$q_2 \in S_2$ be such that~$\xi = q_1^{\tr} q_2$.  The goal is then to have~$q_1^{\tr} q_2 \leq t$.  The following simple fact will be useful:
  \begin{equation}%
    \label{ass:MaximumDistanceOnUnitCircle}
    \assert{If~$S_1$ and~$S_2$ are arcs of the unit circle~$S^1$ such that the infimum $\inf{\{\, x^{\tr} y : x \in S_1, y \in S_2\, \}}$ is attained, then the infimum is attained by an antipodal pair or by endpoints of the arcs.}
  \end{equation}
  By Lemma~\ref{lem:MaxDistanceWhenOpposite} it can be assumed that~$q_1$,~$q_2$, and~$e$ all lie on the same great circle~$C$.  By~\eqref{ass:MaximumDistanceOnUnitCircle}, either the~$q_i$ are endpoints of~$S_i \cap C$ or they are antipodal.

  If the~$q_i$ are endpoints, then~$e^{\tr} q_i = \tau_i$.  Using Lemma~\ref{lem:MaxDistanceWhenOpposite} again gives
  \[
    E = q_1^{\tr} q_2 = \tau_1 \tau_2 - \sqrt{1 - \tau_{\smash{1}}^{\smash{2}}} \sqrt{1 - \smash{\tau_2}^{\smash{2}}}.
  \]
  Hence, in this case the spindle is $(n, t)$-realizable if and only if~$E \leq t$.

  The~$q_i$ are antipodal only if~$\tau_1 \leq -\tau_2$.  In this case, the spindle is $(n, t)$-realizable.  This gives necessary and sufficient conditions for realizability in the~$n > 2$ case.

  If~$n = 2$, then~$k = l = 1$.  The sets~$S_i$ then each contain only two choices for~$q_i$ such that~$e^{\tr} q_i = \tau_i$. A necessary and sufficient condition for realizability is then that~$q_1^{\tr} q_2 = t$.

  To summarize, necessary and sufficient conditions for $(n, t)$-realizability of the $(k,l)$-spindle are:
  \begin{enumerate}
    \item[(i)] $E = t$\quad if~$n = 2$;

    \item[(ii)] $E \leq t$\quad if~$n > 2$ and~$k = l = n - 1$;

    \item[(iii)] $E \leq t$\quad or\quad $\tau_1 \leq -\tau_2$\quad otherwise.
  \end{enumerate}

  Recall from~\eqref{eq:tau-formula} that the~$\tau_i$ are functions of~$k$, $l$, and~$t$, and hence so is~$E$.  The goal is now to determine, for each case above, the values of~$t$ for which the conditions hold.

  In most of the cases below, one has~$k = l$.  Then~$\tau_1 = \tau_2 \eqqcolon \tau$, and so
  \[
    E  = 2\tau^2 - 1.
  \]
  Plug~\eqref{eq:tau-formula} into the right-hand side above to see that~$E \leq t$ if and only if
  \begin{equation}
    8k^2t^3 - (k^2 -10k + 1)t^2 - 2(k - 1)t - 1
    \label{eq:MoserSpindlePolynomial}
  \end{equation}
  is nonnegative, with equality when~$t$ is a root of the polynomial.  In what follows, this and other polynomials considered are seen as polynomials on~$t$ only, that is,~$k$ is fixed.
  \medbreak

  \noindent
  {\sc Case}~(i).  If~$n = 2$, then~$k = l = 1$, and there are only two values of~$t$ for which~$\MS{1}$ is realizable. To see this, factor the polynomial~\eqref{eq:MoserSpindlePolynomial} as
  \[
    8t^3 + 8 t^2 - 1 = (2t + 1) (4t^2 + 2t -1).
  \]
  For the root~$t = -1 / 2$, the points~$p_1$ and~$p_2$ coincide with other points in the spindle. The other roots are~$t = -(1/4)(1 \pm \sqrt{5})$.  These inner products correspond to the pentagon and pentagram.  This gives~\eqref{eq:spindle-a}.
  \medbreak

  \noindent
  {\sc Case}~(ii). If~$n > 2$ and~$k = l = n - 1$, then~(ii) is satisfied if and only if the polynomial~\eqref{eq:MoserSpindlePolynomial} has a nonnegative value at~$t$.  Its discriminant is positive, so it only has real roots. Denote them by~$t_{k,1} \leq t_{k,2} \leq t_{k,3}$.  The constant and linear terms are negative, so~$t_{k,1} \leq t_{k,2} < 0 < t_{k,3}$. At~$t = 0$ the polynomial is negative, thus the polynomial must be nonnegative for~$t_{k,1} \leq t \leq t_{k,2}$ and~$t \geq t_{k,3}$. So~$\MS{n - 1}$ is realizable if and only if~$t_{k,1} \leq t \leq t_{k,2}$ or~$t \geq t_{k,3}$.  This establishes~\eqref{eq:spindle-b}.
  \medbreak

  \noindent
  {\sc Case}~(iii).  It remains to consider~$n > 2$ and~$l < n - 1$.  The discussion splits into two cases: (a).~$k = l$ and (b).~$k < l$.
  \medbreak

  \noindent
  {\sl Case (a).} If~$k = l < n - 1$, either one of the conditions in~(iii) has to be satisfied. The first one is again equivalent to finding~$t$ such that the polynomial~\eqref{eq:MoserSpindlePolynomial} is nonnegative, and so a sufficient condition for realizability is~$t_{k,1} \leq t \leq t_{k,2}$ or $t \geq t_{k,3}$.

  The second condition is~$\tau_1 \leq -\tau_2$.  Since~$\tau_1 = \tau_2 \eqqcolon \tau$ one has~$\tau \leq 0$.  From~\eqref{eq:tau-formula}, this happens if and only if~$g = 2kt^2 - (k - 1)t - 1 \leq 0$. This polynomial has a positive and a negative root and is negative at~$0$. At both roots,~\eqref{eq:MoserSpindlePolynomial} is positive. This can be seen by taking the remainder of~\eqref{eq:MoserSpindlePolynomial} after division by~$g$, and testing it at a convenient value smaller than the smallest root of~$g$ (for example~$t = -1/k$), since the remainder is linear and increasing in~$t$. So~$\MS{k}$ with~$k < n - 1$ is realizable if and only if~$t_{k,1} \leq t$.  This establishes~\eqref{eq:spindle-c}.
  \medbreak

  \noindent
  {\sl Case (b).} The final case is~$n > 2$ and~$k < l \leq n - 1$.  We will see later that it suffices to consider the case~$l = k + 1$.

  So assume~$l = k + 1$.  Let
  \[
    f = 8k^2 \left( k - 1 \right) t^4 - \left( k^3 - 19k^2 + 8k + 4 \right) t^3 - k \left( 3k - 14 \right) t^2 - 3 \left( k - 1 \right) t - 1.
  \]
  The inequality~$E \leq t$ is satisfied if and only if~$f \geq 0$.

  If~$k=1$, then~$f$ is of degree~$3$.  Computing its roots, one gets conditions for the inequality above to be satisfied, obtaining a set of values of~$t$ for which the spindle is realizable.  Similarly, the condition~$\tau_1 \leq -\tau_2$ is satisfied if and only if~$t^3 + 3t^2 - t - 1 \leq 0$.  This gives another set of values of~$t$ for which the spindle is realizable.  Taking the union of both sets, one gets the condition~$t > -1/2 = -1/l$ for realizability.

  If~$k > 1$, then~$f$ has degree~$4$ and its discriminant is negative, so it has exactly two real roots~$f_1 \leq f_2$. At~$t = 0$ it is negative and at~$t = 1$ and at~$t = - 1 / (k + 1)$ it is positive, hence~$-1/(k+1) < f_1 < 0 < f_2 < 1$ and~$f \geq 0$ for~$-1/(k + 1) < t \leq f_1$ and~$f_2 \leq t < 1$.

  The condition~$\tau_1 \leq -\tau_2$ is equivalent to the condition
  \[
    g = \left(2 k^2-1\right) t^3 - \left(k^2 - 3k - 1\right) t^2 - (2k - 1)t - 1 \leq 0.
  \]
  By an analysis similar as before, this polynomial has three real roots given as~$g_1 \leq g_2 < 0 < g_3$. It is negative at~$0$, so it is nonpositive for all~$t$ such that~$t \leq g_1$ or~$g_2 \leq t \leq g_3$. The next objective is to show~$g_2 \leq f_1 \leq f_2 \leq g_3$, so that the result follows; see Figure~\ref{fig:order-of-roots}.
  \begin{figure}[t]
    \centering
    \includegraphics[width=\textwidth]{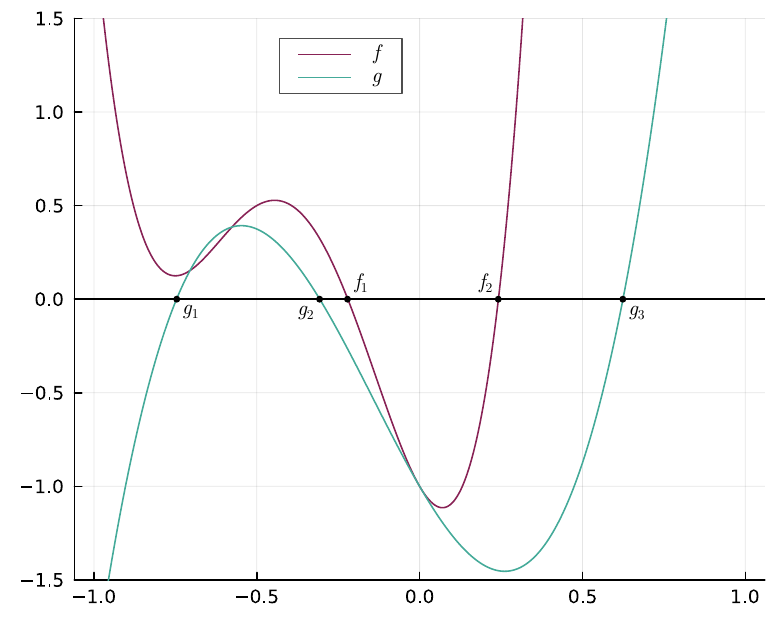}
    \caption{A plot of~$f$ and~$g$ for~$k = 2$. The horizontal axis is the inner product~$t$. Clearly~$-1/2 < g_2 \leq f_1 < 0 < f_2 \leq g_3$. If~$n > 2$, then~$\spindle{1}{2}$ is~$(n, t)$-realizable if and only if~$t > -1/2$, and~$f \geq 0$ or~$g \leq 0$. This plot shows that it is~$(n, t)$-realizable if and only if~$t > -1/2$.}
    \label{fig:order-of-roots}
  \end{figure}

  To determine the order of the roots~$f_1$,~$f_2$,~$g_1$,~$g_2$, and~$g_3$, take the remainder~$r$ of~$f$ after division by~$g$.  The remainder has degree~2 and has two real roots; denote the roots of~$r$ by~$r_1 \leq r_2$. Then~$f$ is nonnegative at a~$g_i$ if and only if~$r$ is. Both roots of~$r$ are negative for any~$k \geq 2$. Moreover,~$g$ is positive at~$r_1$ and~$r_2$, so they lie between~$g_1$ and~$g_2$. The coefficient of the quadratic term of~$r$ is positive, so it has a global minimum, meaning it is positive for all~$t > r_2 > g_1$, so~$f$ is positive at~$g_2$ and~$g_3$. This determines the order of the roots~$g_2 \leq f_1 < 0 < f_2 \leq g_3$. The spindle is realizable if~$-1/(k + 1) < t \leq f_1$,~$f_2 \leq t \leq 1$ and~$g_2 \leq t \leq g_3$, so putting all of this together,~$\spindle{k}{k+1}$ is realizable if and only if~$-1/(k + 1) < t < 1$.

  From~$l = k + 1$ all other cases follow. Indeed,~$\spindle{k}{l}$ with~$k < l - 1$ is a subgraph of~$\spindle{l - 1}{l}$, and so a sufficient condition for realizability is~$t > - 1/l$, which was already seen to be necessary. This settles~\eqref{eq:spindle-d}.
\end{proof}

\subsection*{Some results on non-realizability}

To classify almost-equiangular sets in low dimension, it is necessary to show that given anti-triangle-free graphs are not $(n, t)$-realizable for certain~$n$ and~$t$.

\subsubsection*{The extended rhombus}

Let~$t = - 1/ n$ and take two~$(n - 1)$-rhombi,~$R_1$ and~$R_2$, that intersect in an induced subgraph~$\Sigma$ isomorphic to~$\comp{n}$ (see Figure~\ref{fig:ExtendedRhombus}). Call this graph an \defi{extended~$(n-1)$-rhombus}\index{extended n 1 rhombus@extended~$(n-1)$-rhombus}. Let~$e$ and~$p$ be the endpoints of the unique nonedge of~$R_1$ with~$p \in \Sigma$. Let~$f$ be the endpoint of the nonedge of~$R_2$ not contained in~$\Sigma$.

If~$t = -1/n$ and~$k =  n - 1$, then by Lemma~\ref{lem:kRhombusRealizable},~$e^{\tr} p = -1 / n$ in any realization of~$R_1$. So a realization of~$R_1$ actually forms an~$n$-simplex, and analogously the same holds for~$R_2$. But then~$e$ and~$f$ are uniquely determined by~$\Sigma$, and must coincide, hence the extended $(n-1)$-rhombus is not $(n, -1/n)$-realizable.

\begin{figure}[t]
  \centering
  \includegraphics{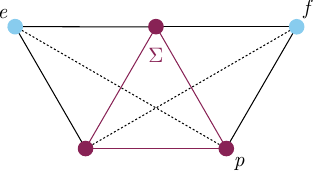}\qquad
  \includegraphics{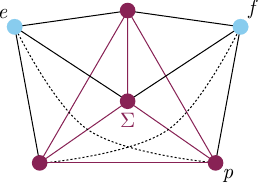}
  \caption{The extended~2-rhombus on the left and the extended~3-rhombus on the right. The dotted lines are edges that follow from Lemma~\ref{lem:kRhombusRealizable}, forcing~$f$ to coincide with~$e$.}
  \label{fig:ExtendedRhombus}
\end{figure}

\subsubsection*{The complement of the split $k$-cycle}

\begin{figure}[t]
  \centerfloat
  \includegraphics[width=0.5\textwidth]{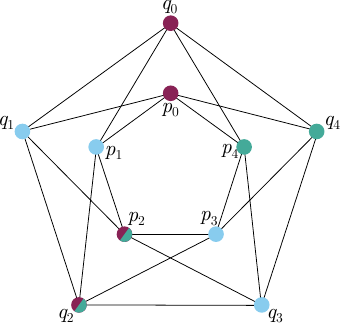}\qquad
  \includegraphics[width=0.5\textwidth]{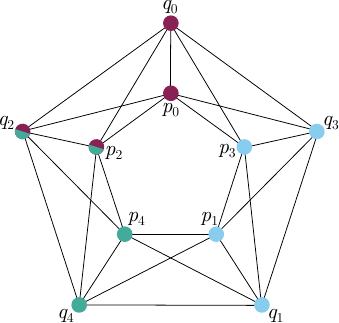}
  \caption{On the left the graph~$W_5$ with~$\Sigma_0$ and~$\Sigma_2$ indicated by color. In this case,~$T_2 = \{p_2, q_2\}$. On the right~$\overline{W_5}$ after rearranging the vertices, with~$\Sigma_0$ and~$\Sigma_2$ colored as well.  The similarity between the two graphs is incidental for~$k = 5$.}
  \label{fig:W5}
\end{figure}

Let~$k \geq 4$. The \defi{split $k$-cycle}\index{split k cycle@split $k$-cycle} is the graph~$W_k$\index{ Wk@$W_k$} on vertices~$p_0,\ldots, p_{k-1}, q_0,\ldots, q_{k-1}$ in which the neighborhood of both~$p_i$ and~$q_i$ is~$\{p_{i-1},q_{i-1},p_{i+1},q_{i+1}\}$ with all indices modulo~$k$ (see Figure~\ref{fig:W5}). It is obtained from a~$k$-cycle by splitting each vertex. Deaett proved~\cite[Theorem 4.11]{Deaett2011TheGraph} that the graph~$\overline{W_n}$, the complement of~$W_n$, is~$(n,0)$-realizable.

For even~$k$, the graph~$W_k$ is bipartite with parts of size~$k$, since the set of all even-indexed points is independent and so is its complement. This means that~$\overline{W_k}$ is a union of two~$(k - 1)$-simplices with some extra edges and therefore does not give a new construction.

For~$k = 5$, the graph~$\overline{W_5}$ is~$(5,0)$-realizable (see Figure~\ref{fig:W5}).  It is the smallest example of an optimal~$(n, 0)$-realizable anti-triangle-free graph that is not a union of two~$(n - 1)$-simplices~\cite{Deaett2011TheGraph}. Balko, Pór, Scheucher, Swanepoel, and Valtr showed~\cite[Theorem 2]{Balko2020Almost-EquidistantSets} that~$\overline{W_5}$ cannot be embedded in~$\R^3$ such that adjacent vertices are at distance~$1$.  Since there are $(4, -1/4)$-realizable graphs of order~$10$, a priori~$\overline{W_5}$ could be $(4, -1 / 4)$-realizable. It turns out, however, that~$\overline{W_k}$ with odd~$k \geq 5$ is not $(k-1, t)$-realizable for any negative~$t$.

Indeed, take~$W_{k}$ with odd~$k \geq 5$. The optimization bound (Theorem~\ref{thm:interpolation}) shows that if~$n < k$ and~$t \in [-1, 0]$, then the maximum cardinality of a $t$-almost-equiangular set on~$S^{n-1}$ is~$\leq 2(n+1)$, with equality only at~$t = -1/n$. Since~$\overline{W_k}$ has order~$2k \geq 2(n+1)$, it can only be $(n,t)$-realizable for~$n < k$ when~$n = k - 1$ and~$t = -1/n$.

Hence, the goal is to show~$\overline{W_k}$ is not $(n, -1/n)$-realizable with~$n = k - 1$.  So assume that~$\overline{W_k}$ is realizable.

In what follows, indices are taken modulo~$k$.  Let~$\Sigma_i$ be the set of all vertices~$p_{i+2j}$ and~$q_{i+2j}$ for~$0 \leq j \leq (k-3)/2$ and set~$T_i= \Sigma_{i-2} \cap \Sigma_i$ (see Figure~\ref{fig:W5}).

The~$\Sigma_i$ are independent in~$W_k$ and so form~$(k - 2)$-simplices in a realization of~$\overline{W_k}$. Take the sets~$\Sigma_0$ and~$\Sigma_2$. Then~$\Sigma_0 \setminus T_2$ and~$\Sigma_2 \setminus T_2$ both consist of two points that lie in the intersection of hyperplanes defined by the equations~$l^{\tr} x = t$ for all~$l \in T_2$. The realization of~$T_2$ is a~$(k - 4)$-simplex, so by Section~\ref{sec:realizability},~$T_2$ consists of~$k - 3$ linearly independent vectors and the dimension of the intersection of these hyperplanes is~2. Therefore,~$p_0$,~$q_0$,~$p_{k-1}$, and~$q_{k-1}$ are coplanar and lie on a circle~$C_1$. Repeat this for~$\Sigma_1$ and~$\Sigma_3$ to see that~$p_0$,~$q_0$,~$p_1$, and~$q_1$ are also coplanar and lie on a circle~$C_2$.

Since~$K = \{p_1, q_1, p_{k-1}, q_{k-1}\}$ is a clique in~$\overline{W_k}$, it defines a regular tetrahedron, hence its affine span is 3-dimensional, and the circles~$C_1$ and~$C_2$ are distinct. Denote the circumsphere of~$K$ by~$S$, which is a 2-sphere. The affine span of~$\{p_0, q_0, p_1, q_1, p_{k-1}, q_{k-1}\}$ is also 3-dimensional, since these points lie on two distinct planes intersecting on a line. Then~$p_0$, $q_0 \in \Aff K$. By uniqueness of the circumsphere of a simplex this means~$p_0$ and~$q_0$ also lie on~$S$.

Since~$\Sigma_0$ can be completed to a regular~$(k - 1)$-simplex for~$t = -1/(k - 1)$ by adding a point on~$z \in C_1$, it follows that~$C_1$ is a circumcircle of a regular triangle on~$S$ whose vertices are~$z$, $p_{k-1}$, and~$q_{k-1}$. However, there are only two such regular triangles on~$S$, namely~$\{p_{k-1}, q_{k-1}, p_1\}$ and~$\{p_{k-1}, q_{k-1}, q_1\}$. So~$C_1$ contains~$p_0$,~$q_0$,~$p_{k-1}$,~$q_{k-1}$ and either~$p_1$ or~$q_1$. By a similar argument,~$C_2$ contains~$p_0$,~$q_0$,~$p_1$,~$q_1$ and either~$p_{k - 1}$ or~$q_{k - 1}$. Then~$C_1$ intersects~$C_2$ in at least four points, a contradiction.

\section{Maximum obtuse almost-equiangular sets}
\label{sec:properties-and-uniqueness}

Theorem~\ref{thm:maximal-almost-equiangular-sets} below establishes that~$\alpha(n, t) \leq 2(n+1)$ for all~$t \leq 0$, with equality only for~$t = -1/n$.  This motivates calling a $(-1/n)$-almost-equiangular set with~$2(n+1)$ points a \defi{maximum obtuse almost-equiangular set}\index{almost equiangular set@almost-equiangular set!maximum obtuse}.

The proof of Theorem~\ref{thm:maximal-almost-equiangular-sets} follows a spectral analysis of matrices associated to the Gram matrix of such a set, done by Rosenfeld~\cite{Rosenfeld1991AlmostEd} and Bezdek and Lángi~\cite{Bezdek1999AlmostSd-1}. Further analysis of these matrices gives useful properties of maximum obtuse almost-equiangular sets; they turn out to be spherical $2$-designs, and are in bijection with certain symmetric orthogonal matrices.

Finally, this leads to a proof that the only maximum obtuse almost-equiangular set is the double-regular $n$-simplex for~$n = 2$, \dots,~$5$.

\subsection*{The spectral analysis}

Bezdek and Lángi prove in~\cite{Bezdek1999AlmostSd-1} that a $t$-almost-equiangular subset of~$S^{n-1}$ with~$t \leq 0$ cannot have more that~$2(n+1)$ points by analyzing the eigenvalues of a certain matrix related to the  Gram matrix of the set. Their method is revisited here to strengthen their result as
follows.

\begin{theorem}\label{thm:maximal-almost-equiangular-sets}
  If~$t \in [-1, 0]$, then~$\alpha(n, t) \leq 2(n+1)$, with equality only at~$t = -1/n$. The Gram matrix of a maximum obtuse almost-equiangular set has rank~$n$, its only nonzero eigenvalue is $2(1+1/n)$, and the all-ones vector~$e$ is in its kernel. In particular, the barycenter of a maximum obtuse almost-equiangular set is~$0$.
\end{theorem}

\begin{proof}
  Following~\cite{Bezdek1999AlmostSd-1}, let $U$ be the Gram matrix of a $t$-almost-equiangular subset of $S^{n-1}$ of cardinality $N$,  let
  $C=U-t J$,  and $B=U-t J-(1-t)I$,  where $J$ is the all-ones matrix and~$I$ is the identity matrix. The diagonal coefficients of~$B$ are~$0$,  hence~$\Tr B = 0$.  The coefficients of~$B$ corresponding to pairs of points with inner product~$t$  are equal to~$0$, hence the set being almost equiangular translates to $B_{ij}B_{jk}B_{ki}=0$ for all $1\leq i,j,k\leq N$, whence~$\Tr(B^3)=0$.

  These two properties give rise to equations for the eigenvalues of $B$.  Because $\rank C \leq n+1$,  the matrix $B$ has at least  $N-(n+1)$ eigenvalues equal to $-(1-t)$.  If $\lambda_1$, \dots,~$\lambda_{n+1}$ denote the remaining ones, then
  \[
    \sum_{i=1}^{n+1} \lambda_i = (N-n-1)(1-t)\quad\text{and}\quad
    \sum_{i=1}^ {n+1} \lambda_i^3 =(N-n-1)(1-t)^3.
  \]

  Since~$t < 0$, the matrix~$C$ is positive semidefinite, and so the smallest eigenvalue of~$B$ is~$-(1-t)$.  Hence, if $y_i=\lambda_i/(1-t)$, then~$y_i \geq -1 > -\sqrt{3}$ and the problem
  \begin{equation}%
    \label{opt:eigs}
    \begin{optprob}
      z^* = \max&\sum_{i=1}^{n+1} y_i\\
      &\sum_{i=1}^{n+1} y_i - y_i^3 = 0,\\
      &y_i \geq -\sqrt{3}\quad\text{for~$i = 1$, \dots,~$n + 1$}
    \end{optprob}
  \end{equation}
  gives an upper bound for~$N - n - 1$.

  Let
  \[
    L(y) = \sum_{i=1}^{n+1} y_i + (1/2)\biggl(\sum_{i=1}^{n+1} y_i - y_i^3\biggr)
    = \sum_{i=1}^{n+1} (3/2) y_i - (1/2) y_i^3
  \]
  and
  \begin{equation}%
    \label{opt:lag}
    d^* = \max\{\, L(y) : \text{$y_i \geq -\sqrt{3}$ for all~$i$}\,\}.
  \end{equation}
  If~\eqref{opt:lag} has an optimal solution~$y^*$ that is feasible for~\eqref{opt:eigs}, then it is also optimal for~\eqref{opt:eigs}. Conversely, if~$z^* = d^*$ and~$y^*$ is optimal for~\eqref{opt:eigs} then it is optimal for~\eqref{opt:lag}.

  In an optimal solution of~\eqref{opt:lag} all the~$y_i$ have the same value, namely
  \[
    \max\{\, p(y) : y \geq -\sqrt{3}\, \},
  \]
  where~$p(y) = (3/2)y - (1/2)y^3$.  A boundary and critical point analysis on~$p$ shows it has a unique maximum for~$y \geq -\sqrt{3}$ given by~$p(1) = 1$.

  Therefore, the problem~\eqref{opt:lag} has a unique optimal solution~$y^*$ with~$y^*_i = 1$ for all~$i$, and its optimal value is~$n + 1$. Since~$y^*$ is also feasible for~\eqref{opt:eigs}, it is its unique optimal solution with optimal value~$n + 1$. So, $N\leq 2(n+1)$, and equality holds if and only if the matrix $B$ has exactly  $n+1$ eigenvalues equal to $1-t$ and $n+1$ eigenvalues equal to $-(1-t)$.  It then follows that if~$N = 2(n+1)$, then~$C$ has exactly one nonzero eigenvalue, namely~$2(1-t)$ with multiplicity~$n+1$.

  Assume that~$N = 2(n+1)$, so the set attains the maximum cardinality. Then, the all-ones vector~$e$ is in the kernel of~$U$, and~$U$ has rank~$n$. Indeed,~$\rank U \leq n < n+1 = \rank C$, and since~$C = U - tJ$ it follows that~$e$ is not in the column space of~$U$, so~$e$ is in the column space of~$C$. The column space~$E$ of~$C$ is the eigenspace of~$C$ with eigenvalue~$2(1-t)$. Let~$S \subseteq E$ be the orthogonal complement to the span of~$e$ in~$E$. If~$x \in S$, then
  \[
    Ux = Cx + tJx = Cx = 2(1-t)x,
  \]
  hence~$x$ is an eigenvector of~$U$. Since~$\rank U < \rank C$ it follows that~$S$ is the only eigenspace of~$U$ with nonzero eigenvalue. Hence,~$Ue = 0$ and~$U$ has rank~$n$.

  The equation~$Ue = 0$ means that the barycenter of the set is~$0$. Moreover, from~$0 = Ue = (C + tJ)e = (2(1-t) + 2(n+1)t)e$ it follows that~$t = -1/n$.
\end{proof}

By the continuous dependence of eigenvalues on the coordinates of a matrix, the bound~$\alpha(n, t) \leq 2(n+1)$ can be extended to~$[-1, \varepsilon(n))$, where~$\varepsilon(n)$ is some (small) positive number depending on~$n$, something Bezdek and Lángi already showed. However, it is not true that this bound is global on~$t \in [-1,1)$, as a construction of Larman and Rogers~\cite{Larman1972TheSpace} shows. Namely, let~$n=5$ and~$S$ be the set of vertices of the cube~$[-1, 1]^5$ that have an odd number of positive signs. Then~$|S| = 16$ with vectors of norm~$\sqrt{5}$.  Rescaling by~$\sqrt{5}$ gives a $(1/5)$-almost-equiangular set on~$S^4$ of cardinality~$16$.

The proof of Theorem~\ref{thm:maximal-almost-equiangular-sets} moreover links the maximum obtuse almost-equiangular sets to the theory of spherical designs; see the survey by Bannai and Bannai~\cite{Bannai2009ASpheres} for more on spherical designs.

\begin{proof}[Proof of Theorem~\textup{\ref{cor:optimal-construction-are-2-designs}}]
  According to Theorem~\ref{thm:maximal-almost-equiangular-sets}, $\sum_{i=1}^{2(n+1)} x_i=0$, the Gram matrix~$U$ of~$S$ satisfies $U^2=2(1+1/n)U$, and~$U$ has rank~$n$.  Moreover, the identity $U^2=2(1+1/n)U$ translates to
  \begin{equation*}
    \sum_{k=1}^{2(n+1)}  (x_i^{\tr} x_k)(x_k^{\tr} x_j) =2(1+1/n)(x_i^{\tr} x_j) \quad\text{for all }1\leq i,j\leq 2(n+1).
  \end{equation*}
  By linearity, $x_i$ and $x_j$ can be replaced by any vector of $\R^n$. In particular, for all $u\in S^{n-1}$,
  \begin{equation*}
    \sum_{k=1}^{2(n+1)}  (u^{\tr} x_k)^2 =2(1+1/n).
  \end{equation*}
  This identity, together with $\sum_{i=1}^{2(n+1)}x_i=0$, characterizes the spherical designs of strength $2$.  For a proof of the latter, see~\cite[Theorem~2.2]{Bannai2009ASpheres}, but note that in property~(6) of this theorem the first appearance of the exponent~$k$ is wrong and should be~$2k$.
\end{proof}

\subsection*{Relation to orthogonal matrices}

The union of two vertex-disjoint regular $n$-simplices, called a double regular
$n$-simplex, is an example of a maximum obtuse almost-equiangular set.  A
natural question is whether this construction is unique.  The affirmative answer
for~$n \leq 5$ is established in Theorem~\ref{thm:double-simplex-uniqueness}.
Theorem~\ref{thm:O-matrix} works towards this proof, and is interesting by
itself.

\begin{proof}[Proof of Theorem~\textup{\ref{thm:O-matrix}}]
  With similar notation as in the proof of Theorem~\ref{thm:maximal-almost-equiangular-sets}, let~$B$ denote the matrix associated to a maximum obtuse almost-equiangular set of unit vectors.  The matrix~$B$ has only two eigenvalues, namely~$\pm (1+1/n)$, and hence satisfies $B^2=(1+1/n)^2I$.  Moreover, $Be=(1+1/n)e$.  Let $O = (1+1/n)^{-1}B$; it is clear from the properties of $B$ that~$O$ is symmetric and orthogonal and that it satisfies the conditions~(i)--(iii).

  Conversely, given a symmetric and orthogonal matrix~$O$ satisfying (i)--(iii), let
  \[
    U = (1+1/n)O-(1/n)J+(1+1/n)I
  \]
  and let $E_{\pm1}$ be the eigenspaces of $O$ associated with the two eigenvalues $\pm 1$. Both of them have dimension $n+1$ because $\Tr(O)=0$ due to (ii).
  The kernel of $U$ is the subspace $E_{-1}\oplus \R e$ of dimension $n+2$; its orthogonal complement is the eigenspace of $U$ associated to the eigenvalue $2(1+1/n)$. So $U$ is the Gram matrix of a set of $2(n+1)$ unit vectors in $\R^n$. Condition (iii) ensures that this set is $(-1/n)$-almost equiangular.
\end{proof}

Any $t$-almost-equidistant set in~$S^{n-1}$ with~$t \leq 0$ can be lifted to an almost-orthogonal set on~$S^n$~\cite{Polyanskii2017OnII}. Since~$\alpha(n+1,0) = 2(n+1)$, every maximum obtuse almost-equidistant set gives a maximum almost-orthogonal set in this way. Since~$\overline{W_5}$ is $(5,0)$-realizable but not $(4,-1/4)$-realizable (see the last subsection of Section~\ref{sec:realizability}), the converse is not the case. Deaett characterized the maximum almost-orthogonal sets by a statement similar to Theorem~\ref{thm:O-matrix}; it differs only by the eigenvector condition~(i). Hence, the eigenvector condition distinguishes between those maximum almost-orthogonal sets that show this form of descent, and those that do not.

\subsection*{The distance graph of a maximum obtuse almost-equiangular set}

A graph is \defi{quadrangular}\index{quadrangular graph} if no two vertices have exactly one neighbor in common.

\begin{lemma}\label{lem:props-maximum-sets}
  The following properties hold for the distance graph $G$ of a maximum obtuse almost-equiangular subset~$S$ of $S^{n-1}$.
  \begin{enumerate}
    \item[(i)] If $G$ contains a $\comp{n+1}$, then~$S$ is a double-regular $n$-simplex.

    \item[(ii)] The graph~$\overline{G}$ is quadrangular.

    \item[(iii)] The degree of a vertex in~$\overline{G}$ lies between $1$ and $n+1$.
      If there is a vertex of degree~$n+1$ in~$\overline{G}$, then~$S$ is a double-regular $n$-simplex.
      If there is a vertex~$x$ with exactly one neighbor~$y$ in~$\overline{G}$, then~$G[S \setminus \{x, y\}]$ is the distance graph of a maximum obtuse almost-equiangular subset of~$S^{n-2}$.
  \end{enumerate}
\end{lemma}

\begin{proof}
  Let $O=(n / (n+1))U+(1 / (n+1))J-I$, where~$U$ is the Gram matrix of~$S$, be the matrix of Theorem~\ref{thm:O-matrix}.  The entries of~$O$ are equal to~$0$ on the diagonal and at pairs of vectors with inner product~$-1/n$,  so the adjacency matrix~$A$ of~$\overline{G}$ is such that~$A_{ij} = 0$ if~$O_{ij} = 0$ and~$A_{ij} = 1$ if~$O_{ij} \neq 0$.

  If~$G$ contains a~$K_{n+1}$, then~$O$ is of the form~$O = \smallpmatrix{ 0& B\\B^{\tr} & D}$ where~$B$ and~$D$ are $(n+1)\times (n+1)$ matrices.  The condition~$O^2=I$ leads to~$BB^{\tr}=I$ and~$BD=0$. But then~$B$ is invertible and so~$D=0$, which proves~(i).

  Property~(ii) follows from the columns of~$O$ being pairwise orthogonal: if two vertices~$x_i$, $x_j$ share a single neighbor~$x_k$ in~$\overline{G}$,  then~$O_{ki}O_{kj} \neq 0$,  while~$O_{li}O_{lj} = 0$ for~$l \neq k$.  But then the columns~$i$ and~$j$ of~$O$ would not be orthogonal.

  To prove~(iii), note that~$\overline{G}$ is triangle free.  Let~$x$ be a vertex and let~$\overline{N}_x$ denote its set of neighbors in~$\overline{G}$.  Two vertices in~$\overline{N}_x$ cannot be adjacent in~$\overline{G}$, otherwise they would form a triangle with~$x$.  So~$\overline{N}_x$ is a clique in~$G$,  that is, it is a regular simplex,  which proves that the degree of~$x$ in~$\overline{G}$ is at most~$n+1$.  Moreover, if~$x$ has degree~$n+1$, then it follows from~(i) that~$G$ contains a~$\comp{n+1}$, and hence that~$S$ is a double-regular $n$-simplex.

  Next, given a vertex~$x$, let~$N_x$ be its neighborhood in~$G$.  All vertices in~$N_x$ have inner product~$-1/n$ with~$x$, and so lie in an affine hyperplane, and hence belong to an $(n-2)$-sphere~$C$.  By scaling and translating~$C$ via an affine transformation, it can be mapped to~$S^{n-2}$, and then~$N_x$ is mapped to a $t$-almost-equidistant set for some~$t \leq 0$.  It then follows from Theorem~\ref{thm:maximal-almost-equiangular-sets} that~$|N_x| \leq 2n$, and so the degree of~$x$ in~$\overline{G}$ is at least~$1$.  Moreover, if~$|N_x| = 2n$, then~$t = -1/(n-1)$.
\end{proof}

\subsection*{Uniqueness of the double-regular simplex.}
The goal in this section is to prove Theorem~\ref{thm:double-simplex-uniqueness}. For a given dimension~$n$, the theorem is false if there is an $(n, -1/n)$-realizable anti-triangle-free graph of order~$2(n+1)$ whose complement is not bipartite.  It turns out that, to prove the theorem, it is enough to show that such a graph whose complement contains a $5$-cycle is not realizable.

\begin{proof}[Proof of Theorem~\textup{\ref{thm:double-simplex-uniqueness}}]
  The distance graph of any maximum obtuse almost-equiangu\-lar set is anti-triangle free and, by Lemma~\ref{lem:props-maximum-sets}, has a quadrangular complement.  Moreover, if the set is not a double-regular $n$-simplex, then the complement is not bipartite.  The goal of the proof is then to show that, if~$G$ is an anti-triangle-free graph of order~$2(n+1)$ whose complement is quadrangular and nonbipartite, then~$G$ is not $(n, -1/n)$-realizable. For~$2 \leq n \leq 5$, this is done below.

  Let~$G$ be an anti-triangle-free graph of order~$2(n+1)$ whose complement is quadrangular.  Say that~$\overline{G}$ does not contain odd cycles of length~$3$, $5$, \dots,~$2k-1$, but contains an odd cycle of length~$2k + 1$ with vertices~$p_0$, \dots,~$p_{2k}$.  Since~$\overline{G}$ is quadrangular, every pair of vertices~$p_i$, $p_{i+2}$, with indices taken modulo~$2k+1$, has at least two common neighbors.  One of the neighbors is~$p_{i+1}$; denote the other by~$q_{i+1}$.  Since~$\overline{G}$ does not contain odd cycles of length less than~$2k+1$, the vertices~$p_i$ and~$q_i$ must all be distinct, and so the order of~$G$ is at least~$2(2k+1)$, whence~$n \geq 2k$. This settles the case~$n = 3$.
  \medbreak

  \noindent
  {\sc Dimensions~$4$ and~$5$.} It follows that, for~$n \leq 5$, if~$G$ is an anti-triangle-free graph of order~$2(n+1)$ whose complement is quadrangular and nonbipartite, then~$\overline{G}$ has an odd cycle of length~$5$, and since~$n \geq 2k$ as shown above, it is necessary that~$n \geq 4$.  So it suffices to show that such a graph~$G$ for~$n = 4$ and~$5$ is not $(n, -1/n)$-realizable.

  \begin{figure}
    \centering
    \includegraphics{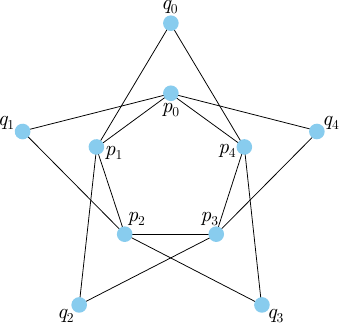}
    \caption{If $\overline{G}$ contains a $5$-cycle,  it contains this subgraph.}
    \label{fig:star}
  \end{figure}

  To this end, note that if~$p_0$, \dots,~$p_4$ is a $5$-cycle in~$\overline{G}$ and if~$q_0$, \dots,~$q_4$ are the common neighbors defined above, then~$\overline{G}$ has the graph in Figure~\ref{fig:star} as a subgraph.  Again since~$\overline{G}$ is quadrangular,  the pairs $p_i$, $q_{i+2}$ must have another common neighbor besides $p_{i+1}$.  If $n=4$,  there are no other vertices available,  so the only possibility is that~$q_0$, \dots,~$q_4$ is a cycle, that is,~$\overline{G}$ is isomorphic to~$W_5$ (see Figure~\ref{fig:W5}).  The graph $\overline{W_5}$ is not $(4,-1/4)$-realizable (see the end of Section~\ref{sec:realizability}), so the proof is finished for~$n=4$.

  The remaining case is~$n = 5$, for which~$G$ has order~$12$.  Call $x$, $y$ the two vertices of~$G$ other than the~$p_i$ and~$q_i$.   By an argument similar to the one above,~$\overline{G}$ contains as a subgraph either~$W_5$, as was the case for~$n = 4$, or, without loss of generality, the graph in Figure~\ref{fig:W5alt}.
  \medbreak

  \begin{figure}
    \centering
    \includegraphics{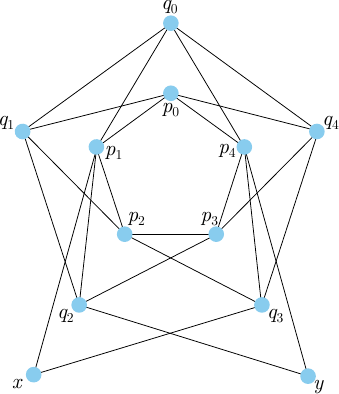}
    \caption{When $n=5$,  the graph~$\overline{G}$ may contain this graph as a subgraph.}
    \label{fig:W5alt}
  \end{figure}

  \noindent
  \textsl{Dimension~$5$ and~$\overline{G}$ contains the graph of Figure~\ref{fig:W5alt}.} If $\overline{G}$ contains the graph of Figure~\ref{fig:W5alt}, then since~$\overline{G}$ is triangle free and~$x$ is adjacent to~$p_1$ and~$q_3$ in~$\overline{G}$, it must be that~$x$ is adjacent to~$p_0$, $q_0$, $p_2$, $q_2$, $p_4$, and~$q_4$ in~$G$.  The same reasoning for~$y$ shows that~$G$ contains as a subgraph the graph~$H_{12}$ from Figure~\ref{fig:H12}.  It will turn out that~$H_{12}$ is not $(5, -1/5)$-realizable.
  \medbreak

  \begin{figure}
    \includegraphics{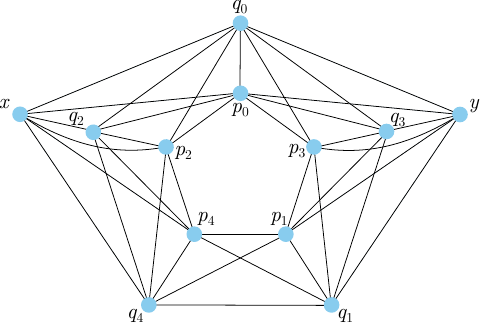}
    \caption{The graph $H_{12}$ is not $(5,-1/5)$-realizable. }
    \label{fig:H12}
  \end{figure}

  \noindent
  \textsl{Dimension~$5$ and~$\overline{G}$ contains $W_5$.} If~$\overline{G}$ contains~$W_5$, then the graph~$G$ contains a subgraph isomorphic to~$H_{12}$ as well.  Indeed, in this case, the vertices~$x$ and~$y$ must be adjacent in~$\overline{G}$ to the subgraph~$W_5$, otherwise by~(iii) of Lemma~\ref{lem:props-maximum-sets} the graph~$\overline{W_5}$ would be the distance graph of a $(-1/4)$-almost-equiangular set in~$S^3$ with~$10$ points that is not a regular double simplex, a contradiction.

  If~$v$ is a vertex of~$W_5$ in~$\overline{G}$, then the neighborhood of~$v$ in~$W_5$ is an independent set, since~$\overline{G}$ is triangle free. The neighborhood forms a clique in~$G$; call it~$C_v$. If~$x$ is adjacent to~$v$ in~$\overline{G}$, again since~$\overline{G}$ is triangle free,~$x$ is adjacent to all vertices of~$C_v$ in~$G$.

  Since~$x$ is adjacent in $\overline{G}$ to at least one vertex~$v$ of $W_5$, without loss of generality say~$x$ is adjacent to~$p_1$. Then,~$x$ is adjacent in~$G$ to~$C_{p_1} = \{p_0, q_0, p_2, q_2\}$. But then without loss of generality~$x$ is adjacent in~$G$ to~$\{p_0,q_0,p_2,q_2,p_4,q_4\}$.  Namely, if~$x$ is not adjacent to any of~$p_3$, $q_3$, $p_4$, and~$q_4$ in~$\overline{G}$, the statement follows immediately. Otherwise, if~$x$ is adjacent, say, to~$p_3$ in~$\overline{G}$, then~$x$ is adjacent in~$G$ to~$C_{p_1} \cup C_{p_3} = \{p_0, q_0, p_2, q_2, p_4, q_4\}$.

  It remains to show that~$y$ is adjacent in~$G$ to all vertices in~$\{p_1,q_1,p_3,q_3\}$; applying the previous reasoning to~$y$ shows that if this is the case,~$y$ is adjacent to all vertices in either~$\{p_4,q_4,p_1,q_1,p_3,q_3\}$ or~$\{p_1,q_1,p_3,q_3, p_0,q_0\}$,  meaning that a subgraph isomorphic to~$H_{12}$ occurs in~$G$.

  To prove that~$y$ is adjacent to all vertices in~$\{p_1, q_1, p_3, q_3\}$, consider the following. In order to arrive at a contradiction, assume~$y$ is adjacent to $p_1$ in~$\overline{G}$, again without loss of generality. Then~$y$ is connected in~$G$ to~$C_{p_1} = \{p_0,q_0,p_2,q_2\}$.  But~$x$ is also adjacent in~$G$ to these vertices,  and if~$G$ contains a~$\comp{6}$, then it is not $(5,-1/5)$-realizable ((i) of Lemma~\ref{lem:props-maximum-sets}), so $\{x, y\}$ is independent in~$G$.  Now the contradiction comes from the quadrangularity of~$\overline{G}$; indeed, if~$x$ and~$y$ are not adjacent in~$G$, then~$y$ is a common neighbor of~$x$ and~$p_1$ in~$\overline{G}$. But it is not possible that~$x$ and~$p_1$ have a second common neighbor because $x$ is not connected to any neighbor of~$p_1$ in~$\overline{G}$ other than~$y$.

  To complete the proof, it remains to show that the graph $H_{12}$ is not $(5,-1/5)$-realizable. This is a specialization of a part of the proof of the nonrealizability of~$W_k$ from the end of Section~\ref{sec:realizability}. In fact, the graph~$H_{12}$ is a subgraph of~$\overline{W_7}$, with two vertices and some edges removed. The removed edges play no role in the proof, and the two vertices only play a role for nonrealizability for~$n=6$, but for~$n=5$ they are superfluous.
\end{proof}

\section{Classification in dimensions 2 and 3}
\label{sec:low-dim}

Section~\ref{sec:realizability} gives exact conditions on the dimension~$n$ and inner product~$t$ for which simplices, rhombi, and spindles are $(n, t)$-realizable. For every integer~$m \geq 1$, dimension~$n$, and inner product~$t$, this gives sufficient conditions for the existence of $(n, t)$-realizable anti-triangle-free graphs of order~$m$.  These realizable graphs then give $t$-almost-equiangular sets of cardinality~$m$ in dimension~$n$.  In this section a converse result is obtained in low dimension: list all maximum-cardinality, almost-equiangular sets in~$S^{n-1}$, with~$n \leq 3$.

Say that an anti-triangle-free graph is \defi{minimal}\index{anti triangle free graph@anti-triangle-free graph!minimal} if the removal of any edge results in a graph that is not anti-triangle free.  Given~$n$ and~$t$, say that an anti-triangle-free graph is \defi{$(n, t)$-optimal}\index{anti triangle free graph@anti-triangle-free graph!n t optimal@$(n, t)$-optimal} if it is $(n, t)$-realizable and if it has order~$\alpha(n, t)$.  If a graph is the unique minimal $(n, t)$-optimal graph up to isomorphism, then it is called a \defi{unique optimal construction}\index{unique optimal construction}.

Finding all minimal $(n, t)$-optimal graphs for low dimension~$n$ is done by performing a graph search. The results from Section~\ref{sec:realizability} provide the conditions for this search. They also give lower bounds on~$\alpha(n, t)$. Theorem~\ref{thm:interpolation} and Theorem~\ref{thm:maximal-almost-equiangular-sets} provide an upper bound of~$\alpha(n, t) \leq 2(n+1)$  for~$t \in [-1, 0]$, which is only attained at~$t = -1/n$. There exists a global lower bound of~$\alpha(n, t) \geq 4$, given by the disjoint union of two edges.

Perform the graph search as follows. Let~$t_{k,i}$ be the~$i$th root of the polynomial~\eqref{eq:MoserSpindlePolynomial} for fixed~$k$. Given~$t$ and~$2 \leq k \leq n$, list all graphs~$G$ of a given order that do not contain a subgraph isomorphic to:
\begin{itemize}
  \item an anti-triangle;
  \item $\comp{n + 2}$;
  \item an~$n$-rhombus;
  \item a~$k$-rhombus if~$t \leq -1/k$;
  \item $\comp{k + 1}$ if~$t < -1/k$;
  \item $\comp{n + 1}$ if~$t > -1 / n$;
  \item an extended $(n - 1)$-rhombus if~$t = -1/n$;
  \item $\MS{k}$ if~$t < t_{k, 1}$;
  \item $\MS{n - 1}$ if~$t > t_{n - 1, 2}$.
\end{itemize}

The search is implemented in SageMath in a script in the supplement of~\cite{Bachoc2025ObtuseSets}. The code is a modified version of the code used in~\cite{Balko2020Almost-EquidistantSets}. Given a dimension~$n$, all anti-triangle-free graphs of cardinality at most~$2(n+1)$ not containing a~$K_{n+2}$ are generated. A second script reduces the size of these sets greatly by only taking the minimal anti-triangle-free graphs. Finally, each graph is searched for the above list of subgraphs.  The results below are summarized in Figure~\ref{fig:dim2and3}.
\medbreak

\noindent
{\sc Dimension~$2$.}
The three-point bound~$\kpb_3(H(n,t))$ for~$t\leq 0$ proves a global upper bound~$\alpha(2, t) \leq 6$, which is only achieved at~$t = -1 / 2$ by the double triangle. A graph search on order~$6$ graphs that are anti-triangle-free and do not contain~$\comp{4}$ or a~$2$-rhombus shows that this is the only minimal order~6 construction on the circle.

The Moser Spindle~$\MS{1}$ is realizable for~$t = -(1 / 4)(1 \pm \sqrt{5})$ and has order~5. Its graph is a~5-cycle, which is the unique anti-triangle-free graph of order~5 containing no~$\comp{3}$. The inner product~$t = -(1/4)(1 + \sqrt{5}) = \cos(-4\pi / 5)$ corresponds to the pentagram and~$t = -(1/4)(1-\sqrt{5}) = \cos(-2\pi / 5)$ corresponds to the regular pentagon.

Every other anti-triangle-free graph of order~$5$ satisfying the constraints above contains the disjoint union of a triangle and an edge, which is only realizable at~$t = -1 / 2$. This shows that the~5-cycle is the unique optimal construction at~$t = -(1 / 4)(1 \pm \sqrt{5})$.

At every other inner product the maximum cardinality is~$4$, attained by two disjoint edges, which is the unique optimal construction of this order.
\medbreak

\noindent
{\sc Dimension~$3$.}
The three-point bound~$\kpb_3(H(n,t))$ for~$t \leq 0$ proves a global upper bound~$\alpha(3, t) \leq 8$. This is achieved by the double tetrahedron for~$t = -1 / 3$. The graph search shows that the double tetrahedron is the only minimal construction of order~$8$.

In the region~$t_{2,1} \leq t \leq t_{2,2}$, the Moser spindle~$\MS{2}$ is~$(3, t)$-realizable and of order~7. Excluding this subgraph from the graph search shows that it is the unique optimal construction for~$t_{2,1} \leq t < -1/3$ and~$-1/3 < t \leq t_{2,2}$.

For~$t \geq -1/2$, the double triangle is realizable. For~$t = -1/2$, it is the unique optimal construction. For~$t > -1/2$, the spindle~$\spindle{1}{2}$ is realizable and of order~6. Excluding these subgraphs from the graph search shows there are no other $(3, t)$-realizable graphs of order~$6$ with~$t > -1/2$. So for~$-1 / 2 < t < t_{2,1}$ and~$t_{2,2} < t \leq 0$, there are two optimal constructions of order~6.

For all~$-(1/4)(1 - \sqrt{5}) \leq t < -1/2$, the Moser spindle~$\MS{1}$ is the unique optimal construction and has order~$5$, as described above.

\section*{Acknowledgments}
We thank Nando Leijenhorst for help with \texttt{ClusteredLowRankSolver.jl} and Willem de Muinck Keizer for helpful discussions.  Alexey Glazyrin has pointed us to references~\cite{Bilyk2024OptimizersSets,
Bilyk2023OptimalPotentials}.

\part{Euclidean space}\label{part:Euclidean-space}
\chapter[Distance-avoiding sets]{Distance-avoiding sets}%
\label{ch:distance-avoiding-sets-euclidean-space}
This chapter is part of ongoing work together with Fernando Mário de Oliveira Filho.

\sectionbreak

The compilation \textit{Problems, problems, problems} by Moser~\cite{Moser1991ProblemsProblems} contains the following question---first posed by Moser's brother, Leo Moser:
\begin{quotation}
  \textbf{LM 25} (1966) Estimate the ``size'' of the largest measurable point set in a large square, which does not determine unit distance.
\end{quotation}
This is precisely Problem~\ref{it:problem-3} from the introduction of this thesis.

Denote the Lebesgue measure on~$\R^n$ by~$\lambda$. In integrals, the notation~$dx$\index{ dx@$dx$} always means~$d\lambda(x)$. For a Lebesgue-measurable set~$S \subseteq \R^n$, define the~\defi{upper density}\index{upper density} of~$S$ by
\index{ udS@$\ud(S)$}\[
  \ud(S) = \limsup_{T \to \infty} \frac{\lambda(S \cap [-T/2,T/2]^n)}{T^n}.
\]
Problem~\ref{it:problem-3} can then be formulated as: what is the largest upper density a Lebesgue-measurable subset of~$\R^n$ not containing pairs at distance 1 can have. We denote this number by~$m_1(\R^n)$.

Erd\H{o}s~\cite{ErdosProblemsGeometry} conjectured that~$m_1(\R^n) < 2^{-n}$. This was recently confirmed for~$n = 2$ by Ambrus, Csisz\'arik, Matolcsi, Varga, and Zs\'amboki~\cite{Ambrus2024TheDistances}, who gave the upper bound~$m_1(\R^2) \leq 0.2470$. The bound was partly based on convex optimization techniques introduced by Oliveira and Vallentin~\cite{Oliveira2010FourierRn} and DeCorte, Oliveira, and Vallentin~\cite{DeCorte2022CompleteSets}, using semidefinite programming. The gap between lower and upper bounds on~$m_1(\R^n)$ is currently still quite large---the lower bound being~$m_1(\R^2) \geq 0.22936$, as the tortoise construction by Croft shows~\cite{Croft1967IncidenceIncidents}. The paper~\cite{Ambrus2024TheDistances} contains a more detailed account of the history of~$m_1(\R^n)$.

Recall that a choice of set~$D \subseteq (0,\infty)$ defines the distance graph~$G(D)$ with vertex set~$\R^n$ and edge set~$E(D)$\index{ GD@$G(D)$} with~$xy \in E(D)$ if and only if~$\|x-y\| \in D$. In this chapter, we study the problem of determining
\index{ mD@$m_D(\R^n)$}\[
  m_D(\R^n) = \sup \{\, \ud(S) : S \subseteq \R^n \text{ is measurable and avoids } D\, \}.
\]
A completely positive formulation for the general distance-avoiding-set problem on~$\R^n$ was introduced by DeCorte, Oliveira, and Vallentin~\cite{DeCorte2022CompleteSets}. They prove this bound is exact, and compute new upper bounds on~$m_1(\R^n)$ in low dimensions by including well-chosen constraints coming from the completely positive cone. They also use the full completely positive formulation to reprove a theorem by Bukh~\cite{Bukh2008MeasurableDistances} about sets avoiding many distances. The theorem says that if~$n$,~$m \geq 2$ are integers, and~$d_1$,~$\ldots$,~$d_m$ are positive numbers, then~$m_{\{d_1, \ldots, d_m\}}(\R^n)$ approaches~$m_1(\R^n)^m$ as the ratios~$d_i/d_{i+1}$ go to infinity.

In the next section, we describe a completely positive hierarchy for these problems, based on the outer approximation of~$\CP(\R^n)_{\inv}$ described in Chapter~\ref{ch:complete-positivity-under-symmetry}, and show it converges to the exact optimal upper density. We end with some notes on how we can use the results from Chapter~\ref{ch:Completely positive programming on compact spaces} to say something about other hierarchies for distance-avoiding-set problems on~$\R^n$.

\section{Lattices, tori, and densities}
Preliminaries on harmonic analysis can be found in Section~\ref{ch:complete-positivity-under-symmetry}.\ref{sec:prel-harmonic-analysis} and the appendix.

Fix an integer~$n \geq 2$. Analysis on~$\R^n$ can often be factored through its tori. It suffices to consider tori of the form~$\R^n/L\Z^n$. The necessary preliminaries on lattices are taken from~\cite[Appendix B.2]{Bekka2008KazhdanT}.

To be precise, let~$L >0$ be a real number, then~$L\Z^n$ is a lattice in~$\R^n$. The associated torus is~$\torus_L = \R^n/L\Z^n$\index{ TL@$\torus_L$}, where we omit~$n$ from the notation, as it is always fixed and clear from context. Let~$p_L: \R^n \to \torus_L$ be the quotient map. The torus~$\torus_L$ is itself a compact group under the quotient topology. Each torus~$\torus_L$ is a metric space when equipped with the metric
\index{ dL@$d_L$}\[
  d_L(p_L(x), p_L(y)) = \inf_{v \in L\Z^n} \|x - y + v\|,
\]
which is a right-invariant density metric.

A~\defi{fundamental domain}\index{fundamental domain} of a lattice~$L\Z^n$ is a Borel set~$F$ with respect to which~$\R^n = \bigcup_{v \in L\Z^n}F + v$ and~$(F + v )\cap (F + v') = \emptyset$ for all~$v$ and~$v'\in L\Z^n$ such that~$v \neq v'$. We choose for~$L > 0$ the fundamental domain~$F_L = [-L/2, L/2)^n$. Thus,~$\R^n$ is a discrete union of translates of~$F_L$ for each~$L$. By Equation~\eqref{eqn:quotient-measure-is-invariant-Radon-measure}, there is a Haar measure~$\mu_L$ on~$\torus_L$ such that
\[
  \lambda(F_L) = \int_{\R^n} \1_{F_L}(x)\, dx = \int_{\torus_L} \sum_{v \in L\Z^n}\1_{F_L}(x+v)d\mu_L(p(x)) = \mu_L(\torus_L).
\]
It is common to work with this Haar measure, so we do so as well; thus, in all that follows,~$\mu_L(\torus_L) = L^n$.

A set~$S \subseteq \R^n$ is called~\defi{periodic with period~$L$}\index{periodic set!with period L@with period $L$} if for all~$v \in L\Z^n$,~$S + v = S$. A~\defi{periodic set}\index{periodic set} is a set that is periodic with period~$L$ for some~$L$. Sets periodic with period~$L$ define a subset of~$\torus_L$ by taking the quotient. On the other hand, a subset~$S \subseteq \torus_L$ defines a subset of~$\R^n$ by the section~$s: \torus_L \to F_L$ of the quotient map, and taking~$\bigcup_{v \in LZ^n} s(S) + v$. This correspondence is bijective.

Likewise, a function~$f: \R^n \to \R$ is called~\defi{periodic with period~$L$}\index{periodic function!with period L@with period~$L$} if it is invariant under translation by elements of~$L\Z^n$, that is,~$f(x+v) = f(x)$ for all~$x \in \R^n$ and~$v \in L\Z^n$. Call~$f$~\defi{periodic}\index{periodic function} if it is a periodic function with period~$L$ for some~$L$. A function on~$\torus_L$ defines a periodic function on~$\R^n$ by composition with the quotient map. Vice versa, a periodic function on~$\R^n$ defines a function on~$\torus_L$ by composition with the section~$s: \torus_L \to F_L$.

Let for~$D \subseteq (0,\infty)$ and~$L > 0$ the graph~$G_L(D) = (\torus_L, E(D))$\index{ GLD@$G_L(D)$} be the distance graph given by~$p_L(x)p_L(y) \in E(D)$ if and only if there is a~$v \in L\Z^n$ such that~$\|x - y + v\| \in D$. Then,~$G_L$ is a homogeneous graph on~$\torus_L$ as defined in Chapter~\ref{ch:completely-positive-formulations-meas-ind-num}. To avoid confusion, we denote the measurable independence number of~$G_L(D)$ under~$\mu_L$ by~$\alpha_{\mu_L}(G_L(D))$.

The key insight for dealing with distance-avoiding sets on~$\R^n$ is that they are approximated by periodic distance-avoiding sets. For a measurable periodic set~$S$ with period~$L$ the upper density is~$\ud(S) = L^{-n}\mu_L(S)$, which relates~$m_D(\R^n)$ to the optimal densities~$\alpha_{\mu_L}(G_L(D))$.

A theorem by Furstenberg, Katznelson, and Weiss~\cite[Theorem A]{Furstenberg1990ErgodicDensity} shows that if~$D$ is unbounded,~$m_D(\R^n) = 0$. The following lemma says that if~$D$ is bounded, we may approximate a measurable $D$-avoiding set by periodic measurable $D$-avoiding sets, that is, independent sets of~$G(D)$ are approximated arbitrarily well by independent sets of~$G_L(D)$. A proof was given by DeCorte, Oliveira, and Vallentin~\cite[Lemma 6.2]{DeCorte2022CompleteSets}.
\begin{lemma}%
  \label{lem:approximate-density-by-density-on-tori}
  If~$n \geq 2$ and~$D \subseteq (0, \infty)$ is bounded, then
  \[
    m_D(\R^n) = \limsup_{L \to \infty} \frac{\alpha_{\mu_L}(G_L(D))}{L^n}.
  \]
\end{lemma}

This behavior is reflected by the optimization upper bounds on~$m_D(\R^n)$ we consider. Define the operator
\[
  Mf = \limsup_{T \to \infty} T^{-n} \int_{[-T/2,T/2]^n} f(x)\, dx.
\]
For~$\cC \subseteq L^{\infty}(\R^n)$ a convex cone, let
\index{ theta@$\vartheta(G, \cC)$!Euclidean space}\[
  \begin{optprob}
    \vartheta(G(D), \cC) = \sup &\onerow{Mf}\\
    &\onerow{f(0) = 1}\\
    &f(x) = 0 &\text{if } x \in D\\
    &\onerow{f \in \cC\text{ is of positive type},}
  \end{optprob}
\]
where we use that a function of positive type is in particular continuous.

As usual, denote~$C_r(V) = C_r(V, 2)$ and~$\CP(V) = \CP(V, 2)$. For a homogeneous graph~$G=(V,E)$ with~$V$ compact and a convex cone~$\cC \subseteq C_{\sym}(V)$, we have~$\vartheta_{\st}(G, \cC) = \vartheta_{\bt}(G, \cC)$; denote this program by~$\vartheta(G, \cC)$. DeCorte, Oliveira, and Vallentin showed that, when~$D$ is closed,~$m_D(\R^n)$ has the completely positive formulation~$\vartheta(G(D), \CP(\R^n)_{\inv})$ by proving
\begin{equation}%
  \label{eqn:completely positive-formulation-approximated-by-tori}
  \vartheta(G(D), \CP(\R^n)_{\inv}) = \limsup_{L \to \infty}\frac{\vartheta(G_L(D), \CP(\torus_L)_{\cont})}{L^n},
\end{equation}
for all closed~$D \subseteq (0,\infty)$~\cite[Theorem 6.3]{DeCorte2022CompleteSets}. Together with Lemma~\ref{lem:approximate-density-by-density-on-tori}, Equation~\eqref{eqn:completely positive-formulation-approximated-by-tori} implies~$m_D(\R^n) = \vartheta(G(D), \CP(V)_{\inv})$.

The requirement that~$D$ is closed is unnecessary for us: it is only there to ensure that the graphs~$G_L(D)$ are locally independent. However, we managed to prove
Theorem~\ref{thm:completely-positive-is-exact-homogeneous-graph} without this assumption, and an inspection of the proof of~\cite[Theorem 6.3]{DeCorte2022CompleteSets} shows that this implies the following theorem.
\begin{theorem}%
  \label{thm:completely positive-is-exact-Rn-distance-avoiding}
  If~$n \geq 2$,~$D \subseteq (0, \infty)$ is bounded, and~$m_D(\R^n) > 0$, then
  \[
    m_D(\R^n) = \vartheta(G(D), \CP(\R^n)_{\inv}).
  \]
\end{theorem}
\begin{proof}
  See the discussion above.
\end{proof}

\sectionbreakafterproof

It may seem that Theorem~\ref{thm:completely positive-is-exact-Rn-distance-avoiding} misses an assumption: it is not clear that~$\vartheta(G(D), \CP(\R^n)_{\inv})$ is feasible. For example, if~$D = (0,1)$, then the conditions~$f(x) = 0$ if~$\|x\| \in D$ and~$f(0)=1$ are in contradiction with continuity of~$f$. On the other hand, this is the only thing that can prevent feasibility. Likewise, if~$m_D(\R^n) > 0$, the set~$D$ must be bounded away from~$0$. Thus, positivity of the density is equivalent to feasibility of~$\vartheta(G(D), \CP(\R^n)_{\inv})$. We already saw this in the finite-measure setting of Chapter~\ref{ch:completely-positive-formulations-meas-ind-num}.

\section{Convergence of the completely positive hierarchy}
Given Lemma~\ref{lem:approximate-density-by-density-on-tori}, we can hope that to prove convergence of a completely positive hierarchy for~$m_D(\R^n)$, we can somehow factor the hierarchy through a converging completely positive hierarchy for each~$\alpha_{\mu_L}(G_L(D))$. However, for nonthick edge sets, we only managed to prove convergence for the unit sphere and similar spaces, a result which dependent highly on the representation theory of these spaces. Particularly, in Section~\ref{ch:Completely positive programming on compact spaces}.\ref{sec:discussion-finite-measure-spaces} we saw that our method does not work for nonthick distance graphs on tori. Thus, this approach does not work.

Luckily, like for distance-avoiding sets on~$S^{n-1}$, the representation theory of~$\R^n$ saves us. We will show it is enough to consider feasible solutions that are~\defi{radial}\index{radial function}: these are the functions~$f : \R^n \to \R$ such that~$f(Tx) = f(x)$ for all~$T \in \ortho(n)$. Thus, radial functions only depend on the norm of their argument. Radial functions of positive type are given by the integral of a function in~$C_0(\R)$ under a nonnegative Borel measure, so that a weak*-converging sequence of such measures gives a sequence of radial functions of positive type that converges pointwise, except at~$0$.

As we will see, the programs~$\vartheta(G(D), \cC)$ with~$\cC = C_r(\R^n)_{\inv}^*$ or~$\CP(\R^n)_{\inv}$ are invariant under the action of~$\ortho(n)$. We can thus assume that a feasible solution~$f$ is radial and positive type. Schoenberg~\cite[Theorem 1]{Schoenberg1938MetricFunctions} showed that for any complex-valued radial function of positive type, there exists a Borel measure~$\alpha \in M(\R)_{\geq 0}$ such that
\index{ Omegan@$\Omega_n$}
\begin{multline}%
  \label{eqn:measure-of-radial-function}
  f(x) = \int_0^{\infty} \Omega_n(t \|x\|)\, d\alpha(t),\text{ with}\\
  \Omega_n(\|x\|) = 1/\omega_n \int_{S^{n-1}} e^{i x^{\tr}\xi}\, d\omega(\xi),
\end{multline}
with~$\omega$ the standard surface measure on the sphere and~$\omega_n = \omega(S^{n-1})$. The proof goes as follows: by Bochner's theorem~\cite[Theorem 4.19]{Folland2016AAnalysis}, if~$f \in \PT(\R^n)$, then there exists a measure~$\nu \in M(\R^n)_{\geq 0}$ such that for all~$x \in \R^n$ we have~$f(x) = \int_{\R^n}e^{i x^{\tr}\xi}\, d\nu(\xi)$. Then, by interchanging integrals,
\[
  \omega_n^{-1}\int_{S^{n-1}} \int_{\R^n}e^{i \|x\| y^{\tr}\xi}\, d\nu(\xi)d\omega(x) = \int_{\R^n} \Omega_n(\|x\| \|\xi\|)d\nu(\xi).
\]
Hence, define~$\alpha$ by~$\alpha(g) = \int_{\R^n} g(\|x\|)\, d\nu(x)$.

The functions~$\Omega_n$ lie in~$C_0(\R)$ for~$n \geq 2$. This follows, for example, from the expansion~$\Omega_n(t) = \Gamma(n/2)(2/t)^{(n-2)/2} J_{(n-2)/2}(t)$ for~$t > 0$, where~$J_{(n-2)/2}$ is a Bessel function of the first kind---Equation (1.8) in~\cite{Schoenberg1938MetricFunctions}---together with an asymptotic formula for the Bessel functions---equation (1) in \S 7.21 of~\cite{Watson1922AFunctions}---which says that~$J_{(n-2)/2}(t) \to 0$ as~$t \to \infty$.

\begin{theorem}%
  \label{thm:convergence-cp-hierarchy-euclidean-space}
  If~$n \geq 2$,~$D \subseteq (0, \infty)$, and~$m_D(\R^n) > 0$, then
  \[
    m_D(\R^n) = \lim_{r \to \infty} \vartheta(G(D), \bC_r(\R^n)_{\inv}^*).
  \]
\end{theorem}
\begin{proof}
  The inequality~$\lim_r \vartheta(G(D), \bC_r(\R^n)_{\inv}) \geq \vartheta(G(D), \CP(\R^n)_{\inv})$ follows from the inclusions~$\CP(\R^n)_{\inv} \subseteq \bC_r(\R^n)_{\inv}$.

  The first step to show the other inequality, is to prove that the image of each of the cones~$\bC_r(\R^n)_{\inv}^* \cap\PT(\Gamma)$ and~$\CP(\R^n)_{\inv}$ under~$\Av_{\ortho(n)}$ is contained in the original cone. Then show that the objective and constraints are also preserved under~$\Av_{\ortho(n)}$, which implies that it suffices to consider radial functions.

  Denote the Borel sets with finite Borel measure by~$\B_{\fin}$, and recall that for all~$A \in \B_{\fin}$,~$\convol_A f(x,y) = f(x - y)$ for all~$x$ and~$y \in A$. By Lemma~\ref{lem:local-definition-bC}
  \[
    \bC_r(\R^n)_{\inv}^* \cap \PT(\R^n) = \bigcap_{A \in \B_{\fin}} \convol_A^{-1} C_r(A)^* \cap \PT(\R^n).
  \]
  Let~$\nu$ be the Haar measure on~$\ortho(n)$, and let~$\Av_{\ortho(n)}f(x) = \int_{\ortho(n)} f(Tx)\, d\nu(T)$ for all~$f \in L^{\infty}(\R^n)$ and~$x \in \R^n$. Take~$f \in \bC_r(\R^n)_{\inv}^*$, take~$A \in \B_{\fin}$, and take~$K \in C_r(A)$, then by an application of Fubini-Tonelli,
  \[
    \begin{split}
      \langle \convol_A \Av_{\ortho(n)}f, K\rangle &= \int_A\int_A\int_{\ortho(n)}f(Tx-Ty)K(x, y)\, d\nu(T) dx dy\\
      &= \int_{\ortho(n)}\int_{TA}\int_{TA} f(x-y)K(T^{-1}x, T^{-1}y)\, dx dy d\nu(T)\\
      &= \int_{\ortho(n)} \langle \convol_{TA} f, TK\rangle\, d\nu(T) \geq 0,
    \end{split}
  \]
  where~$TK(x,y) = K(T^{-1}x, T^{-1}y)$, so that~$TK \in C_r(TA)$ for all~$T \in \ortho(n)$. This uses that the Lebesgue measure is also invariant under the orthogonal group.

  With a similar calculation it follows that
  \[
    \langle \Av_{\ortho(n)}f, \rho \cor \rho\rangle = \int_{\ortho(n)} \langle f, \rho_T \cor \rho_T \rangle\, d\nu(T) \geq 0
  \]
  for all~$\rho \in L^1(\R^n)$, where~$\rho_T(x) = \rho(Tx)$. Thus,~$\Av_{\ortho(n)}f \in \PT(\R^n)$. It is similarly clear that~$\Av_{\ortho(n)}\CP(\R^n)_{\inv} \subseteq \CP(\R^n)$.

  Taking~$\omega_n^{-1} \Av_{\ortho(n)} f$ preserves the objective. Indeed, let~$f \in \PT(\R^n)$, by Bochner's theorem there exists~$\alpha \in M(\R^n)_{\geq 0}$ such that~$f(x) = \int_{\R^n} e^{ix^{\tr}\xi}\, d\alpha(\xi)$. Then,~$Mf = \alpha(0)$; to see this, let~$\phi_T = T^{-n}\int_{[-T/2,T/2]^n} e^{iu^{\tr}x}\, dx$. Then, by pointwise convergence~$\phi_T \to \1_{\{0\}}$ and the dominated convergence theorem,~$\alpha(0) = \lim_{T \to \infty} \alpha(\phi_T)$. The measure from Bochner's theorem corresponding to~$\omega_n^{-1} \Av_{\ortho(n)}f$ is given by~\eqref{eqn:measure-of-radial-function}, so~$Mf = \Omega_n(0)\alpha(0) = \alpha(0)$, since~$\Omega_n(0) = 1$. The normalization and edge are constraints also preserved under this operation. Thus, in the remainder assume that all solutions are radial.

  From here on, the proof follows steps similar to that of Theorem~\ref{thm:convergence-cp-hierarchy-sphere}. That is, take for all~$r \in \N$ a radial feasible solution~$f_r$ of~$\vartheta(G(D), \bC_r(\R^n)_{\inv})$ such that~$Mf_r \geq m_D(\R^n) / 2$. Then, show that there exists a feasible solution of~$\vartheta(G(D), \CP(\R^n)_{\inv})$ such that for all~$r_0 \in \N$ there is an~$r \geq r_0$ such that~$f(0)^{-1}Mf \geq Mf_r$. This and~$\lim_r \vartheta(G(D), \bC_r(\R^n)_{\inv}^*) \geq m_D(\R^n)$ concludes the proof.

  Since each~$f_r$ is radial, there exists for all~$r$ an~$\alpha_r \in M(\R)_{\geq 0}$ such that~$f_r(x) = \int_{\R} \Omega_n(t\|x\|)\, d\alpha_r(t)$. Then,~$1 = f_r(0) = \alpha_r(\R) = \|\alpha_r\|$, thus the sequence~$(\alpha_r)_r$ lies in the unit ball of~$M(\R)$, which is compact by Banach-Alaoglu~\cite[Theorem V.3.1]{Conway2010AAnalysis}. Since~$C_0(\R)$ is separable, the weak* topology on the unit ball is metrizable~\cite[Theorem V.5.1]{Conway2010AAnalysis}, thus it is sequentially compact, and~$(\alpha_r)_r$ has a converging subsequence; assume the sequence itself converges to~$\alpha \in M(\R)$. Then,~$\alpha$ is nonzero, since for all~$r$ and continuous functions~$\phi$ with~$0 \in \supp \phi$,~$\alpha(\phi) = \lim_r \alpha_r(\phi) \geq \lim_r \alpha_r(0) \geq m_D(\R^n)/2$.

  Define~$f \in L^{\infty}(\R^n)$ by~$f(x) = \int_{\R} \Omega_n(t\|x\|)\, d\alpha(t)$, which by Bochner's theorem is positive type, thus continuous, and~$f(0) = \|\alpha\| \leq 1$. Moreover, note that~$f_r \to f$ under the weak* topology in~$L^{\infty}(\R^n)$.

  Since~$\bC_{r+1}(\R^n)_{\inv}^* \subseteq \bC_r(\R^n)_{\inv}^*$ for all~$r \in \N$, and each cone is weak* closed, it follows that~$f \in \bigcap_r \bC_r(\R^n)_{\inv} = \CP(\R^n)$ by Theorem~\ref{thm:Polyas-theorem-for-invariant-cones}. Moreover, since~$\Omega_n \in C_0(\R)$, it follows that for all~$x \in \R^n\setminus\{0\}$ the map~$t \mapsto \Omega_n(t\|x\|)$ is in~$C_0(\R)$, so~$f(x) = \int_\R \Omega_n(t\|x\|)\, d\alpha(t) = \lim_r \int_{\R} \Omega_n(t\|x\|)\, d\alpha_r(t) = \lim_r f_r(x)$. This shows that~$f_r$ converges to~$f$ pointwise on~$\R^n$, except perhaps at~$0$. In particular, if~$x \in D$,~$f(x) = 0$.

  It follows that~$f(0)^{-1} f$ is feasible for~$\vartheta(G(D), \CP(\R^n)_{\inv})$. It is left to show that~$f(0)^{-1} Mf \leq \lim_r Mf_r$.  Moreover, with~$\phi_T =T^{-n}\int_{[-T/2,T/2]^n}e^{ix^{\tr}\xi}\, d\alpha(\xi)$ as before, since~$\phi_T(0) = 1$ for all~$T$ and~$\alpha_r \geq 0$ for all~$r$, it follows that,
  \[
    \alpha(\phi_T) = \lim_{r \in \N} \alpha_r(\phi_T) \geq \lim_{r \in \N} \alpha_r(0).
  \]
  Therefore,~$\alpha(0) = \lim_T \alpha(\phi_T) \geq \lim_{T} \lim_r \alpha_r(0) = \lim_r \alpha_r(0)$. Since~$\|\alpha\| \leq 1$, indeed~$M(f(0)^{-1}f) \geq \lim_r Mf_r$, and the result follows.
\end{proof}

\section{Some discussion of the implementation}
By analogy with Witsenhausen's problem (Chapter~\ref{ch:witsenhausen}) we might hope that an implementation of~$\vartheta(G(\{1\}), \bC_1(\R^2)^*_{\inv})$ leads to better bounds on~$m_1(\R^2)$. Although there indeed is again a closely related optimization problem that gives provable upper bounds on~$m_1(\R^2)$ which are at least as strong as the best known bounds, we have not yet found strictly better bounds. It is unclear whether this is a shortcoming in the implementation or in the bound itself.

\chapter[Sphere packing]{Sphere packing}%
\label{ch:sphere-packing}
This chapter is part of ongoing work with David de Laat and Fernando Mário de Oliveira Filho.

\sectionbreak

A \defi{sphere packing}\index{sphere packing} is a collection of congruent balls in~$\R^n$ having pairwise-disjoint interiors. Problem~\ref{it:problem-4} from the introduction asks for the largest fraction of Euclidean space that can be covered by a sphere packing; this is called the~\defi{maximal sphere-packing density}. The problem of determining this number is the~\defi{sphere-packing problem}\index{problem!sphere packing}. Linear programming bounds were introduced to the sphere-packing problem by Cohn and Elkies~\cite{Cohn2003NewI} and proved highly successful. In dimensions 8 and 24 they are even sharp~\cite{Viazovska2017The8,Cohn2017The24}.

Attention has recently moved to improving the upper bound by restricting the linear programming bound. Cohn, De Laat, and Salmon~\cite{cohn2022ThreepointPacking} introduced a three-point bound inspired by the block moment hierarchy for compact packing graphs. Cohn and Salmon~\cite{cohn2021SphereRescaling} made an in-depth study of the relation between bounds on compact packing graphs and bounds on the sphere-packing problem, leading to a definition of a moment hierarchy for sphere packing, including a proof of its convergence.

In this chapter, we complement this work by introducing a copositive formulation of the maximal sphere-packing density and proving convergence of a corresponding hierarchy. The proof that the copositive formulation is an upper bound is a straight-forward modification of the proof that Cohn and Elkies gave that their linear programming bound upper bounds the sphere packing density. Exactness then follows by a reduction to compact packing graphs.

We will optimize over functions in a vector space~$X$ that lies between the space of compactly supported smooth functions~$C_c^{\infty}(\R^n)$\index{ Cczzz@$C_c^{\infty}(\R^n)$} and~$C_0(\R^n) \cap L^1(\R^n)$. There is a good reason to treat the space~$X$ as a parameter. From the perspective of the theory it seems sufficient to only ask that the functions are integrable and continuous. On the other hand, the space of Schwartz functions has proven especially effective for calculating bounds~\cite{Cohn2003NewI, Cohn2002NewII, cohn2022ThreepointPacking}. The~\defi{Schwartz space}\index{Schwartz space}, denoted by~$\swrtz(\R^n)$\index{ S Rn@$\swrtz(\R^n)$}, is defined as the space of all smooth functions~$f(x)$ such that all partial derivatives decay faster than any negative power of~$\|x\|$. Viazovska showed~\cite{Viazovska2017The8} that for the sphere-packing problem in~$\R^8$, the Cohn-Elkies programming bound has an optimizer~$f_{\rm opt}$ such that~$f_{\rm opt} \in \swrtz(\R^n)$, and whose objective value is exactly the maximal sphere packing density~$\R^8$. This result was later also obtained for~$\R^{24}$~\cite{Cohn2017The24}.

It is an open question whether there exists an optimizer that is Schwartz for every dimension. We will see, however, that for any space~$X$ between~$C_0(\R^n)$ and~$C_c^{\infty}(\R^n)$, the maximal sphere-packing density has a copositive formulation over~$X$. Hence, there is a copositive formulation over~$\swrtz(\R^n)$.

The Fourier transform plays an important role in the optimization problems in this chapter; for~$f \in L^1(\R^n)$, let
\[
  \widehat{f}(x) = \int_{\R^n} f(y) e^{-2\pi i x^{\tr}y}\, dy
\]
be the Fourier transform of~$f$. In this chapter, the optimization bounds we will investigate are of the form
\index{ theta SP@$\vartheta_{\SP}^X(\cC)$}
\begin{equation}%
  \label{eqn:sphere-packing-bound}
  \begin{optprob}
    \vartheta^X_{\SP}(\cC) = \inf & \onerow{f(0) + g(0)}\\
    &\onerow{\widehat{f}(0) = 1,}\\
    &f(x) + g(x) \leq 0 &\text{for all }\|x\|\geq 1,\\
    &\onerow{f \in \PT(\R^n) \cap X,\ g \in \cC \cap X,}
  \end{optprob}
\end{equation}
where~$\cC \subseteq L^1(\R^n)$ is a convex cone and~$C_c^{\infty}(\R^n) \subseteq X \subseteq C_0(\R^n) \cap L^1(\R^n)$ a linear subspace. The Cohn-Elkies bound is of this form, with~$X$ a suitable space and~$\cC = \{0\}$.

\section{Sphere-packing densities}
The~\defi{maximal sphere-packing density} of~$\R^n$ is the quantity
\[
  \sup\{\, \ud(S) : S\text{ is a sphere packing}\, \}.
\]
This number is scaling invariant, in the sense that the radius of the spheres making up the packing does not matter, as long as they are all the same. In the rest of this chapter, the radius is~$1/2$, which matches much of the literature. Denote a ball of radius~$r$ and center~$p$ by~$B_r(p)$\index{ Brp@$B_r(p)$}. The volume of the ball~$B_{1/2}(0) \subseteq \R^n$ is denoted~$V_{1/2}^n = \lambda(B_{1/2}(0))$\index{ V 12@$V_{1/2}$}.

It is more practical to work with the~\defi{maximal center density}\index{maximal center density}~$\Delta_n$. Write the cube with side length~$T$ centered at~$0$ as~$C_T = [-T/2,T/2]^n$\index{ CT@$C_T$}. The maximal center density is the average number of centers of spheres in a packing per unit volume:
\begin{multline}%
  \label{eqn:sphere-packing-density}
  \Delta_n = \sup \biggl\{\, \limsup_{T \to \infty} \frac{|P \cap C_T|}{T^n} : P \subseteq \R^n \text{ such that } \|x - y\| \geq 1\\
  \text{ for all } x,\ y \in P \, \biggr\}.
\end{multline}
Then,~$V_{1/2}^n \Delta_n$ is exactly the maximal sphere-packing density of~$\R^n$.

In Chapter~\ref{ch:distance-avoiding-sets-euclidean-space} we saw that the density of a $D$-avoiding set can be approximated by the density of a periodic set, which is given by a single integral over the fundamental domain. Something similar holds for the center density, as explained by Cohn and Elkies~\cite[Appendix A]{Cohn2003NewI}.

Recall that for~$n \in \N$ and~$L > 0$ a real number,~$\torus_L = \R^n/L\Z^n$. Define the graph~$G_L = (\torus_L, E_L)$ by saying that~$xy \in E_L$ if and only if~$d_L(x,y) < 1$; then,~$G_L$ is the graph~$G_L(D)$ of Chapter~\ref{ch:distance-avoiding-sets-euclidean-space} with~$D = (0,1)$. The graphs~$G_L$ are compact packing graphs. It can be shown that
\begin{equation}%
  \label{eqn:SP-density-as-limit-over-tori}
  \Delta_n = \lim_{L \to \infty} \frac{\alpha(G_L)}{L^n}.
\end{equation}

It is sometimes easier to work with the closure of the fundamental domain~$C_L$ instead of with the torus itself. Let~$G_{C_L} = (C_L, E_L)$ be the graph with vertex set~$C_L$ and~$xy \in E_L$ if and only if~$\|x - y\| < 1$. The graphs~$G_{C_L}$ are also compact packing graphs. From~\eqref{eqn:SP-density-as-limit-over-tori}, one can obtain
\begin{equation}%
  \label{eqn:sphere-packing-density-limit-over-cubes}
  \Delta_n = \lim_{L \to \infty} \frac{\alpha(G_{C_L})}{L^n}.
\end{equation}

\section{Some more harmonic analysis}
The following lemmas describe mechanics by which we can approximate feasible solutions of~$\vartheta^{X}_{\SP}(\cC)$ by solutions with compact support, having additional properties. In Sections~\ref{sec:upper-bound-sphere-packing} and~\ref{sec:exactness-and-convergence-sphere-packing} we will use these as reduction steps to proof that there are copositive formulations for~$\Delta_n$, and that there is a converging copositive hierarchy. See Section~\ref{ch:complete-positivity-under-symmetry}.\ref{sec:prel-harmonic-analysis} for background and important inequalities.

\begin{lemma}%
  \label{lem:action-of-pt-approximate-identity}
  Let~$g \in L^1(\R^n)$. If~$\rho \in \PT(\R^n)$, then~$(g \cor g) \cor \rho \in \PT(\R^n)$. If~$\rho \in \COP(\R^n)_{\inv}$ and~$g \geq 0$, then~$(g \cor g) \cor \rho \in \COP(\R^n)_{\inv}$.
\end{lemma}
\begin{proof}
  The correlation of a function~$\rho \in L^1(\R^n)$ and~$f \in L^p(\R^n)$ is in~$L^p(\R^n)$ for all~$1 \leq p \leq \infty$, see Section~\ref{ch:complete-positivity-under-symmetry}.\ref{sec:prel-harmonic-analysis}. So,~$(g \cor g) \cor \rho \in L^1(\R^n)$, and similar for the convolution. In particular, if~$f \in L^2(\R^n)$ and~$g \in L^1(\R^n)$, then~$f * g \in L^2(\R^n)$.

  Let~$\rho$ and~$g \in L^1(\R^n)$ and~$f \in L^{2}(\R^n)$, then by a change of variables and Fubini-Tonelli,
  \[
    \langle g \cor \rho, f \rangle = \int_{\R^n} \int_{\R^n} g(y) \rho(x) f(x-y)\, dydx = \langle \rho, g * f\rangle.
  \]
  Thus, for all~$\rho$ and~$g \in L^1(\R^n)$ and~$f \in L^2(\R^n)$,
  \[
    \langle (g \cor g) \cor \rho, f \cor f\rangle = \langle \rho, (g \cor g) * (f \cor f)\rangle.
  \]
  Furthermore, for all~$x \in \R^n$, by Fubini-Tonelli and the two changes of variables~$y' = y+z$ and~$w' = w+z$,
  \[
    \begin{split}
      (g \cor g) * (f \cor f)(x) &= \int_{\R^n} (g \cor g)(y) (f \cor f)(x-y)\, dy\\
      &= \int_{\R^n} \int_{\R^n} \int_{\R^n} g(z)g(y+z) f(w)f(x-y+w)\, dwdzdy\\
      &= \int_{\R^n} \int_{\R^n} \int_{\R^n} g(z)g(y) f(w - z)f(x + w -y)\, dwdydz\\
      &= \int_{\R^n} (g * f)(w) (g * f) (x+w)\, dw\\
      &= (g*f) \cor (g*f)(x).
    \end{split}
  \]

  Hence, if~$\rho \in \PT(\R^n) \cap L^1(\R^n)$ and~$g \in L^1(\R^n)_{\geq 0}$, then for every~$f \in L^2(\R^n)$ we have~$f * g \in L^2(\R^n)$, and
  \[
    \langle (g \cor g) \cor \rho, f \cor f\rangle =  \langle \rho, (g*f) \cor (g*f) \rangle \geq 0,
  \]
  thus~$(g \cor g) \cor \rho \in \PT(\R^n)$. If~$\rho \in \COP(\R^n)_{\inv}$ and~$f \geq 0$, since~$g\geq 0$ it follows in the same way that~$(g \cor g) \cor \rho \in \COP(\R^n)_{\inv}$.
\end{proof}

The next lemma is a well-known method of bounding the support of a positive-type function, while preserving the constraints that figure in the definition of~$\vartheta_{\SP}^X(\cC)$, see for example the proof of~\cite[Theorem 3.1]{DeLaat2014UpperRadii}. This approximation is used to prove that~$\vartheta_{\SP}^{C_0(\R^n) \cap L^1(\R^n)}(\COP_{\inv}(\R^n))$ is an upper bound on~$\Delta_n$. For a bounded set~$C \subseteq \R^n$, let~$\psi_C = \lambda(C)^{-1} \1_C \cor \1_C$.
\begin{lemma}%
  \label{lem:approximate-by-restrictions}
  Let~$\rho \in C_0(\R^n)$, then
  \begin{enumerate}
    \item[(i)] $\psi_{C_T}\rho \in C_c(\R^n)$ for all~$T > 0$,
    \item[(ii)] if~$\rho \in \PT(\R^n)$, then~$\psi_{C_T} \rho \in \PT(\R^n)$ for all~$T > 0$,
    \item[(iii)] if~$\rho \in \COP(\R^n)_{\inv}$, then~$\psi_{C_T} \rho  \in \COP(\R^n)_{\inv}$ for all~$T > 0$, and
    \item[(iv)] $\lim_T\|\psi_{C_T} \rho - \rho\|_{\infty} = 0$.
  \end{enumerate}
\end{lemma}
\begin{proof}
  Let~$\rho \in C_0(\R^n)$ and~$T > 0$ a real number. The support of~$\psi_{C_T}$ is contained in~$C_T + C_T$, and the correlation of a bounded function with an integrable function is continuous, so the first statement follows immediately.

  The second and third statement can be understood as a consequence of the fact that the Hadamard product of two positive-semidefinite matrices is positive semidefinite, and likewise, the Hadamard product of a positive-semidefinite matrix and a copositive matrix is copositive. Indeed, a continuous function is positive type if and only if for all~$U \subseteq \R^n$ the matrices~$f[U] = (f(u_i - u_j))_{u_i, u_j \in U}$ are positive semidefinite~\cite[Proposition 3.35]{Folland2016AAnalysis}. Take~$x$ and~$y \in \R^n$, then by a change of variables
  \begin{multline*}
    T^n (\psi_{C_T}\rho)(x - y) = \int_{\R^n} \1_{C_T}(z) \1_{C_T}(x - y + z)\, dz \rho(x-y)\\
    = \langle \1_{C_T +x}, \1_{C_T + y} \rangle \rho(x-y),
  \end{multline*}
  and the inner product is given by a positive definite form, which proves the second statement. The third follows by a similar argument.

  Finally, for convergence, note that
  \[
    \|\psi_{C_T} \rho - \rho\|_{\infty} = \sup_{x \in \R^n} \biggl|\frac{\lambda(C_T \cap (C_T - x))}{T^n} - 1\biggr||\rho(x)|.
  \]
  Since~$\rho \in C_0(\R^n)$, for every~$\epsilon > 0$ there is a compact~$K_{\epsilon}$ such that~$|\rho(x)| \leq \epsilon$ on~$\R^n \setminus K_{\epsilon}$. Surely, since~$\lambda(C_T \cap (C_T - x)) \leq T^n$, this implies
  \[
    \sup_{x \in \R^n \setminus K_{\epsilon}} \biggl|\frac{\lambda(C_T \cap (C_T - x))}{T^n} - 1\biggr||\rho(x)| \leq \sup_{x \in \R^n \setminus K_{\epsilon}} |\rho(x)| \leq \epsilon.
  \]
  On the other hand, the sequence~$T^{-n}\lambda(C_T \cap (C_T - x))$ can be shown to converge to~$1$, uniformly for~$x$ in a compact set. Take~$T$ large enough such that~$|T^{-n}\lambda(C_T \cap (C_T - x)) - 1| \leq \epsilon / \|\rho\|_{\infty}$ for all~$x \in K_{\epsilon}$, then
  \[
    \sup_{x \in K_{\epsilon}} \biggl|\frac{\lambda(C_T \cap (C_T - x))}{T^n} - 1\biggr||\rho(x)| \leq \epsilon \|\rho\|_{\infty} / \|\rho\|_{\infty} = \epsilon,
  \]
  and the result follows.
\end{proof}

For~$r \in \R$, let~$\rho_{r}(x) = \rho(r x)$. We will use the next lemma to smoothen functions, while again preserving the constraints from~$\vartheta_{\SP}^X(\cC)$.
\begin{lemma}%
  \label{lem:approximate-by-correlation}
  If~$\rho \in C_c(\R^n)$ such that~$\rho(x) \leq 0$ for all~$x$ with~$\|x\| \geq 1$, if~$r_i = 1/(1-1/i)$ for all~$i \in \N_{\geq 2}$, and if~$(\phi_i)_{i \in \N_{\geq 2}}$ is an approximate identity such that~$\supp \phi_i \subseteq B_{1/i}(0)$ for all~$i$, then
  \begin{enumerate}
    \item[(i)] $r_i^n \phi_i \cor \rho_{r_i}$ is in~$C_c(\R^n)$ for all~$i$,
    \item[(ii)] $(r_i^n \phi_i \cor \rho_{r_i})(x) \leq 0$ for all~$x$ with~$\|x\| \geq 1$ and all~$i$,
    \item[(iii)] $\int_{\R^n} (r_i^n \phi_i \cor \rho_{r_i}) (x)\, dx = \int_{\R^n}\rho(x)\, dx$, and
    \item[(iv)]  $\lim_i|(r_i^n \phi_i \cor \rho_{r_i})(x) - \rho(x)| = 0$ uniformly on bounded sets.
  \end{enumerate}
\end{lemma}
\begin{proof}
  The first statement follows from the fact that the cross-correlation of two bounded functions with compact support is continuous, for instance by Lemma~\ref{lem:averaged-rank-one-is-continuous}, and the support of the correlation is contained in the sum of the support, which follows from a direct calculation.

  The second statement follows from the definition of the correlation and the reverse triangle inequality. Indeed, let~$\|x\|\geq 1$ and~$\|y\|\leq 1/i$, then
  \[
    \|r_i(x - y)\| \geq \frac{|\|x\| - \|y\||}{1-1/i} \geq 1.
  \]
  Then, for all~$x$ with~$\|x\| \geq 1$,
  \[
    \phi_i \cor \rho_{r_i}(x) = \int_{B_{1/i}(0)} \phi_i(-y) \rho(r_i(x-y))\, dy \leq 0.
  \]

  For the third statement, use~$\int \phi_i(x)\, dx = 1$ for all~$i$ and Fubini's theorem, invariance of the Haar measure, and the change of variables~$x' = r_i(x - y)$, which gives
  \begin{multline*}
    \int_{\R^n} r_i^n \phi_i \cor \rho_{r_i}(x)\, dx = \int_{\R^n} \int_{\R^n} r_i^n \phi_i(-y) \rho(r_i(x-y))\, dydx\\
    = \int_{\R^n} \int_{\R^n} \phi_i(y) \rho(x)\, dxdy = \int_{\R^n} \rho(x)d\, x.
  \end{multline*}

  For the fourth statement, by Hölder's inequality,
  \begin{multline*}
    r_i^n\phi_i \cor \rho_{r_i}(x) - \rho(x) = \int_{\R^n}\phi_i(-y) (r_i^n\rho(r_i(x-y)) - \rho(x))\, dy\\
    \leq \sup_{y \in B_{1/i}(0)}| r_i^n\rho(r_i(x-y)) - \rho(x) |.
  \end{multline*}
  Whence, it suffices to find for all~$\epsilon > 0$, an~$i_{\epsilon}$ such that for all~$i \geq i_{\epsilon}$
  \[
    \sup_{y \in B_{1/i}(0)} |r_i^n\rho(r_i(x-y)) - \rho(x)| \leq \epsilon,
  \]
  which only depends on~$\|x\|$.

  Fix a real number~$M \geq 0$, a vector~$x \in \R^n$ such that~$\|x\| \leq M$, and a vector~$y \in B_{1/i}(0)$. Then
  \[
    \|x - r_i(x-y)\| = \biggl\|\frac{x/i - y}{1 - 1/i}\biggr\| \leq \frac{\|x\| + i\|y\|}{i-1} \leq \frac{M+1}{i-1},
  \]
  Hence, for all~$\delta > 0$ there exists an~$i_{\delta}$ such that for all~$i \geq i_{\delta}$,~$\|x - r_i(x-y)\| \leq \delta$ for all~$y \in B_{1/i}(0)$.

  The function~$\rho$ is uniformly continuous, so for every~$\epsilon > 0$ there exists a~$\delta(\epsilon)$ such that, if~$\|x-y\| \leq \delta(\epsilon)$, then~$|\rho(x) - \rho(y)| \leq \epsilon$. In particular, for all~$x \in B_M(0)$ and~$y \in B_{1/i}(0)$ with~$i \geq i_{\delta(\epsilon)}$,~$|\rho(r_i(x-y)) - \rho(x)| \leq \epsilon$.

  Finally, let~$\epsilon > 0$,~$\bar{\epsilon} = \epsilon / \|\rho\|_{\infty} + 1$, and~$i_0 = \max\{i_{\delta(\epsilon)}, \lceil\bar{\epsilon}^{1/n} / (\bar{\epsilon}^{1/n} - 1)\rceil\}$. Then, for all~$i \geq i_0$, since~$i \geq 2$ and~$i / (i-1)$ is decreasing in~$i$,
  \[
    \begin{split}
      |r_i^n\rho(r_i(x-y)) - \rho(x)| &= \biggl|\frac{\rho(r_i(x-y)) - \rho(x)(1-1/i)^n}{(1-1/i)^n}\biggr|\\
      &\leq \frac{|\rho(r_i(x-y)) - \rho(x)|}{(1-1/i)^n} + \frac{|((1-1/i)^n - 1)\rho(x)|}{(1-1/i)^n}\\
      &\leq 2^n|\rho(r_i(x-y)) - \rho(x)| + \biggl(\biggl(\frac{i}{i-1}\biggr)^n - 1\biggr)\|\rho\|_{\infty}\\
      &\leq 2^n|\rho(r_i(x-y)) - \rho(x)| + (\bar{\epsilon} - 1)\|\rho\|_{\infty}\\
      &\leq \epsilon(2^n+1),
  \end{split}\]
  which goes to~$0$ as~$\epsilon$ goes to~$0$.
\end{proof}

\section{Upper bounds on the maximal sphere-packing density}%
\label{sec:upper-bound-sphere-packing}
Because~$\vartheta_{\SP}$ is a minimization problem, for any~$X \subseteq Y \subseteq L^1(\R^n) \cap C_0(\R^n)$ and any convex cone~$\cC \subseteq L^1(\R^n)$,
\begin{equation}%
  \label{eqn:SP-theta-under-inclusion}
  \vartheta_{\SP}^Y(\cC_Y) \leq \vartheta_{\SP}^X(\cC_X).
\end{equation}
We will use this principle to show that any subspace~$X$ of~$C_0(\R^n) \cap L^1(\R^n)$ that contains~$C^{\infty}_c(\R^n)$ satisfies~$\vartheta_{\SP}^X(\COP(\R^n)_{\inv}) = \Delta_n$. We will first show the inequality~$\Delta_n \leq \vartheta^{C_0(\R^n) \cap L^1(\R^n)}_{\SP}(\COP(\R^n)_{\inv})$.

Recall that for a compact packing graph~$G = (V, E)$ and a convex cone of continuous symmetric kernels~$\cC \subseteq C_{\sym}(V)$,
\[
  \begin{optprob}
    \vartheta^*(G, \cC) = \inf & \onerow{t}\\
    &T(x, x) \leq t-1 & \text{for all } x \in V,\\
    &T(x, y) \leq -1 & \text{for all } xy \in E,\\
    &\onerow{T \in \cC.}
  \end{optprob}
\]
Dobre, Dür, Frerick, and Vallentin~\cite{Dobre2016AGraphs} showed that~$\alpha(G) = \vartheta^*(G, \COP(V)_{\cont})$. Together with identity~\eqref{eqn:SP-density-as-limit-over-tori}, this implies
\begin{equation}%
  \label{eqn:limit-over-copositive-theta-tori}
  \Delta_n = \lim_{L \to \infty} \frac{\vartheta^*(G_L, \COP(\torus_L)_{\cont})}{L^n}.
\end{equation}

To see how the problems~$\vartheta^*(G_L, \COP(\torus_L)_{\cont})$ correspond to the problem~$\vartheta^{C_0(X) \cap L^1(\R^n)}_{\SP}(\COP(\R^n)_{\inv})$, think of the functions~$f$ and~$g$ as the image under~$\Av_{\R^n}$ of kernels on~$\R^n$. The main difference between the~$\vartheta^*$ for compact packing graphs as above and~$\vartheta_{\SP}$ is the presence of the auxiliary function~$f$ in the latter.

Fix a compact packing graph~$G = (V, E)$. Let us first bring~$\vartheta^*(G, \cC)$ into a similar form. Of course,~$\vartheta^*(G, \cC)$ is equivalent to
\[
  \begin{optprob}
    \inf & \onerow{t}\\
    &T(x, x) \leq t & \text{for all } x \in V,\\
    &T(x, y) \leq 0 & \text{for all } xy \in E,\\
    &\onerow{T - J \in \cC,}
  \end{optprob}
\]
where~$J = \1^{\otimes 2}$. If~$\cC$ is a convex cone that contains~$\PSD(V)$, then~$T - J \in \cC$ if and only if~$T = F + G$ where~$F - J\in \PSD(V)$ and~$G \in \cC$; indeed, if~$T - J \in \cC$, then take~$F = J$ and~$G = T-J$. On the other hand, if~$F - J \in \PSD(V)$ and~$G \in \cC$, then~$G + F - J \in \cC$ because~$\cC$ is convex. It follows that for every convex cone~$\cC$ containing the positive-semidefinite cone,
\[
  \begin{optprob}
    \vartheta^*(G, \cC) = \inf & \onerow{t}\\
    &F(x, x) + G(x, x) \leq t & \text{for all } x \in V,\\
    &F(x, y) + G(x, y) \leq 0 & \text{for all } xy \in E,\\
    &\onerow{F - J \in \PSD(V) \text{ and } G \in \cC.}
  \end{optprob}
\]

The split of~$T$ into~$F$ and~$G$ is important. The constraint~$F - J \in \PSD(V)$ comes down to a simple eigenvalue condition, whereas, a priori, the constraint~$T - J \in \COP(V)_{\cont}$ is a difficult one.

The identity~\eqref{eqn:limit-over-copositive-theta-tori} suggests an approach for the proof of the Theorem~\ref{thm:C0-theta-is-upper-bound}: show that for every feasible solution~$(f, g)$ of~$\vartheta_{\SP}^{C_0(\R^n) \cap L^1(\R^n)}(\COP(\R^n)_{\inv})$, for large enough~$L$, there is a feasible solution~$(t, F, G)$ of~$\vartheta^*(G_L, \COP(\torus_L)_{\cont})$ such that~$f(0) + g(0) \geq t/L^n$.

\begin{theorem}%
  \label{thm:C0-theta-is-upper-bound}
  For all~$n \geq 2$,~$\vartheta_{\SP}^{C_0(\R^n) \cap L^1(\R^n)}(\COP(\R^n)_{\inv}) \geq \Delta_n$.
\end{theorem}
\begin{proof}
  Let~$(f, g)$ be a feasible solution of~$\vartheta_{\SP}^{C_0(\R^n) \cap L^1(\R^n)}(\COP(\R^n)_{\inv})$. First, they may be assumed to have compact support: let~$\psi_T = T^{-n} \1_{C_T} \cor \1_{C_T}$. By Lemma~\ref{lem:approximate-by-restrictions}, the functions~$\psi_T f$ and~$\psi_T g$ converge uniformly to~$f$ and~$g$ respectively, have compact support, and~$\psi_T f \in \PT(\R^n)$ and~$\psi_T g \in \COP(\R^n)_{\inv}$. Moreover,~$\psi_T f(x) + \psi_T g (x) \leq 0$ if~$\|x\| \geq 1$.

  If~$\tau_T = \widehat{\psi_T f}(0) = \int_{\R^n} \psi_T(x) f(x)\, dx$, by the dominated convergence theorem,~$\lim_T \tau_T =\int_{\R^n} f(x)\, dx = \widehat{f}(0)$. So, the pair~$(\tau_T^{-1}\psi_T f, \psi_T g)$ is feasible for~$\vartheta_{\SP}^{C_0(\R^n) \cap L^1(\R^n)}(\COP(\R^n)_{\inv})$ and converges uniformly to~$(f, g)$. Whence,~$f$ and~$g$ can be assumed to have compact support.

  Let~$(f, g)$ be a feasible solution to~$\vartheta_{\SP}^{C_0(\R^n) \cap L^1(\R^n)}(\COP(\R^n)_{\inv})$ with compact support, and let~$L$ be large enough such that~$\supp f$ and~$\supp g$ lie entirely in~$[-L/2, L/2)^n$. Let~$p : \R^n \to \torus_L$ be the quotient map. Define the functions
  \[
    \phi(x) = \sum_{v \in L\Z^n} f(x + v)\qquad \text{and}\qquad \gamma(x) = \sum_{v \in L\Z^n} g(x + v);
  \]
  the sums are finite, since~$f$ and~$g$ have compact support. Since these functions are invariant under the action of~$\torus_L$, the kernels~$F(p(x), p(y)) = \phi(x - y)$ and~$G(p(x), p(y)) = \gamma(x - y)$ are well-defined on~$\torus_L$. Moreover,
  \[
    F(x,x) + G(x,x) = \phi(0) + \gamma(0) = \sum_{v \in L\Z^n} f(v) + g(v) = f(0) + g(0),
  \]
  since~$\supp f \cup \supp g \subseteq [-L/2, L/2)^n$. So, take~$t = L^n f(0) + L^n g(0)$. Since we can take~$L \geq 1$, it follows that~$F(x, x) + G(x, x) \leq t$.

  It is left to prove that~$(t, F, G)$ satisfies the edge and cone constraints of~$\vartheta^*(G_L, \COP(\torus_L)_{\cont})$. Recall that~$G_L = (\torus_L, E)$ with~$xy \in E$ if and only if~$d_L(x, y) \geq 1$. If~$d_L(x, y) \geq 1$, then for every~$v \in L\Z^n$,~$\|x - y + v\|\geq 1$. Hence,
  \[
    F(x,y) + G(x,y) = \sum_{v \in L\Z^n} f(x - y + v) + g(x - y + v) \leq 0.
  \]

  That~$F \in \PSD(\torus_L)$ and~$G \in \COP(\torus_L)$ follows from the following. Let~$\mu_L$ be the Haar measure on~$\torus_L$, normalized so that~$\mu_L(\torus_L) = L^n$. Take a function~$\rho \in L^2(\torus_L)$, and denote the inner product on~$L^2(\torus_L)$ by~$\langle \cdot, \cdot \rangle_{\torus_L}$ and the cross-correlation on~$\torus_L$ by~$\cor_{\torus_L}$.
  Then, by Equation~\eqref{eqn:quotient-measure-is-invariant-Radon-measure}, Lemma~\ref{lem:approximate-by-restrictions}, and because~$(\rho \cor_{\torus_L}\hspace{-1.5pt} \rho) \smallcirc p$ is continuous, bounded, and periodic,
  \[
    \begin{split}
      \langle F, \rho \otimes \rho\rangle_{\torus_L} &= \langle \phi, \rho \cor_{\torus_L}\hspace{-1.5pt} \rho\rangle_{\torus_L} = \langle f, (\rho \cor_{\torus_L}\hspace{-1.5pt} \rho)\smallcirc p \rangle\\
      &=\lim_{T \to \infty} \frac{1}{T^n} \langle f (\1_{C_T} \cor \1_{C_T}), (\rho \cor_{\torus_L}\hspace{-1.5pt} \rho)\smallcirc p \rangle\\
      &=\lim_{T \to \infty} \frac{1}{T^{n}} \int_{C_T^n}\int_{C_T^n} f(x-y) (\rho \cor_{\torus_L}\hspace{-1.5pt} \rho)(p(x-y))\, dydx\\
      &=\lim_{T \to \infty} \frac{1}{T^{n}} \int_{C_T^n} \int_{C_T^n} f(x-y)\\
      &\phantom{= \lim_{T \to \infty} \frac{1}{T^{n}} \int_{C_T^n}}\int_{\torus_L} \rho(p(x)+z) \rho(p(y)+z)\, d
      \mu_L(z)dydx\\
      &= \lim_{T \to \infty} \frac{1}{T^{n}} \int_{\torus_L}\\
      &\phantom{= \lim_{T \to \infty} \frac{1}{T^{n}}}\int_{C_T^n} \int_{C_T^n} f(x-y) \rho(p(x)+z) \rho(p(y)+z)\, d
      \mu_L(z)dydx\\
      &\geq 0.
  \end{split}\]
  Thus,~$F \in \PSD(\torus_L)$, and similarly~$\langle G, \rho \otimes \rho\rangle_{\torus_L} \geq 0$ for all~$\rho \in L^2(\torus_L)_{\geq 0}$, hence,~$G \in \COP(\torus_L)$.

  Finally, since~$F$ is invariant under the torus, to show that~$F - J \in \PSD(\torus_L)$ it suffices that the Fourier coefficient of~$\phi$ at~$0$ as a function on~$\torus_L$ is at least~$1$. Indeed, since~$f$ has compact support, again by Equation~\eqref{eqn:quotient-measure-is-invariant-Radon-measure},
  \[
    \begin{split}
      \int_{\torus_L} \phi(x) \, d\mu_L(p(x)) &= \int_{\torus_L} \sum_{v \in L\Z^n} f(x+v)\, d\mu_L(p(x)) = \int_{\R^n} f(x)\, dx = \widehat{f}(0) = 1.
  \end{split}\]
  This concludes the proof.
\end{proof}

\section{Exactness and convergence of the copositive hierarchy}%
\label{sec:exactness-and-convergence-sphere-packing}
We now prove exactness of the copositive bound~$\vartheta_{\SP}^X(\COP(\R^n)_{\inv})$ for any~$C_c^{\infty}(\R^n) \subset X \subseteq C_0(\R^n)$. As a consequence, we will see that if~$X$ is the space of Schwartz functions, the bound is sharp; this is Corollary~\ref{cor:cop-sp-theta-exact}. The following inequality does most of the work to show exactness of the copositive formulation of~$\Delta_n$ and convergence of the copositive hierarchy. Let the set~$\K_{\fin}(\R^n)$ be the family of compact subsets of~$\R^n$ with nonzero measure, and recall the operators~$\convol_K$ from Chapter~\ref{ch:complete-positivity-under-symmetry} given by~$\convol_Kf(x,y) = f(x-y)$ for all~$x$ and~$y \in K$.

\begin{lemma}%
  \label{lem:sp-theta-smaller-than-torus-theta}
  Let~$n \geq 2$ be an integer and~$K \in \K_{\fin}(\R^n)$. If~$\cC(K) \subseteq C_{\sym}(K)$ is a convex cone and~$\cC \subseteq C_0(\R^n)$ is a convex cone containing~$\convol_K^* \cC(K)$, then
  \[
    \vartheta_{\SP}^{C_c(\R^n)}(\cC) \leq \frac{\vartheta^*(G_{K}, \cC(K))}{\lambda(K)}.
  \]
\end{lemma}
\begin{proof}
  If~$\vartheta^*(G_K, \cC(K))$ is infeasible, the conclusion follows. So, let~$(t, F, G)$ be a feasible solution to~$\vartheta^*(G_K, \cC(K))$, and let~$f = \convol_{K}^* F$ and~$g = \convol_{K}^* G$. Then,~$f$ and~$g$ are in~$C_c(\R^n)$. Indeed, the support of both functions is contained in~$K+K$. Since~$\convol_K^* = \Av_{\R^n} \Res_K^*$, a straightforward calculation shows that, since~$F$ and~$G$ are uniformly continuous,~$f$ and~$g$ are continuous.

  If~$\|x\| \geq 1$, then
  \[
    f(x) + g(x) = \int_{\R^n} \bigl(\Res_K^*F(y,x+y) + \Res_K^*G(y,x+y)\bigr)\, dy \leq 0.
  \]
  That~$\convol_K^* G \in \cC(K)$ follows from the assumption on~$\cC$.

  Since~$F \in \PSD(K)_{\cont}$, for all finite~$U \subseteq \R^n$ the matrix~$\Res_U \Res_K^*F$ is positive semidefinite. Since~$\convol_K^* = \Av_{\R^n} \Res_K^*$, it follows that
  \[
    \sum_{u, v\in U} f(u-v) x_u x_v = \int_{\R^n} \sum_{u, v \in U + z} \Res_K^*F(u, v) x_{u - z} x_{v - z}\, dz \geq 0
  \]
  for all finite~$U \subseteq \R^n$ and~$x \in \R^U$. By~\cite[Proposition 3.35]{Folland2016AAnalysis}, it follows that~$f \in \PT(\R^n)$.

  Moreover, since~$F - J \in \PSD(K)$,
  \[
    0 \leq \langle F - J, J\rangle = \int_{\R^n} \int_{\R^n} F(y, x+y)\, dydx - \lambda(K)^2 = \widehat{f}(0) - \lambda(K)^2,
  \]
  so that~$\widehat{f}(0) \geq \lambda(K)^2$. Finally,
  \[
    f(0) + g(0) = \int_{\R^n} F(z,z) + G(z,z)\, dz \leq t \lambda(K),
  \]
  so that~$(f(0) + g(0)) / \widehat{f}(0) \leq t / \lambda(K)$.
\end{proof}

\begin{theorem}%
  \label{thm:smooth-cop-sp-theta-is-exact}
  For all integers~$n \geq 2$,~$\Delta_n = \vartheta_{\SP}^{C^{\infty}_c(\R^n)}(\COP(\R^n)_{\inv})$.
\end{theorem}
\begin{proof}
  The inequality~$\vartheta_{\SP}^{C^{\infty}_c(\R^n)}(\COP(\R^n)_{\inv}) \geq \Delta_n$ follows from the inequality~\eqref{eqn:SP-theta-under-inclusion} and Theorem~\ref{thm:C0-theta-is-upper-bound}.

  For the other inequality, the strategy is to use Theorem~\ref{thm:cop-theta-exact-packing-graphs} and the expression~\eqref{eqn:sphere-packing-density-limit-over-cubes}, which says that
  \[
    \Delta_n = \lim_{L \to \infty} \frac{\alpha(G_{C_L})}{L^n} =  \lim_{L \to \infty} \frac{\vartheta^*(G_{C_L}, \COP(C_L)_{\cont})}{L^n}.
  \]

  By Theorem~\ref{thm:Polyas-theorem-for-invariant-cones} and~\ref{thm:polar-identities},~$\COP(\R^n)_{\inv} = \ccone \bigcup_{A \in \B_{\fin}}\convol_A^* \COP(A)$, thus it satisfies the conditions of Lemma~\ref{lem:sp-theta-smaller-than-torus-theta}, and
  \[
    \vartheta^{C_0(\R^n)}_{\SP}(\COP(\R^n)_{\inv}) \leq \vartheta^{C_c(\R^n)}_{\SP}(\COP(\R^n)_{\inv}) \leq \frac{\vartheta^*(G_{C_L}, \COP(C_L)_{\cont})}{L^n}
  \]
  for all~$L$. Taking the limit over~$L$ shows that~$\Delta_n = \vartheta_{\SP}^{C_c(\R^n)}(\COP(\R^n)_{\inv})$.

  It remains to show that every feasible solution~$(f, g)$ of~$\vartheta_{\SP}^{C_c(\R^n)}(\R^n)$ can be approximated in a suitable way by smooth feasible solutions. To achieve this, make use of Leibniz's rule for differentiation under the integral: if~$\rho$ is smooth and~$\psi$ continuous with compact support, then for all~$i$,
  \[
    (\partial/\partial x_i) (\rho \cor \psi) = \rho \cor ((\partial / \partial x_i)\psi) = -\psi*((\partial / \partial x_i)\rho),
  \]
  whence the cross-correlation of a smooth function with a compactly supported continuous function is smooth.

  Thus, take a smooth approximate identity~$(\phi_i')_{i \in \N}$ where each~$\phi_i'$ has support in~$B_{1/(2i)}(0)$ for~$i \geq 2$. For example, the bump functions defined by
  \[
    \phi_i'(x) =
    \begin{cases}
      N_i \exp^{-1/(1 - 2i\|x\|^2)} &\text{if } \|x\|< 1/(2i),\\
      0 & \text{otherwise,}
    \end{cases}
  \]
  with~$N_i$ a suitable normalization constant, satisfy the requirements. Then, the functions~$\phi_i = \phi_i' \cor \phi_i'$ form a smooth approximate identity where every~$\phi_i$ is supported in~$B_{1/i}(0)$.

  Recall that if~$r \in \R$ and~$\rho : \R^n \to \R$, then~$\rho_r(x) = \rho(rx)$. Note that~$f_r \in \PT(\R^n)$ and~$g_r \in \COP(\R^n)_{\inv}$ for all~$r \in \R$. Let~$r_i = 1/(1-1/i)$. Then, by Lemma~\ref{lem:action-of-pt-approximate-identity} the function~$r_i^n \phi_i \cor f_{r_i}$ is in~$\PT(\R^n)$ and~$r_i^n \phi_i \cor g_{r_i}$ is in~$\COP(\R^n)_{\inv}$. Moreover, all these functions are smooth with compact support, and by Lemma~\ref{lem:approximate-by-correlation} they are nonpositive on~$x$ such that~$\|x\|\geq 1$, converge uniformly to~$f$ and~$g$ respectively, and~$\int r_i^n \phi_i \cor f_{r_i}(x)\, dx = \int f(x)\, dx$ for all~$i$. This shows the theorem.
\end{proof}

\begin{corollary}%
  \label{cor:cop-sp-theta-exact}
  Let~$n \geq 2$ be an integer. If~$X$ is a vector space such that~$C_c^{\infty}(\R^n) \subseteq X \subseteq C_0(\R^n)\cap L^1(\R^n)$, then~$\Delta_n = \vartheta_{\SP}^X(\COP(\R^n)_{\inv})$. In particular,~$\Delta_n = \vartheta_{\SP}^{\swrtz(\R^n)}(\COP(\R^n)_{\inv})$.
\end{corollary}
\begin{proof}
  This follows immediately from Theorem~\ref{thm:C0-theta-is-upper-bound}, Theorem~\ref{thm:smooth-cop-sp-theta-is-exact}, and Equation~\eqref{eqn:SP-theta-under-inclusion}.
\end{proof}

We conclude with the convergence of a copositive hierarchy for~$\Delta_n$, which is a simple corollary of all the work we have done before.
\begin{theorem}%
  \label{thm:convergence-copositive-hierarchy-SP}
  If~$n \geq 2$ is an integer, then
  \[
    \Delta_n = \lim_{r \to \infty} \vartheta_{\SP}^{C_0(\R^n) \cap L^1(\R^n)}(\bC_r(\R^n)_{\inv}) = \lim_{r \to \infty}\vartheta_{\SP}^{C_c(\R^n)}(\bC_r(\R^n)_{\inv}).
  \]
\end{theorem}
\begin{proof}
  Convergence of the copositive hierarchy for compact packing graphs, Theorem~\ref{thm:convergence-copositive-hierarchy-packing-graphs}, says that
  \[
    \Delta_n = \lim_{L \to \infty} \lim_{r \to \infty} \frac{\vartheta^*(G_{C_L}, C_r(C_L)_{\cont})}{L^n}.
  \]
  By Lemma~\ref{lem:local-definition-bC} and Theorem~\ref{thm:polar-identities}, the inclusions~$\bC_r(\R^n)_{\inv} \supseteq \convol_K^*C_r(K)$ hold for all~$K \in \K_{\fin}(\R^n)$, hence by Lemma~\ref{lem:sp-theta-smaller-than-torus-theta}, for all~$L$,
  \[
    \vartheta_{\SP}^{C_c(\R^n)}(\bC_r(\R^n)_{\inv}) \leq \frac{\vartheta^*(G_{C_L}, C_r(C_L)_{\cont})}{L^n}.
  \]
  First taking the limit over~$r$ and then over~$L$ concludes the proof.
\end{proof}

\section{Comparison to other known bounds}%
\label{sec:comparison-with-known-SP-bounds}
Can we use these copositive hierarchies to improve known bounds? A first step would be to see whether it improves on the Cohn-Elkies bound, which in our notation is~$\vartheta_{\SP}^{C_0(\R^n) \cap L^1(\R^n)}(\{0\})$; clearly, it is at most as strong as~$\vartheta_{\SP}^{C_0(\R^n) \cap L^1(\R^n)}(\bC_r(\R^n)_{\inv})$ for each~$r$.

To improve on the Cohn-Elkies bound, it is not necessary to optimize over all of~$\bC_r(\R^n)_{\inv}$; it is enough to only take a suitable subset. However, finding feasible solutions that are not feasible for the Cohn-Elkies bound seems to be difficult, and is work in progress.

\bookmarksetup{startatroot}

\chapter*{Conclusion and discussion}
Exact completely positive and copositive formulations were introduced for four classes of extremal problems in geometry that are modeled as an independence number problem on a graph. These formulations were used to define a completely positive or copositive hierarchy, and prove they converge to the independence number. Table~\ref{tab:current-standing} shows the current status of these results. These hierarchies were then compared to extensions of the moment and block moment hierarchy for two classes of these problems. Two applications were worked out in detail.

\newcommand{\PreserveBackslash}[1]{\let\temp=\\#1\let\\=\temp}
\newcolumntype{L}[1]{>{\PreserveBackslash\raggedright}p{#1}}
\begin{table}[b]
  \centering
  \begin{tabular}{@{}L{11em}L{10em}L{10em}@{}} \toprule
    & \textsl{Exact completely positive / copositive formulation} & \textsl{Converging completely positive / copositive hierarchy}\\\midrule
    \textsl{Measurable $k$-uniform hypergraph}&&\\
    $\smallbullet$ with a thick edge set& This thesis.& If~$k=2$; this thesis.\\
    $\smallbullet$ compact homogeneous& If the group admits a right-invariant density metric; this thesis and~\cite{DeCorte2022CompleteSets}.& If~$k=2$ on a continuous, compact, two-point homogeneous space with real dimension at least~$2$; this thesis and~\cite{Bekker2026OptimizationSpaces}.\\\midrule
    \textsl{Compact packing graph}& If the vertex set is a metric space~\cite{Dobre2016AGraphs}.& If the vertex set is a metric space~\cite{Kuryatnikova2017Approximating,Kuryatnikova2019TheProblems}.\\\midrule
    \textsl{Distance-avoiding sets in~$\R^n$}& If~$n \geq 2$~\cite{DeCorte2022CompleteSets}.& If~$n \geq 2$; this thesis.\\\midrule
    \textsl{Sphere packing in~$\R^n$}.& If~$n \geq 2$; this thesis.& If~$n \geq 2$; this thesis.\\
    \bottomrule
  \end{tabular}
  \bigskip

  \caption{Known sufficient conditions for the existence of an exact completely positive or copositive formulation a corresponding hierarchy for the independence number of a (measurable) graph, and where they were first described. It is assumed everywhere that the independence number is nonzero, and that underlying measure space of a measurable graph is assumed to be finite and countably generated.}%
  \label{tab:current-standing}
\end{table}

\section*{Discussion of the theory}
\subsection*{Finite measure spaces and density}
The exactness statements of the completely positive hierarchy for measurable graphs on a finite countably generated measure space are quite complete. In both the thick as in the homogeneous setting, we used the concept of density points to move from a set that might have a zero-measure subset of edges to an independent set. The use of density seems inevitable. However, the notions of density system and thick set as presented in Chapter~\ref{ch:completely-positive-formulations-meas-ind-num} are crude, and there might be more refined definitions. This would form an interesting topic of study.

For example, consider the following argument in the homogeneous setting. Let~$\ortho(n)$ be the orthogonal group on~$\R^n$ with Haar measure~$\mu$, let it act on~$S^{n-1}$, and let~$\nu$ be the pushforward onto~$S^{n-1}$ of~$\mu$. Clearly, the edge set~$E = \{\, (x,y) : x, y \in S^{n-1}, x^{\tr} y = 0\, \}$ has measure~$\nu^2(E) = 0$, hence it is not thick.

However, we can directly a measure on~$E$ that has better properties. Let~$x_0$ and~$y_0 \in S^{n-1}$ be any choice of points such that~$x_0^{\tr} y_0 = 0$. The orbit of a pair under the diagonal action of~$\ortho(n)$ is characterized by the inner product of the coordinates, thus~$E = \ortho(n) \{ (x_0, y_0) \}$. We can then define a measure on~$E$ by
\[
  \rho(S) = \mu(\{ T : (T x_0, T y_0) \in S\}).
\]
It is fair to say that~$E$ is thick with respect to~$\rho$; Castro-Silva~\cite{Castro-Silva2023GeometricalConfigurations} used such arguments to formulate zero-measure-removal lemmas for~$S^{n-1}$ and~$\R^n$ for similar edge sets. Thus, it seems that thickness should be measured not by a measure on~$V$, but by a measure on~$E$.

\subsection*{Converging hierarchies in the thick setting}
In contrast to the completeness of the exactness statements, Table~\ref{tab:current-standing} also shows several interesting gaps. For a measurable graph on a finite and countably generated measure space, the view adopted in this thesis is that the programs essentially optimize over a set of $\mi{p}$-integrable functions. However, in neither setting does the proof of convergence of the completely positive hierarchy extend to $k$-uniform hypergraphs with~$k > 2$.

The shortcoming of the $L^p$ formulation can be understood by thinking of the normalization constraint of~$\vartheta_{\bt}$ in the thick setting. For a measurable graph~$G = (V, E)$ with thick edge set~$E$, the normalization~$\|A\|_{\B^1} = 1$ bounds the feasible region, since it bounds the $L^2$-norm, which we used to show that the feasible regions of the programs~$\vartheta_{\bt}(G, C_r(V)^*)$ lie in a common compact space. However, this result can be executed analogously in the space of trace-class operators~$\B^1(L^2(V))$, since the edge constraints can be written in terms of linear functionals in a predual of the trace class: the space of compact operators. Thus, using Banach-Alaoglu, we obtain the same compactness result under the weak* topology on~$\B^1(L^2(V))$ with respect to the duality with the space of compact operators, which is a more specialized statement. This seems to be a more natural setting, since we only consider trace class kernels to begin with.

For this reason, it seems that the correct extension to $k$-uniform hypergraphs with~$k \geq 2$ would use a normalization constraint given by a \textit{symmetric nuclear norm}
\[
  \|A\|_{\mathcal{N}_p} = \inf \biggl\{ \sum_{n \in \N} |\lambda_n| \|\phi_n\|_{p}^k : A = \sum_{n \in \N} \lambda_n \phi_n^{\otimes k}\biggr\}.
\]
The infimum is over every such expansion of~$A$ with~$\lambda_n \in \R$ and~$\phi_n \in L^{p}(V)$ for suitable~$p$. Whether bounding this norm forces all required properties requires closer investigation.

\subsection*{Converging hierarchies in the homogeneous setting}
As mentioned in Section~\ref{ch:Completely positive programming on compact spaces}.\ref{sec:discussion-finite-measure-spaces}, the convergence of the completely positive hierarchy in the homogeneous setting can be readily extended to the continuous, compact, two-point homogeneous spaces with real dimension at least~$2$: the sphere, the real, complex and quaternionic projective spaces, and the octonionic projective plane. Analogues on $k$-uniform homogeneous hypergraphs on these spaces also seem within reach, using the results on configuration-avoiding sets by Castro-Silva~\cite{Castro-Silva2023GeometricalConfigurations}.

The challenge lies in extending the convergence result to other homogeneous spaces. As explained in Section~\ref{ch:Completely positive programming on compact spaces}.\ref{sec:discussion-finite-measure-spaces} the method we used already fails for homogeneous graphs~$G$ on the circle, even though~$\vartheta_{\st}(G, C_0(S^1)) = \alpha(G)$.

If the vertex set is a compact group, we can limit the feasible regions to a space spanned by the functions of positive type on the group: the Fourier algebra. This approach is already present in the graph case: when~$V$ is compact, the image of the trace class kernels under the averaging operator lies in the Fourier algebra. It would be interesting to better understand the role of the Fourier algebra, for example to better understand the relation between the behavior of the hierarchies and the representation theory of the group. It is, however, unclear whether this will help to extend convergence of the completely positive hierarchy to a larger class of compact groups. On the positive side, there are $k$-tensor analogues of the Fourier algebra~\cite{Todorov2010MultipliersAlgebras}, which could perhaps form a setting for hierarchies for $k$-uniform hypergraphs on~$S^{n-1}$ and related spaces.

\subsection*{Compact packing hypergraphs}
How to define compact packing hypergraphs is still an open problem. There are simple conditions on a $k$-uniform hypergraph with~$k > 2$ that would imply all results on copositive programming in this thesis. For example, if~$V$ is a metric space with an edge set~$E \subseteq \Sub{V}{=k}$, asking for all~$r \geq 0$ that~$\ind_{r} = \bigsqcup_{i = 0}^r \ind_i$ as topological spaces under the restriction of the standard topology would be sufficient, as is the case for compact packing graphs. On the other hand, this definition seems too limiting, as the application studied in Chapter~\ref{ch:almost-equiangular} shows. In the case of that particular chapter, simply replacing the standard topology on~$\Sub{V}{r}$ by the disjoint union topology~$\bigsqcup_{i=0}^r \Sub{V}{i}$ seems satisfactory, but the precise details need more thought, since such a space differs topologically from~$\Sub{V}{r}$; most notably, compactness of~$\Sub{V}{r}$ under this topology is no longer implied by compactness of~$V$. See also the discussion in Section~\ref{ch:Completely positive programming on compact spaces}.\ref{sec:what-about-hypergraphs}.

\subsection*{Euclidean space}
It is likely that the results for problems on Euclidean space can be extended to problems on $k$-uniform homogeneous hypergraphs with vertex set~$\R^n$ by similar methods. Again, to do this properly, a first step would be to investigate the relation of these optimization problems and the Fourier algebra.

In Chapter~\ref{ch:sphere-packing} we saw that the optimal sphere-packing density can be approximated using copositive Schwartz functions. It is not clear that this also holds when one only considers those Schwartz functions that lie in the cone~$\bigcup_r\bC_r(\R^n)_{\inv}$. It would be interesting to know whether this is true, since these are the functions that are used in practice.

\section*{Discussion of the computations}
For distance-avoiding set on~$\R^n$, the completely positive hierarchy offers the first bound taking correlations between more than two points into account. Attempts were made to implement a version of a related three-point bound for the 1-avoiding-sets problem on~$\R^2$. Although in theory this would lead to bounds that are stronger than the optimization method introduced by DeCorte, Oliveira, and Vallentin~\cite{DeCorte2022CompleteSets}, in practice it turns out that within a reasonable runtime, there is no improvement. Whether this is inherent to this optimization problem, or whether it has to do with the way the problem is modeled is not clear. This is still work in progress.

A similar conclusion holds for the sphere-packing problem, see also the discussion in Section~\ref{ch:sphere-packing}.\ref{sec:comparison-with-known-SP-bounds}. Although new feasible solutions coming from~$\bC_r(\R^n)_{\inv}$ were found, none of these make an improvement on the objective value. This too requires more experimentation.

\appendix
\chapter[Preliminaries]{Preliminaries from analysis}%
\label{ch:Appendix}
This appendix contains the preliminaries to this thesis that I consider more elementary, or which did not have a natural place in the relevant chapters.

\section{Locally convex spaces, duality, and cones}%
\label{sec:appendix-topological-vector-spaces}
Our reference for all matters convex is the book \textit{Convexity} by Simon~\cite{Simon2011Convexity}. All vector spaces are over~$\R$, all topological vector spaces are Hausdorff.

A~\defi{locally convex space}\index{locally convex space} is a topological vector space~$X$ with a topology that is generated by a set of seminorms~$\{ \rho_j \}_{j \in J}$. This means that a net~$(x_i)_{i \in I}$ converges to~$x$ if and only if for all~$j \in J$ the net~$(\rho_j(x_i - x))_{i \in I}$ converges to~$0$. The space~$X$ is Hausdorff if and only if~$\bigcap_j \{\, x : \rho_j(x) = 0\, \} = \{0\}$.

Let~$X$ and~$Y$ be vector spaces. A~\defi{duality}\index{duality} between~$X$ and~$Y$ is a bilinear map~$\langle \cdot\, , \cdot\rangle : X \times Y \to \R$\index{ <>@$\langle \cdot\, , \cdot\rangle$} such that if~$x \neq 0$ there is a~$y$ such that~$\langle x, y \rangle \neq 0$, and if~$y \neq 0$ there is an~$x$ such that~$\langle x, y \rangle \neq 0$. If there exists a duality between~$X$ and~$Y$, we call~$(X, Y)$ a~\defi{dual pair}\index{dual pair} of vector spaces.

Let~$(X, Y)$ be a dual pair of vector spaces with duality~$\langle \cdot\, , \cdot \rangle$. Then~$X$ is a locally convex space under the topology generated by the set of seminorms~$\{|\langle \cdot\, , y\rangle|\}_{y \in Y}$. This topology is called the~\defi{$\sigma(X, Y)$-topology}\index{topology!s topology@$\sigma(X,Y)$-topology}. We will call it the~\defi{weak topology}\index{topology!weak topology} when the duality is clear. The seminorms~$\{|\langle x, \cdot\rangle|\}_{x \in X}$ make~$Y$ into a locally convex space. The induced topology on~$Y$ is called the~\defi{weak* topology}\index{topology!weak* topology}; it is the same as the~$\sigma(Y, X)$-topology when we interchange~$X$ and~$Y$ everywhere.

The map~$y \mapsto \langle \cdot\, , y\rangle$ identifies~$Y$ with a subspace of~$X'$. Then,~$X^* = Y$ under the weak topology~\cite[Proposition 5.1]{Simon2011Convexity}. Any locally convex topology on~$X$ such that~$X^* = Y$ is called an~\defi{$(X, Y)$-dual topology}\index{topology!dual topology}. The weak topology is by definition the weakest dual topology.

A linear map~$A : X_1 \to X_2$ induces a unique~\defi{linear adjoint}\index{adjoint!linear adjoint}~$A'$ from~$X_2'$ to~$X_1'$ by~$A'(l)(x) = l(A(x))$ for all~$x \in X_1$ and~$l \in X_2'$. With these notions, continuity of~$A$ under weak topologies is reduced to an algebraic condition. If~$A^*$ from the following lemma exists and is continuous, it is called the~\defi{continuous adjoint}\index{adjoint!continuous adjoint} of~$A$. In this case,~$(A^*)^* = A$.
\begin{lemma}%
  \label{lem:continuity-dual-topologies}
  Let~$(X_1, Y_1)$ and~$(X_2, Y_2)$ be dual pairs of vector spaces with respective dualities~$\langle \cdot\, , \cdot\rangle_1$ and~$\langle \cdot\, , \cdot\rangle_2$. Let~$A : X_1 \to X_2$ be a linear map.

  If~$A$ is continuous under a dual topology on~$X_1$ and any topology on~$X_2$ such that~$\langle \cdot\, , y\rangle_2$ is continuous for all~$y \in Y_2$, then~$A'(Y_2) \subseteq Y_1$.

  If~$A'(Y_2) \subseteq Y_1$,~$X_1$ and~$Y_2$ are equipped with dual topologies, and~$X_2$ and~$Y_1$ are equipped with weak topologies, then both~$A$ and the map~$A^* : Y_2 \to Y_1$ given by~$A^*(y) = A'(y)$ are continuous.
\end{lemma}
\begin{proof}
  Let~$A$ be continuous with respect to an $(X_1,Y_1)$-dual topology on~$X_1$ and any topology on~$X_2$. Suppose that for all~$y \in Y_1$, the functionals~$\langle \cdot, y\rangle_2$ are continuous, and fix~$y \in Y_2$. Then, the functional~$x \mapsto \langle A(x), y\rangle_2$ is continuous. The definition of~$A'$ and the identification~$Y_1 = X_1^*$ together warrant~$\langle A(x), y\rangle_2 = \langle x, A'(y)\rangle_1$, and~$A'(y) \in Y_1$.

  If~$A'(Y_2) \subseteq Y_1 = X_1^*$, the functional~$x \mapsto \langle A(x), y\rangle_2 = \langle x, A'(y)\rangle_1$ is continuous for all~$y \in Y_2$. So, if~$(x_i)_{i \in I}$ is a converging net in~$X_1$, then~$(A(x_i))_{i \in I}$ converges weakly in~$X_2$, and~$A$ is continuous under a dual topology on~$X_1$ and the weak topology on~$X_2$.

  Continuity of~$A^*: Y_2 \to Y_1$ follows by applying the previous argument to~${A'}'$ restricted to~$Y_2$, and~${A'}' = A$.
\end{proof}

Let~$X$ and~$Y$ be vector spaces, and suppose~$X$ has an $(X,Y)$-dual topology. Recall that here, all cones are convex. The closure of a convex set~$S \subseteq X$ under the chosen dual topology is equal to its closure under the weak topology~\cite[Theorem 5.2]{Simon2011Convexity}. Let~$S \subseteq X$ be any subset. The \defi{dual cone}\index{dual cone} of~$S$ is
\[
  S^* = \{\, y \in Y : \langle x, y\rangle \geq 0\text{ for all } x \in S\, \}.
\]
It is a convex cone and closed under all $(Y, X)$-dual topologies on~$Y$. Since the duality is bilinear and continuous under dual topologies, $S^* = (\ccone S)^*$. Taking the dual reverses inclusion: if~$S_1 \subseteq S_2$ then~$S_2^* \subseteq S_1^*$.

We refer to the following identities many times.
\begin{theorem}%
  \label{thm:polar-identities}
  Let~$(X_1, Y_1)$ and~$(X_2, Y_2)$ be dual pairs of vector spaces, each of~$X_1$,~$X_2$,~$Y_1$, and~$Y_2$ with dual topologies. Let~$A: X_1 \to X_2$ be a continuous linear map with continuous adjoint~$A^*$, let~$S_X \subseteq X_1$, let~$S_Y \subseteq Y_1$, and let~$\{S_i\}_{i \in I}$ be an arbitrary family of subsets of~$X_1$. Then:
  \begin{enumerate}
    \item[(i)] $(S_X^*)^* = \ccone S_X$;
    \item[(ii)] $(A S_X)^* = {A^*}^{-1} S_X^*$;
    \item[(iii)] $({A^*}^{-1} \ccone S_Y)^* = \cl A S_Y^*$;
    \item[(iv)] $\bigl(\bigcap_{i \in I} S_i^*\bigr)^* = \ccone \bigcup_{i \in I} S_i$.
  \end{enumerate}
\end{theorem}
\begin{proof}
  \textup{(i)} This follows from~$S_X^* = (\ccone S_X)^*$ and~\cite[Proposition 5.5, Example 5.9]{Simon2011Convexity}, since~$\ccone S_X$ is convex and contains~$0$.

  \textup{(ii)} Denote by~$\langle \cdot\, , \cdot \rangle_i$ be the duality between~$X_i$ and~$Y_i$. Let~$s \in S_X$ and~$y \in Y_2$. From~$\langle As, y\rangle_2 = \langle s, A^* y\rangle_1$ it follows immediately that~$\langle As, y\rangle_2 \geq 0$ if and only if~$A^*y \in S_X^*$. So~$(AS_X)^* = {A^*}^{-1}S_X^*$.

  \textup{(iii)} Since~$A$ is linear and~$S_Y^*$ is a convex cone,~$AS_Y^* = \cone AS_Y^*$. Then, by~\textup{(i)} and~\textup{(ii)}:
  \[
    \cl AS_Y^* = \ccone AS_Y^* = ((AS_Y^*)^*)^* = ({A^*}^{-1}(S_Y^*)^*)^* = ({A^*}^{-1}\ccone S_Y)^*.
  \]

  \textup{(iv)} The inclusion~$\ccone \bigcup_i S_i \subseteq \bigl(\bigcap_i S_i^*\bigr)^*$ holds if~$\bigcup_i S_i \subseteq \bigl(\bigcap_i S^*_i\bigr)^*$ holds. Indeed, each~$S_i^*$ is a closed convex cone, and an intersection of closed convex cones is again a closed convex cone. Then, apply~$\ccone$ to both sides.

  Let~$j \in I$,~$s \in S_j$, and~$y \in \bigcap_i S^*_i$ In particular,~$y \in S_j^*$, so~$\langle s, y\rangle_1 \geq 0$. Thus,~$\bigcup_i S_i \subseteq \bigl(\bigcap_i S_i^* \bigr)^*$ and~$\ccone \bigcup_i S_i \subseteq \bigl( \bigcap_i S_i^*\bigr)^*$.

  Suppose that~$y \in Y_1 \setminus S_j^*$ for some~$j \in I$. By the definition of the dual there exists an~$s \in S_j$ such that~$\langle s, y\rangle_1 < 0$. But~$S_j$ is contained in~$\ccone \bigcup_i S_i$, so~$y$ is not in~$\bigl( \ccone \bigcup_i S_i\bigr)^*$. Then,~$\bigl( \ccone \bigcup_i S_i \bigr)^* \subseteq \bigcap_i S_i^*$. By taking the dual again and applying~\textup{(i)}, it follows that~$\bigl(\bigcap_i S_i \bigr)^* = \ccone \bigcup_i S_i$.
\end{proof}

\section{\texorpdfstring{$p$-integrable functions}{p-integrable functions}}%
\label{subsec:spaces-of-meas-funcs}
All functions and measures here are real valued.

Let~$(V, \A, \mu)$ be a measure space with $\mi{\sigma}$-algebra~$\A$ and measure~$\mu$. The integrals~$\|f\|_p = \mu(|f|^p)^{1/p}$\index{ \textbar p @$\lVert\, \cdot\, \rVert_p$} are defined for all~$0 < p < \infty$. The \defi{essential supremum}\index{essential supremum} of~$f$ is~$\|f\|_{\infty} = \inf \{\, a : \mu(\{\, x : f(x) \geq a\,\})=0\,\}$\index{ \textbar zzz@$\lVert\, \cdot\, \rVert_{\infty}$}. It is the smallest upper bound on~$f$ outside a set of measure~$0$.

For~$0 < p < \infty$ a function~$f$ from~$V$ to~$\R$ is called \defi{$\mi{p}$-integrable}\index{p integrable@$\mi{p}$-integrable} if~$\|f\|_p$ is finite. When~$\|f\|_{\infty}$ is finite~$f$ is called \defi{essentially bounded}\index{essentially bounded}. Sometimes we will just say that~$f$ is bounded to mean the same. We say that two measurable functions are \defi{$\mi{\mu}$-equivalent}\index{mu equivalent@$\mi{\mu}$-equivalent} if they are equal \defi{almost everywhere}\index{almost everywhere}, that is, they are equal outside a set of measure~$0$. For~$0 < p \leq \infty$ the space of equivalence classes of real-valued $\mi{p}$-integrable functions on~$(V, \A, \mu)$ is denoted~$L^p(V, \A, \mu)$ or simply~$L^p(V)$\index{ L p@$L^p(V)$}.

The map~$f \mapsto \|f\|_p$\index{ \textbar p@$\lVert\, \cdot\, \rVert_p$} induces a norm on~$L^p(V)$ whenever~$p$ is at least~$1$.\index{norm!Lp norm@$L^p$ norm} The spaces~$L^p(V)$ with~$1 \leq p \leq \infty$ are Banach spaces under their respective norms. Under this norm, the dual of~$L^p(V)$ is~$L^q(V)$ with~$1 \leq p, q \leq \infty$ such that~$1/p+ 1/q = 1$, or if~$p = 1$ and~$q = \infty$. For such~$p$ and~$q$, denote the duality between~$L^p(V)$ and~$L^q(V)$ by~$\langle \cdot\, , \cdot\rangle$. We have the following inequalities.

If~$1 \leq p, q \leq \infty$ such that~$1/p + 1/q = 1$, then~$|\langle f, g\rangle| \leq \|f\|_p \|g\|_q$. This is called~\defi{Hölder's inequality}\index{inequality!Hölder's inequality}.

If~$V_1$ and~$V_2$ are two $\sigma$-finite measure spaces with respective measures~$\mu_1$ and~$\mu_2$, and~$F \in V_1 \times V_2 \to \R$ is measurable, then for all~$1 \leq p < \infty$, we have
\begin{multline*}
  \biggl(\int_{V_2} \biggl|\int_{V_1} F(v, w)\, d\mu_1(v)\biggr|^pd \mu_2(w)\biggr)^{1/p}\\
  \leq \int_{V_1}\biggl(\int_{V_2}|F(v,w)|^p\, d\mu_2(w)\biggr)^{1/p}\, d\mu_1(v),
\end{multline*}
and similar for~$p = \infty$. This is called~\defi{Minkowski's inequality}\index{inequality!Minkowski's inequality}.

Let~$V$ be a measure space with measure~$\mu$,~$W$ a set equipped with a $\sigma$-algebra, and~$q : V \to W$ a measurable function. Then~$q$ induces a measure~$q_*\mu$\index{ q*mu@$q_*\mu$} on~$W$ by~$\int_W f(w)\, dq_*\mu(w) = \int_V f(q(v))\, d\mu(v)$, called the~\defi{pushforward}\index{measure!pushforward measure}\index{pushforward} of~$\mu$ under~$q$. This is nothing more than the change of variables formula. A function~$f \in L^p(W, q_*\mu)$ if and only if~$f \smallcirc q \in L^p(V, \mu)$.

\section{Radon measures}
Cohn~\cite{Cohn2013MeasureTheory} gives a complete overview of elementary measure theory. Let~$V$ be a locally compact Hausdorff space. A~\defi{Borel measure}\index{measure!Borel measure} on~$V$ is a measure on the Borel algebra, which we denote~$\mathscr{B}$\index{ B@$\mathscr{B}$}. A~\defi{Radon measure}\index{measure!Radon measure} is a measure on~$\mathscr{B}$ that is finite on compact sets, outer regular on Borel sets, and inner regular on open sets. That is,~$\mu$ is Radon if and only if it is finite on compact sets, for every Borel set~$A$
\[
  \mu(A) = \inf\{\, \mu(U) : A \subseteq U,\, U\text{ open}\, \},
\]
and if~$U \subseteq V$ is open, then
\[
  \mu(U) = \sup\{\, \mu(K) : K \subseteq U,\, K \text{ compact}\, \}.
\]
We denote the space of Radon measures by~$M(V)$. It is isometrically isomorphic to the continuous dual of~$C_0(V)$, with duality~$\langle f, \mu\rangle = \mu(f)$ for all~$\mu \in M(V)$ and~$f \in C_0(V)$.

\section{Spaces of continuous linear operators}%
\label{sec:operator-spaces}
Suppose~$X$ and~$Y$ are topological vector spaces. Denote the set of continuous linear operators~$X \to Y$ by~$\L(X, Y)$\index{ LXY@$\L(X, Y)$}. This space has many useful topologies. Here is a list of some of them. If~$X = Y$ is a Hilbert space, and the duality is the inner product, the list goes from the weakest topology to the strongest.

When~$(Y, Z)$ is a dual pair, the~\defi{weak operator topology}\index{topology!weak operator topology} is the topology of pointwise convergence under the weak topology on~$Y$. It is the topology generated by the set of seminorms~$\{A \mapsto |\langle Ax, z\rangle|\}_{(x,z) \in X \times Z}$. Convergence under this topology is called~\defi{weak convergence}\index{convergence!weak convergence}, or we say a net of operators~\defi{converges weakly}\index{converges!weakly}.

When~$Y$ is a normed space with norm~$\| \cdot \|$, the~\defi{strong operator topology}\index{topology!strong operator topology} is the topology of pointwise convergence under the norm. It is the topology generated by the seminorms~$\{ A \mapsto \|Ax\| \}_{x \in X}$. Convergence under this topology is called~\defi{strong convergence}\index{convergence!strong convergence}, or we say that a net of operators~\defi{converges strongly}\index{converges!strongly}.

When~$X$ and~$Y$ are normed spaces with norms~$\|\,\cdot\,\|_X$ and~$\|\,\cdot\,\|_Y$, then~$A$ is in~$\L(X, Y)$ if and only if it is bounded. That is, if and only if there is a~$c \geq 0$ for which~$\|Ax\|_Y \leq c \|x\|_X$ for all~$x \in X$. Denote the infimum over such~$c$ by~$\|A\|_{\op}$\index{ \textbar op@$\lVert\, \cdot\, \rVert_{\op}$}, the~\defi{operator norm}\index{operator norm} of~$A$. Equivalently,
\[
  \|A\|_{\op} = \sup \{\, \|Ax\|_Y : x \in X, \|x\|_X = 1\, \}.
\]
The topology of the operator norm is the topology of uniform convergence on bounded sets.

These topologies are all defined by seminorms, so they turn~$\L(X, Y)$ into a locally convex space. When~$X$ and~$Y$ are normed spaces, denote~$\L(X, Y)$ by~$\B(X, Y)$\index{ BXY@$\B(X, Y)$}. This is a normed space under the operator norm. When~$Y$ is a Banach space,~$\B(X, Y)$ is a Banach space. Denote~$\B(X) = \B(X,X)$\index{ BX@$\B(X)$}.

If~$(V, \A, \mu)$ is a measure space, then~$L^2(V)$ is a Hilbert space, and the space of bounded operators~$\B(L^2(V))$ has two subspaces that are important to us. The first is the space of~\defi{Hilbert-Schmidt operators}\index{Hilbert-Schmidt operator}, which can be identified as a Hilbert space with~$L^2(V^2)$. An element of~$L^2(V^2)$ is called a~\defi{kernel}\index{kernel}.

The second space of interest is the~\defi{trace class}\index{trace class}~$\B^1(L^2(V))$\index{ B1L2@$\B^1(L^2(V))$}. We say that a kernel~$K$ is of trace-class if there are orthonormal sets~$\Phi$ and~$\Psi$ and complex numbers~$\lambda_{\phi \psi}$ such that
\[
  K = \sum_{\phi \in \Phi, \psi \in \Psi} \lambda_{\phi \psi} \phi \otimes \psi
\]
such that~$\sum_{\phi, \psi} |\lambda_{\phi \psi}| < \infty$. If finite, the latter sum is independent of~$\Phi$ and~$\Psi$, and is known as the~\defi{trace norm}\index{norm!trace norm} of~$K$; we will denote it~$\|K\|_{\B^1}$\index{ \textbar B1@$\lVert \, \cdot\, \rVert_{B^1}$}. It can be shown that~$\|K\|_2 \leq \|K\|_{\B^1}$. See~\cite{Simon2005TraceApplications} for a detailed account

\section{Invariant measures}%
\label{sec:app-invariant-measures}
We mostly follow~\cite{Folland2016AAnalysis}. A \defi{topological group}\index{group!topological group} is a group~$\Gamma$ with a Hausdorff topology such that the multiplication and inversion maps are continuous. A~\defi{locally compact group}\index{group!locally compact group} is a topological group with a locally compact topology.

Denote the Borel algebra of~$\Gamma$ by~$\mathscr{B}$. A \defi{left Haar measure}\index{measure!Haar measure} on~$\Gamma$ is a nonzero Radon measure~$\mu$ on~$\Gamma$ such that for every~$E \in \mathscr{B}$ and~$g \in \Gamma$, we have~$\mu(g E) = \mu(E)$. If~$\Gamma$ is locally compact, a left Haar measure exists and is unique up to a constant multiple. We denote a choice of left Haar measure by~$\mu$. The group is called~\defi{unimodular}\index{group!unimodular group} if the left Haar measure is also right invariant, that is, if, for all~$g \in \Gamma$ and all~$E \in \mathscr{B}$, we have~$\mu(gE) = \mu(Eg) = \mu(E)$. If~$\Gamma$ is~$\sigma$-compact,~$(\Gamma, \mathscr{B}, \mu)$ is~$\sigma$-finite.

If~$H = \Stab(v_0)$ and we identify~$V$ with~$\Gamma / H$, then~$V$ is locally compact and Hausdorff. When both~$\Gamma$ and~$H$ are unimodular, then there is a $\Gamma$-invariant Radon measure on~$\Gamma / H$. Indeed, if~$\mu$ is a Haar measure on~$\Gamma$ and~$\nu$ a Haar measure on~$H$, then~$f \mapsto \bigl(x \mapsto \int_H f(x h)\, d\nu(h)\bigr)$ is a surjective map~$C_c(\Gamma) \to C_c(\Gamma / H)$. Let~$p: \Gamma \to \Gamma / H$ be the quotient map. Theorem 2.51 of~\cite{Folland2016AAnalysis} implies that if both~$\Gamma$ and~$H$ are unimodular, there exists a $\Gamma$-invariant Radon measure~$\omega$ on~$\Gamma / H$, i.e. for all~$g \in \Gamma$ and all measurable subsets~$E \subseteq V$, we have~$\omega(gE) = \omega(Eg) = \omega(E)$. This measure is unique up to a constant multiple, and is related to~$\mu$ and~$\nu$ by the Fubini-type formula
\begin{equation*}
  \int_{\Gamma} f(g)\, d\mu(g) = \int_{\Gamma / H} \int_H f(v h)\, d\nu(h)d\omega(p(v)),
\end{equation*}
for all~$f \in C_c(\Gamma)$. The formula extends to all integrable functions~$f$, as Reiter and Stegeman~\cite[\S 3.4]{Reiter2000ClassicalGroups} discuss in detail.

On the other hand, when~$H$ is compact, the quotient map~$p$ induces the pushforward measure~$p_*\mu$, which is Radon. Since~$p_*\mu$ defines a continuous linear functional on~$C_c(\Gamma / H)$ and is $\Gamma$-invariant, it is a $\Gamma$-invariant Radon measure, hence a multiple of~$\omega$. We therefore set~$\omega = p_*\mu$ if~$H$ is compact, and call it the~\defi{quotient measure}\index{measure!quotient measure} on~$\Gamma/H$.


\backmatter

\newcommand{\arXiv}[1]{arXiv:#1}
\newcommand{\doi}[1]{}
\newcommand{\version}[1]{}
\newcommand{\publisher}[1]{}

\bibliographystyle{amsinit}
\bibliography{references}

\printindex

\chapter*{Dankwoord / Acknowledgements}%
\label{ch:acknowledgements}

De eerste regels van dit dankwoord gaan naar mijn ouders, uiteraard. Voor de liefde, het geduld en dat jullie er altijd zijn geweest. Ze gaan ook naar Bronke, ook voor je liefde en voor het luisteren, en omdat je ook op jouw beurt geduld hebt moeten hebben, zeker het afgelopen jaar. Ik ga snel weer mee om met de paarden te wandelen (zonder aan wiskunde te denken).

Roos, dankjewel voor het minutieuze werk dat je hebt gestoken in het tekenen van de mooiste kaft van een proefschrift ooit. (Merk op, lezer, dat Hallard de Schildpad niet alleen zeer gedetailleerd is, maar bovendien Crofts schildpadschild-constructie accuraat verbeelt~\cite{Croft1967IncidenceIncidents}).

Next, a thank you to those that are mathematically closest to me. Fernando, not only for your expert guidance, but also for the (genuine!) fun while editing the same section of text line-by-line multiple times until it was up to your standard, which I hope has become my standard as well. In that spirit I should also thank Jackie Daytona (regular human bartender) for telling it straight when we are talking The Bullshit. Willem and Nando, besides the many insightful discussions we have had, I am lucky to have had such good friends as my closest colleagues. Willem,  having \textsc{covid} together in Norway is in my top three favorite memories of the past four years---a magical time. David, thanks for the many discussions, mathematical and otherwise.

The DMO group is filled with wonderful people, all of whom made my life in Delft enjoyable. In fear of missing names, the following list is purposefully incomplete and in no particular order, but everyone in the group has been a pleasure to be around. My thanks to Naqi and Qiaochu, for the fond memories Casper has of you; to Yuki, for the \scalebox{2}[1]{gains}; to Niels, for maximizing shareholder value three days per day; to Nicolaas, for making the entire group believe we do not know basic Dutch---\textit{schabouwelijk}; to Esther, our family-style communal bread lunches are something to remember; and to Cindy, Lara, Ananth, and everyone else, for the lunches and coffee breaks that delayed submission of this thesis by several weeks.

Finally, a thank you to the many amazing people of the optimization community that I have been lucky to meet. In particular, Christine and Philippe were kind enough to host me during research visits to Bordeaux and Toulouse, respectively. In all, the community is a fun, friendly, and welcoming bunch, enough so to convince me to stay.

\sectionbreak

I am not the only person who worked hard on this thesis. First and foremost, Dorothée, I would not have been able to work my way through the bureaucracy that is TU Delft without you. Joffrey, thank you for the extensive support with everything computer related, and maintaining the computational cluster at TU Delft that most of the optimization problems were solved on. Nando and David, thanks for all the help with \texttt{ClusteredLowRankSolver.jl}. Fernando, Philippe, and Dion, thank you for proofreading this thesis---not an insignificant feat.

\chapter*{List of publications}
{
  \renewcommand{\arraystretch}{1.8}
  \begin{tabular}{p{0.9\linewidth}}
    B. Bekker, O. Kuryatnikova, F. M. de Oliveira Filho, and J. C. Vera Lizcano, \textit{Optimization hierarchies for distance-avoiding sets in compact spaces}, Trans. Amer. Math. Soc. \textbf{379} (2026), no. 1, 33--70\\
    C. Bachoc, B. Bekker, P. Moustrou, and F. M. de Oliveira Filho, \textit{Obtuse almost-equiangular sets}, preprint (2025). arXiv:2504.11086\\
    B. Bekker and F. M. de Oliveira Filho, \textit{On the convergence of the $k$-point bound for topological packing graphs}, preprint (2023). arXiv:2306.02725
  \end{tabular}\par
}

\chapter*{Curriculum vitae}

{
  \renewcommand{\arraystretch}{1.5}
  \centering
  \begin{tabular}{p{3cm}p{6cm}}
    \multicolumn{2}{c}{\textsc{Abraham Johannes Franciscus Bekker}}\\
    Sep. 30, 1993  & Born in Tilburg, The Netherlands\\
    2014 & Graduated VWO Atheneum \newline Staatsexamen\\
    2014--2018 & Double BSc in Physics \& Mathematics \newline Radboud University Nijmegen\\
    2018--2021 & MSc in Mathematics \newline Radboud University Nijmegen\\
    2022--2026 & PhD in Mathematics \newline Delft University of Technology
  \end{tabular}\par
}

\end{document}